\documentclass[12pt, letterpaper]{amsart}
\usepackage{amssymb,mathtools,latexsym,graphics,enumerate,color}
\usepackage[margin=1in]{geometry}
\usepackage{slashed}
\usepackage{accents}
\usepackage{mathrsfs}
\usepackage{bbm}
\usepackage[usenames,dvipsnames]{xcolor}
\usepackage[colorlinks=true,linkcolor=Red,citecolor=Green]{hyperref}
\usepackage[depth=2]{bookmark}
\usepackage{tikz}
\usepackage{tikz-3dplot}

\title[Causal propagators for the Klein-Gordon equation]{The Klein-Gordon equation on asymptotically Minkowski
  spacetimes: causal propagators}

\author[D. Baskin]{Dean Baskin}
\email{dbaskin@math.tamu.edu}
\address{Department of Mathematics, Texas A\&M University \\ College Station, TX 77843, USA}
\author[M. Doll]{Moritz Doll}
\email{moritz.doll@unimelb.edu.au}
\address{School of Mathematics and Statistics, University of Melbourne \\ VIC 3010 \\ Australia}
\author[J. Gell-Redman]{Jesse Gell-Redman}
\email{jgell@unimelb.edu.au}
\address{School of Mathematics and Statistics, University of Melbourne \\ VIC 3010 \\ Australia}

\newtheorem{theorem}{Theorem}
\newtheorem{lemma}[theorem]{Lemma}
\newtheorem{proposition}[theorem]{Proposition}
\newtheorem{corollary}[theorem]{Corollary}
\theoremstyle{remark}
\newtheorem{definition}[theorem]{Definition}
\newtheorem{remark}[theorem]{Remark}

\numberwithin{equation}{section}
\numberwithin{theorem}{section}

\newcommand\NP{\mathrm{NP}}
\newcommand\SP{\mathrm{SP}}

\newcommand{\Rad}{\mathcal{R}}
\newcommand{\fibeq}{\operatorname{fibeq}}
\newcommand{\scrI}{\mathscr{I}}

\newcommand{\PtV}{P_{\widetilde{V}}}

\newcommand{\pa}{\partial}
\DeclarePairedDelimiter\abs{\lvert}{\rvert}
\DeclarePairedDelimiter\norm{\lVert}{\rVert}
\DeclarePairedDelimiter\ang{\langle}{\rangle}

\DeclarePairedDelimiter\curl\{\}
\newcommand{\eps}{\varepsilon}
\DeclarePairedDelimiter\set{\{}{\}}

\DeclarePairedDelimiter\ico{[}{)}
\DeclarePairedDelimiter\ioo{(}{)}

\newcommand{\cB}{\mathcal{B}}

\newcommand{\cF}{\mathcal{F}}

\newcommand{\cV}{\mathcal{V}}
\newcommand{\cX}{\mathcal{X}}
\newcommand{\cY}{\mathcal{Y}}

\newcommand{\ellvar}{\ensuremath{\boldsymbol{\ell}}}

\DeclareMathOperator{\forw}{for}
\DeclareMathOperator{\sources}{src}
\DeclareMathOperator{\sinks}{snk}

\newcommand{\CI}{C^\infty}
\newcommand{\CmI}{C^{-\infty}}
\newcommand{\CcI}{C_c^\infty}
\newcommand{\CdI}{\dot{C}^\infty}
\newcommand{\schwartz}{\mathcal{S}}
\DeclareMathOperator{\Diff}{Diff}
\DeclareMathOperator{\Sym}{Sym}

\newcommand{\LinOp}{\cB}
\newcommand{\cl}{\mathrm{cl}}
\DeclareMathOperator{\Op}{Op}
\DeclareMathOperator{\supp}{supp}
\DeclareMathOperator{\esssupp}{ess-supp}
\DeclareMathOperator{\Ell}{Ell}
\DeclareMathOperator{\WF}{WF}
\DeclareMathOperator{\Char}{Char}

\newcommand{\spec}{\sigma}

\newcommand{\scl}{\mathrm{scl}}

\newcommand\Psisclsc{\Psi_{\scl, \sc}}

\newcommand\Sobscl{H_{\scl}}
\newcommand\sclsymb[1]{\sigma_{\scl,#1}}

\renewcommand{\sc}{\mathrm{sc}}
\newcommand\poles{C}
\newcommand{\tsc}{3\sc}
\newcommand\Psitsc{{}^{\tsc}\Psi}
\newcommand\Psisc{{}^{\sc}\Psi}
\newcommand\Sobsc{H_{\sc}}
\newcommand\Sobres{\Sobsc^{-N, -M}}
\newcommand\normres[1]{\norm{#1}_{-N, -M}}
\newcommand\Ssc{{}^{\sc}S}
\newcommand\Tsc{{}^{\sc}T}
\newcommand\Lamsc{{}^{\sc}\Lambda}
\newcommand\Ttsc{{}^{\tsc}T}
\newcommand\Stsc{{}^{\tsc}\!S}
\newcommand\WFtsc{{}^{\tsc}\!\WF}
\newcommand\WFsc{{}^{\sc}\WF}
\newcommand\Elltsc{{}^{\tsc}\!\Ell}
\newcommand\Chartsc{{}^{\tsc}\Char}
\newcommand\esssupptsc{{}^{\tsc}\esssupp}
\newcommand\Wperpo{\overline{W^\perp}}
\newcommand\Tsco{{}^{\sc}\overline{T}}
\newcommand\Ttsco{{}^{\tsc}\overline{T}}
\newcommand\Diffsc{\Diff_{\sc}}
\newcommand\Difftsc{\Diff_{\tsc}}

\newcommand{\Norop}{\Psi_{\scl, \sc, \pm 1/\tau}}

\newcommand{\Tdot}{{}^{\sc}\dot{T}}
\newcommand{\Tdoto}{{}^{\sc}\dot{\overline{T}}}
\newcommand\Radtsc{{}^{\tsc}\Rad}

\newcommand{\Ctsc}{C_{\tsc}[X; C]}
\newcommand{\Ctscd}{\dot{C}_{\tsc}[X; C]}

\newcommand\prinsymb[1]{j_{\tsc,#1}}        
\newcommand\fibsymb[1]{\sigma_{\tsc,#1}}    
\newcommand\mfsymb[1]{\hat{N}_{\mf,#1}}     
\newcommand\ffsymb[1]{\hat{N}_{\ff,#1}}     
\newcommand\prinsymbz{j_{\tsc}}
\newcommand\fibsymbz{\sigma_{\tsc}}
\newcommand\mfsymbz{\hat{N}_{\mf}}
\newcommand\ffsymbz{\hat{N}_{\ff}}
\newcommand\ffsymbpmz{\hat{N}_{\ff,\pm}}
\newcommand\ffsymbpz{\hat{N}_{\ff,+}}
\newcommand\ffsymbmz{\hat{N}_{\ff,-}}

\newcommand\scprinsymb[1]{j_{\sc,#1}}        
\newcommand\scfibsymb[1]{\sigma_{\sc,#1}}    
\newcommand\scnormsymb[1]{\hat{N}_{\sc,#1}}     
\newcommand\scprinsymbz{j_{\sc}}
\newcommand\scfibsymbz{\sigma_{\sc}}
\newcommand\scnormsymbz{\hat{N}_{\sc}}

\newcommand{\mf}{\operatorname{mf}}
\newcommand{\ff}{\operatorname{ff}}
\newcommand{\fib}{\operatorname{fib}}

\newcommand{\CC}{\mathbb{C}}

\newcommand{\NN}{\mathbb{N}}

\newcommand{\RR}{\mathbb{R}}
\newcommand{\Sph}{\mathbb{S}}
\DeclareMathOperator{\id}{Id}
\DeclareMathOperator{\Id}{Id}

\renewcommand{\Im}{\operatorname{Im}}

\newcommand{\Hamf}{H_p}
\newcommand{\Hamsc}{{}^{\sc}\!H}
\newcommand{\Hamscp}{{}^{\sc}\!H_p}

\newcommand\lra{\longrightarrow}

\newcommand{\modspan}{\operatorname{span}}

\newcommand\BSspm{\mathcal{B}_{V, \pm}}
\newcommand\BSsp{\mathcal{B}_{V, +}}

\newcommand\aff{a_{\ff}}

\begin{document}

\begin{abstract}
We construct the causal (forward/backward) propagators for the
massive Klein-Gordon equation perturbed by a first order operator
which decays in space but not necessarily in time.  In particular,
we obtain global estimates for forward/backward solutions to the
inhomogeneous, perturbed Klein-Gordon equation, including in the
presence of bound states of the limiting spatial Hamiltonians.

To this end, we prove propagation of singularities estimates in
all regions of infinity (spatial, null, and causal) and use the estimates to
prove that the Klein-Gordon operator is an invertible mapping between
adapted weighted Sobolev spaces. This builds off work of
Vasy in which inverses of hyperbolic PDEs are obtained via
construction of a Fredholm mapping problem using radial points
propagation estimates. To deal with the presence of a perturbation
which persists in time, we employ a class of pseudodifferential operators first
explored in Vasy's many-body work.
\end{abstract}

\maketitle

\tableofcontents

\section{Introduction}

We consider the inhomogeneous Klein--Gordon equation on $\RR^{n +
  1} = \RR_{t} \times \RR^{n}_{z}$ with coordinates $(t, z)$:
\begin{align}\label{eq:KG inhomo intro static}
\left[  D_t^2 - (\Delta + m^2 + V) \right] u(t, z) = f(t, z)
\end{align}
where $D_t = - i \pa_t, D_{z_{j}} = - i \pa_{z_{j}}$,
$\Delta = D_z \cdot D_z$ is the positive Laplacian,
$m \in \RR, m > 0$, and $V = V(t, z)$ is a smooth, real-valued
potential function with spatial decay.  In this paper we give quantitative
estimates in phase space for the solution $u$ to \eqref{eq:KG inhomo
  intro static} in terms of the forcing $f$ in all regions of
spacetime infinity.  We use these phase space estimates to provide a
novel construction of the causal (forward/backward) propagators
$G_{+/-}$.

To give a simplified version of our results, consider
$V = V(z) \in S^{-2}(\RR^{n}_{z}; \RR)$, meaning $V$ is real-valued and decays like
$1/|z|^{2}$ with corresponding derivative estimates. If the spatial
Hamiltonian in equation~\eqref{eq:KG inhomo intro static} is positive and $\Delta + V$
has no eigenvalue or resonance at $0$, then,
letting $\chi_{< 0}(t)$ be a smooth cutoff to $t < 0$, we have, for
any $\epsilon > 0$, 
\begin{equation*}
    \norm{ \left( \ang{t, z}^{-1/2 - \epsilon} +
    \chi_{< 0}(t) \ang{t, z}^{-1/2 + \epsilon} \right) G_+ f(t, z)}_{H^1(\RR^{n + 1})}
    \le C\norm{\ang{t, z}^{1/2 + \epsilon} f(t, z)}_{L^2(\RR^{n + 1})}.
\end{equation*}
Such an estimate gives a global
spacetime weighted $H^{1}(\RR^{n + 1})$ bound for the
forward solution $u_{+} = G_{+} f$ in terms of a corresponding weighted
$L^{2}(\RR^{n + 1})$ norm of the source $f$, with weight
depending on the time direction; faster growth is allowed in the $t \to
+ \infty$ region, with slower growth as $t \to - \infty$,
relative to a $\ang{t, z}^{1/2}$ threshold.  This estimate is
refined substantially in Theorem \ref{thm:basic propagator result}, 
below, in which we prove finer mapping properties for $G_{+/-}$
in a wider class of spaces with
variable order spacetime weights and arbitrary differential
orders.  The theorem applies to a large class of Klein-Gordon
operators which are asymptotically static and non-trapping.

By a ``global'' solution to \eqref{eq:KG inhomo intro static}, we mean one for which both the forcing $f(t,z)$
and the solution $u(t,z)$ are defined on the whole of the spacetime
$\RR^{n + 1}_{t,z}$.  In this work, we study primarily the two
special solutions operators (i.e.\ propagators) to \eqref{eq:KG inhomo
  intro static} which exhibit temporal causality, namely the forward
and backward propagators $G_{+/-}$.  These define global
solutions $u_{+/-}  = G_{+/-} f$ with the property that they propagate
solutions in the forward ($+$) or backward ($-$) time directions,
meaning, for example for $G_{+}$, that
\begin{equation}\label{eq:forward solution}
  \supp f  \subset \left\{ t \ge T \right \} \implies  \supp G_{+} f  \subset \left\{ t \ge T \right \}.
\end{equation}

Our main results on the causal propagators are Theorem
\ref{thm:fredholm causal 3sc}, Theorem \ref{thm:result with bound
  states}, and Theorem~\ref{thm:result for static} below.  We give a
simplified version of these theorems now, assuming that
\[
V = V_{0}(z) + V_{1}(t, z)
\]
where $V_{0} \in S^{-1}(\RR^{n}; \RR)$ and $V_{1} \in \CI(\RR^{n+1}; \RR)$ a perturbation
which is Schwartz in space and decaying in time:
\begin{equation*}
    \abs*{\pa_{z}^{\alpha} \pa_{t}^{j} V_{1}(t, z)} \le C \ang{t}^{-1-j} \ang{z}^{-N} 
\end{equation*}
for any $j \in \mathbb{N}_{0}$, multiindex $\alpha \in \NN^n_0$ and $N \in
\RR$ and some $C$ depending on $j$ and $N$.  The assumptions in the theorem pertain to
the spectral properties of the limiting spatial Hamiltonian
\begin{equation*}
H_{V_0} = \Delta +  m^2 + V_{0}(z).
\end{equation*}
Namely, if $V \equiv V_0(z)$ is static, we assume only that $0$ is not
an eigenvalue of $H_{V_0}$ (see Theorem~\ref{thm:result for static}).
Otherwise if $V_1$ is non-zero, we
assume that $H_{V_0} \ge c > 0$, i.e.\ the limiting Hamiltonian
$H_{V_0}$ is strictly positive.  We assume in all cases that the
spectrum of $H_{V_{0}}$ is strictly absolutely continuous near $[m^2, \infty)$.
If $V_0 \in S^{-2}$, then this is implied by the absence of eigenvalues or resonances
for $\Delta +V_{0}$ at $0$.
\begin{theorem}\label{thm:basic propagator result}
  With $V$ as above
  the forward propagator exists as a mapping
  between weighted $L^2$-based Sobolev spaces: for any $\epsilon > 0$,
  $s \in \RR$,
  \[
  G_+ \colon \ang{t, z}^{- 1/2 - \epsilon} H^{s-1}(\RR^{n +
  1}_{t, z}) \lra \ang{t, z}^{1/2 + \epsilon} H^s(\RR^{n +
  1}_{t, z}).
  \]
Moreover, if a spacetime function (see
Section~\ref{sec:fred for scat} for details)
\begin{align*}
    \ellvar_+(t, z) \in S^{0}_\cl(\RR^{n + 1}_{t, z}; \RR)
\end{align*}
satisfies
\begin{enumerate}
    \item $\ellvar_+ < -1/2 $ in a neighborhood of future causal (timelike) infinity
(i.e.\ in $t \gg 0 , |z/t| < 1 + \epsilon$ for small $\epsilon > 0$),
    \item $\ellvar_+ > -1/2 $ in a neighborhood of past causal (timelike) infinity
(i.e.\ in $t \ll 0 , |z/t| < 1 + \epsilon$ for small $\epsilon > 0$),
    \item $\ellvar_+$ is monotone decreasing along future-directed null rays,
    \item \label{it:ellvar_const} $\ellvar_+$ is constant in a neighborhood of causal infinity,
\end{enumerate}
then $G_+$ extends to a bounded operator
\begin{align*}
    G_+ : \ang{t, z}^{\ellvar_+(t,z) + 1} H^{s-1}(\RR^{n+1}_{t,z}) \lra \ang{t, z}^{\ellvar_+(t,z)} H^s(\RR^{n+1}_{t,z})\,.
\end{align*}
In particular, the weight $\ellvar_+$ can be taken arbitrarily high away
from future causal infinity, yielding additional decay if $f$
has additional decay.

The analogous statements are true for the backward propagator $G_-$, where now $\ellvar_- < -1/2$ in a neighborhood of past causal infinity,
$\ellvar_- > -1/2$ in a neighborhood of future causal infinity, and monotone decreasing on past-directed null rays.
\end{theorem}

One of the main features of our approach is that it applies to
perturbations of free Klein-Gordon within a large class of operators
which limit to static potentials $V_\pm = V_\pm(z)$ as
$t \to \pm \infty$. In particular, there is no need for the limiting
potentials $V_{\pm}$ to be equal to each other as they are in Theorem
\ref{thm:basic propagator result}.  Though we allow more general $V$
below, to fix the idea we begin by assuming that $V = V(t, z)$ is a
smooth potential function with rapid spatial decay and smooth time
dependence, approaching a static potential as $t \to \pm \infty$:
\[
V(t, z) \in \CI(\RR_t; \schwartz(\RR^n_z)),
\mbox{ and }  V(t, z) - V_\pm(z)  = O( |t|^{-1})  \mbox{ as } t \to \pm \infty,
\]
and corresponding derivative estimates.
The influence
of the perturbation $V$ on the analysis being the main object of
study, we define the operator with explicit dependence on $V$:
\begin{align}\label{eq:def of PV}
    P_V \coloneqq D_t^2 - H_V = D_t^2 - (\Delta + m^2 + V)\,.
\end{align}
In fact, we can allow $V$ to be a differential operator of order $1$
with symbolic behavior in $z$.  Below in Section \ref{sec:V assump}, we
construct a compactification of spacetime on which $V$ is smooth, and this global
smoothness condition is more general than the assumptions above.  We can
also allow $\square_{g_{\mathrm{mink}}} = D_t^2 - \Delta$ to be
replaced by the d'Alembertian $\square_{g}$ of a non-trapping,
globally hyperbolic, asymptotically Minkowski metric $g$ as described
in prior work on the massless wave equation~\cite{BVW15}.

Thanks to a famous result of Klainerman~\cite{Kla1993}, the forward
solution $u_{+} = G_{+} f$ to $P_{V} u_{+} = f$ with a sufficiently
smooth and decaying source $f$ is known to lie in
$t^{-n/2} L^\infty(\RR^{n + 1}) \subset \ang{t, z}^{1/2 + 0}L^2(\RR^{n
  + 1})$ near future causal infinity.  Our results below refine the
$L^{2}$ statement, in that they allow us also to localize in
frequency space in a neighborhood of the future radial set, i.e.\ the
limit locus of bicharacteristic rays, a phase space subset over future
null infinity defined below.   One upshot is that the
  frequency-localization of the solution to the radial set effectively
  carries all of the asymptotic data of $G_{+}f$ as $t\to +\infty$.
  In other words, there are frequency localizers $Q_{rad}$ which cut
  off in phase space to the radial set, such that, for
  $f \in \CcI(\RR^{n + 1})$, 
\begin{equation}\label{eq:Qrad}
(\Id - Q_{rad}) G_{+} f \in \schwartz(\RR^{n + 1}),\ \mbox{ while
  }  Q_{rad}
G_{+} f \in \langle t, z \rangle^{1/2 + \epsilon} H^{s}(\RR^{n
  + 1})
\end{equation}
for any $s \in \RR$ and any $\epsilon > 0$.  We prove our result using
propagation of singularities and radial points estimates, as we
describe in detail shortly.

Propagator estimates of this type, i.e.\ global spacetime estimates
using $L^{2}$-based, weighted Sobolev spaces, were established by Vasy
for a large class of ``scattering'' operators, including the free
Klein-Gordon operator \cite{V18}.  They are a combination of principal
type propagation of singularities estimates, radial points estimates,
and microlocal elliptic estimates proven on non-compact regions of
phase space.  More accurately, they are proven up to and including
infinity, meaning in phase space regions on a compactified spacetime.
Propagation estimates in the scattering setting were first established
by Melrose~\cite{Melrose94} in the case of the Helmholtz equation,
generalizing H\"ormander's propagation of singularities
theorem~\cite{Hormander71} to scattering operators on non-compact
scattering spaces.  Melrose's paper~\cite{Melrose94} further
introduced the radial point\footnote{As far as the authors know,
  radial points as such were first studied in \cite{GuilSchae}}
propagation estimates (also in the scattering setting).  For the free
Klein-Gordon equation (and scattering perturbations thereof) one can
use the phase space picture of Melrose/Vasy, based on the radial
compactification $X = \overline{\RR^{n + 1}_{t, z}}$, a compact space,
diffeomorphic to a closed ball, whose boundary points are the limits
of geodesic rays.  This compactification, illustrated in
Figure~\ref{fig:radialset}, includes spatial and causal infinities as
open subsets of the boundary (the shaded piece of the boundary in
Figure~\ref{fig:radialset} also represents causal infinity), while
future null infinity and past null infinity are compressed to
codimension one submanifolds of the boundary.  The relevant phase
space for the free Klein-Gordon, the scattering cotangent space
$\Tsc^* X$, has $X$ as the underlying spacetime and the standard
$\tau, \zeta$ (dual to $t, z$) momenta.  This is described in detail
in Section \ref{sec:model-case}.

To understand the phase space nature of the estimates, consider
the family of solutions to the free Klein-Gordon equation $P_0 u = 0$ obtained
directly by Fourier transformation in the spatial variable, namely, for
$g_\pm \in \CcI(\RR^n)$, with
\[
\hat{u}(t, \zeta) = \sum_{\pm} e^{\pm i t \sqrt{|\zeta|^2 + m^2}} g_\pm(\zeta),
\]
let $u(t, z) = \cF_{z \to \zeta}^{-1} \hat u (t, \zeta)$.  The
infinite dimensional space of such $u$, in regions of the form $|z| /t
\le c < 1$, behave asymptotically as
\[
u = t^{-n/2} \left( a_+(z/t) e^{- i m \sqrt{t^2 - |z|^2}} + b_+(z/t) e^{ i m \sqrt{t^2 - |z|^2}} \right) ( 1 + O(1/t))
\]
as $t \to + \infty$ with $a_+, b_+$ smooth functions easily computed
in terms of the inverse Fourier transforms of the coefficients $g_{\pm}$.  Here $y = z/t$
parametrizes future causal infinity.  A similar
expression holds as $t \to - \infty$.  If one denotes the phase function
$\phi_\pm(t, z) =\pm  m \sqrt{t^2 - |z|^2}$, then
the radial set over future causal infinity is the subset of
phase space defined by
\begin{equation*}
  \tau = D_t \phi_\pm = \frac{\pm m t}{\sqrt{t^2 - |z|^2}}, \quad \zeta =
   D_z \phi_\pm = \frac{\mp m z}{\sqrt{t^2 - |z|^2}}.
\end{equation*}
This future radial set is the union of two smooth manifolds in
$\Tsc^*_{\iota^+} X$ which extend smoothly out to the compactified
boundary of the fibers in $\Tsco^* X$.  Here $\iota^+$ denotes future
causal (timelike) infinity.  The $Q_{rad}$ in \eqref{eq:Qrad} can be
taken to be the quantization of a cutoff function supported near this
radial set, and, for these simple solutions, the conclusion of
\eqref{eq:Qrad} follows from stationary phase.

The principal novelty of the current work lies in its treatment,
via microlocal methods, of perturbations $V$ that persist in time, in
particular when there is asymptotic time dependence and separation of
variables between time and space cannot be employed directly.  If one thinks of the solutions
of \eqref{eq:KG inhomo intro static} in terms of their behavior along
classical trajectories, then a geodesic
\begin{align*}
(t_0 + s\tau, z_0 - s \zeta) \mbox{ with } \zeta \not = 0
\end{align*}
exits every compact spatial set, and thus as
$s \to \pm \infty$, one expects solutions to behave asymptotically
like those for free Klein-Gordon.  However, on timelike trajectories
with zero 
spatial momentum, the nature of the spatial Hamiltonians
$\Delta  + m^2+ V_\pm$ influences the asymptotic behavior
substantially.  The limits of such rays form two special points at
causal infinity, one in the future and one in the past, which we call
the ``north pole'', $\NP$, and the ``south pole'', $\SP$, and the
lion's share of our work below is in proving estimates near $\NP$ and $\SP$.

In particular, for non-zero $V_0$, the operator $P_V$ does not lie
in the class of scattering operators $\Diffsc(\RR^{n + 1})$
defined by Melrose \cite{Melrose94}.  However, it does lie in a class
of differential operators studied earlier by Vasy~\cite{V2000,Vasy13}
in his treatment of many-body Hamiltonians, called ``3-body'' or
``many-body'' or, as we write more frequently below,
$\tsc$-operators. This family of operators on $\RR^N$ includes not
only many-body Hamiltonians, but also operators which decompose
analogously near a family of collision planes, modeled as submanifolds
$C$ of infinity.  In our case, the picture is comparatively simple; we
have just one analogue of a collision plane and it is the straight line $z = 0$ in
$\RR^{n + 1}$. For $P_V$, the relevant $C$ that arises is simply the
two poles,
\[
C = \NP \cup \SP,
\]
Those two points lying in causal infinity are precisely the points
where $P_V$ fails to be a scattering operator.  For a description of
how these operators, which we follow Vasy in denoting by
$\Difftsc(\RR^{n+1})$, arise in our setting, see Section
\ref{sec:operator-p_v}.  Correspondingly, we work
  on the resolved space $[X; \poles]$ obtained by blowing up the two
  poles.  This blow-up introduces two new boundary
  hypersurfaces, each of which is essentially a copy of $\RR^n_z$,
and is in fact the limit locus of geodesics with zero spatial
momentum, as depicted in Figure \ref{fig:X and XC}.

The result of this $\tsc$-analysis is a family of non-elliptic Fredholm
problems for $P_V$, analogous to the Fredholm problems on
asymptotically hyperbolic spaces established by Vasy \cite{Vasy13}.
We prove global estimates on
a priori spaces of distributions which satisfy prescribed asymptotics, above threshold
near past causal infinity and below threshold
near future causal infinity.  We let
\begin{align*}
    \Sobsc^{s,\ellvar}(\RR^{n + 1}) \coloneqq \{ u \in \schwartz'(\RR^{n + 1}) :
  \ang{t, z}^{\ellvar} u \in H^s(\RR^{n + 1})\}\,, \quad \norm{u}_{\Sobsc^{s,\ellvar}} =
  \norm{u}_{s,\ellvar} \coloneqq \norm{\ang{t, z}^{\ellvar} u}_{H^s}\,.
\end{align*}
denote the relevant weighted Sobolev spaces, where $\ellvar(t, z)$ is a smooth spacetime weight and
$\Sobsc^{s,0}(\RR^{n + 1}) = H^s(\RR^{n + 1})$ is the standard Sobolev space of order
$s$ on $\RR^{n + 1}$.
The a priori spaces are defined as
\begin{gather}\label{eq:first X spaces}
    \cX^{s,\ellvar} \coloneqq \{u \in \Sobsc^{s,\ellvar}(\RR^{n + 1}) \colon P_V u \in \Sobsc^{s-1, \ellvar+1}(\RR^{n + 1})\}\,,\quad
    \cY^{s,\ellvar} \coloneqq \Sobsc^{s,\ellvar}(\RR^{n + 1})
    \intertext{with}
    \norm{u}_{\cX^{s,\ellvar}}^2 = \norm{u}_{s,\ellvar}^2 + \norm{P_V u}_{s-1, \ellvar+1}^2\,.
\end{gather}
In general, the weight $\ellvar$ can have phase space dependence,
although for time-persistent $V$ it must be constant in a neighborhood
of $\poles$.  Genuine phase space dependence is permitted but not needed in the
present paper, though it will play a role in future work.  Our construction of the propagators works by first
establishing a Fredholm mapping property, and then proving
invertibility of this mapping by proving that the kernel and cokernel are trivial.
The Fredholm property holds more
generally than the invertibility, and in that case there are still
causal propagators in the sense that there is a generalized inverse
for any Fredholm map.

Given $V$, to obtain a Fredholm result (from which invertibility may
or may not be concluded) we do \emph{not} need to assume that the
limiting Hamiltonians $H_{V_\pm}$ are positive.  Instead, we assume
only that
\begin{enumerate}
    \item $H_{V_\pm} = \Delta + m^2 + V_\pm$
        has purely absolutely continuous spectrum near $[m^2, \infty)$, i.e.,\ the 
        absolutely continuous spectrum $[m^2, \infty)$ of $H_{V_\pm}$ is disjoint from the singular
        and point spectra, and
    \item $H_{V_\pm}$ has no eigenvalue at $0$.
\end{enumerate}
The second assumption avoids the presence of linearly growing modes.
In this case we have that $P_V$ is a Fredholm operator between
appropriate spaces.
\begin{theorem}\label{thm:fredholm}
  Let $V$ be any of the asymptotically static potentials described in
  Section~\ref{sec:V assump} satisfying the assumptions just
  described.
  If $\ellvar_{\pm}$ is a spacetime weight as
  in Theorem \ref{thm:basic propagator result}, then
    \begin{align*}
        P_V \colon \cX^{s,\ellvar_{\pm}} \lra \cY^{s-1,\ellvar_{\pm}+1}
    \end{align*}
    is a Fredholm operator.
\end{theorem}

Analyzing $P_V$ as a $\tsc$-operator allows us to prove estimates
for $P_V u = f$ near $\poles$.  In accordance with Vasy's
results for many-body Hamiltonians, estimates at $C$ are proven at each
level of the temporal momentum component $\tau$.  This is accomplished
in particular by analysis of the so-called \emph{indicial operator} $\ffsymbz(P_V)$.
Near $\NP$,
the indicial operator is the Fourier transform in time of the limiting operator:
\begin{align*}
    \ffsymbz(P_V)(\tau) = \tau^2 - \Delta_z - m^2 - V_+ = \tau^2 - H_{V_+}\,,
\end{align*}
with an analogous indicial operator at $\SP$.
This operator's behavior has three distinct types:
\begin{enumerate}
    \item if $|\tau| > m$ then $\tau^2$ lies in the continuous spectrum of $H_{V_+}$,
    \item if $|\tau| < m$ then the operator $\ffsymbz(P_V)(\tau)$ is elliptic on $\RR^n$ in the scattering sense, and
    \item the borderline cases $\tau = \pm m$, which is where the radial set lies over $\NP$.
\end{enumerate}
In each of these three cases, we prove a different type of
estimate.  The three estimates are analogues of (1) principal
type-propagation, (2) elliptic estimates, and (3) radial points
estimates.
The
difference between our estimates and standard scattering propagation
estimates
is that over $\poles$, frequency localization is local in $\tau$
but global in the spatial momentum $\zeta$.  As a result, the
principal type propagation and radial points estimates always have assumptions on the
flow on all bicharacteristics passing over $C$ at the relevant
$\tau$-level.

In the case $|\tau_0| > m$, we prove a principal type propagation
estimate of the form ``if nothing comes in then nothing goes out'' 
that appears in Vasy's original work, described there using broken
geodesics.  In other words, if $P_V u$ lies in some
$\Sobsc^{s - 1, \ell + 1}(\RR^{n + 1})$ when microlocalized near
$\tau = \tau_0$, and $u$ lies in $\Sobsc^{s, \ell}(\RR^{n + 1})$ on all
bicharacteristics which flow into $\NP$ at $\tau = \tau_0$, then $u$
itself lies in $\Sobsc^{s, \ell}(\RR^{n + 1})$ at $\tau = \tau_0$ over
$\NP$ and satisfies a corresponding estimate.  This estimate is the
content of Proposition \ref{prop:propagation_localized}.

At $\tau = \pm m$, we prove above and below threshold radial point estimates.
Below the threshold decay rate $\ell < - 1/2$, the estimates mimic principal type estimates,
namely if $P_V u \in \Sobsc^{s - 1, \ell + 1}$ microlocally near $\NP$ and $\tau = m$ and $u \in \Sobsc^{s, \ell}$ 
on all bicharacteristics that flow into $\NP$ at $\tau = m$, then $u$ is microlocally in $\Sobsc^{s,\ell}$ at
$\NP$ and $\tau = m$.
On the other hand for $\ell > -1/2$, the above threshold estimate is similar to an elliptic estimate;
if $P_V u \in \Sobsc^{s - 1, \ell + 1}$ microlocally near $\NP$ and $\tau = m$ and $u \in \Sobsc^{s,\ell'}$
for $\ell' \in (-1/2,\ell)$, then $u \in \Sobsc^{s,\ell}$ near $\NP$ and $\tau = m$.
We refer to Proposition~\ref{prop:tsc_radial_above} and Proposition~\ref{prop:tsc_radial_below} for details.



The scattering calculus provides a mechanism for frequency localization
in $\tau$ and $\zeta$ uniformly up to infinity.
In particular, the localizer to the radial set $Q_{rad}$ can be taken to be a scattering operator of order $(0,0)$.
However, in the $\tsc$-calculus, operators obtained by quantization of $\zeta$-dependent symbols do not
have good commutation relations with $P_V$ near $\poles$.
Therefore, in Vasy's
treatment, frequency localizers include functions of the many-body
Hamiltonian $H$, i.e.,\ $\phi(H)$, where $\phi \in \CcI(\RR)$.  In that case, because $H$ is
elliptic, $\phi(H)$ is a smoothing pseudodifferential operator.
However, in our case, functions $\phi(P_V)$ are not well-behaved
pseudodifferential operators, and we must first compose $P_{V}$ with
an invertible, globally elliptic $\tsc$-operator and then apply the
functional calculus to obtain a pseudodifferential localizer to the characteristic set.
We use, for $E \ge 0$ sufficiently large,
\begin{equation*}
    G_{\psi} \coloneqq \psi\left( (D_t^2 + H_{V_0} + E)^{-1} P_{V_0} \right)\,.
\end{equation*}
which we show lies in $\Psitsc^{0,0}(\RR^{n + 1})$.
The fact that $G_\psi$ is not smoothing corresponds to the non-compactness of the characteristic set of $P_0$.
We have the commutation relation
\[
[G_\psi, P_{V_{0}}] = 0
\]
in our commutator construction and for $P_V$ this commutation
holds to leading order.
Frequency localizers over $\poles$
will be of the form $Q G_\psi$ where $Q$ localizes to a $\tau$-level
over $C$ (see Section \ref{sec:full func calc sec}).

Another novelty of our work is the treatment of large $\tau$ at $\poles$.
For many-body Hamiltonians, non-compact regions of the non-interacting
dual variable lie in the elliptic region, and are therefore not of
direct interest.  (It should be noted that a global Fredholm framework
for many-body operators would still need to address these elliptic
regions.)  In contrast, for the $\tsc$ formulation of Klein-Gordon,
large $\tau$ lie in the non-radial part of the characteristic set and
therefore exhibit principal type propagation.  To establish
propagation of singularities estimates for $\tau = \pm \infty$, we
recognize the indicial family as a semiclassical scattering
differential operator of the form introduced by Vasy--Zworski~\cite{VaZw2000},
\begin{equation*}
  \ffsymbz(P_V)(\tau) \in \Psisclsc^{2,0,2}(\RR^n),
\end{equation*}
where the semiclassical parameter is $h = \pm 1/ \tau$, say as
$\tau \to \pm \infty$.  An attractive picture emerges, in which the
semiclassical principal symbol of $\ffsymbz(P_V)(\tau)$ is exactly the
restriction of the scattering principal symbol of $P_V$ restricted to
an appropriate hypersurface of a fiber compactified phase space
(see Section~\ref{sec:norm fam}).

Similar to Melrose~\cite{Melrose94} and Vasy~\cite{V18}, we exploit the global structure
of the Hamiltonian flow in which all bicharacteristics flow from two
components of the radial set to the other two; the former act as
global sources for the flow and the latter act as sinks.  On the
characteristic set away from the radial set, we prove real principal
type propagation estimates along bicharacteristic rays form the
Hamiltonian flow.  For the free Klein-Gordon operator, the Hamilton
vector field $H_{p} = \tau \partial_{t} - \zeta \cdot \partial_{z}$
rescales to a smooth vector field on a the fiber compactification
$\Tsco^{*} X$ of the scattering cotangent bundle.  There it induces a
smooth extension of the bicharacteristic flow which preserves the
characteristic set.
The propagator is the inverse of a Fredholm problem
constructed using radial points estimates, as detailed below in
Section \ref{sec:fred for scat} and \ref{sec:scat perturb}.

As in Vasy \cite{V2000}, the flow we use to analyze $P_{V}$ is the
same flow as that for $P_{0}$.  Away from $\poles$, this makes sense
immediately since the assumptions on $V$ make it so that $P_{V}$ has
the same scattering principal and subprincipal symbol as $P_{0}$.  At
$\poles$, as in Vasy~\cite{V2000}, we use that the functional calculus
localizers to the characteristic set can be approximated by those same
functions of the corresponding localizers for the free operator.
Specifically, we use that, under the given assumptions, with
$G_{\psi}$ the localizer for $P_{V}$ and $G_{\psi,0}$ the free
localizer, that
\[
\ffsymbz(G_{\psi})(\tau) \mbox{ and } \ffsymbz(G_{\psi,0})(\tau) \mbox{ have the same semiclassical
  principal symbol,}
\]
which fits nicely below thanks to the semiclassical formulation of the
indicial family.  This is discussed in Section~\ref{sec:full func
  calc sec}, together with the basic construction of the commutants 
that go into the positive commutator argument.

The mathematical study of the long-time behavior of solutions to
massive wave equations is extensive and dates back at least to the
pioneering work of Morawetz--Strauss~\cite{MoSt1972}, in which the
authors found one of the first decay results for the Klein--Gordon
equation.  Quite a bit more work on the spectral and scattering theory
of Klein--Gordon equations ensued, including the works of
Lundberg~\cite{Lundberg1973} and Weder~\cite{Weder1978}.  Klainerman's
use of energy techniques~\cite{Kla1993}, discussed above, is perhaps
the result most directly related to the current work.  More recent
work has focused on energy decay (such as the works by
Kopylova~\cite{Kopylova2018} or by 
Komech and Kopylova~\cite{KoKo2010b,KoKo2010a} or those described in the survey
article of Kopylova~\cite{Kopylova-survey}), on Strichartz estimates
(such as those by Kubo--Lucente~\cite{KuLu2007}), or on asymptotics
for related equations (such as the work of
Bejenaru--Herr~\cite{BeHe2015} for the Dirac equation).

The Fredholm approach to the construction of resolvents and
propagators for non-elliptic operators using radial points estimates
and anisotropic spaces is due to Vasy \cite{HaVa2015,Vasy13,Vasy_asymp_hyper_high-energy}, while
the radial points estimates used in his construction are due to
Melrose \cite{HMV2008, Melrose94}.  Adaptation of the method to more general non-elliptic
scattering operators is due to Vasy \cite{V18}.  G{\'e}rard--Wrochna~\cite{GeWr19} used the method to
construct the Feynman propagator for the Klein--Gordon equation on
asymptotically Minkowski spaces.

There is closely related work on the wave equation on Lorentzian
scattering manifolds due to Baskin--Vasy--Wunsch~\cite{BVW15, BVW18} and Hintz--Vasy
\cite{HintzVasy2015}.  These papers prove not only linear
invertibility properties for the wave equation, but also semilinear
results using weighted global spacetime estimates akin to those above,
only there the relevant estimates are $\mathrm{b}$-spaces.
Subsequent work of Hintz and Vasy, which use in particular related microlocal methods
including radial points estimates on non-compact spacetimes,
establishes stability properties in mathematical GR
\cite{HaefnerHintzVasy, HintzVasy2018}.  Use of the radial points
estimates in semiclassical analysis to study resonances on
asymptotically hyperbolic manifolds and
Pollicot-Ruelle resonances,
\cite{DatchevDyatlov, DatchevDyatlovZworski}.  Many of these approaches use anisotropic
spaces, as they allow for threshold conditions to vary between components
of the radial set \cite{FaureSjostrand}.

Ours is not the only recent work in microlocal analysis on hyperbolic
PDE which uses a many-body approach; Hintz has used $\tsc$-operators,
and a related class of $3\mathrm{b}$-pseudo\-differential operators
\cite{Hintz2023Gluing,Hintz2023Flat,Hintz2023,Hintz2024,HintzVasy2023}.
Also recently, Sussman~\cite{Sussman2023} has analyzed the asymptotics of solutions to
Klein-Gordon near null infinity. This analysis includes a substantially
more detailed study of solutions near null infinity, where we use only
propagation of singularities estimates.

\subsection*{Outline of the paper}
The paper is organized as follows.  In Section \ref{sec:model-case},
we analyze the free Klein-Gordon operator $P_0$, which serves as an
instructive model.
In doing so, we
recall the relevant features of the scattering calculus, including the
global structure of the bicharacteristic flow of $P_0$ on the
scattering cotangent bundle.  We then prove Theorem \ref{thm:basic
  propagator result} for $P_0$ and for scattering perturbations of
$P_0$.  This provides both an outline for our general approach to
analysis of $P_V$ and the actual estimates that will be used for
$P_{V}$ away from $\poles$.  In
Section \ref{sec:operator-p_v} we discuss how $P_V$ is not a
scattering operator in general; it is instead a $\tsc$-differential
operator, a class which we define and discuss the basic properties of,
including the indicial operator of $P_V$.  In Section \ref{sec:3body
  klein} we display the basic properties of $\tsc$-pseudodifferential
operators, including their symbol mappings and quantization,
commutators, wavefront sets and elliptic sets.  In Section
\ref{sec:reduced elliptic} in particular we develop an elliptic
theorem of $\tsc$-operators, including fiber infinity.  Section
\ref{sec:full func calc sec} extends the work of Section
\ref{sec:3body klein} to prove functional calculus statements for
non-elliptic $\tsc$-operators and we compute the symbols of some
commutators which are used in the propagation proofs.
In Section~\ref{sec:3sc_propagation} and Section~\ref{sec:3sc_radial}
we prove the
propagation estimates over the poles $C$
(Section~\ref{sec:3sc_propagation} proves the analogue of
H{\"o}rmander's theorem and Section~\ref{sec:3sc_radial} establishes
the radial point estimates).  Finally, in Section
\ref{sec:3sc_propagators}, we put the foregoing work together to prove
the existence of propagators using the Fredholm framework discussed
above.  In particular the main theorems are proven in this final numbered
section.  We finally include an index of notation after the main sections of
the paper.

\subsection*{Acknowledgements}
This research was supported in part by the Australian Research Council
grant DP210103242 (JGR, MD) and National Science Foundation grant
DMS-1654056 (DB).  We acknowledge the support of MATRIX through the
program ``Hyperbolic PDEs and Nonlinear Evolution Problems'' September
18-29, 2023 and DB and MD were supported by the ESI through the
program ``Spectral Theory and Mathematical Relativity'' June 5-July
28, 2023.

Moreover, we are grateful to Andrew Hassell, Peter Hintz, Yilin Ma, Ethan Sussman, and Andr\'as Vasy for valuable discussions.

\section{The model case}
\label{sec:model-case}

In this section we prove Theorem~\ref{thm:basic propagator result} for
the model operator
\begin{align}\label{def:P0}
    P_0 \coloneqq D_t^2 - (\Delta + m^2).
\end{align}
The majority of this work was described by Vasy~\cite{V18}; we follow
that development and then at the end of the section we describe
related settings in which the same proof applies with minimal
modifications.

In the case of $P_0$, the adapted Sobolev spaces are
\begin{align}\label{eq:P0_Sobolev_spaces}
  \cX^{s,\ellvar} = \curl{ u \in \Sobsc^{s,\ellvar} : P_{0} u \in \Sobsc^{s-1,\ellvar+1}}\,,\quad
  \cY^{s,\ellvar} = \Sobsc^{s,\ellvar}\,,
\end{align}
for smooth function $\ellvar = \ellvar(t, z) \in S_{\cl}^{0}(\mathbb{R}^{n + 1})$.
We prove the invertibility of $P_0$ as a bounded operator between $\cX^{s,\ellvar}$ and $\cY^{s-1, \ellvar+1}$
for suitable $\ellvar$.
\begin{theorem}
  \label{thm:sec2thm}
  If $s \in \RR$ and $\ellvar_{+}$ a forward weight,
  \begin{equation}
    \label{eq:model mapping}
    P_{0} : \cX^{s, \ellvar_{+}} \lra \cY^{s-1, \ellvar_{+}+1}
  \end{equation}
  is an isomorphism. Its inverse is the forward propagator.
  The same is true if $\ellvar_{+}$ is replaced by a backward weight $\ellvar_{-}$, in
  which case the inverse is the backward propagator.
\end{theorem}

\subsection{Outline}
\label{sec:scattering-outline}

As described in the introduction, the main step in the proof of
Theorem~\ref{thm:sec2thm} is to show that $P_{0}$ and $P_{0}^{*}$ are
Fredholm operators between $\cX^{s,\ellvar}$ and $\cY^{s-1,\ellvar+1}$.
In particular, $P_{0} : \cX^{s,\ellvar} \lra \cY^{s-1,\ellvar+1}$
is Fredholm provided we can establish the following
two global estimates:
\begin{align*}
  \norm{u}_{s,\ellvar} &\leq C\left( \norm{P_{0}u}_{s-1,\ellvar+1} + \normres{u}\right), \\
  \norm{u}_{1-s,-1-\ellvar} &\leq C\left( \norm{P_{0}^{*}u}_{-s, -\ellvar} + \normres{u}\right),
\end{align*}
where $M,N$ are sufficiently large that $\Sobres \hookrightarrow
\Sobsc^{s,\ellvar}$ and $\Sobres \hookrightarrow
\Sobsc^{1-s,-1-\ellvar}$ are compact.  These will hold only for
suitable $\ellvar$, and the outline of the estimates here motivates
the choice of the forward and backward weights $\ellvar_{\pm}$ in
Section \ref{sec:fred for scat}. Although for the
free case we have $P_{0}^{*} = P_{0}$, it is useful to distinguish
these as the second estimate above is used to analyze the cokernel of
the Fredholm maps for $P_{0}$.

The global estimates for $P_{0}$ and $P_{0}^{*}$ follow from
several types of microlocal estimates. 
Indeed, via a microlocal
partition of unity, we decompose phase space (the scattering cotangent
bundle, described below) into open sets $U$ where
\begin{enumerate}
\item $U$ localizes to the elliptic set of $P_{0}$.  
\item $U$ is a neighborhood of a point in the characteristic set of
  $P_{0}$ at which the Hamilton vector field is non-radial.  Here we bound $u$ by $P_0 u$ and $E u$, where $E$
  microlocalizes to a neighborhood that is in the past of the
  Hamiltonian flow, see
  Proposition~\ref{prop:scattering-propagation-variable}.  This is the
  typical microlocal propagation of singularities estimate.  
\item $U$ is a neighborhood of a point in the characteristic set of
  $P_{0}$ at which the Hamilton vector field is radial.
\end{enumerate}

We prove estimates in the corresponding regions. 
\begin{enumerate}
\item   We have microlocal elliptic estimates in Propositions \ref{prop:scat
    elliptic} and \ref{prop:ellip-specific-scattering-variable}

\item   In a neighborhood of a point in the characteristic set and away from
  the radial set, we bound $u$ by $P_0 u$ and $E u$, where $E$
  microlocalizes to a neighborhood that is in the past of the
  Hamiltonian flow, see
  Proposition~\ref{prop:scattering-propagation-variable}.  This is the
  typical microlocal propagation of singularities estimate.

\item   Near the radial set, we have two different types of estimates: below
  the decay rate $\ellvar < -1/2$, we have an estimate that is similar to
  the propagation of singularities estimate, see
  Proposition~\ref{prop:localized below thresh}. Whereas for
  $\ellvar > -1/2$, we have an estimate that mimics the elliptic estimate, but
  we need to assume that $u$ is a priori above the $-1/2$ threshold,
  see Proposition~\ref{prop:localized above thresh}.
\end{enumerate}

Given these four estimates -- elliptic, principal type, and above and
below threshold radial points estimates -- we obtain the desired Fredholm estimate
for $P_{0}$ (and, by working with adjoints, for $P_{0}^{*}$) in the
space $\Sobsc^{s,\ellvar}$ provided that $\ellvar > -1/2$ (to apply the above
threshold estimate) and $\ellvar < -1/2$ (to apply the below threshold
estimate).  This motivates the use of a variable weight described
below in Section~\ref{sec:scat prop} and Section~\ref{sec:fred for scat}.

In microlocal regions near the radial set we always assume that
$\ellvar$ is constant, and emphasize this by denoting it by $\ell$. 

The rest of this section is devoted stating the estimates together with a sketch of the proofs and then the proof of Theorem~\ref{thm:sec2thm}.

\subsection{The scattering calculus}
\label{sec:scattering-calc}

The operator $P_{0}$ is an element of Melrose's scattering
calculus~\cite{Melrose94}, which quantizes functions of the standard
translation-invariant vector fields on $\RR^{n+1}$.  Indeed, the
\emph{scattering differential operators} on a vector space
$\RR^{N}$ are given by
\begin{equation}
  \label{eq:sample-scattering-diff-op}
  L \in \Diffsc^{m}(\RR^{N}_{w}) \iff L = \sum_{\abs{\alpha}\leq m}
  a_{\alpha}(w) D_{w}^{\alpha}, \quad a_{\alpha} \in S^{0}_{\cl}(\RR^{N}),
\end{equation}
where $S^{0}_{\cl}(\RR^{N})$ is the space of classical symbols of
order zero on $\RR^{N}$.  Observe that
$P_{0} = D_{t}^{2} - \left( \Delta + m^{2}\right) \in
\Diffsc^{2}(\RR^{n+1})$.  We now describe some relevant features of
the scattering calculus.

Scattering operators are most easily understood as operators 
on a compactified space.  Here we use the notation
$X = \overline{\RR^{n+1}}$ for this space, which is the
\textit{simultaneous radial compactification of spacetime}.  In
particular, points on the boundary of $X$ (below simply called
``infinity'') represent the loci of endpoints of geodesics of
arbitrary type (timelike/spacelike/null), where asymptotically parallel geodesics
limit to the same point at infinity.
\begin{align}\label{eq:def of X}
    X \coloneqq \overline{\RR^{n+1}_{t, z}} = \Sph^{n+1}_+\,.
\end{align}
Here $X$ is simply a hemisphere of the unit sphere, $\Sph^{n+1}_+$;
it is a compactification of $\RR^{n + 1}$ explicitly via the inclusion
\begin{align*}
    (t, z) \mapsto \frac{(1, t, z)}{\ang{t, z}}, \qquad
    \ang{t, z} = (1 + t^2 + |z|^2)^{1/2}
\end{align*}
The boundary of $\Sph^{n+1}_{+}$ is diffeomorphic to $\Sph^{n}$ and is
given by
\begin{align}\label{eq:paX}
    \pa \Sph^{n+1}_+ = \curl{(0,t, z) \colon t^2 + \abs{z}^2 = 1}\,.
\end{align}
In particular, a global boundary defining function on $X$ is given by
$\left(1 + t^2 + \abs{z}^{2}\right)^{-1/2}$.  In the region
where $y = z/t$ is bounded, we may use $x = 1/t$ and $y$ as local
coordinates.  

The smooth structure on $X$ (as a manifold with boundary) distinguishes
the classical symbols of order zero on $\RR^{n+1}$ as smooth, i.e., 
\[
\CI(X) = S^0_{\cl}(\RR^{n + 1}),
\]
and thus one can rephrase the definition of $\Diffsc^{m}(\RR^{n+1})$ by
demanding coefficients $a_{\alpha}$ in
\eqref{eq:sample-scattering-diff-op} satisfy
\[
a_{\alpha} \in \CI(X).
\]
We also write
\[
\Diffsc^{m}(X) = \Diffsc^{m}(\RR^{n + 1}),
\]
and
\begin{equation*}
    \Diffsc^{m,r}(X) = \ang{t,z}^{r}\Diffsc^{m}(X).
\end{equation*}

The space $\Diffsc^{m}(X)$ is the universal enveloping algebra of
the space of \emph{scattering vector fields}
$\cV_{\sc}(X) \coloneqq x\cV_b(X)$, where $x$ is a total
boundary defining function for $X$ and $\cV_b(X)$ is the
space of vector fields tangent to $\pa X$ \cite{Melrose:APS, Melrose95}.  The space of scattering
vector fields is independent of the specific choice of boundary
defining function and forms a Lie algebra.  The scattering tangent
bundle $\Tsc X$ is the vector bundle whose sections are scattering
vector fields.  If $(x,y) \in \RR_+ \times \RR^n$ are local
coordinates on $X$ with $x$ a boundary defining function, then
$\Tsc X$ is locally spanned over $C^{\infty}(X)$ by
\begin{align*}
    \curl{ x^2 \partial_x, x \partial_y}\,.
\end{align*}
The dual bundle is the \emph{scattering cotangent bundle} $\Tsc^* X$ and it is locally given by
\begin{align*}
    \curl*{ \frac{dx}{x^2}\,, \frac{dy}{x}}\,.
\end{align*}

The (total) symbol of a scattering differential operator is best
understood as a function on the doubly compactified phase space
\begin{equation}
  \label{eq:scattering phase space}
  \Tsco^* X = X \times \overline{\RR^{n + 1}_{\tau, \zeta}}
\end{equation}
where $\tau, \zeta$ are dual to
$t, z$ respectively, and $\overline{\RR^{n + 1}_{\tau, \zeta}}$
is the radial compactification of the momentum factor, or, as we
will refer to it, the ``fiber'', as it is the fiber of the fibration
$\Tsco^* X \longrightarrow X$.

We define, for $m, r \in \RR$,
\begin{equation}
    \Ssc^{m,r}(X) = \ang{t,z}^r \ang{\tau, \zeta}^m \CI(X \times
\overline{\RR^{n + 1}}) \label{eq:scattering symbols}
\end{equation}
to be the space of classical scattering symbols of order
$m,r$, and
\[
\Psisc^{m,r}(\RR^{n + 1})= \Op_L(\Ssc^{m,r}(\RR^{n + 1})) 
\]
be the (classical) scattering pseudodifferential operators as defined
by Melrose~\cite{Melrose94}.  The 
principal symbol map, sending $\Op_L(a) = A$ to the equivalence class
of $a$ in
\begin{equation}
  \label{eq:principal}
  \begin{gathered}
      \scprinsymb{m,r} \colon \Psisc^{m,r} \lra
    \Ssc^{m,r}(X)/ \Ssc^{m- 1,r - 1}(X),
  \end{gathered}
\end{equation}
is equivalent to the mapping taking $a$ to its restriction to the
boundary of $X \times
\overline{\RR^{n + 1}}$:
\begin{align*}
    \scprinsymb{m,r}(A) = \ang{t, z}^{-r} \ang{\tau, \zeta}^{-m} a \rvert_{\partial (X \times \overline{\RR^{n + 1}})}.
\end{align*}
Here $\partial (X \times \overline{\RR^{n+1}})$ is a union of two
boundary hypersurfaces (bhs's), which we denote
\begin{equation}
    C_\sc(X) \coloneqq (\pa X \times \overline{\RR_{\tau, \zeta}^{n + 1}}) \cup
    (X \times \partial \overline{\RR^{n + 1}_{\tau, \zeta}})
    =  \partial \Tsco^* X. \label{eq:square}
\end{equation}
The boundary hypersurface $\partial X \times \overline{\RR_{\tau, \zeta}^{n + 1}}$ is ``spacetime infinity''
while $X \times \partial \overline{\RR^{n + 1}_{\tau, \zeta}}$ is ``momentum''
or ``fiber infinity''.  The functions
\begin{equation}
  \label{eq:bhss}
  \rho_{\mathrm{base}} \coloneqq \ang{t, z}^{-1} \mbox{ and }
  \rho_{\mathrm{fib}} \coloneqq \ang{\tau, \zeta}^{-1},
\end{equation}
are boundary defining functions for spacetime and fiber infinity
respectively.

We define
\begin{align}
    \label{eq:scnormsymb}
    \scnormsymb{m,r}(A) \coloneqq \ang{t, z}^{-r} \ang{\tau, \zeta}^{-m} a \rvert_{\pa X \times \overline{\RR_{\tau, \zeta}^{n+1}}}\,,\\
    \label{eq:scfibsymb}
    \scfibsymb{m,r}(A) \coloneqq \ang{t, z}^{-r} \ang{\tau, \zeta}^{-m} a \rvert_{X \times \pa \overline{\RR_{\tau,\zeta}^{n+1}}}\,.
\end{align}
We have that for all $p \in \pa X \times \pa \overline{\RR_{\tau, \zeta}^{n+1}}$,
\begin{align*}
    \scnormsymb{m,r}(A)(p) = \scfibsymb{m,r}(A)(p)\,.
\end{align*}

Note in particular that the symbol of $L$ in
equation~\eqref{eq:sample-scattering-diff-op}, obtained by replacing
$D_t^j D_z^\alpha$ by $\tau^j \zeta^\alpha$, is smooth on $\Tsco^* X$
after it is multiplied by $\ang{\tau, \zeta}^{-m}$, and
more generally the symbol of $L \in \Diffsc^{m,r}$ is smooth on
$\Tsco^* X$ after it is multiplied by
$\ang{t,z}^{-r} \ang{\tau,\zeta}^{-m}$.

\subsection{Elliptic estimates}
\label{sec:elliptic-estimates-scattering}

We now recall the standard notions of operator wavefront set,
elliptic set, and characteristic set for operators on the scattering
calculus, which are useful in this section as well as in comparison
with the corresponding $\tsc$ notions below.

The \textbf{(scattering) operator wavefront set} $\WF'(A) =
\WF'_{\sc}(A)\subset C_{\sc}(X)$ is the essential support of the symbol $a$, where we
recall that $\alpha \not \in \esssupp(a)$ if an only if there is an
open set $U \subset X \times
\overline{\RR^{n + 1}}$ with $\alpha \in U$ such that $a$ is
Schwartz in $U$.

For $A \in \Psisc^{0,0}(X)$, the \textbf{(scattering) elliptic set}
$\Ell(A)$ is defined by $\alpha \in \Ell(A)$ if and only if
$\sigma_{\sc, 0,0}(A)(\alpha) \neq 0$.  For $A \in \Psisc^{m,\ell}$,
$\alpha \in \Ell(A)$ if and only if $\sigma_{\sc,m,\ell}(A)
\neq 0$.  The \textbf{(scattering) characteristic set} is the
complement of the elliptic set:
\[
\Char(A) = C_{\sc}(X) \setminus \Ell(A).
\]

Note that this definition of ellipticity is equivalent to the standard
one in which we demand that, in a neighborhood of $\alpha$ in $\Tsco^{*}X$, the
principal symbol $\scprinsymb{m,r}(A)$ is bounded below by 
\begin{equation*}
    \abs{\scprinsymb{m, r}(A)} \geq c > 0.
\end{equation*}

Given $\alpha \in \Ell(A)$, the
standard  elliptic parametrix construction holds in the scattering
calculus, and there is $B \in \Psisc^{-m, -r}(X)$ such that
\[
\alpha \not \in \WF'(\Id - BA) \cup \WF'(\Id - AB)
\]
and thus we obtain the (scattering) elliptic estimates:
\begin{proposition}[Corollary 5.5, \cite{V18}]\label{prop:scat elliptic}
  Let $A \in \Psisc^{m,r}(X)$.  Let $B, G \in \Psisc^{0,0}$, and
  assume $\WF'(G) \subset \Ell(A)$ and $\WF'(B) \subset
  \Ell(G)$.  Then for any $M, N \in \RR$, there is a $C> 0$
  such that
\begin{align*}
    \norm{B u}_{s, \ell} \leq C \left( \norm{G A u}_{s - m, \ell - r} + \normres{u}\right).
\end{align*}
\end{proposition}

\subsection{Variable weight spaces}
\label{sec:vari-weight-spac}

As discussed briefly at the end of
Section~\ref{sec:scattering-outline}, our Fredholm estimates require
Sobolev spaces with variable growth/decay order $\ellvar$, which we
review briefly, referring to Vasy~\cite[Sect.\ 3]{V18} for further
properties and details.  In our discussion of causal propagators, we
require only that $\ellvar$ depend on spacetime variables, but we
discuss the general case for completeness.

Suppose $\ellvar \in \CI(\Tsco^* X; \RR)$, and $0 < \delta < 1/2$.  We define
$a \in S^{m,  \ellvar}_\delta(\RR^{N}_{w})$ if $a \in
\CI(\RR^{N}_{{w}} \times
\RR^{N}_{\theta})$ and
\[
|D_{w}^\alpha D_\theta^\beta a | \le C_{\alpha\beta} \ang{w}^{\ellvar - |\alpha| + \delta |(\alpha, \beta)|} \ang{\theta}^{m - |\beta| + \delta |(\alpha, \beta)|}.
\]
(Vasy uses two distinct $\delta, \delta'$, but this is not needed
here.)  Then
\[
\Psisc^{m, \ellvar}_\delta(\RR^N) = \Op_L(S^{m,  \ellvar}_\delta(\RR^{N})),
\]
Standard symbolic constructions still work in the variable order
setting, but one must work with equivalence classes of symbols rather
than with restrictions.  As an example, the principal symbol of
$\Op_{L}(a)\in \Psisc^{m, \ellvar}(\RR^{N})$ is the equivalence
class $[a]$ of $a$ in $S^{m, \ellvar}_{\delta} / S^{m-1+2\delta, \ellvar-1+2\delta}_{\delta}$.

A paradigmatic example of such a variable order symbol is the product
\begin{align}\label{eq:def_asl}
    a_{m, \ellvar}(w, \theta) \coloneqq \ang{w}^{\ellvar} \ang{\theta}^m,
\end{align}
with the $\delta$ loses incurred by differentiation because the exponent
$\ell$ is a function.  Note that $A_{m, \ellvar} = \Op_L(a_{m, \ellvar})$ is
\emph{not} a classical symbol but is still globally
scattering elliptic in the sense that, for some $\epsilon > 0$,
\begin{equation*}
  \abs{a_{m,\ellvar}(t, z, \tau, \zeta)} \geq \epsilon \ang{t,z}^{\ellvar}\ang{\tau, \zeta}^{m}.
\end{equation*}
We further note that $A_{m, \ellvar}$ is invertible as a map on
$\schwartz (X) = \schwartz (\RR^{n+1})$  and hence as a map on
$\schwartz' (X) = \schwartz' (\RR^{n+1})$.
We may therefore define the variable order Sobolev space $\Sobsc^{m, \ellvar}$ by
\begin{align*}
    \Sobsc ^{m, \ellvar} = \curl{ u \in \schwartz' (X): A_{m, \ellvar} u \in L^2 }.
\end{align*}
If $\ellvar = \ellvar(w)$, then
$\Op_L(a_{m, \ellvar}) = \ang{w}^{\ellvar} \Op_L(\ang{\theta}^{m})$ and
$\Sobsc^{m, \ellvar} = \ang{w}^{\ellvar} \Sobsc^m$ where
$ \ang{w}^{\ellvar}$ simply acts as a spacetime-dependent weight.  It
follows essentially from Arzela-Ascoli that, for any $m > m'$ and
$\ellvar > \ellvar'$ (the latter interpreted pointwise) then
\begin{equation*}
  \Sobsc ^{m, \ellvar} \hookrightarrow \Sobsc ^{m', \ellvar'}
\end{equation*}
is a compact inclusion.

Operators in $\Psisc^{0,0}(X)$ are bounded on $\Sobsc^{m,\ellvar}$ and
$A \in \Psisc^{m,\ellvar}$ maps $\Sobsc^{m, \ellvar}\to \Sobsc^{0,0}$, so
the standard elliptic estimates still apply.  In particular, we have
the following generalization of
Proposition~\ref{prop:scat elliptic} for variable order spaces.
\begin{proposition}
  \label{prop:ellip-specific-scattering-variable}
  Suppose $A \in \Psisc^{m,r}$ $G, B \in \Psisc^{0,0}$ satisfy
  \begin{equation*}
    \WF'(B) \subset \Ell(G)\subset \WF'(G) \subset \Ell(A).
  \end{equation*}
  For each $s, M, N \in \RR$ and $\ellvar \in
  C^{\infty}(\Tsco^{*}X)$, there is a constant $C$ so that
  \begin{equation*}
    \norm{Bu}_{s,\ellvar} \leq C \left( \norm{GAu}_{s-m,\ellvar-r} + \normres{u}\right).
  \end{equation*}
\end{proposition}

\subsection{Hamiltonian flow and radial sets}

Having obtained estimates on the elliptic set, we now analyze the
operator $P_{0}$ near its characteristic set.  The estimates
obtained by Vasy~\cite{V18} include both standard propagation
estimates and estimates near the radial points, i.e., the submanifold
of points where the rescaled Hamilton vector field vanishes on
$\partial \Tsco^{*}X$.  More precisely, the global structure of the
characteristic set is used in the compilation of the global estimates
we obtain below.  In this section, we prove that the characteristic
set $\Char(P_{0})$ is comprised of two connected components, each of
which admits a source-sink structure.  The sources/sinks are the
radial sets, i.e.\ vanishing loci of the rescaled Hamilton vector field
Therefore we must analyze the characteristic
set and the Hamiltonian flow of $P_{0}$, which we proceed to do now.

The full symbol of $P_0$ is
\begin{align*}
    p(t, z, \tau, \zeta) = \tau^2 - (\abs{\zeta}^2 + m^2)\,.
\end{align*}
The standard coordinates on phase space we
write as $(t, z, \tau, \zeta)$, so that the characteristic set is given by
\begin{equation}
  \label{eq:char set}
  \Char(P_0) = \curl{ (t, z, \tau, \zeta) : \tau^2 - \abs{\zeta}^2 - m^2 = 0}.
\end{equation}\label{eq:can one form}
In open sets of the form $|z| / t < C$ (including for large $C$), we may use coordinates
\[
x = 1/t , \quad y = z / t, \quad \xi, \quad \eta,
\]
on $\Tsco^*X$.  We write the canonical one-form on $\Tsc^{*}X$ as
\begin{equation}\label{eq:scattering coords cotan}
\tau dt + \zeta \cdot dz = \xi \frac{dx}{x^2} + \eta \cdot
\frac{dy}{x},\qquad \text{i.e.\ we write } \quad
\zeta = \eta, \quad \tau = - \xi - \eta \cdot y.
\end{equation} 
The Hamilton vector field is defined by
\begin{align*}
    H_p \coloneqq \frac{\pa p}{\pa \tau} \pa_t + \frac{\pa p}{\pa \zeta} \pa_z
    - \frac{\pa p}{\pa t} \pa_\tau - \frac{\pa p}{\pa z} \pa_\zeta\,.
\end{align*}
In the above coordinates we see that the Hamilton vector field is
\begin{equation}
  \begin{split}\label{eq:hammy 1}
    \frac 12 \Hamf &=  \tau \partial_t - \zeta \cdot \partial_z  \\
&= x \left( \tau (-x \partial_x - y \cdot \partial_y +(\eta \cdot y)
  \partial_\xi)  - \zeta \cdot  (\partial_y - \eta \partial_\xi)
\right) \\
&= x \left(  (\xi + \eta \cdot y) x \partial_x - (\eta - (\xi + \eta
  \cdot y) y) \cdot \partial_y + \eta \cdot (\eta -(\xi + \eta \cdot
  y) y) \partial_\xi \right)
  \end{split}
\end{equation}
It is a general fact that for $a \in \Ssc^{m,r}(X)$ with
$A = \Op_L(a)$, a classical scattering symbol, one can rescale the
Hamilton vector field to obtain a new vector field $\Hamsc_a$ that
is tangent to the boundary of $\Tsco^*X$.  One can take, for example,
\begin{equation}
    \Hamsc_a = \ang{t, z}^{-l + 1} \ang{\tau, \zeta}^{-m + 1} H_a \in \cV_b(\Tsco^* X),\label{eq:scattering hammy def}
\end{equation}
where the latter containment means exactly that $\Hamsc_a$ extends to a smooth vector field on $\Tsco^* X$
and is tangent to the boundary; in particular the flow on $\Tsco^* X$ restricts to
the boundary $C_\sc(X)$ and defines a flow on there.  The flow of
$\Hamsc_a$ on $\Char(A)$ is what we refer to as the
bicharacteristic or Hamiltonian flow; over the spacetime interior this is the
standard formulation of the bicharacteristic flow as a homogeneous
degree zero restriction to the sphere bundle.

\begin{remark}\label{rem:positive prefactor}
  This tangency to the boundary in \eqref{eq:scattering hammy def} is
  important in that it allows us to extend the Hamiltonian flow to the
  whole of $\Tsco^* X$ and thus to extend the propagation of
  singularities estimates to infinite spacetime and/or large
  momenta regions.  Since this tangency is unchanged under multiplication by a
  positive non-vanishing prefactor, there is some ambiguity in the
  definition of $\Hamsc_a$.  Moreover, multiplication by such a
  prefactor does not affect the propagation estimates.  We use that in
  regions
  \begin{equation}
  0 \le x \le C, |y| < C, \label{eq:region space}
\end{equation}
that $x \sim \ang{t, z}^{-1}$, meaning
\[
0 < c < x \cdot \ang{t, z} < 1/c,
\]
and in such regions $\ang{t, z}^{-1}$ can be replaced by $x$ in
\eqref{eq:scattering hammy def}.  A similar remark holds for the
fiber variable $(\tau, \zeta)$.  That is, if we define $\rho =
1/\tau, \mu = \zeta / \tau$, then in regions
\begin{equation}
0 \le \rho \le C, |\mu| \le
C, \label{eq:region fib}
\end{equation} we have
\[
0 < c < \rho \cdot \ang{\tau, \zeta} < 1/c,
\]
and in such regions $\ang{\tau, \zeta}^{-1}$ can be replaced by $\rho$ in
\eqref{eq:scattering hammy def}.  It will be useful below that
$|\rho|, |\mu| < C $ in a neighborhood of the characteristic set of
$\tau^2 - |\zeta|^2 - m^2 = 0$.

One can obtain a global definition of $\Hamsc_a$ which uses $x^s$
in place of $\ang{t, z}^{-s}$ and $\rho^s$ in place of
$\ang{\tau, \zeta}^{-s}$ in regions \eqref{eq:region space} and
\eqref{eq:region fib} by choosing global boundary defining
functions which are equal to $x$ and $\rho$ respectively in those
regions, but we do not make this process formal here; we simply use
powers of $x$ and $\rho$ to rescale in regions where to do so is valid.
\end{remark}

The relationship between the commutator $[A, B] \coloneqq AB - BA$
and the Hamiltonian flow
underpins the propagation estimates in the next section.  Indeed, if
$A = \Op_{L}(a) \in \Psisc^{m_{1},r_1} (X)$ and $B =
\Op_{L}(b)\in \Psisc^{m_{2},r_2}(X)$, then
\begin{equation}
  \label{eq:scattering-commutator-hammy}
\begin{aligned}
  [A,B] &\in \Psisc^{m_{1}+m_{2}-1, r_1+r_2-1}(X)\,, \\
  \scprinsymb{m_{1}+m_{2}-1, r_{1}+r_{2}-1}\left( i [A,B]\right)
        &= \ang{t,z}^{-r_{2}}\ang{\tau, \zeta}^{-m_{2}}\cdot \Hamsc_{a}(b)\,.
\end{aligned}
\end{equation}
It is of course possible to phrase this relationship in terms of the
scattering principal symbols of $A$ and $B$, our normalization of the
principal symbol would then make the symbol's dependence on the orders
explicit.  The same relationship also holds if one of $A$ or $B$ lies
in a variable order space; the only change in this case is that the
principal symbol must be interpreted as an equivalence class of
(variable order) symbols.

The radial set of $\Hamf$ is given by
\begin{align}\label{eq:radial set first def}
    \Rad \coloneqq \Char(P_0) \cap \Hamscp^{-1}(0) \subset \partial \Tsco^* X\,.
\end{align}
The rescaled vector field $\Hamscp$ is non-vanishing on the portion of
$C_{\sc}X$ lying over the interior of $X$; to locate the radial set
we therefore consider it over $\partial X \times \overline{\RR^{n+1}}$.

Despite the fact that $\xi, \eta$ are dual (in the rescaled scattering
sense) to $x, y$, it is sometimes useful to have different coordinates
on the cotangent bundle.  
For example in the coordinates $(x, y, \tau,
\zeta)$ we have 
\begin{equation}\label{eq:hammy 3} 
    \begin{split}
        \Hamf &= 2\tau (-x^2 \partial_x - x y \cdot \partial_y) - x \zeta \cdot \partial_y \\
        &= -2x \left( \tau x \partial_x + (\zeta + \tau y) \cdot \partial_y \right)
    \end{split}
\end{equation}
Thus,  in regions with $x \le C, |y| \le C$ and $0 \le \rho = 1 / \tau \le
C, |\mu|  = |\zeta / \tau| < C$, by Remark \ref{rem:positive
  prefactor}, we have
\begin{equation}
(1/2) \Hamscp = - x \partial_x -  (\mu +  y)\cdot \partial_y,\label{eq:hammy scat}
\end{equation}
This latter expression indicates how we can realize the zero locus of
$\Hamscp$ as a family of radial sinks and sources; namely, as long as $t, \tau > 0$, we can work in the following coordinates 
\begin{equation}
x = \frac{1}{t}, \quad w = \frac{\zeta}{\tau} + y, \quad \rho = \frac{1}{\tau}, \quad \mu =
\frac{\zeta}{\tau} ,\label{eq:great coords}
\end{equation}
to obtain
\begin{align}\label{eq:hamf_simple}
    (1/2) \Hamf  = (x/\rho) \left( - x \partial_x - w \cdot \partial_w\right),
\end{align}

This expression indeed shows that, in $\curl{t>0} \cap \curl{\tau > 0}$,
the radial set is $x = 0, w = 0$ in the characteristic set, that is,
in these coordinates it is exactly
\[
\Rad^f_+ = \curl{ (x, w, \rho, \mu) \in \RR_+ \times \RR^n
\times \RR_+ \times \RR^n \colon   x = 0, w = 0,  1 - |\mu|^2 - m^2 \rho^2 = 0 }
\]
which is a smooth submanifold of $\Tsco^* X$ intersecting the boundary
$\rho = 0$ normally \emph{and is a radial sink of the flow.}

However, only the region $t > 0, \tau
> 0$ is preserved by the flow, and in that region $x = 1/t$ and $\rho
= 1 / \tau$ are valid boundary defining functions for spacetime and
fiber infinity respectively near the radial set; in the other three
regions, i.e.\ the other combinations of signs in $\pm  t > 0, \pm \tau >
0$, one can use analogous coordinates.  For example, in $t < 0 , \tau
> 0$, using $-1/t$ and $1/\tau$ as bdf's, only the overall sign of the
expression in \eqref{eq:hamf_simple} changes and in that region the
radial set is a radial source.
As for the set $\{ t = 0 \} \cup \{ \tau = 0 \}$ where no such
decomposition holds, the latter set, $\{ \tau = 0 \}$, lying
in the elliptic set of $P_0$, is irrelevant, whereas the former lies
away from the radial sets but intersects the characteristic set; we analyze that part only in Proposition
\ref{thm:global char set scattering} when we look at the global
properties of $\Char(P_0)$.

First, though, we require some basic definitions.  We define future and past causal infinity as
    \begin{align*}
        \iota^+ &\coloneqq \overline{\{(t,z) \colon  t \geq \abs{z}\}} \cap \pa X\,,\\
        \iota^- &\coloneqq \overline{\{(t,z) \colon -t \geq \abs{z}\}} \cap \pa X\,,
    \end{align*}
    where the closure is taken in $X$.  Its boundary is a compressed null infinity,
    \begin{align*}
        S^+ \coloneqq \pa\iota^+\,,\quad 
        S^- \coloneqq \pa\iota^-\,.
    \end{align*}
It is ``compressed'' in the sense that it is lower dimensional than
null infinity $\scrI^\pm$, which is typically $n-$dimensional.  Null infinity can
be shown to be naturally identified with the faces introduced by blow
up of $S^\pm$, a formalism not used here but very useful in the study
of radiation fields \cite{BVW15, BVW18}.
    
We then define the future and past radial sets as
\begin{align*}
    \Rad^f \coloneqq \Rad \cap \Tsco^*_{\iota^+} X\,,\\
    \Rad^p \coloneqq \Rad \cap \Tsco^*_{\iota^-} X\,.
\end{align*}
and, moreover, using that $\tau = 0$ does not intersect the
characteristic set, we have the further decomposition
\begin{align*}
    \Rad^f = \Rad^f_+ \sqcup \Rad^f_-\,,\quad
    \Rad^p = \Rad^p_+ \sqcup \Rad^p_-\,,
\end{align*}
where
\begin{align*}
    \Rad^\bullet_+ &\coloneqq \Rad^\bullet \cap \{\tau \geq m\}\,,\\
    \Rad^\bullet_- &\coloneqq \Rad^\bullet \cap \{\tau \leq -m\}\,.
\end{align*}

From \eqref{eq:hamf_simple} we can already see that away from fiber
infinity, the radial set has the form:
\begin{align}\label{eq:rad_future}
    \Rad^f \cap \{ \abs{(\tau, \zeta)} < \infty \} = \{ x = 0, \zeta = - \tau y, \tau^2= |\zeta|^2 + m^2 \}\,.
\end{align}

\begin{proposition}\label{thm:global char set scattering}
  The characteristic set $\Char(P_0)$ consists of two connected
  components, $\Char(P_0)_{\pm} = \Char(P_0) \cap \overline{ \{\pm \tau >
    m - \epsilon \}}$.  On the component $\Char(P_0)_+$, the radial set
    $\Rad^f_+$ is a global sink, and $\Rad^p_+$ a global source, for
    the Hamiltonian flow, while on $\Char(P_0)_-$, $\Rad^f_-$ is a
    global source, and $\Rad^p_-$ a global sink.  These sinks /
    sources are radial in the sense that the Hamilton vector field is
    of the form \eqref{eq:hamf_simple} in neighborhoods of each
    sink component, and satisfies the analogue of
    \eqref{eq:hamf_simple} with positive sign near the source components.

  The projection of the radial set to the base variables is causal infinity,
    \begin{align*}
       \pi(\Rad^f)  &= \iota^+, \pi(\Rad^p) = \iota^-\,,
    \end{align*}
    and this projection is a diffeomorphism on each component of
    $\Rad$ away from fiber infinity.
\end{proposition}

\begin{proof}
We wish to use \eqref{eq:hamf_simple}, and the three other expressions
obtained by using the other combinations of $\rho = \pm 1/\tau, x =
\pm 1/t$.    First off, defining
\begin{equation}
    \phi_t = t / \ang{t, z}, \quad \phi_z = z / \ang{t, z},\label{eq:time angle}
\end{equation}
we note that for $\delta_0  > 0$, on the region
\[
\Char(P_0) \cap \{ |\phi_t| \ge \delta_0 \} \mbox{ we have } \ang{\tau, \zeta} \sim |\tau| \mbox{ and } \ang{t, z} \sim t,
\]
where $\ang{\tau, \zeta} \sim \abs{\tau}$ means there is a constant
$c> 0$ so that
$c \leq \frac{\abs{\tau}}{\ang{\tau,\zeta}} \leq c^{-1}$.  Thus in
this region, if $t , \tau > 0$: (1) \eqref{eq:hamf_simple} holds and
(2) $\Hamscp = a \Hamf$ where $0 < c \le a \le C$ is a smooth, bounded,
positive function.  All this is to say that \eqref{eq:hamf_simple} and
the analogous expressions for the other sign choices
$\rho = \pm 1/\tau, x = \pm 1/t$ completely describe the behavior of
$\Hamscp$ on $\Char(P_0)$ in regions where $|\phi_t| \ge \delta_0$.  In
particular, for example where $t > 0, \tau > 0$, $\Rad$ is a smooth
radial sink given by $w = 0 = x$ and intersects fiber infinity (in
this region where $\rho = 0$) normally.  This region contains only
$\Rad^f_+,$ and the expressions in neighborhoods of the other
components are easily derived.

It remains only to show that in the set $\Char(P_0) \cap \{ | \phi_t| \le \delta_0 \}$
all trajectories of the Hamiltonian flow leave this
region.  As it is useful below in our definition of spacetime weights
for the causal propagators, we do this by showing that $\phi_t$ is a
monotone quantity on the flow in regions $\phi_t \in (- (1 / \sqrt{2})
+ \delta , (1 / \sqrt{2}) - \delta)$.  (Note that on the spacetime
boundary,  $|\phi_t| = 1 /
\sqrt{2}$ is null infinity.) Indeed,
\[
(1/2) H_p \phi_t = \frac{1}{\ang{t, z}} \left( \tau (1 -
  \phi_t^2) + \phi_t \zeta \cdot \phi_z \right)
\]
Thus on $\Char(P_0)$, in the
$\tau > 0$ region, where $\tau = \sqrt{|\zeta|^2 + m^2}$, using
\begin{equation*}
  \begin{split}
    (1/2) H_p \phi_t &\ge \sqrt{|\zeta|^2 + m^2} (1 -
    \phi_t^2) - |\phi_t||\zeta ||  \phi_z | \\
    &\ge \sqrt{|\zeta|^2 + m^2} \sqrt{1 -
  \phi_t^2} \left( \sqrt{1 -
  \phi_t^2} - |\phi_t| \right).
  \end{split}
\end{equation*}
Using that $\sqrt{1 -
  \phi_t^2} - |\phi_t| \ge (1 / \sqrt{2}) - |\phi_t|$ where $|\phi_t|
\le 1 / \sqrt{2}$ we have, for some $c > 0$, that
\[
 c \Hamscp \phi_t \ge (1/ \sqrt{2} ) - |\phi_t|,
 \]
 when $|\phi_t| \le 1 / \sqrt{2}$, so $\phi_t$ is monotone increasing
 on the $\tau > 0$ component of $\Char(P_0)$, and a similar argument
 shows it is decreasing on the $\tau < 0$ component.

Thus every trajectory eventually leaves a neighborhood of $\phi_t = 0$, and the
global structure of $\Char(P_0)$ is as stated.
\end{proof}

\begin{figure}
    \centering
    \begin{tikzpicture}[scale=2]
        \draw (0,0) circle (1cm);
        \draw[thick] (0.707, 0.707) arc (45:135:1cm); 
        \draw[thick] (0.707, -0.707) arc (315:225:1cm); 
        \filldraw (0.707,0.707) circle (0.3pt);
        \draw (0.8, 0.7) node [right] {$\abs{y} = 1$};
        \filldraw (-0.707,0.707) circle (0.3pt);
        \filldraw (-0.707,-0.707) circle (0.3pt);
        \filldraw (0.707,-0.707) circle (0.3pt); 
        \draw (0,0.66) node {$\iota^{+} = \pi(\Rad^f)$};
        \draw (0,-0.66) node {$\iota^{-} = \pi(\Rad^p)$};
        \draw[densely dotted] (0.707, 0.707) -- (-0.707, -0.707);
        \draw[densely dotted] (0.707, -0.707) -- (-0.707, 0.707);
    \end{tikzpicture}
    \caption{The projection of the radial set $\Rad \subset \Tsco^* X$
      to $X$. The dotted lines denote the light cone, $\{\abs{t} =
      \abs{z}\}$. The forward weights satisfy $\ellvar_{+} < -1/2$ in a neighborhood $\iota_{+}$
      and $\ellvar_{+}>-1/2$ in a neighborhood of $\iota_{-}$.}
    \label{fig:radialset}
\end{figure}
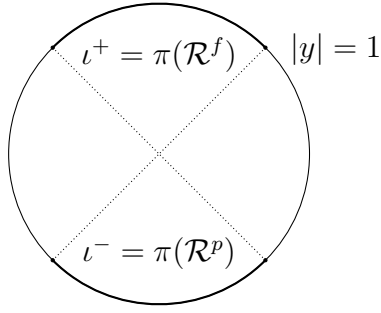

\subsection{Propagation of singularities in the scattering calculus}\label{sec:scat prop}

We now discuss the propagation estimate away from the radial sets.
Specifically, we have the standard real principal type propagation
estimates for $P_0$,
following Vasy~\cite[Theorem 4.4]{V18}.

To make the statement cleaner, we define the following notion of
control.  Given sets $U_{1}, U_{2},U_{3}\subset \Tsco^{*}X$ open, we say
that $U_{1}$ is controlled along $\Hamscp$ by $U_{2}$ through $U_{3}$
if for every $\alpha \in U_{1}\cap \Char(P_{0})$, there is a $\sigma < 0$ so that the
bicharacteristic $\gamma$ of $\Hamscp$ with $\gamma(0) = \alpha$ has
$\gamma (\sigma) \in U_{2}$ and $\gamma ([\sigma, 0])\subset U_{3}$.
In other words, each point in $U_{1}$ lies on a bicharacteristic
segment contained in $U_{3}$ that originates in $U_{2}$.

\begin{proposition}
  \label{prop:scattering-propagation-variable}
  Suppose $B, E, G \in \Psisc^{0,0}$ and that $\WF'(B)$ is controlled
  along $\Hamscp$ by $\Ell (E)$ through $\Ell (G)$.  Assume further
  that $\ellvar \in \CI (\Tsco^{*}X)$ is decreasing monotonically
  along the Hamiltonian flow.

  For any $M, N$ there is $C > 0$ such that if
  $Eu \in \Sobsc^{s,\ellvar}$ and $GP_{0} u \in \Sobsc^{s - 1, \ellvar + 1}$ ,
  then $Bu \in \Sobsc^{s, \ellvar}$ and
\begin{align*}
  \norm{B u}_{s, \ellvar} \leq C\left( \norm{E u}_{s, \ellvar} + \norm{G P_{0} u}_{s - 1, \ellvar + 1} + \normres{u}\right).
\end{align*}
The same is true if $\WF'(B)$ is controlled along $\Hamsc_{-p}= -\Hamscp$ by $\Ell
(E)$ through $\Ell (G)$ and $\ellvar$ is monotone increasing along the $\Hamscp$-flow.
\end{proposition}

\begin{remark}
  The differential order $s$ may also be taken variable in this
  proposition, monotone along the flow.  This is irrelevant for our
  purposes and thus omitted.  The operator $P_{0}$ may also be taken to
  be an element $P \in \Psisc^{m, r}$ in which case the $s-1$
  on the RHS is replaced by $s - m + 1$ and the $\ellvar + 1$ replaced by
  $\ellvar - r + 1$; again this is irrelevant for our purposes.
\end{remark}

As it informs the proof of the estimates in
Section~\ref{sec:3sc_propagation} below, we provide a sketch of the
proof of Proposition~\ref{prop:scattering-propagation-variable}.  The
proof proceeds in several steps.  We first establish the following
weaker estimate by a positive commutator argument:
\begin{lemma}
  \label{lemma:weaker-scattering-prop-estimate-schwartz}
  Suppose $B, E, G \in \Psisc^{0,0}(X)$, and
  that $\WF'(B)$ is controlled along $\Hamscp$ by $\Ell(E)$ through
  $\Ell(G)$, and $\ellvar\in \CI(\Tsco ^{*}X)$ is decreasing
  monotonically along the Hamiltonian flow.  For any $M, N$, there is
  a  $C> 0$ so that for all $u\in \Sobsc^{s, \ellvar}$ with $P_{0}u
  \in \Sobsc^{s-1, \ellvar+1}$,
  \begin{equation*}
    \norm{Bu}_{s,\ellvar} \leq C \left( \norm{Eu}_{s,\ellvar} + \norm{GP_{0}u}_{s-1,\ellvar+1} + \norm{Gu}_{s-1/2, \ellvar - 1/2} + \normres{u}\right).
  \end{equation*}
  The same is true if $\WF'(B)$ is controlled along $- \Hamscp$ by
  $\Ell(E)$ through $\Ell(G)$ and $\ellvar$ is monotone increasing.
\end{lemma}

\begin{proof}[Proof of Lemma~\ref{lemma:weaker-scattering-prop-estimate-schwartz}]
  We first prove the lemma for $u \in \schwartz (X)$; a standard
  approximation argument
  extends the result to $u$ as in the statement of the lemma.

  Our aim is to exploit the
  relationship~\eqref{eq:scattering-commutator-hammy} between the
  commutator and the Hamiltonian flow.  Indeed, given
  $\alpha\in \WF'(B)\cap \Char(P_{0})$, we construct an operator
  \begin{equation*}
    Q = A_{s-\frac{1}{2}, \ellvar+\frac{1}{2}}Q_{0}\in
    \Psisc^{s-\frac{1}{2}, \ellvar + \frac{1}{2}}(X), \quad Q_{0} \in \Psisc^{0,0}(X),
  \end{equation*}
  where $A_{s,\ellvar}=\Op_{L}(a_{s, \ellvar})$ with $a_{s,\ellvar}$ the symbol defined in \eqref{eq:def_asl}.
  We then consider
  \begin{equation*}
      \frac{i}{2}\ang{ [P_{0}, Q^{*}Q]u, u} = \Im \ang{Qu, QP_{0}u}.
  \end{equation*}
  On the one hand, for each $\epsilon > 0$, this quantity bounds the
  following from above:
  \begin{align*}
     \ang{ A_{\frac{1}{2}, -\frac{1}{2}}A_{s-\frac{1}{2},
    \ellvar + \frac{1}{2}}Q_{0}u, A_{-\frac{1}{2},
    +\frac{1}{2}}A_{s-\frac{1}{2}, \ellvar+ \frac{1}{2}}Q_{0}P_{0}u}
    &\geq - \frac{1}{4\epsilon}\norm{Q_{0}P_{0}u}_{s-1,\ellvar +1}^{2} - \epsilon \norm{Q_{0}u}_{s,\ellvar}^{2}.
  \end{align*}
  
  On the other hand, we know the relationship between the principal
  symbol of $i [P_{0}, Q^{*}Q]$ and the Hamiltonian flow, which we
  exploit to construct $Q$ symbolically.  In principle, the construction requires
  three cases; when $\alpha \in \pa X \times \overline{\RR^{n+1}}$,
  when $\alpha \in X \times \pa \overline{\RR^{n+1}}$, and when
  $\alpha \in \pa X \times \pa \overline{\RR^{n+1}}$.   We treat
  the first two cases simultaneously; the third case is handled by including an
  additional boundary defining function in the definition of the
  symbol of $Q$.

Following Vasy~\cite[Section 4.3]{V18}, we first construct
  $q_{0}\in \Ssc^{0,0}(X)$ so that $q_{0}$ is elliptic at $\alpha$ and
  there are $c > 0$ and a compact set $K\subset \Ell(e)$ and
  \begin{equation*}
    - \left( a_{s-\frac{1}{2},\ellvar+\frac{1}{2}}^{2}q_{0} \Hamf
      q_{0} + q_{0}^{2}a_{s-\frac{1}{2}, \ellvar+\frac{1}{2}}\Hamf
      a_{s-\frac{1}{2}, \ellvar+\frac{1}{2}}\right) - c a_{s,\ellvar}^{2}q_{0}^{2} \geq 0 \quad \text{off of
    }K. 
  \end{equation*}
  In a small neighborhood of $\alpha$ we employ coordinates
  $q_{1}, \dots, q_{2(n + 1)}$ on an open neighborhood $W$ of $\alpha$
  contained in $\Ell(G) $in $\Tsco^{*}X$ so that $\Hamscp = \pa_{q_{1}}$
and $q_{2(n+1)}$ is a boundary defining function.  We
  further assume that $\alpha = (0, \dots, 0)$ in this system.  This
  is possible because $\alpha$ is on the boundary and $\Hamscp$ is
  non-vanishing and tangent to the boundary.  We suppose that $\gamma$
  is the integral curve of $\Hamscp$ passing through $\alpha$ with
  $\gamma (0) = \alpha$ and $\gamma (\sigma) = \beta$, $\sigma < 0$.
  Given open neighborhoods $U_{1}$ of $\gamma([\sigma, 0])$ and
  $U_{2}$ of $\beta$ in $W$ we take $\epsilon > 0$ so that the open
  rectangles
  \begin{align*}
    &\{ q : q_{1} \in (\sigma-\epsilon, \epsilon) , q_{2(n+1)} < \epsilon,
    \abs{q_{j}}<\epsilon, j = 2, \dots , 2n+1\}, \\ 
    &\{ q : q_{1}\in (\sigma - \epsilon, \sigma + \epsilon), q_{2(n+1)} <
      \epsilon, \abs{q_{j}} < \epsilon, j = 2 , \dots, 2n+1\}
  \end{align*}
  around $\gamma ([\sigma, 0])$ and $\beta$ are contained in $U_{1}$
  and $U_{2}$, respectively.  Here $q_{1}$ acts as the flow parameter
  and $q_{2(n + 1)}$ acts as the spacetime boundary defining function;
  powers $q_{2(n + 1)}^{2 \ellvar}$ are spacetime weights.

  We then set
  \begin{equation*}
    \varphi (s) =
    \begin{cases}
      \exp (- \kappa s + (s-\epsilon)^{-1} - (s - \sigma + \epsilon)^{-1})
      & \sigma - \epsilon <s < \epsilon, \\
      0 & \text{otherwise}
    \end{cases},
  \end{equation*}
  and let $\chi \in \CI_{c}(\RR)$ be a smooth function that is
  identically one in a neighborhood of $0$.  We finally set
  \begin{equation*}
    q_{0}  = \varphi ( q_{1})
    \chi^{2}\left(\frac{q_{2}}{\delta}\right)\dots
    \chi^{2}\left(\frac{q_{2n-1}}{\delta}\right) \coloneqq \varphi \chi^{2}.
  \end{equation*}
For $\delta > 0$ sufficiently small, this function $q_{0}$ is
supported in the rectangle above.  Moreover, the explicit definition
of $\varphi$ allows us to bound $\varphi$ in terms of $\varphi ' = \Hamsc_{p} (\varphi(q_{1}))$.
  
  With $q = a_{s, \ellvar} q_{0}$, the principal symbol of
  $(i/2)[P_{0}, Q^{*}Q]$ is then
  \begin{equation*}
   q H_{p} q =  a_{s-\frac{1}{2},\ellvar + \frac{1}{2}}^{2} \varphi ' \varphi
    \chi^{2} + \varphi^{2} \chi^{2} a_{s - \frac{1}{2}, \ellvar +
      \frac{1}{2}}\Hamf a_{s-\frac{1}{2},\ellvar+\frac{1}{2}}.
  \end{equation*}
  In the region of interest, the weight $a$ is a non-vanishing smooth
  multiple of $q_{2(n+1)}^{-\ellvar- \frac{1}{2}}$, i.e., $a_{s-1/2,
    \ellvar + 1/2} = b
  q_{2(n+1)}^{-\ellvar-\frac{1}{2}}$, so
  \begin{equation*}
    a_{s-\frac{1}{2}, \ellvar+\frac{1}{2}} \Hamf a_{s-\frac{1}{2},
      \ellvar + \frac{1}{2}} = q_{2(n+1)}^{-2\ellvar - 1}b\Hamf b +
    b^{2}q_{2(n+1)}^{-2\ellvar-1}(\log q_{2(n+1)}) \left( - \Hamf \ellvar\right).
  \end{equation*}
  The second term has a favorable sign as $\ellvar$ is decreasing
  along the flow, while the first term will be controlled by the main
  term in the commutator.
  
Taking $\kappa > 0$ sufficiently large allows
  for the main term (arising from $\varphi '$) to control the
  others.  
  
  We now set $e \in \Ssc^{0,0}(X)$ so that $E = \Op_{L}(e)$; as $K
  \subset \Ell(E)$, there is thus some $C> 0$ so that, with $q = a_{s-\frac{1}{2},\ellvar+\frac{1}{2}}q_{0}$,
  \begin{equation}
    \label{eq:scattering-nonneg-quant}
    C a_{s,\ellvar}^2 e^{2} - q \Hamf q - c a_{s, \ellvar}^{2}q_{0}^{2}\geq 0.
  \end{equation}
  The construction of $q_{0}$ also guarantees that
  $\WF'(Q_{0}) \subset \Ell(G)$; explicit choice of the commutant
  $Q_{0}$ will also guarantee that this nonnegative
  quantity~\eqref{eq:scattering-nonneg-quant} above is a smooth sum of
  squares.\footnote{This allows us to avoid the use of the sharp
    G\aa{}rding inequality here.}
  The G\aa{}rding inequality in
  Lemma~\ref{lemma:scattering-easy-garding} below then shows that
  \begin{equation*}
    C\ang{E^{*}Eu, u} - \ang{\frac{i}{2}[P_{0},Q^{*}Q]u, u} - c
    \ang{Q_{0}^{*}A_{s,\ellvar}^{*}A_{s,\ellvar}Q_{0}u, u} \gtrsim -
    \norm{Gu}_{s-1/2,\ellvar-1/2}^{2} - \normres{u}^{2}.
  \end{equation*}
  In other words, we have the bound
  \begin{equation*}
      \ang{\frac{i}{2}[P_{0},Q^{*}Q]u, u} + c \norm{Q_0 u}_{s,\ellvar}^2 \lesssim
    \norm{Eu}_{s, \ellvar}^{2} + \norm{Gu}_{s-1/2, \ellvar-1/2}^{2} + \normres{u}^{2}.
  \end{equation*}

  Putting the two bounds together and taking $\epsilon = c / 2$ yields the estimate 
  \begin{equation*}
    \norm{Q_{0}u}_{s,\ellvar}^{2} \lesssim
      \norm{Gu}_{s - 1/2, \ellvar - 1/2}^{2}
      + \norm{Eu}_{s,\ellvar}^{2} +
      \norm{Q_{0}P_{0}u}_{s-1,\ellvar+1}^{2} + \normres{u}^{2}.
  \end{equation*}
\end{proof}

For completeness, we include the ``easy'' version of the G\aa{}rding
inequality used in the proof sketch above:
\begin{lemma}
  \label{lemma:scattering-easy-garding}
  If $A \in \Psisc^{m,\ellvar}(X)$ has a nonnegative principal
  symbol that is a sum of squares, i.e., there are $B_{1}, \dots,
  B_k \in \Psisc^{m/2, \ellvar/2}(X)$ with
  \begin{equation*}
    \sigma_{\sc, m,\ellvar} (A) = \sum_{j=1}^k \abs{\sigma_{\sc, m/2, \ellvar/2}(B_j)}^{2},
  \end{equation*}
  and $G \in \Psi^{0,0}(X)$ satisfies $\WF'(A) \subset \Ell(G)$, then
  for every $M, N \in \RR$, there is a constant $C$ so that for all
  $u \in \Sobsc^{m/2, \ellvar/2}$, 
  \begin{equation*}
    \ang{Au, u} \geq - C \norm{Gu}_{\frac{m-1}{2},
        \frac{\ellvar-1}{2}}^{2} - C \normres{u}^{2}.
  \end{equation*}
\end{lemma}

\begin{proof}
    This is an easy consequence of elliptic regularity, see \cite{Baskin2024} for details.
\end{proof}

An inductive application of
Lemma~\ref{lemma:weaker-scattering-prop-estimate-schwartz} (combined
with repeated adjustments to the supports of the symbols) then
establishes Proposition~\ref{prop:scattering-propagation-variable} for
$u \in \Sobsc^{s, \ellvar}$ with $P_{0}u \in \Sobsc^{s-1, \ellvar +
  1}$.  Finally, a regularization argument finishes the proof.  The
regularization argument is technical and involves replacing the
symbols in the proof of
Lemma~\ref{lemma:weaker-scattering-prop-estimate-schwartz} with weaker
approximating symbols.  As the relevant symbol classes are not used
in the rest of the paper, we refer the reader to Vasy~\cite[Section
4.4]{V18} and defer our discussion of regularization arguments to
Section~\ref{sec:3sc_propagation} below.

\subsection{Radial points estimates}\label{sec:localized rad est}

We also need estimates that hold near the radial set $\Rad$.  In
addition to illuminating the proof for the model setting, we also must
employ estimates which hold microlocally near open subsets of $\Rad$
located away from $\poles$.  In Section~\ref{sec:3sc_radial} below, we
establish estimates that hold in a more general setting; these are then
combined with estimates that are local on $\Rad$.  For simplicity, we use
the radial points estimates of \cite{V18} directly, and then show that
these can easily be localized on $\Rad$.

We use the radial points estimates from Vasy~\cite{V18} proven for a
real principal type operator $P \in \Psisc^{m,r}$ near a smooth
submanifold of radial points (there denoted $L$).  Specifically, we
use the estimates in Proposition 4.12 of \cite{V18}, which are the
``first pass" estimates in which the norm $\norm{B u}_{s, \ell}$ is
controlled by a norm $\norm{G u}_{s-1/2, \ell'}$ with
$r' \in [\ell - 1/2, \ell)$.  An induction argument then removes this norm
on the right hand side.  We use these estimates because they clarify
how easily the estimates can be localized along the radial set.  The
assumptions in Vasy's proposition about the Hamiltonian flow of $P$
near the radial set are given in equation~(4.12) there; these
assumptions are satisfied for radial vector fields as in
equation~\eqref{eq:hamf_simple}.  Vasy's quantity $\tilde \beta$ is $0$
if $P - P^* \in \Psi^{m-2, r-2}$ (i.e. is lower than the expected
order for real principal type operators).  As $P_{0}$ is self-adjoint,
we do not include that correction here.

As discussed by Vasy~\cite[Section 4.7]{V18}, the regularity and decay
orders $s, \ell$ can be exchanged thanks to the symmetry in scattering
analysis realized by the Fourier transform, which in particular maps
$\Sobsc^{s, \ell}$ isometrically to $\Sobsc^{\ell, s}$.  In particular, even though
the radial set in Vasy's notes is at fiber infinity, it is a trivial
change to adapt it to our setting.  In contrast with the propagation
estimates, the positivity near radial points is due entirely to the
weight (in our case $\ellvar + \frac{1}{2}$).  This leads to two cases
depending on the sign of the weight; we refer to these as ``above
threshold'' and ``below threshold'' estimates.  We first state the
above threshold results, which bounds $\Sobsc^{s, \ellvar}$ norm of
$u$ near a component $\Rad'$ of the radial set of $P_0$.  

\begin{proposition}[Proposition 4.11 of
  \cite{V18}, above threshold]\label{thm:vasy above prop}
  Let $\Rad'$ denote any one of the four components of the radial set
  $\Rad$ of $P_0$.  Suppose $\ellvar$ is constant near $\Rad'$ and
  that $\ellvar > -1/2$ there.
  
  Let $s, \ell' \in \RR$, $\ell' > -1/2$,
  $\ell' \in [\ellvar - 1/2, \ellvar)$.  Then there are
  $B,G \in \Psisc^{0,0}$ with $\Rad' \subset \Ell(B)$ and
  $\WF'(B) \subset \Ell(G)$, such that: if
  $G u \in \Sobsc^{s - 1/2, \ell'}$ and
  $G P_0 u \in \Sobsc^{s -1, \ellvar + 1}$, then $B u \in \Sobsc^{s, \ellvar}$,
  and for all $M,N$ there is $C > 0$ such that
\begin{equation}
  \norm{Bu}_{s,\ellvar} \leq C \left( \norm{GP_{0}u}_{s-1,\ellvar+1} + \norm{Gu}_{s-\frac{1}{2}, \ell'} + \normres{u}\right).
  \label{eq:first pass radial scattering}
\end{equation}
\end{proposition}

The operators $B$ and $G$ in the proposition are essentially
microlocal cutoffs microsupported near $\Rad'$, with $B$
microsupported in a compact subset of the elliptic set of $G$.  Such
operators can easily be constructed in the coordinates in
\eqref{eq:great coords} near $\Rad^f_+$ (and the analogous coordinates
near the three other components of $\Rad$.)  Indeed, for $\delta_1 > 2 \delta_0 > 0$ sufficiently
small and $c > 0$, we choose smooth bump functions $\phi_{0}, \phi_1,
\chi_{>c}$ so that
$\phi_i(s) = 1$ for $|s| < \delta_i$, $\phi_i(s) = 0$ for
$|s| > 2\delta_i$, $i = 1,2$, and 
$\chi_{> c}(s) = 1$ for $s \ge c $ and $\chi_{> c}(s) = 0$ for
$s \le c / 2$.  With these choices, we define:
\[
B = \Op_L(b), \quad b = \chi_{> m - \delta_0}(1/\rho) \phi_0(x) \phi_0(|w|) \phi_0( 1 -
|\mu|^2 - m^2 \rho^2), 
\]
with a similar construction for $G$:
\begin{equation*}
  G = \Op_{L}(g), \quad g = \chi_{>m - \delta_{1}} (1/\rho) \phi_{1}(x)\phi_{1}(\abs{w})\phi_{1}(1-\abs{\mu}^{2}-m^{2}\rho^{2}).
\end{equation*}
Here, $\phi_i( 1 - |\mu|^2 - m^2 \rho^2)$ cuts off to the characteristic set, $\phi_i(x)
\phi_i(|w|)$ to the radial set, and $\chi_{> m -\delta_i}(1/\rho) =
\chi_{> m - \delta_i}(\tau)$ cuts of to the $\tau > 0$ portion of the
radial set.  It is not hard to check that, despite the presence of $1/
\rho$, $b$ and $g$ are smooth scattering symbols microlocalized near the radial set.   Note that
\[
b,g \in S^{0,0}(\RR^{n + 1}),
\]
as they are smooth functions on $\Tsco^* X$.  Indeed, they are supported in a set in which all the coordinates in
$(x, w, \rho, \mu)$ are bounded, and thus they form
smooth coordinates on $\Tsco^* X$ in which $x, \rho$ are boundary
defining functions.
In particular,
\[
b g = b.
\]
The proof of such an estimate follows from a commutator argument with
a commutant $Q^{*}Q$ with $Q = \Op_{L}(q)$,
\begin{equation*}
  q = \chi_{> m - \delta_{1}}\left( \frac{1}{\rho}\right)
  \phi_{0}(x)\phi_{0}(\abs{w}^{2})\phi_{0}(1-\abs{\mu}^{2}-m^{2}\rho^{2})
  x ^{-\ellvar- \frac{1}{2}}\rho^{-s+\frac{1}{2}}.
\end{equation*}
The commutator $\frac{i}{2}[P_{0},Q^{*}Q]$ then has principal symbol
\begin{equation}
  \label{eq:radial-hamvf-calc-scattering}
  q \Hamf q = \frac{x}{\rho}\left(  (\ellvar+\frac{1}{2}) - x
    \frac{\phi_{0}'(x)}{\phi_{0}(x)} -
    \abs{w}^{2}\frac{\phi_{0}'(\abs{w}^{2})}{\phi_{0}(\abs{w}^{2})} \right)q^{2}.
\end{equation}
As $\ellvar > -\frac{1}{2}$, the first two terms have the same sign,
which explains the absence of a term $\norm{Eu}$ on the right side of
the estimate.  The other term can be absorbed into the first one if
$\delta_{0}$ is sufficiently small by observing that
$\delta_{0}^{2}\abs{w}^{2} < 4\delta_{0}^{2}$ on the support of
$\phi_{0}'$.  The
remaining terms in the estimate arise as before.

We now localize the proposition to neighborhoods of particular closed
subsets of $\Rad'$, allowing us later to combine them with more
specialized estimates.  In particular, we localize to $\Rad^{f}_{+} \cap \{
\abs{y} > 0\}$.  Indeed, as $\mu= 0$ is the only point in
$\Rad^{f}_{+}$ lying above $y = 0$, we use symbols as in the global radial point
estimate together with an additional localizer in $\mu$.  In
particular, we replace $b$ with $b' = \chi_{>c_{0}}(\abs{\mu})b$ with
a similar definition for $g' = \chi_{>c_{1}}(\abs{\mu})g$.  Then,
provided $c_{i}>2\delta_{i}$, we find that
\begin{equation*}
  \esssupp (b) \subset \{ \abs{y}\geq c_{0}'-2\delta_{0}\}, \quad
  \esssupp (g) \subset \{\abs{y}\geq c_{1}' - 2\delta_{1}\}.
\end{equation*}
We assume moreover that $2c_{1}' < c_{0}'$ so that
\begin{equation*}
  b'g' = b' \quad \text{and}\quad \WF'(B')\subset \Ell(G').
\end{equation*}
We therefore have a family of cutoffs localized near the radial set
and away from $\{ y = 0\}$.  For any $c>0$, we can choose these
operators so that
\begin{equation*}
  \Rad^{f}_{+} \cap \{ \abs{y}>c\} \subset \Ell(B') \quad
  \text{and}\quad \{ y=0\} \cap \WF'(G) = \varnothing.
\end{equation*}

From the proposition above and considerations using these microlocalizers
on $\Rad$, we have the following:
\begin{proposition}[Localized above threshold estimate]
  \label{prop:localized above thresh}
  Let $\Rad'$ denote any one of the four components of the radial set
  $\Rad$ of $P_{0}$.  Suppose $\ellvar$ is constant near $\Rad'$ and
  that $\ellvar > -1/2$ there.

  Let $s, \ell' , M, N\in \RR$, $\ell' > -1/2$ and $B', G' \in
  \Psisc^{0,0}(X)$ as above.
  If $G'u \in \Sobsc^{-N, \ell'}(X)$ and $G' P_{0}u \in \Sobsc^{s-1,\ellvar+1}$,
  then $B'u \in \Sobsc^{s, \ellvar}$  Moreover, there is a constant
  $C>0$ depending on $M, N, \ellvar, \ell', s$ so that
  \begin{equation*}
    \norm{B'u}_{s,\ellvar} \leq C \left( \norm{G'P_{0}u}_{s-1, \ellvar + 1} + \norm{G'u}_{-N, \ell'} + \normres{u}\right).
  \end{equation*}
\end{proposition}

\begin{proof}
  We first show that the estimates in Proposition \ref{thm:vasy above
    prop} hold with the $B$ and $G$ replaced by $B'$ and $G'$.

  Thus, we assume we are given $\ellvar, \ell', s$ as Proposition \ref{thm:vasy above
    prop} and, for $G'$ as in the current proposition, that
  $G' u \in \Sobsc^{s-1/2, \ell'}$ and $G' P_0 u \in \Sobsc^{s - 1, \ellvar- 1}$.
  We wish to deduce that $B' u \in \Sobsc^{s, \ellvar}$.  Taking $G$ as
  in Proposition \ref{thm:vasy above prop}, note that, if $\tilde Q$ a
  Fourier localizer to $|\mu > c|$ e.g.,
   \[
  \tilde Q \coloneqq \Op_L(\chi_{>  c}(|\mu|) \phi(1 - |\mu|^2 - m^2 \rho^2))
  \in \Psisc^{0,0},
  \]
  for $\phi$ a bump
function supported near $0$ and $\chi_{> c}$ a localizer to $|\mu| \ge
c$,  then
\[
G\tilde Q u \in  \Sobsc^{s-1/2, \ell'}.
\]
Indeed, for $c > 0$ and $\delta_1 > 0$ sufficiently small, $\WF'(G
\tilde Q) \subset \Ell G'$.

As for $G P_0 \tilde Q u$, we have that $G P_0 \tilde Q u = G
\tilde Q  P_0 u + G [ P_0 , \tilde Q] u$, and $G [ P_0 , \tilde Q]
u$ can be made lower order than expected because 
\[
\WF'([P_0, \tilde Q]) \cap \WF'(G) = \varnothing,
\]
which follows directly from the form of the Hamilton vector
field \eqref{eq:hamf_simple} and the definition of $\tilde Q$.\footnote{In fact,
in this case $[P_{0}, \tilde Q] \equiv 0$.  In the case where $P_{0}$ is perturbed by a
lower order scattering operator only the wavefront set containment holds.}
Hence,
\[
G [P_0, \tilde Q] \in \Psisc^{0,-2}.
\]
  Thus
we have, again taking $c, \delta_1$ sufficiently small,
\begin{equation}
  \norm{G P_{0}\tilde{Q}u}_{s-1,\ellvar+1} \leq C \left(
    \norm{G' P_{0}u}_{s-1,\ellvar+1} + \norm{G'
      u}_{s-1,\ell'} + \normres{u}\right).
  \label{eq:an estimate}
\end{equation}

Thus, the hypotheses of Proposition \ref{thm:vasy above prop} apply to
$\tilde Q u$, and we obtain \eqref{eq:first pass radial scattering}
for $u = \tilde Q u$ for $B$ as in that estimate.  But then there is
$B'$ with
\[
\WF'(B' ) \subset \Ell(B \tilde Q),
\]
obtained simply by taking $c > 0$ sufficiently large and $\delta_0$
sufficiently small. We may thus use
\begin{equation*}
  \norm{B'u}_{s,\ellvar} \leq C \left( \norm{B\tilde{Q}u}_{s,\ellvar} + \normres{u}\right)
\end{equation*}
in combination with
\eqref{eq:first pass radial scattering} and \eqref{eq:an estimate} 
to give \eqref{eq:first pass radial scattering} for $B'$ and $G'$.  In
other words, if
$G' u \in \Sobsc^{s - 1/2, \ell'}$, $G' P_0 u \in \Sobsc^{s -1, \ellvar + 1}$,
then $B' u \in \Sobsc^{s, \ellvar}$ together with the estimate
\begin{equation*}
  \norm{B'u}_{s,\ellvar} \leq C \left( \norm{G' P_{0}u}_{s-1, \ellvar+1} + \norm{G' u}_{s-1/2 \ell'} + \normres{u}\right).
\end{equation*}

A standard argument using induction on $s$ in half-integer steps then
finishes the proposition.
\end{proof}

We similarly have the below threshold and localized below threshold
estimates.  Note the presence of an additional term on the right side
owing to the sign change in equation~\eqref{eq:radial-hamvf-calc-scattering}.
\begin{proposition}[Proposition 4.11 of~\cite{V18}]
  \label{thm:vasy below thresh}
  Let $\Rad'$ denote any one of the four components of the radial set
  $\Rad$ of $P_{0}$.  Suppose $\ellvar$ is constant near $\Rad'$ and
  that $\ellvar < -1/2$ there.

  Let $s\in \RR$ and $B, E, G\in \Psisc^{0,0}(X)$ be such that
  $\WF'(B) \setminus \Rad'$ is controlled along $\Hamscp$ by $\Ell (E)$
  through $\Ell(G)$.  If $Eu \in \Sobsc^{s, \ellvar}$, $GP_{0}u \in
  \Sobsc^{s-,\ellvar+1}$, and $G u \in \Sobsc^{s-\frac{1}{2},
    \ellvar-\frac{1}{2}}$, then $Bu \in \Sobsc^{s, \ellvar}$ and, for
  any $M, N \in \RR$, there is a constant $C$ so that
  \begin{equation*}
    \norm{Bu}_{s,\ellvar} \leq C \left( \norm{Eu}_{s,\ellvar} + \norm{GP_{0}u}_{s-1, \ellvar + 1} + \norm{Gu}_{s-1/2, \ellvar-1/2} + \normres{u}\right).
  \end{equation*}
\end{proposition}

For the localized version of the below threshold estimate, we also
introduce $e' = \chi_{>c_{1}}(\abs{\mu})e$ in addition to the
definitions of $b'$ and $g'$ given for the localized above threshold
estimate.  The proof is essentially the same as the proof of Proposition~\ref{prop:localized above thresh}.

\begin{proposition}[Localized below threshold estimate]
  \label{prop:localized below thresh}
  Let $\Rad'$ be any one of the four components of the radial set
  $\Rad$ of $P_{0}$.  Suppose $\ellvar$ is constant near $\Rad'$ and
  that $\ellvar < -1/2$ there.

  Let $s, M, N \in \RR$ and let $B', E', G' \in \Psisc^{0,0}(X)$ be as
  above.  If $E' u \in \Sobsc^{s, \ellvar}$, $G'P_{0}u \in
  \Sobsc^{s-1, \ellvar+1}$, then $B' u \in \Sobsc^{s, \ellvar}$, and
  for any $M, N\in \RR$, there is a constant $C$ so that
  \begin{equation*}
    \norm{B'u}_{s,\ellvar} \leq  C\left( \norm{E'u}_{s, \ellvar} + \norm{G'P_{0}u}_{s-1,\ellvar+1} + \normres{u}\right).
  \end{equation*}
\end{proposition}

\subsection{Weight functions, Fredholm estimates, and propagators} \label{sec:fred for scat}

We now describe in more detail how to combine the estimates of the
previous sections to obtain Fredholm estimates for $P_{0}$.  We pass
to a microlocal partition of unity subordinate to a cover of $\Tsco ^{*}X$ by
open neighborhoods of the components of the radial set $\Rad$, the
characteristic set $\Char P_{0}$, and then the elliptic set.  On each
of these neighborhoods, we appeal to the estimates of previous
sections.  Away from the characteristic set, we appeal to elliptic
estimates.  Near the characteristic set, the above threshold estimate
propagates regularity from a source component of the radial set to a
small neighborhood of it; we then use the propagation estimates to
conclude regularity in a neighborhood of a sink component, which is
then propagated into the radial set by the below threshold estimate.

As noted earlier, it is impossible to satisfy simultaneously the
conditions for the above threshold and below threshold estimates with
a constant weight, so we appeal to variable weights.  Because the
radial estimates have no threshold conditions in the regularity order
$s$, there is no need to allow variable order regularity.  It is only
the spacetime weight $\ellvar$ that must vary, and the conditions it
must satisfy are summarized in the following definition:
\begin{definition}\label{def:forward backward weights}
  Let $\ellvar \in \CI (\Tsco^{*}X; \RR)$.  We call $\ellvar$ admissible if
  $\ellvar$ is monotone along the Hamiltonian flow within each
  component of the characteristic set and constant near the components
  of $\Rad$.  Moreover, we say $\ellvar$ is:
  \begin{enumerate}
  \item \emph{forward} if $\ellvar > -1/2$ on $\Rad_{+}^{p}\cup \Rad_{-}^{p}$
    and $\ellvar <-1/2$ on $\Rad_{+}^{f}\cup \Rad_{-}^{f}$,
  \item \emph{backward} if $\ellvar <-1/2$ on $\Rad_{+}^{p}\cup
    \Rad_{-}^{p}$ and $\ellvar > -1/2$ on $\Rad_{+}^{f}\cup
    \Rad_{-}^{f}$, 
  \item \emph{Feynman} if $\ellvar < -1/2$ on $\Rad_{+}^{f}\cup
    \Rad_{-}^{p}$ and $\ellvar > -1/2$ on $\Rad_{+}^{p}\cup
    \Rad_{-}^{f}$, and 
  \item \emph{anti-Feynman} if $\ellvar > -1/2$ on $\Rad_{+}^{f}\cup
    \Rad_{-}^{p}$ and $\ellvar <-1/2$ on $\Rad_{+}^{p}\cup \Rad_{-}^{f}$.
  \end{enumerate}
  We write forward weights as $\ellvar_{+}$ and backward weights as $\ellvar_{-}$.
\end{definition}
Note that the four types of weight functions described here correspond
to the four distinguished parametrices of
Duistermaat--H{\"o}rmander~\cite{DuHo72}.  We encode which propagator we
are considering by selecting an appropriate weight for the function
spaces.

In particular, a forward (resp.\ backward) weight function decreases
(resp.\ increases) as $t$ increases, while a Feynman (resp.\
anti-Feynman) weight function decreases (resp.\ increases) along the
global Hamiltonian flow.  Forward and backward weight functions
distinguish the causal (i.e., forward and backward) propagators.

For the causal propagators, the weights can be taken to be functions
on spacetime (i.e., independent of $\tau, \zeta$), while the Feynman
and anti-Feynman weights must be genuinely pseudodifferential.
We focus now on the causal propagators, but the construction in the
Feynman and anti-Feynman settings follows similar lines (though with a
less explicit weight).

To construct the causal weights, we seek a function on spacetime that
has the desired monotonicity and is equal to $-1/2 \pm \epsilon$ at
$\Rad^{p/f}$.  We therefore employ the function
\begin{equation*}
  \phi_{t} = \frac{t}{\ang{t,z}}
\end{equation*}
from the proof of Proposition~\ref{thm:global char set scattering}.
Indeed, we show there that $\phi_{t}$ is monotone increasing along the
$\tau > 0$ component of the flow and decreasing on the $\tau < 0$
component.  For any $\epsilon, \delta > 0$, we then let $f:[-1,1] \to
\RR$ be any smooth, non-increasing function with $f(s) = -1/2 +
\epsilon$ for $s < -1/ \sqrt{2} + \delta$ and $f (s) = -1/2 -
\epsilon$ for $s > 1/\sqrt{2}-\delta$ and set
\begin{equation*}
  \ellvar (t,z) = f(\phi_{t}).
\end{equation*}
With this definition, $\ellvar$ is a forward weight function, and $-1
- \ellvar$ is a backward weight function.

We now introduce some notation for the function spaces on which we
expect $P_{0}$ to be Fredholm.  To simplify matters, we focus on the
the construction leading to the forward propagator, but the same
argument shows that the other three choices of weights lead to
Fredholm estimates.  We let $\ellvar_{+}$ denote a forward weight
function, so that $-1 - \ellvar_{+}$ is a backward weight function.
We recall that
\begin{equation*}
  \cX^{s,\ellvar_{+}} = \left\{ u \in \Sobsc^{s,\ellvar_{+}}
    \colon P_{0} u \in \Sobsc^{s-1, \ellvar_{+}+1}\right\}\,,\quad
  \cY^{s,\ellvar_{+}} = \Sobsc^{s,\ellvar_{+}}
\end{equation*}
and $\cY^{s,\ellvar_+}$ is equipped with the norm of $\Sobsc^{s,\ellvar_+}$, whereas
\begin{align*}
    \norm{u}_{\cX^{s,\ellvar_+}}^2 = \norm{u}_{s,\ellvar_+}^2 + \norm{P_0}_{s-1, \ellvar_+ + 1}^2\,.
\end{align*}
Although we do not need this fact
here, the space $\cX^{s, \ellvar_{+}}$ depends only on the principal
symbol of $P_{0}$ and operators with the same principal symbol induce
equivalent norms.

As $P_{0} : \cX^{s,\ellvar_{+}}\lra \cY^{s-1,\ellvar_{+}+1}$ is
continuous, showing that it is Fredholm therefore reduces to
the following two estimates:
\begin{align}
  \label{eq:global scat fredholm}
  \norm{u}_{s, \ellvar_{+}} &\leq C \left( \norm{P_{0}u}_{s-1,\ellvar_{+}+1} + \normres{u}\right), \\
  \norm{u}_{1-s, -1 - \ellvar_{+}} &\leq C\left( \norm{P_{0}u}_{-s, -\ellvar_{+}} + \norm{u}_{-N', - M'}\right), \notag
\end{align}
for some $M,M', N, N'$ are such that the inclusions
$\Sobsc^{s, \ellvar_{+}}\hookrightarrow \Sobsc^{-N, -M}$ and
$\Sobsc^{1-s, -1-\ellvar_{+}}\hookrightarrow \Sobsc^{-N', -M'}$ are
compact.

We take now an open cover $O_{1}, O_{2}, O_{3}, O_{4}$ of $\Tsco^{*}X$
so that:
\begin{enumerate}
\item $\Rad_{\pm}^{p} \subset O_{1} \subset \{ \ellvar_{+} = -1/2 +
  \epsilon\}$, 
\item $\Rad_{\pm}^{f}\subset O_{2}\subset \{ \ellvar_{+}=
  -1/2-\epsilon\}$, 
\item $\Char P_{0} \subset O_{1}\cup O_{2}\cup O_{3}$, 
\item $O_{3}$ is controlled along $\Hamscp$ by $O_{1}$, 
\item $O_{2}\setminus \Rad$ is controlled along $\Hamscp$ by $O_{3}$,
  and 
\item $O_{4} \subset \Ell(P_{0})$.
\end{enumerate}
Now we take a microlocal partition of unity $\Id = B_{1} + B_{2} +
B_{3} + B_{4}$, $B_{i} \in \Psisc^{0,0}(X)$ with 
\begin{equation}
\WF' (B_{i})\subset O_{i}.\label{eq:same old WF condition}
\end{equation}

If $u \in \cX^{s, \ellvar_{+}}$, then, by assumption, the above
threshold estimate applies to $u$ near $\Rad_{\pm}^{p}$ and so, by
Proposition~\ref{thm:vasy above prop} (with $G = \Id$),
\begin{equation}\label{eq:B1}
  \norm{B_{1}u}_{s, \ellvar_{+}} \leq C\left( \norm{P_{0}u}_{s-1,\ellvar_{+}+1} + \norm{u}_{s-1/2, \ellvar'}\right),
\end{equation}
where $\ellvar ' < \ellvar_{+}$ and $-1/2 < \ellvar' < \ellvar_{+}$ on
$O_{1}$.  Now, as $O_{3}$ is controlled by $O_{1}$,
Proposition~\ref{prop:scattering-propagation-variable} tells us
\begin{equation}\label{eq:B3}
  \norm{B_{3}u}_{s, \ellvar_{+}} \leq C\left( \norm{B_{1}u}_{s, \ellvar_{+}} + \norm{P_{0}u}_{s-1, \ellvar_{+}+1} + \norm{u}_{s-1/2, \ellvar'}\right).
\end{equation}
The hypotheses for Proposition~\ref{thm:vasy below thresh} are now
fulfilled and so we obtain
\begin{equation}\label{eq:B2}
  \norm{B_{2}u}_{s, \ellvar_{+}} \leq C\left( \norm{B_{3}u}_{s, \ellvar_{+}} + \norm{P_{0}u}_{s-1, \ellvar_{+}+1} + \norm{u}_{s-1/2, \ellvar'}\right).
\end{equation}
Because $\WF'(B_{4})\subset \Ell(P_{0})$, the elliptic estimates of
Proposition~\ref{prop:scat elliptic} tell us
\begin{equation}\label{eq:B4}
  \norm{B_{4}u}_{s, \ellvar_{+}} \leq C \left( \norm{P_{0}u}_{s-2,\ellvar_{+}} + \norm{u}_{s-1/2, \ellvar'} \right)
  \leq C\left( \norm{P_{0}u}_{s-1,\ellvar_{+}+1} + \norm{u}_{s-1/2, \ellvar'} \right).
\end{equation}
Because $\Id = B_{1} + B_{2} + B_{3} + B_{4}$, we then have the
estimate
\begin{equation*}
  \norm{u} _{s, \ellvar_{+}} \leq C \left( \norm{P_{0}u}_{s-1,\ellvar_{+}+1} + \norm{u}_{s-1/2, \ellvar'}\right),
\end{equation*}
and the inclusion
$\cX^{s, \ellvar_{+}}\hookrightarrow \Sobsc^{s-1/2,\ellvar'}$
is compact.

To obtain the estimate
\begin{equation*}
  \norm{u}_{1-s, -1-\ellvar_{+}} \leq C \left( \norm{P_{0}^{*} u}_{-s, -\ellvar_{+}} + \norm{u}_{-s-1/2, \ellvar'}\right),
\end{equation*}
we use the same chain of estimates, but with the roles of $\Rad^{p}$
and $\Rad^{f}$ exchanged.  In other words, we propagate regularity
from $\Rad^{f}$ to $\Rad^{p}$ along the Hamiltonian flow in
$\Char (P_{0})$.  Here $\ellvar'$ must be chosen analogously, i.e.,
$\ellvar' < -1-\ellvar_{+}$ must also be greater than the threshold
$-1/2$ near $\Rad^{f}$.

By the standard argument of
  iterating by $1/2$ differential orders we replace the $-1/2$ on the
  right by an arbitrary differential order $-N$ and deduce \eqref{eq:global scat
    fredholm}. Note that the formulation of the estimate in
  \eqref{eq:global scat fredholm} with the lower order error term on
  the right hand side follows from bounding the $\ellvar'$ term by a
  small factor times that $\ellvar$ term on the left; specifically,
  using that for any $M > 0$ and $\epsilon > 0$ there is
  $C(\epsilon) > 0$ such that
  $x^{-\ellvar'} \le C(\epsilon) x^M + \epsilon x^{-\ellvar}$, based
  off which we have, for any $N \in \RR$, a $C > 0$ such that
  \begin{equation}\label{eq:youngs inequality}
    \norm{u}_{-N, \ellvar'} \le
    C \norm{u}_{-N,  -M} + \epsilon \norm{u}_{-N, \ellvar}\,,
  \end{equation}
so for $\epsilon$ sufficiently small the last term can be absorbed
onto the left hand side.

With these estimates, we have nearly proved Theorem~\ref{thm:sec2thm}:
\begin{proof}[Proof of Theorem~\ref{thm:sec2thm}]
  The estimates above in equation~\eqref{eq:global scat fredholm} show
  that $P_{0}$ is Fredholm between the stated spaces.  Indeed, the
  estimates directly imply that the operators have closed range and
  finite dimensional kernel.  That they have finite dimensional
  cokernel follows from the identification of the cokernel with the
  kernel of the operator on the adjoint space, which for the forward
  weight $\ellvar_{+}$ is the backward weight $-1 -\ellvar_{+}$.

  The fact that $P_{0}$ is invertible on the stated space then follows
  if its kernel and cokernel are zero.  This claim follows from the
  energy/Gr\"{o}nwall argument given below in the proof of Theorem
  \ref{thm:fredholm causal 3sc}. The cokernel of
  the forward problem is the kernel of a corresponding backward
  problem, and vice versa, so this completes the proof.

  For the statement that the inverse is the forward propagator,
let $f \in
  \Sobsc^{s - 1, \ellvar_{+} + 1}$ for any forward weight
  $\ellvar_{+}$, and assume that for some $T \in \mathbb{R}$,
  \[
\supp f \subset \{ t \ge T \}.
\]
Then the inverse mapping of \eqref{eq:model mapping} applied to $f$
gives a solution $u_{+}$ to $P_{0} u_{+} = f$ with
$u_{+} \in \Sobsc^{s, \ellvar_{+}}$.  Then $u_{+}$ satisfies the above
threshold condition near the past radial sets, and just as with
elements in the kernel, $u_{+}$ is Schwartz near past causal infinity.
The same energy/Gr\"{o}nwall argument then shows that $u_{+}$ is
identically zero in $t \le T$, i.e. $u_{+}$ is the forward
solution.
\end{proof}

\subsection{Scattering perturbations}\label{sec:scat perturb}

We finally observe that the estimates above are all symbolic in
nature, so the same estimates hold for any operator with the same
principal symbol and sub-principal symbol as $P_{0}$.  In fact, we
require only that these agree at $\partial X \times \overline{\RR^{n}}$
as long as the underlying Lorentzian metric is non-trapping.

We therefore consider a ``potential''
\begin{equation*}
    V \in \ang{t,z}^{-1} \Diffsc^{1}(X)
\end{equation*}
with $V - V^{*} \in \ang{t,z}^{-2}\Diffsc^{0}(X)$.  We also
allow the differential part
\begin{equation*}
    D_{t}^{2}-D_{z}\cdot D_{z} \eqqcolon \square_{g_0}
\end{equation*}
of $P_{0}$ to be replaced by the wave operator for an asymptotically
Minkowski metric $g$ satisfying
\begin{equation*}
  g - g_{0}\in S^{-2}\left( \RR^{n+1} ; \Sym^{0,2}\right),
\end{equation*}
where
\begin{equation*}
  g_{0} = dt^{2} - \sum_{j=1}^{n}dz_{j}^{2}.
\end{equation*}
We further demand that $g$ be non-trapping, i.e., that all null
  geodesics of $g$ approach $\scrI^{\pm}$ (or equivalently compressed
  null infinity $S^\pm$) in both directions along the
  flow.  Then, with
$\square_{g} = -\sqrt{|g|}^{-1}\partial_{i}\sqrt{|g|}g^{ij}\partial_{j}$, we
treat the operator
\begin{equation*}
  P_{g, V} \coloneqq \square_{g} - m^{2} - V.
\end{equation*}
We observe that $P_{0}$ and $P_{g,V}$ have the same principal symbol
at spatial infinity $\partial X \times \overline{\RR^{n}}$ and the
structure of the Hamiltonian flow is the same as that given in
Proposition~\ref{thm:global char set scattering}.  Additionally, we
note that $P_{g,V} - P_{g,V}^{*} \in \ang{t,z}^{-2}\Diffsc^{0}(X)$,
so that the same propagation and radial point estimates over the
boundary hold.  The propagation estimates at fiber infinity hold with
the same proof, and so we in fact have the following theorem:
\begin{theorem}
  \label{thm:sec2thmredux}
  If $g$ is a non-trapping asymptotically Minkowski metric in the
  sense that
  \begin{equation*}
    g - g_{0} \in S^{-2}\left( \RR^{n+1}; \Sym^{0,2}\right),
  \end{equation*}
  and $V \in \ang{t,z}^{-1}\Diffsc^{1}(X)$ satisfies $V - V^{*} \in
  \ang{t,z}^{-2}\Diff^{0}(X)$, then Theorem~\ref{thm:sec2thm} holds
  for the operator $P_{g,V}$.  In other words, if $s\in \RR$ and
  $\ellvar_{+}$ is a forward weight, then
  \begin{equation*}
    P_{g,V} : \cX^{s, \ellvar_{+}} \lra \cY^{s-1,\ellvar_{+}+1}
  \end{equation*}
  is an isomorphism and its inverse is the forward propagator.  The
  same is true when $\ellvar_{+}$ is replaced by a backward weight
  $\ellvar_{-}$, in which case the inverse is the backward propagator.
\end{theorem}

\section{Asymptotically static potentials}
\label{sec:operator-p_v}

In this section we describe the spacetime geometry adaptations for the
operator $P_{V}$.  For static potentials $V = V(z)$, the operator
$P_{V}$ is not a scattering operator in the sense of
Section~\ref{sec:model-case}; \emph{$V$ fails to be smooth on the north
pole and south pole of $X$ as described below.}  We therefore pass to the
minimal resolution of $X$ on which $V$ is smooth and thereby recognize
our operator as an element of Vasy's many-body scattering calculus.

\subsection{The resolution of \texorpdfstring{$X$}{X}}

Recall that, for the spacetime compactification $X = \overline{\RR^{n
    + 1}_{t, z}}$, the coordinates $x = 1/t, y = z/ t$ are valid up to
the boundary $\pa X$ in any region in which $(x, y)$ are bounded.  This
clarifies the failure of smoothness of $V$ at the point which $x = 0,
y = 0$, which we refer to as the
``north pole'' and denote $\NP$.  Indeed,
there, even for Schwartz potentials $V \in \schwartz(\RR^n)$, we see that $V(z) = V(y/x)$ fails even to be continuous at
$\NP$.  The same goes for the ``south pole'', $\SP$, where the
coordinates $x = - 1/t$, $y = z/t$ vanish.  (We only write $x = -1/t$
when working near $\SP$.)  We set $C = \{\NP, \SP\} \subset \pa X$.
Thus
\begin{equation}
  \label{eq:def of C}
  \poles = \pa X \cap \overline{\{ z = 0 \}}.
\end{equation}
and
\begin{align*}
    \NP =  \pa X \cap \overline{ \{ z = 0 \} } \cap \overline{\{ t > 0
  \}},\quad \SP  = \pa X \cap \overline{\{ z = 0 \}} \cap \overline{\{ t > 0 \}}.
\end{align*}
On the other hand, $V \in \schwartz(\RR^n)$ is smooth on $X
\setminus \poles$.

We are therefore led to consider the blow-up of $C$ in $X$ equipped
with the blow-down map
\begin{align}\label{eq:X blow C def}
    \beta_C : [X; C] \lra X\,.
\end{align}
Static potentials that are Schwartz functions $V = V(z) \in \schwartz(\RR^n)$ are smooth on
$[X;C]$.\footnote{More general $V$ must have a classical symbol
  expansion at infinity to be smooth on $[X;C]$.}

The space $[X; C]$ is the \emph{$\tsc$-single space} and is a manifold
with corners possessing three boundary hypersurfaces,
\begin{align*}
    \ff_+ &\coloneqq \beta_C^*(\NP)\,,\\
    \ff_- &\coloneqq \beta_C^*(\SP)\,,\\
    \mf &\coloneqq \beta_C^*(\partial X)\,.
\end{align*}
Note that $\ff_\pm$ has interior which is isomorphic to
$\RR^n_z$.  In what follows, we typically restrict our
attention to $\ff_{+}$ and write $\ff = \ff_{+}$.

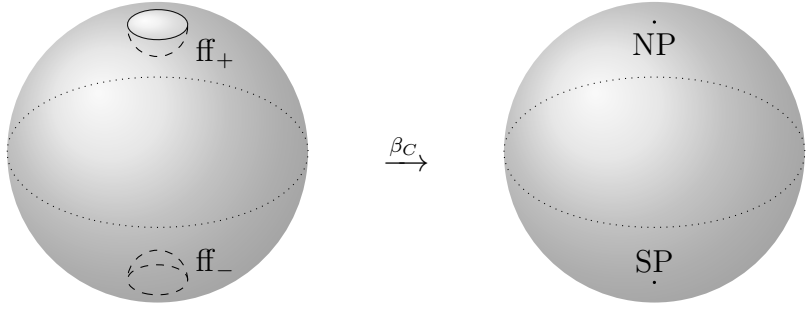
\begin{figure}
    \tdplotsetmaincoords{60}{115}
    \centering
    \begin{minipage}{0.3\textwidth}
        \centering
        \begin{tikzpicture}[tdplot_main_coords, scale = 2]
            \shade[ball color = lightgray, opacity = 0.5] (0,0,0) circle (1cm);
            \filldraw (0,0,1) circle (1pt);
            \begin{scope}[canvas is xy plane at z=0.98]
                \draw [fill=white](0,0) circle (0.20cm);
                \shade[ball color = lightgray, opacity = 0.25] (0,0) circle (0.20cm);
            \end{scope}
            \draw[dashed] (0.0,0.22,1.02) arc (0:-180:0.20cm);
            \begin{scope}[canvas is xy plane at z=-0.98]
                \draw[dashed] (0,0) circle (0.20cm);
            \end{scope}
            \draw[dashed] (0.0,0.22,-0.93) arc (0:180:0.20cm);
            \begin{scope}[canvas is xy plane at z=0]
                \draw[dotted] (0,0) circle (1cm);
            \end{scope}
            \draw (0,0.2,1) node [below right] {$\ff_+$};
            \draw (0,0.2,-1) node [above right] {$\ff_-$};
        \end{tikzpicture}
    \end{minipage}%
    \begin{minipage}{0.1\textwidth}
        \centering
        $\xrightarrow{\beta_C}$
    \end{minipage}%
    \begin{minipage}{0.3\textwidth}
        \centering
        \begin{tikzpicture}[tdplot_main_coords, scale = 2]
            \shade[ball color = lightgray, opacity = 0.5] (0,0,0) circle (1cm);
            \begin{scope}[canvas is xy plane at z=0]
                \draw[dotted] (0,0) circle (1cm);
            \end{scope}
            \filldraw (0,0,1) circle (0.2pt) node [below] {$\NP$};
            \filldraw (0,0,-1) circle (0.2pt) node [above] {$\SP$};
        \end{tikzpicture}
    \end{minipage}
    \caption{The blow-down map $\beta_C : [X;C] \to X$ of the
      $\tsc$-single space.}
    \label{fig:X and XC}
\end{figure}

Coordinates on $[X; C]$ can be understood in terms of those on $X$ as  
follows.  First off, coordinates near the boundary of $X$ can be
taken near any point $p \in \partial X$ near which $t \to \infty$, to
be
\begin{align*}
    x = 1/t,\quad y = z/t
\end{align*}
where here $x$ is a boundary defining function (bdf) of $\partial X$.
Then in the region
$\abs{z} < C$, $t > 0$, we have the simple coordinates
\begin{align*}
    x = 1/t, \quad z 
\end{align*}
with $x$ being a bdf of $\ff$ in this region, while near the intersection
of $\ff \cap \mf$, near any point there is at least one $z_k$ for
which $\hat{x} = 1/z_k$ is a bdf for
$\ff$ and there one can use
\begin{equation}
    \hat x = 1/z_k, \quad \hat{Y}_j = z_j / z_k (j \neq k), \quad y_k = z_k /
    t,\label{eq:bdry ff coords}
\end{equation}
and here $y_k$ is a bdf of $\mf$.

\subsection{\texorpdfstring{$\tsc$-differential operators}{3sc-differential operators}}

Differential operators in the $\tsc$-calculus are given by 
\begin{equation}
  \label{eq:1}
  \Difftsc^{m}(X) \coloneqq \Diffsc^{m}(X) \otimes_{\CI(X)} C^{\infty}([X;\poles])\,.
\end{equation}
More concretely, $L \in \Difftsc^{m}$, if
\begin{align}
\label{eq:tsc diff model}
    L = \sum_{\abs{\alpha} + k \leq m} a_{k,\alpha} D_t^k D_z^\alpha\,,
\end{align}
where the coefficients $a_{k,\alpha}$ are smooth on the blown up space $[X;C]$.

Using the $x, z$ coordinates, it is easy to see that
\begin{equation}  \label{eq:operator is in space duh}
  P_V = D_t^2 - (\Delta + m^2 + V(z)) \in \Difftsc^{2}(X).
\end{equation}
On $[X;\poles]$, general differential operators in the $\tsc$-calculus
are simply
\begin{equation*}
\Difftsc^{m,r} = \ang{t, z}^{r} \Difftsc^{m},
\end{equation*}
which
is to say we do not distinguish, in the notation, the rates of spatial
decay or blow up of coefficients at the faces $\ff$ and $\mf$.  Thus,
in particular
\begin{equation*}
  \Difftsc^{m} = \Difftsc^{m,0}.
\end{equation*}

The principal symbol of the operator $P_V$ will have three
components.  Two of them are inherited directly from the scattering
calculus; they are localized away from $\poles$, i.e.\ to the region
of $X$ where $V$ is smooth, and are essentially the same as the scattering principal
symbol.  The other component of the principal symbol, defined only
above $\poles$, is the ``indicial operator'', and is essentially the
time Fourier transform of $P_V$ restricted to $\poles$.

\subsection{\texorpdfstring{$\tsc$-geometry}{3sc-geometry}}
\label{sec:tsc-geometry}

We now aim to describe the domain of the principal symbol of a
$\tsc$-differential operator.

The (radial compactification) of the three-body scattering cotangent bundle is,
by definition, the pullback bundle
\begin{equation}
    \Ttsco^* [X;C] \coloneqq \beta_C^* \Tsco^* X.\label{eq:three scat cotangent}
\end{equation}
Since we are working over $\RR^{n + 1}$, these bundles are
trivial and thus there is a natural decomposition
\begin{align*}
    \Ttsco^* [X;C] = [\overline{\RR^{n + 1}} ; C] \times \overline{\RR^{n + 1}}.
\end{align*}
The manifold with corners $\Ttsco^* [X;C]$ has three boundary
hypersurfaces, namely
\begin{equation}
  \label{eq:Ttsc_bhss}
  \Ttsco_{\ff}^* [X; C], \quad \Ttsco_{\mf}^* [X; C], \mbox{ and } \Stsc^* [X; C]\,,
\end{equation}
where the latter is the ``fiber'' boundary of $\Ttsco^* X$, i.e. the
$\Stsc^* [X; C] = [\overline{\RR^{n + 1}} ; C] \times \partial \overline{\RR^{n + 1}}$.

We denote the corresponding boundary defining functions as
\begin{align}\label{eq:bdfs}
    \rho_{\ff}\,,\quad \rho_{\mf}\,,\quad \rho_{\fib}\,.
\end{align}
Moreover, we also define the total boundary defining function for the
spacetime boundary $\rho_\infty \coloneqq \rho_{\ff} \rho_{\mf}$.  
As discussed in Remark \ref{rem:positive
  prefactor}, such boundary defining functions, which are used in
particular in re-weighting of symbols and distributions below, can be
multiplied by positive function without effecting the estimates in
which they are used.  One global choice of $\rho_{\ff}$ would be
$\ang{t, z}^{-1}$, but (again see Remark \ref{rem:positive
  prefactor}), it is more convenient to assume that
\[
\rho_{\infty} = x = 1/t
\]
in regions $0 \le x \le C, |y| \le C$.  Similarly, a global
choice of $\rho_{\fib}$ would be $\ang{\tau, \zeta}^{-1}$, but it is more convenient for us to choose
$\rho_{\fib}$ so that
\[
\rho_{\fib} = \rho = 1 / \tau
\]
in regions $0 \le \rho \le C, |\mu| \le C$.  We can take $\rho_{\mf}$
globally as
\[
\rho_{\mf} = \ang{z}^{-1}.
\]

In our analysis using the three-body calculus, near $\ff$ we will prove estimates which are global on the
$\{\tau = const. \}$ subsets of phase space.  
Thus, following the notation of \cite{V2000}, we define the vector bundle
\begin{equation}\label{eq:Wperp}
    W^\perp \coloneqq \modspan_{\RR}\left( \frac{dx}{x^2}\right) \subset \Tsc^*_C X.
\end{equation}
Thus $W^\perp$ is parametrized by $\tau \in \RR$ corresponding
to the form $- \tau (dx / x^2)$, both at $\NP$ and $\SP$, recalling
that near
$\SP$ we write $x = -1/t$.  Formally, $W^{\perp} \subset \Tsc^*_C X$ is defined as the
annihilator of the subset $W \subset \Tsc_{C} X$ consisting of vectors
arising from vector fields $U = x U' $ with $U' \in
\mathcal{V}_{\mathrm{b}}(X)$ with $U'$
tangent to $\poles$, where $\mathcal{V}_{\mathrm{b}}(X)$ are the
vector fields tangent to the boundary.  But the simple definition
above suffices.  In particular,
\begin{equation}
  \label{eq:8}
  W^{\perp} = \mathbb{R} \sqcup \mathbb{R},
\end{equation}
where both copies of $\mathbb{R}$ are parametrized by $\tau$, one over $\NP$ and the other over $\SP$.

The orthogonal projection 
\begin{align}\label{eq:piWperp}
    \pi \colon \Tsc^*_C X \longrightarrow W^\perp\,,
\end{align}
with action $(\tau, \zeta) \mapsto \tau$, does \emph{not} extend
smoothly to $\Tsco^*_C X$, but the closures of the fibers
$\pi^{-1}(\tau)$ are smooth submanifolds with common boundary, as
depicted in Figure~\ref{fig:indicial_operator}.
Pulling back to the three-body space, we have
\[
\beta_C^{-1} \mathrm{cl}(\pi^{-1}(\tau)) = \ff \times \mathrm{cl}\{
(\tau, \zeta )\} \subset \ff \times \overline{\RR^{n + 1}}
\subset \Ttsc^* X,
\]
and thus
\begin{equation}
  \label{eq:tau const fibs}
  \beta_C^{-1} \mathrm{cl}(\pi^{-1}(\tau)) \simeq \ff \times
  \overline{\RR^{n}} \simeq \Tsco^* \RR^n,
\end{equation}
where the equivalence is induced simply by dropping the $\tau$.  This will arise below in the analysis of indicial operators of
three-body operators, which will be $\tau$-dependent families of
scattering operators on $\RR^n$ for which $\pm 1 / \tau$ acts as
semiclassical parameter.

We can now extend the defintion of the radial set to the
$\tsc$-setting for $P_{0}$.  It is given by
\begin{align}\label{eq:tsc radial set}
    \Radtsc \coloneqq \Rad \cap \left( \Ssc^*_{X \setminus \poles} X \cup \Tsc^*_{\pa X  \setminus \poles} X \right) \cup \left( \poles \times \set{ \pm m } \right)
\end{align}
and
\begin{align*}
    \Radtsc_{\sources} &\coloneqq \Rad_{\sources} \cap \left( \Ssc^*_{X \setminus \poles} X \cup \Tsc^*_{\pa X  \setminus \poles} X \right) \cup \Rad_{\ff,\sources}\,,\\
    \Radtsc_{\sinks} &\coloneqq \Rad_{\sinks} \cap \left( \Ssc^*_{X \setminus \poles} X \cup \Tsc^*_{\pa X  \setminus \poles} X \right) \cup \Rad_{\ff, \sinks}\,.
\end{align*}
We have that $\Radtsc = \Radtsc_{\sources} \sqcup \Radtsc_{\sinks}$
and $\Radtsc_{\bullet} \subset \gamma_{\tsc}( \Rad_{\bullet} )$ for $\bullet \in \set{\sources, \sinks}$.

  The coincidence of the boundaries of the  $\{ \tau = const. \}$
  sets in $\Tsco^{*}_{C} X$ will be significant, below, as it will
  force the scattering symbols of indicial operators to be constant in
  $\tau$.  We will denote this common boundary, the fiber equator, by
  $\fibeq$, so, for fixed $\tau_{0} \in \mathbb{R}$, 
  \begin{equation}
    \label{eq:fibeq def}
    \fibeq = \partial \overline{\mathbb{R}^{n + 1}_{\tau, \zeta}} \cap
    \overline{ \{ \tau = \tau_{0}  \} } \subset \partial \overline{\mathbb{R}^{n + 1}_{\tau, \zeta}}.
  \end{equation}
  This set is independent of $\tau_{0}$ and depicted in Figure~\ref{fig:indicial_operator}.  See Section \ref{sec:principal symb
    sec} for further discussion of the role played by $\fibeq$ in
  quantization and as the locus of definition of the fiber symbol of
  the indicial operator.

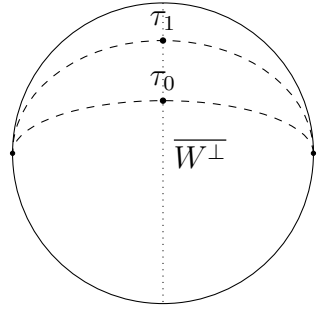
\begin{figure}
    \centering
    \begin{tikzpicture}[scale=2]
        \draw[dashed] (1,0) arc [start angle=0, end angle=180, x radius=1cm, y radius=0.75cm];
        \draw[dashed] (1,0) arc [start angle=0, end angle=180, x radius=1cm, y radius=0.35cm];
        
        \draw (0,0) circle (1cm);
        \filldraw (0,0.75) circle (0.5pt) node [above] {$\tau_1$};
        \filldraw (0,0.35) circle (0.5pt) node [above] {$\tau_0$};
        \filldraw (-1cm,0) circle (0.4pt);
        \filldraw (1cm,0) circle (0.4pt);
        \draw[dotted] (0,1) -- (0,-1);
        \draw (0,0) node [right] {$\Wperpo$};
    \end{tikzpicture}
    \caption{
    The dashed lines are the sets $\beta_C^{-1} \mathrm{cl}(\pi^{-1}(\tau))$ for $\tau = \tau_0, \tau_1$.
    The common boundary of these sets is the fiber equator, $\fibeq$.
  }
  \label{fig:indicial_operator}
\end{figure}

The principal symbol of an element of $\Difftsc^{m}(X)$ therefore has
three pieces corresponding to the three boundary components of $\Ttsco
[X;C]$.  The first two pieces are the two components of scattering
principal symbol, \eqref{eq:scnormsymb} and \eqref{eq:scfibsymb}, while the
new piece is called the indicial operator.   For differential
operators in our context, it is the Fourier transform in time with the
dual variable $\tau$ (dual to $t$) entering as a parameter.  This
piece of the principal symbol then yields a parametrized family of
scattering differential operators on $\ff \simeq \overline{\RR^{n}}$.

Indeed, the components $a_{k, \alpha}$ of a differential operator
restrict to $\ff$ to smooth functions, and the indicial operator for
$L$ as in equation \eqref{eq:tsc diff model} is given by
\begin{align}
\label{eq:tsc diff model indicial}
\ffsymbz(L)(\tau) = \widehat{L}_{\ff}(\tau) \coloneqq \sum_{\abs{\alpha} + k \leq m} (
  a_{k,\alpha} \rvert_{\ff})  \tau^{k} D_z^\alpha\,.
\end{align}
Note that if $L$ were a scattering operator, then the coefficients
$a_{k,\alpha} \rvert_{\ff}$ would be constant at $\ff$.  Thus, when $L
\in \Diffsc$, $\ffsymbz(L)(\tau)$ is translation invariant in $z$, and can
be identified via the Fourier tranform with its total symbol.

More generally, $\ffsymbz(L)(\tau)$ must be regarded as
operator-valued function of $\tau$.  In particular, as we will see in our discussion of
ellipticity below, global ellipticity of a $\tsc$-operator is
equivalent to invertibility of \emph{all three components} of the
principal symbol, in particular it requires invertibility of the
indicial operator for each $\tau$.

The $\tsc$-principal symbol of $P_V$ is
\begin{align*}
\prinsymbz(P_V) = \left( \fibsymbz(P_V), \mfsymbz(P_V),
  \ffsymbz(P_V)(\tau) \right).
\end{align*}
Here, $\fibsymbz(P_V)$ is the ``standard'' interior principal symbol,
i.e.\ the scattering fiber principal symbol, $\mfsymbz(P_V)$ is the
spacetime boundary symbol, and $\ffsymbz(P_V)(\tau)$ is the indicial
operator.  Concretely, $\fibsymbz(P_V)$ and $\mfsymbz(P_V)$ are given
by $(\tau^2 - |\zeta|^2 - m^2) \rho_{\fib}^2$ restricted to momentum
infinity and $\mf$, respectively.  Convenient expressions for these are
\begin{align*}
  \fibsymbz(P_V)(t, z, \tau, \zeta) = \tau^2 - |\zeta|^2,
\end{align*}
while for $\mfsymbz(P_V)$, a function defined at the spacetime boundary away
from $\poles$, 
\begin{align*}
  \mfsymbz(P_V)(y, \tau, \zeta) = \tau^2 - |\zeta|^2 - m^2.
\end{align*}
These functions ``match'' in the sense that if you multiply by
$\rho_{\fib}^2$ they are equal as $\ang{t, z} \to \infty,
\ang{\tau, \zeta} \to \infty.$ For our static potential $V = V(z)$, the \textbf{indicial operator} is simply
\begin{align}\label{eq:indicial operator of PV}
  \ffsymbz(P_V)(\tau) = \tau^2 - (\Delta_z + m^2 + V(z)).
\end{align}
The absence of $V$ in the second component of the symbol is a
consequence of the assumption that $V$ decays in $z$.

In fact, as described below in Section \ref{sec:principal symb sec},
the components of this principal symbol must satisfy matching
conditions on the intersection of their domains.  This is
straightforward for the components $\fibsymbz(P_V)$ and
$\mfsymbz(P_V)$, which are restriction of \emph{functions}, so matching
simply means the values of their restrictions are equal.  The
component $\ffsymbz(P_V)$, however, is a family of scattering
operators, and the matching condition for $P_V$ is that its scattering
symbols are exactly the restrictions of $\fibsymbz(P_V)$ and
$\mfsymbz(P_V)$ to the boundary components of phase space over the front face.  This
is clarified for a general $\tsc$-operator below.
This a generalization of the matching condition that holds for
scattering symbols \eqref{eq:scnormsymb}-\eqref{eq:scfibsymb}.

\subsection{Asymptotically static potentials \texorpdfstring{$V$}{V} and generalizations}
\label{sec:V assump}

Our results apply both to potential functions and to more general
perturbations $V$.  For potential functions, we treat smooth functions
$V$ which approach fixed spatial functions as $t \to \pm \infty$.

We assume in general that 
\begin{equation}
V = V(t, z) \in  \rho_{\mf}  \CI([X; C]; \RR).\label{eq:potential
assump 1}
\end{equation}
Thus (recalling $\rho_{\mf} = \ang{z}^{-1}$) these are
potential functions which are smooth on the whole of $[X; C]$, vanish
at least to order one as $z \to \infty$ (i.e.\ at $\mf$) and have, in
general, non-vanishing limits at $\ff$.  

In particular, we have, in regions with $|z/t| < C$, 
\begin{align*}
  V = V_+(z) + V'(t, z), \quad V' = O(1/t) \mbox{ as } t \to + \infty
\end{align*}
where $V_+$ is a symbol of order $-1$, meaning
\begin{align*}
\abs{ \pa_z^\alpha V_+(z)} &\lesssim_{\alpha}  \ang{z}^{-\abs{\alpha}}
\end{align*}
and $V'$ satisfies the following estimates
\begin{align*}
\abs{ \pa_t^k \pa_z^\alpha V'(t,z)} &\lesssim_{k,\alpha} \ang{t,z}^{-1-k} \ang{z}^{-\abs{\alpha}}\,,
\end{align*}
The estimates here are equivalent to the containment \eqref{eq:potential
assump 1} if you assume in addition that the $V_+$ and $V'$ has
asymptotic expansions in $1/t$ and $\ang{z}^{-1}$.
A simple example of such a potential is a
$V = V(t, z) \in \CI(\RR_t ; \schwartz(\RR^n))$,
with $V(t, z) \equiv V_\pm (z) \in \schwartz(\RR^n)$ for
$\pm t \gg 0$.

More generally, for complex-valued $V$, we assume assume that $V \in \rho_{\mf}  \CI([X; C]; \CC)$ with
\begin{equation*}
    2 \Im V =   V - \overline{V}  \in \ang{t, z}^{-2} \CI([X; C]; \CC).  
\end{equation*}
This ensures in particular that the ``subprincipal symbol'' does not
influence the threshold weight of $-1/2$  that appears in Theorem \ref{thm:basic propagator result}.  In terms of the asymptotic
decomposition above, this means that 
\begin{align*}
  V = V_+(z) + V'(t, z) + i \Im V, 
\end{align*}
where $V_+$ and $V'$ are as above and
\begin{align*}
              \abs{ \pa_t^k \pa_z^\alpha \Im V (t,z)} &\lesssim_{k,\alpha} \ang{t,z}^{-2-k} \ang{z}^{-\abs{\alpha}}\,.
\end{align*}

In fact, our results more generally, including to differential and
pseudodifferential $V$ with suitable regularity and decay hypotheses,
and with an assumption on the subprincipal symbol $V - V^*$ which
generalizes the assumption on $\Im V$ above.

\section{The three-body scattering calculus for
  Klein-Gordon}\label{sec:3body klein}

In this section we will recall the key features of the $\tsc$-calculus
adapted to our setting.  As noted above, our $P_V$ is not a scattering
operator on the whole of $\RR^{n + 1}_{t, z}$ in the sense of
Melrose, but it is a $\tsc$-operator in the sense of Vasy.  We will
now describe what that means in detail, what the $\tsc$-operators and
their features look like in our setting, and the basic properties that
inform our analysis.  

Generally, the $\tsc$-calculus introduced
by Vasy \cite{V2000}, is defined with respect to data which includes
both the total space and collision planes.  In contrast with the
general case, in our setting, we have only the $\{ z = z_0 \}$ collision
planes (really lines) for $z_0 \in \RR^n$ fixed, corresponding to the points $\poles$ on the boundary at
infinity.

We can summarize the main features of this introductory section to the
$\tsc$-calculus as follows.
\begin{itemize}
\item The $\tsc$-operators in our setting are the natural pseudodifferential
  generalization of the $\tsc$-differential operators defined above,
  exactly in the standard sense that they are quantizations of symbols
  whose behaviour is analogous to the behavior of the total symbols of
  elements of $\Difftsc^{m,r}$.  These pseudodifferential operators
  are denotes $\Psitsc^{m,r}$ when they have differential order $m$
  and spacetime weight order $r$.
\item As with $P_V$ \eqref{eq:indicial operator of PV}, the principal
  symbols of these operators have three components.  The first two,
  like the scattering symbols, are local, i.e.\ they are functions.
  The third is the \textbf{indicial operator}, which has both a
  $t = + \infty$ and a $t = - \infty$ component itself, is defined
  only from data that lives over $\poles$, and like the symbol of
  $P_V$ for static $V$ in \eqref{eq:indicial operator of PV}, is a
  family of operators parametrized by $\tau$, the dual variable to
  $t$.  Given $A \in \Psitsc^{m,r}$ and focusing, as we do below, on
  $\NP$, the indicial operator is denoted
  $\ffsymb{r}(A) = \ffsymb{r}(A)(\tau)$. \textit{The principal symbol is
  multiplicative,} in the sense that the indicial operator of the
  composition of two $\tsc$-operator is the composition of the
  indicial operators.  The behavior of $\ffsymb{r}(A)(\tau)$ in $\tau$
  is semiclassical as $\tau \to \pm \infty$.
\item Ellipticity is still appropriately construed as
  invertibility of the principal symbol.  Namely, global ellipticity
  is exactly the assumption that the first two components of the
  principal symbol (which are functions) are non-zero, and that the
  indicial operator is semiclassically elliptic and invertible for all
  $\tau$.  
\item There is also an appropriate notion of microlocal
  $\tsc$-ellipticity, and corresponding $\tsc$-microlocal elliptic
  estimates.
\item As in Vasy's treatment, we use the standard $L^2$-based Sobolev spaces
on $\RR^{n + 1}$.  (In particular, we do not introduce spaces
specifically adapted to the $\tsc$ setting.)  Thus, our Sobolev spaces
are exactly those used in Section \ref{sec:model-case} above.  All the
estimates we state and prove are for distributions in the scattering
(i.e.\ standard!) weighted Sobolev spaces $\Sobsc^{s,\ell}(\RR^{n +
  1}_{t, z})$. 
\end{itemize}

Moreover, in this section we confront perhaps this most striking
difference between $\tsc$ and $\sc$ operators, namely that general commutators
of $\tsc$ operators do not have the standard loss of one order in
comparison to composition.  Namely, the analogue of \eqref{eq:scattering-commutator-hammy} for $\tsc$-operators fails.  Indeed, a static, potential function
$V(z) \in \schwartz(\RR^n)$ lies in $\Psitsc^{0,0}(X)$ (because it does
not decay in time) and the partial derivative in a spatial coordinate
$\pa_{z_j} \in \Diffsc^{1}(X) \subset \Psitsc^{1,0}(X)$ while the commutator
$[\pa_{z_j}, V] = \pa_{z_j} V$ has no additional time decay, and thus
  one concludes only that $[\pa_{z_j}, V] \in \Difftsc^{0,0}(X) \subset
    \Psitsc^{0,0}(X)$.  Regarding commutators, we make the additional point.
    \begin{itemize}
    \item As in the examples just discussed, if
      $A \in \Psitsc^{m_1, r_1}$ and $B \in \Psitsc^{m_2, r_2}$, then
      in general $[A, B] \in \Psitsc^{m_1 + m_2 - 1, r_1 + r_2}$.  If either one
      of $A$ or $B$ satisfies a ``centrality condition'', which is
      essentially that $\ffsymb{r}(A)(\tau)$ is a function (as opposed
      to an operator) then in fact
      $[A, B] \in \Psitsc^{m_1 + m_2 -1, r_1 + r_2 -1 }$.  In case
      this centrality condition is satisfied, a formula for the
      principal symbol of $[A, B]$ is given.  See Section
      \ref{sec:commutators}.
    \end{itemize}

\bigskip

Our treatment is simplified in comparison to the general case of the
$\tsc$-calculus, in which more complex arrangements of collision planes are treated.
In particular, we present a simplified commutator formula for the
indicial operator below.  More important is the fact that
in our setting, as we treat a hyperbolic operator, the characteristic
set extends to fiber infinity; we must therefore discuss the behavior
of the indicial operator for large $\pm \tau$, and we do so below in
our treatment of the indicial operator as a semiclassical scattering operator.

\subsection{\texorpdfstring{$\tsc$-pseudodifferential operators}{3sc-pseudodifferential operators}}
The space of (classical) three body scattering symbols is
\begin{equation}
  \label{eq:symbol space}
  \Stsc^{m,r}(X; \poles) \ = \rho_{\fib}^{-m} \rho_{\infty}^{-r} \CI(\Tsco^*[X;C])
\end{equation}
and the space of (classical) $\tsc$-pseudodifferential operators of
order $m,r$ is 
\begin{equation}
  \label{eq:3}
  \Psitsc^{m,r} = \Op_L(\Stsc^{m,r})
\end{equation}
by \cite[Lemma 3.5]{V2000}, where, concretely, for $a \in \Stsc^{m,r}$,
\[
A = \Op_L(a) = \int e^{i (t - t') \tau + i (z - z') \cdot \zeta} a(x, z, \tau, \zeta)
d\tau d\zeta
\]
(This is taken as the definition of
$\Psitsc^{m,r}$, whereas in the cited paper the space of PsiDO's is
defined as an appropriate set of integral on the three-scattering
double space; we do not use this latter definition directly in our
work.)
In the original $(t,z)$ coordinates, a $\tsc$-symbol is a smooth
function on $\RR^{n+1}_{t,z} \times \RR^{n+1}_{\vartheta}$ such that
each seminorm
\begin{align*}
\norm{a}_{\tsc, M}  = \sum_{k + |\alpha| + |\beta| \le M}  \sup_{t, z,
  \vartheta}    \ang{t,z}^{k} \ang{z}^{- r+ \abs{\alpha}}
  \ang{\vartheta}^{- m + \abs{\beta}}\, \abs{ \pa_t^k \pa_z^\alpha
  \pa_\vartheta^{\beta} a(t,z, \vartheta)}
\end{align*}
is finite.

\begin{remark}
  The spaces of classical scattering and $\tsc$-symbols can also be
  defined by reference to the compactified spaces.  Namely,
  $\Ssc^{0,0}(\RR^{n+1})$ are exactly the smooth functions on the
  compactified scattering cotangent bundle $\Tsco^* X$ while
  $\Stsc^{0,0}(\RR^{n + 1})$ is exactly the smooth functions
  on the compactified $\tsc$-cotangent bundle $\Ttsco^* X$.
\end{remark}

To define the indicial operator, we need a lemma to the effect that a $\tsc$-operator defines an
operator on $\ff$ via extension and restriction to the boundary.  From \cite[Corollary 3.4]{V2000} we have the mapping properties for
smooth functions (the differential order plays no significant role
here): 
\begin{lemma}\label{lem:mapping_prop}
    If $A \in \Psitsc^{m,r}(X)$, then
    \begin{align*}
        A : \CdI(X) &\lra \CdI(X)
        \intertext{and}
        A : \rho_{\mf}^k \rho_{\ff}^{k'} \CI([X; C]) &\lra \rho_{\mf}^{k + r} \rho_{\ff}^{k' + r} \CI([X; C])\,.
    \end{align*}
\end{lemma}
We also note that $\sc$-operators are $\tsc$-operators
\begin{align*}
\Psisc^{m,r}(X) \subset \Psitsc^{m,r}(X)\,,
\end{align*}
since the scattering symbol estimates imply in particular the
$\norm{\bullet}_{\tsc, M}$ estimates above.  (All our operators are
assumed classical throughout.)

We recall the main boundedness property for
$\tsc$-operations, which is proven using the standard square-root trick.
\begin{proposition}[\cite{V2000}, Cor.\ 8.2]\label{prop:bounded_Sobolev}
  For $A = \Op_L(a) \in \Psitsc^{m,r}$ and $s, \ell \in \RR$,
  \begin{equation}
  \label{eq:mapping on scat spaces}
  A \colon \Sobsc^{m + s, r + \ell}(\RR^{n + 1}) \lra \Sobsc^{s, \ell}(\RR^{n + 1}).
\end{equation}
is bounded, with operator norm bounded by a seminorm 
$\| a \|_M$.
\end{proposition}

Having introduced the total symbols of $\tsc$-operators, we now
consider the appropriate definition of their principal symbols.
Indeed, recall from Section \ref{sec:scattering-calc} that for a scattering operator, $\Op_L(a) = A \in
\Psisc^{m,l}(X)$, the two components of the principal symbol
$\scprinsymb{m,r}(A) = (\scfibsymb{m,r}(A), 
\scnormsymb{m,r}(A))$ are the restriction sof the function
$\ang{\tau, \zeta}^{-m} \ang{t, z}^{-r} a$ to the two components of
$\partial \Ttsco^* X$.

In principal, one could define the $\tsc$-principal symbol of $\Op_L(a) = A \in
\Psitsc^{m,r}$ to be the restriction of $a$ to the four components of
the boundary of $\Ttsco^* X$.  However, such a definition would have
the limitation that it
would not be multiplicative over $\ff$, a limitation which is
addressed by using, instead of the front face restriction, the family
of indicial operators mentioned above and described in
Section~\ref{sec:principal symb sec} and Section~\ref{sec:norm fam} below.

When $A = \Op_L(a) \in \Psitsc^{0,0}$, the symbols over $\mf$ and at
fiber infinity are simply the boundary component restrictions
\begin{align*}
    \mfsymbz(A) =  a \rvert_{\Ttsco_{\mf}^* [X; C]} \in \CI(\Ttsco_{\mf}^* [X; C])
\end{align*}
and
\begin{align*}
    \fibsymbz(A) = a \rvert_{\Stsc^* [X;C]} \in \CI(\Stsc^* [X; C]),
\end{align*}
on in terms of the boundary defining functions in~\eqref{eq:bdfs}, they
are the restrictions on $\rho_{\mf} = 0$ and $\rho_{\fib} = 0$,
respectively.  This is the $\tsc$ generalization of the two scattering
principal symbol components \eqref{eq:scnormsymb}-\eqref{eq:scfibsymb}.
When $A = \Op_L(a) \in \Psitsc^{m,r}$ and $a$ needs to be
re-weighted to have boundary restrictions, then we use the boundary
defining functions most suitable to our analysis.  We proceed to
formalize this now.

\subsection{The principal symbol and the indicial operator}\label{sec:principal symb sec}

Generalizing the indicial operator $\ffsymbz(P_V)$ \eqref{eq:indicial operator of PV}, we describe now how $A \in \Psitsc^{m,r}$ has an indicial operator
$\ffsymb{r}(A) = \ffsymb{r}(A) (\tau)$, a smoothly-parametrized family of scattering
operators on $\ff$.  The indicial family is one of the three components of the
total $\tsc$-principal symbol of $A$, together with the (rescaled)
restrictions of the symbol of $A$ to $\mf$ and fiber infinity
\cite[Chap.\ 6]{V2000}. The main goals of this subsection are three-fold:
\begin{itemize}
\item to define the indicial family $\ffsymb{r}(A) $ and
  show that it is the quantization of a boundary restriction of the
  symbol of $A$,
\item to prove that $\ffsymb{r}(A)$ is in fact a
  semiclassical scattering operator, and to characterize
  those semiclassical scattering operators which arise as indicial
  operators of $\Psitsc$ operators,
\item to recall the $\tsc$-principal symbol $\prinsymbz$, show that it
  is multiplicative, and define left-quantization on appropriate
  principal symbols.  This is Proposition
  \ref{prop:principal_symbol_map} below.
\end{itemize}
The properties of the indicial operator are simplified in our case, in
comparison with the general $\tsc$-calculus, due $\poles$ being
zero-dimensional, and correspondingly $\Wperpo$ being one-dimensional.
Thus, in particular $\ffsymb{r}(A)$ depends on a single parameter.
What we develop here is the precise sense in the operator is
semiclassical, and the precise sense in which a special class of
semiclassical scattering operators over $\ff$ can be quantized into
$\tsc$-operators.

Also note that, while we typically work near $\NP$ for brevity, there
are in fact two components of the indicial operator corresponding to
the two boundary hypersurfaces $\ff = \ff_+,  \ff_-$, and the
indicial operator in fact is two separate families, one over $\NP$ and
one over $\SP$.  We often elide separate discussion of the $\SP$
component since the details are nearly identical to that of the $\NP$ component.

Beginning with $A \in \Psitsc^{m,0}$, we assume here and throughout
that $A = \Op_L(a)$ for a classical symbol $a$.  
By \cite[Cor. 3.4]{V2000}, if $u \in \CI([X;\poles])$ then
$Au \in \CI([X;\poles])$.  If $f \in \CI([X;C])$ and
$u \rvert_{\ff} = f$, we define
\begin{align*}
    A_{\partial} f \coloneqq (A u)\ \rvert_{\partial [X;C]},
\end{align*}
which, by Lemma~\ref{lem:mapping_prop}, is independent of the
extension $u$ of $f$.  In particular, identifying $\ff$ with $\RR^n$, if
$f \in \mathcal{S}(\RR^n)$, then the operator
\[
A_{\ff} f \coloneqq A_\partial(f) \rvert_{\ff}
\]
is well defined.

Choosing $\tau_0 \in \RR$, by Vasy's work~\cite[Lemma 6.1]{V2000},
\[
e^{ i \tau_0 / x}A e^{- i \tau_0/x} \in \Psitsc^{0,0}
\]
which gives the definition of the indicial operator at $\tau_0$, namely:
\begin{equation}
  \label{eq:ff boundary symbol}
\ffsymbz(A)(\tau_0) \coloneqq   \left( e^{ i \tau_0 / x}A e^{- i
      \tau_0/x}  \right)_{\ff}.
\end{equation}
(In particular, $A_{\ff} = \hat{A}_{\ff}(0)$.)  We write
\begin{equation}
  \label{eq:ff zero notation}
  \widehat{A}_{\ff}(\tau_0) =   \ffsymbz(A)(\tau_0) 
\end{equation}

Given $A \in \Psitsc^{m,r}$ we have $x^{r} A \in \Psitsc^{m,0}$ and 
define
\begin{equation}\label{eq:general ff boundary symbol}
\ffsymb{r}(A)(\tau_0) = \widehat{x^r A}_{\ff}(\tau_0),
\end{equation}
so in particular for $A \in \Psitsc^{0,0}$, $\widehat{A}_{\ff, 0} = \widehat{A}_{\ff}$.
We also write for short
\begin{align*}
    \hat{A}_{\ff, r} \coloneqq \ffsymb{r}(A)\,.
\end{align*}

Given $A = \Op_L(a)$,
we now derive a simple formula for $\hat{A}_{\ff, r}$ in terms of $a$.
Indeed in Lemma~\ref{thm:indicial family explicit} below we see that $\hat{A}_{\ff, \ell}$ is the left quantization in $z$ of 
an appropriate boundary restriction of $a$ to $\ff$.

Vasy~\cite[p.~23]{V2000} shows that $A_{\ff}$ can be obtained by
restriction of the integral kernel of $A \in \Psitsc^{m,0}$ to a
boundary hypersurface $\mathrm{sf}_C$ of the scattering triple
space $X^2_{\tsc}$; although we do not discuss this space in detail
here, we provide an illustrative formal computation.  In particular, in the coordinates, 
\[
x = 1/t,\quad   S = (x - x')/x^2, \quad z, \quad Y = (y-y')/x, 
\]  
the kernel of $A$ is conormal to the diagonal $S = 0,  Y = 0$
smoothly down to $x = 0$.
\begin{equation}
  \label{eq:A and the boundary}
  \begin{split}
    A u &= \int K_A(t, z, t', z') u(t', z') dt' dz' \\
    &= \int K_A(t, z, S/(1-xS), \frac{z - Y}{1 - xS}) u\left(\frac{S}{1-xS}, \frac{z - Y}{1 - xS}
    \right)  \frac{1}{(1 - xS)^{n + 2}}dSdY 
  \end{split}
\end{equation}
Then defining
\[
A(x, S, z, Y) = (1 -
xS)^{-(n + 2)} K_A(t, z, \frac{S}{1-xS}, \frac{z - Y}{1 - xS})
\]
on $x = 0$, we obtain, as in \cite[Eq.\
4.15]{V2000}
\[
A u \rvert_{x = 0} = A_{\ff} u \rvert_{x = 0} = \int A(0, S, z, Y)
u(0, z - Y) dS dY,
\]
which is to say that, as a Schwartz kernel,
\[
A_{\ff} = \int A(0,S,z,Y) dS
\]

The functions
\[
x, \quad \tilde t = t - t', \quad z, \quad \tilde z = z - z'
\]
are also coordinates near $\mathrm{sf}_C$.  (These are the $W_a, W^a$
functions from \cite{V2001}.)  At $x = 0$ we have $S = \tilde t$ and
$\tilde z = Y$.
If $\Op_L(a) = A$, then, with $\aff(z, \tau, \zeta) =  a
\rvert_{\Ttsco_{\ff} X}$ and $a_0(x, z, \tau, \zeta) = a(1/x, z, \tau, \zeta)$
\begin{equation}
  \begin{split}\label{eq:7}
    K_A &= \int e^{i \tilde t \tau + i \tilde z \cdot \zeta} a(t, z,
    \tau, \zeta) d\tau d\zeta  \\
    &= \int e^{i \tilde t \tau + i \tilde z \cdot \zeta} a_0(x, z,
    \tau, \zeta)d\tau d\zeta  \\
    \implies A_{\ff} &= \int e^{i \tilde t \tau + i \tilde z \cdot \zeta} \aff(z,
    \tau, \zeta) d\tau d\zeta dt = \int e^{i \tilde z \cdot \zeta} \aff(z,
    0, \zeta) d\zeta.
  \end{split}
\end{equation}
Conjugating by $e^{i \tau / x}$ we then obtain:
\begin{lemma}\label{thm:indicial family explicit}
  Let $A = \Op_L(a)$ for $a \in \Stsc^{m, r}$, and denote the
  (rescaled) restriction of $a$ to $\ff$ by
\begin{equation}
  \label{eq:ff restrict}
  \aff \coloneqq (x^r a) \rvert_{\Ttsco_{\ff} [X; \poles]},
\end{equation}
Then $\hat A_{\ff}(\tau_0) \in \Psisc^{m,0}(\RR^n_z)$ and has kernel
\begin{equation}
\widehat{A}_{\ff}(\tau_0) = \int e^{i (z - z') \cdot \zeta} \aff(z,
\tau_0, \zeta) d\zeta = \Op_{L,z} (\aff(z, \tau_0, \zeta)),\label{eq:norm
  fam explicit}
\end{equation}
\end{lemma}
\begin{proof}
We simply refer to \cite[Chap.\ 6]{V2000} where it is shown the
formal computations given above agree with the value of the operator.
\end{proof}

We need to compute the next term in the expansion of $Au(x, z)$ as $x
\to 0$.  In terms of the distribution $A(x, S, z, Y)$ above, this is
straightforward.  Indeed, we have \cite[Lemma 7.1]{V2000}
\begin{equation}
  \label{eq:Vasy thing}
  Au = A_{\ff} u + x\left( (\partial_x A)_{\ff} u + A_{\ff} \partial_x
  u - D_\tau \hat{A}_{\ff}(0) (z \cdot \partial_z u)\right) + O(x^2).
\end{equation}
Here $\partial_x A$ denotes the derivative of $A(x, S, z, Y)$ (in
these coordinates) restricted to $x = 0$, and \eqref{eq:Vasy thing} is
derived by simply applying $\partial_x$ to \eqref{eq:A and the boundary} and using the chain rule.
The last two terms are computed from boundary values of $A$, and thus
from the restricted symbol $\aff$.  The term $\partial_x A$, however,
depends on the interior values of the symbol.

If $u \in x^{-\ell}\CI_{\ff}$, we set
\begin{align*}
    u_{\ff} &\coloneqq (x^\ell u)|_{\ff}\,,\\
    u'_{\ff} &\coloneqq \pa_x (x^\ell u)|_{\ff}\,,
\end{align*}
so that
\begin{align*}
    u = x^{-\ell} \left( u_{\ff} + x u_{\ff}' + O(x^2) \right)
\end{align*}
as $x \to 0$.

We also provide a more intuitive form in terms of symbols.
If $A = \Op(a)$, where $a = a(x, z, x', z', \tau, \zeta)$
\begin{align*}
    A_{\ff, r}' \coloneqq \Op( x^{r} (\pa_x a + \pa_{x'} a)|_{x = x' = 0} )\,.
\end{align*}
In particular, if $A = \Op_L(a)$, then $A_{\ff, r}' = \Op_L( x^{r} \pa_x a|_{x = 0})$.

Then, the expansion takes a more natural form.
\begin{lemma}\label{lem:expansion_symbol}
    Let $A = \Op(a)$ and $u \in x^{-\ell} \CI_{\ff}$, then
    \begin{align*}
        x^{\ell + r} (Au)(x,z) = A_{\ff} u_{\ff} + x \left( A_{\ff}' u_{\ff} + (\ell - 1)D_{\tau} \hat{A}_{\ff}(0) u_{\ff} + A_{\ff} u_{\ff}' \right) + O(x^2)\,.
    \end{align*}
\end{lemma}
\begin{proof}
    We use different coordinates than Vasy, but one can check that the result is the same. 
    Specifically, we take
    \begin{align*}
        x, z, S = \frac{1}{x} - \frac{1}{x'}\,, Y = z - \frac{z'}{1 - xS}\,,
    \end{align*}
    and we calculate that
    \begin{align*}
        x' = \frac{x}{1-xS}\,, z - z' = Y + xS(z-Y)\,, z' = (1 - xS)(z-Y)\,.
    \end{align*}
    We can write $(Au)(x,z)$ in these coordinates as
    \begin{align*}
        x^{\ell + r} (Au)(x,z) = \int \tilde{A}(x, z, S, Y) \tilde{u}\left( \frac{x}{1-xS}, (1 - xS)(z - Y)\right)\, dS\, dY\,,
    \end{align*}
    where
    \begin{align*}
        \tilde{u}(x', z') = (x')^\ell u(x', z')\,,
    \end{align*}
    and the integral kernel $\tilde{A}$ is given by
    \begin{align*}
        \tilde{A} = (2\pi)^{-(n+1)} \int e^{i\varphi} x^r (1 - xS)^r a\left(x, z, \frac{x}{1-xS}, (1 - xS)(z-Y), \tau, \zeta\right) \,d\tau \, d\zeta
    \end{align*}
    with $\varphi = S \tau + Y \zeta + x S(z-Y) \zeta$.

    We also set
    \begin{align*}
        \tilde{a}(x,z,x',z', \tau, \zeta) \coloneqq x^r a(x,z,x',z',\tau, \zeta)\,.
    \end{align*}

    We have that
    \begin{align*}
        \tilde{u}(0,z') &= u_{\ff}(z')\,,\\
        \pa_{x'} \tilde{u}(0,z') &= u'_{\ff}(z')
        \intertext{and}
        \tilde{A}(0,z,S,Y) &= (2\pi)^{-(n+1)} \int e^{iS\tau + iY\zeta} \tilde{a}(0, z, 0, z-Y, \tau, \zeta) \,d\tau\,d\zeta\,.
    \end{align*}
    
    We calculate the derivative of $Au$ as
    \begin{align*}
        \pa_x (x^{\ell + r} Au)(0,z) &= \int \pa_x \tilde{A}(0, z, S, Y) \tilde{u}(0,z-Y) \, dS \,dY \\
        &\phantom{=} + \int \tilde{A}(0, z, S, Y) (\pa_x - S(z-Y) \pa_z) \tilde{u}(0, z- Y) \, dS \,dY\,.
    \end{align*}
    By a straight-forward calculation using integration by parts, we obtain
    \begin{align*}
        \pa_x (x^{\ell + r} Au)(0,z) &= \int e^{i(S\tau + Y\zeta)} \left( (\pa_x + \pa_{x'}) \tilde{a}(0,z,0,z-Y,\tau, \zeta) \right) \tilde{u}(0,z-Y) \, d\tau \, d\zeta \, dS \, dY\\
        &\phantom{=} + (\ell - 1) \int e^{i(S\tau + Y\zeta)} D_\tau \tilde{a}(0,z,0,z-Y,\tau, \zeta) \tilde{u}(0,z-Y) \, d\tau \, d\zeta \, dS \, dY\\
        &\phantom{=} + \int e^{i(S\tau + Y\zeta)} \tilde{a}(0,z,0,z-Y,\tau, \zeta) \pa_x \tilde{u}(0,z-Y) \, d\tau \, d\zeta \, dS \, dY \\
        &= A_{\ff}' u_{\ff} + (\ell - 1) D_\tau \hat{A}_{\ff}(0) u_{\ff} + A_{\ff} u_{\ff}'\,.
    \end{align*}
\end{proof}

Let $Q = \Op_L(x^{-r} q)$ for $q \in \Ssc^{0,0}$.
The first term in the expansion is given by
\begin{align*}
    Q_{\ff} = \Op_L(q(0,0,0,\zeta))\,.
\end{align*}
We have that $x^r \pa_x x^{-\ell} q(x, xz) = \pa_x q(x, xz) + z \pa_y q(x, xz) + O(x)$ and consequently
\begin{align*}
    Q_{\ff,r}' = \Op_L(\pa_x q(0,0,0,\zeta) + z \pa_y q(0,0, 0, \zeta))\,.
\end{align*}
Importantly, $Q_{\ff,r}'$ is \emph{not} a Fourier multiplier.

In the case of the free Klein--Gordon operator, we have for $u \in \CI_{\ff}$ that
\begin{align*}
    P_0 u = -(\Delta + m^2) u_{\ff} - x (\Delta + m^2) u_{\ff}' + O(x^2)\,.
\end{align*}

\medskip

Using this development of the indicial operator, we can now define the \textbf{principal symbol of a $\tsc$ operator},
essentially by putting the indicial operator together with the $\mf$ and fiber infinity parts
of the scattering principal symbol. Concretely, for $A = \Op_L(a) \in
\Psitsc^{m,r}$, we have the two restrictions
\begin{align*}
    \mfsymb{m,r}(A) = \ang{\tau, \zeta}^{-m} x^r a \rvert_{\Ttsco_{\mf}^* [X; C]} \in \CI(\Ttsco_{\mf}^* [X; C])
\end{align*}
and
\begin{align*}
    \fibsymb{m,r}(A) = \ang{\tau, \zeta}^{-m} x^r a \rvert_{\Stsc^* [X;C]} \in \CI(\Stsc^* [X; C])
\end{align*}
and the $\tau$-dependent family of operators
\begin{align*}
    \ffsymb{r}(A) \in \CI(\RR_\tau ; \Psisc^{m,0}(\ff)).
\end{align*}
These three objects jointly form the \textbf{$\tsc$-principal symbol:}
\begin{equation} \label{eq:actual symbol}
    \begin{aligned}
        \prinsymb{m,r} : \Psitsc^{m,r} &\lra \CI(\Ttsco_{\mf}^* [X; \poles]) \times \CI(\Stsc^* [X;C]) \times \CI(\RR_\tau ; \Psisc^{m,0}(\ff)) \\
        A &\mapsto \left(\fibsymb{m,r}(A) , \mfsymb{m,r}(A) , \ffsymb{m,r}(A) \right)
    \end{aligned}
\end{equation}
Note that we have dropped the inclusion of the $\ff_-$ component of
the indicial operator as otherwise the notation becomes too
cumbersome.

We note that $\ffsymb{r}(A)$ is 
not an arbitrary element in $\CI(\RR_\tau; \Psisc^{m,0}(\ff))$.
Indeed, if we define the space of symbols
\begin{align*}
    S^{m,r}(\RR^n_z ; \RR^{n + 1}_{\tau, \zeta})
    = \curl{ a \in \CI(\RR^n ; \RR^{n + 1}) : |\ang{z}^{-r + |\alpha|} \ang{\tau, \zeta}^{-m + j + |\beta|} D_z^\alpha D_\tau^j D_\zeta^\beta a | < \infty}
\end{align*}
Then $\aff = x^{-r} a \rvert_{\ff}$ is a classical symbol, i.e. in
$S_{\cl}^{m,0}(\RR^n_z ; \RR^{n + 1}_{\tau, \zeta}) =
\ang{\tau, \zeta}^m
\CI(\overline{\RR^n_z} \times \overline{\RR^{n +
    1}_{\tau, \zeta}})$ and
\begin{align*}
    \ffsymb{r}(A) \in \Op_z(S^{m,0}_{\cl}(\RR^n_z ; \RR^{n + 1}_{\tau, \zeta}))
\subset \CI(\RR_\tau ; \Psisc^{m,0}(\ff)).
\end{align*}

In Section \ref{sec:norm fam} below, we discuss the consequences
of the fact that the indicial operator $\ffsymb{r}(A)$ is in fact a semiclassical-scattering operator in $h = \pm
1/\tau$ as $\tau \to \pm \infty$.  These
are exactly the symbols, defined for $m,r,k \in \RR$, by
\begin{equation}
  \label{eq:norop symbols}
  S_{\scl, \sc, \pm 1/\tau}^{m,r,k}(\RR^n) = \left\{ a \in
    \CI(\RR^n_z \times \RR^n_\mu \times
    \RR_\tau ) : |D_z^\alpha D_\mu^\beta D_h^j a| \lesssim
    \ang{\mu}^{m - |\beta|} \ang{z}^{r - |\alpha|}
    h^{-k} \ \forall \alpha, \beta, j  \right\},
\end{equation}
where $h = 1 / \ang{\tau}$.  In our case we often write
$S_{\scl, \sc, \pm 1/\tau}^{m,r,k}(\ff)$ as the operators we consider
are more naturally viewed as functions on $\ff =
\overline{\RR^n_z}$.  Then their quantizations are
\begin{equation}
  \label{eq:norop}
  \Norop^{m,r,k} = \Op_{L, \scl}( S_{\scl, \sc, \pm 1/\tau}^{m,r,k}).
\end{equation}
This definition admits an obvious generalization to the case that
$\RR^n$ is an arbitrary scattering manifold is
straightforward.  In Proposition~\ref{prop:comp of ff symb} below we show
that for $A \in \Psitsc^{m,r}$, $\ffsymb{r}(A) \in \Norop^{m,0,m}$.

We note that, for $\Op_L(a) = A \in \Psitsc^{m,r}$,
the indicial operator $\ffsymb{r}(A)$ is equivalent to
the data $\aff = x^{-r} a \rvert_{\Ttsco^*_{\ff} X}$
(since $\aff (z, \tau, \zeta)$ is the left reduction of $\ffsymb{r}(A)(\tau)$).
Thus $\prinsymb{m,r}(A)$, like the scattering symbol in
Section \ref{sec:model-case}, is still given by the restriction data
of the (appropriately weighted) symbol.  Thus, it is automatic that the
three components of $\prinsymb{m,r}(A)$ satisfy \emph{matching conditions}
at the intersections of the boundary hypersurfaces of
$\Ttsco^*[X;C]$, i.e.\ 
$\ang{\tau, \zeta}^{-m} \aff (z, \tau, \zeta)$,
$\mfsymb{m,r}(a)$ and $\fibsymb{m,r}(a)$ are equal on the
restriction to their common boundaries (see \eqref{eq:matching conditions}.)

Conversely, we have the following proposition, which tells us that we
can quantize a $\tsc$-principal symbol provided the appropriate
matching conditions of the three symbol components are satisfied,
where again we do not include the $\ff_-$ component.
\begin{proposition}\label{prop:3sc_quantization} 
    Let
    \begin{align*}
        (a_1, a_2, A_\tau) \in \CI(\Stsc^*[X;C]) \times \CI(\Ttsco^*_{\mf} X) \times \Norop^{m,0,m}
    \end{align*}
    and let $\Op_{L,z}(a_0) = A_\tau$, i.e.\ let $a_0(z, \tau, \zeta)$ be the
    $\tau$-dependent family of symbols quantizing $A_\tau \in
    \Psisc^{m,0}(\RR^n)$.  Then there is $A \in
    \Psitsc^{m,0}(\RR^{n + 1})$ with
    $\prinsymb{m,0}(A) = (a_1, a_2, A_\tau)$ if and only if
    $\ang{\tau, \zeta}^{-m} a_0 \in \CI(\ff \times \overline{\RR^{n + 1}_{\tau, \zeta}})$ and 
    \begin{equation}
    \begin{gathered}
        a_1 \rvert_{\Stsc^*_{\mf} [X;C]} = a_2 \rvert_{\Stsc^*_{\mf} [X;C]},
        \quad \ang{\tau, \zeta}^{-m} a_0 \rvert_{\Stsc^*_{\ff} [X;C]} = a_1 \rvert_{\Stsc^*_{\ff} [X;C]} \label{eq:matching conditions} \\
        \ang{\tau, \zeta}^{-m} a_0 \rvert_{\Ttsco^*_{\ff \cap \mf} [X;C]} = a_2 \rvert_{\Ttsco^*_{\ff \cap \mf} [X;C]}
    \end{gathered}
    \end{equation}
\end{proposition}
This follows from the simple fact that the agreement of the functions
$a_1, a_2$ and $\ang{\tau, \zeta}^{-m} a_0$ implies the
existence of a function $a$ extending all three functions; then
$A = \Op_L(\ang{\tau, \zeta}^{m} a)$ is the desired
operator.  Note that in this case, with $\aff$ in \eqref{eq:ff restrict} that
\[
a_{0} = \aff.
\]

We now discuss multiplicativity of $\prinsymbz$, using the composition theorem~\cite[Proposition 5.2]{V2000}:
\begin{proposition}\label{prop:comp of ff symb}
    Let $A \in \Psitsc^{m_1,r_1}(X)$ and $B \in \Psitsc^{m_2, r_2}(X)$. The composition is well-defined as an operator
    \begin{align*}
        A \circ B \in \Psitsc^{m_1 + m_2, r_1 + r_2}(X)
    \end{align*}
    and
    \begin{align*}
        (A \circ B)_\partial = A_\partial B_\partial\,.
    \end{align*}
\end{proposition}
Applying the latter equation to the conjugated operators in
\eqref{eq:ff boundary symbol}, we see that the indicial operator of a
composition is the composition of the indicial operators.  Summarizing
these considerations, we obtain the basic features of the
$\tsc$-principal symbol.  
\begin{proposition}\label{prop:principal_symbol_map}
    The kernel of the principal symbol mapping
    \eqref{eq:actual symbol} is exactly $\Psitsc^{m-1,r-1}$, while the image is the set of those
    $(q_1, q_2, \{ Q_\tau \})$ such that, with $q_{0}$ the left reduction of
    $Q_{\tau}$, we have $\ang{\tau, \zeta}^{-m} q_0 \in \CI(\ff \times
    \overline{\RR^{n + 1}_{\tau, \zeta}})$ and the matching condition
    \eqref{eq:matching conditions} hold.
    Moreover, the mapping $\prinsymbz$ is multiplicative (but not commutative):
    for $A \in \Psitsc^{m_1, r_1}, B \in \Psitsc^{m_2, r_2}$
    \begin{align*}
        \prinsymb{m_1 + m_2, r_1 + r_2}(A B) = \prinsymb{m_1 , r_1}(A) \prinsymb{m_2, r_2}(B)\,,
    \end{align*}
    where the product denotes component wise composition (i.e.\ multiplication in the first two components).
\end{proposition}
\begin{proof}
Because this is merely a statement about restriction of
  the symbol to boundary components of $\Ttsco^*[X;C]$, the only parts
of the proposition that require attention are: (1) the
multiplicativity of the symbol, which follows from Proposition \ref{prop:comp of ff symb}
and (2) the statement that the indicial family lies in $\Norop^m$.  The latter is precisely the content of Section \ref{sec:norm fam}.
\end{proof}

\begin{remark}\label{rem:scat symb of indicial}
  These propositions clarify that we cannot quantize any family of
  semiclassical scattering operators to be the indicial family of a
  $\tsc$-operator.  In fact, the left reduction $\aff(z, \tau, \zeta)$ of the
  indicial operator $\ffsymbz(A)(\tau)$ is smooth on
  $\ff \times \overline{\mathbb{R}^{n + 1}_{\tau, \zeta}}$, its fiber symbol, being
  the restriction to fiber infinity on each $\{ \tau = const. \}$
  slice, is independent of $\tau$, i.e. with $\fibeq$ as in
  \eqref{eq:fibeq def}, 
    \begin{align*}
        \scfibsymbz\left( \ffsymbz(A)_{\ff}(\tau) \right) \equiv
      \fibsymbz(A) \rvert_{\ff \times \fibeq}\in \CI(\ff \times \fibeq).
    \end{align*}
    Here we use the $\tau$ dependent weight
    \begin{equation}
      \label{eq:4}
      \scfibsymbz\left( \ffsymbz(A)_{\ff}(\tau) \right) = \ang{\tau,
        \zeta}^{-m} \aff(z, \tau, \zeta) \rvert_{\ff \times \pa
        \overline{\mathbb{R}^{n + 1}_{\tau,\zeta}}}
    \end{equation}
\end{remark}

\begin{lemma}\label{lem:adjoint_3sc}
    Let $A \in \Psitsc^{m,r}$, then
    $A^* \in \Psitsc^{m,r}$ with
    \begin{align*}
        \prinsymbz(A^*) = \prinsymbz(A)^*\,,
    \end{align*}
    meaning that
    \begin{align*}
        \fibsymbz(A^*) = \overline{\fibsymbz(A)}\,,\\
        \mfsymbz(A^*) = \overline{\mfsymbz(A)}\,,\\
        \ffsymbz(A^*) = \ffsymbz(A)^*\,.
    \end{align*}
\end{lemma}
\begin{proof}
    The fact that the adjoint is again a $\tsc$-operator is a simple consequence of the definition via the $\tsc$-double space in Vasy~\cite[Eq.\ 3.13]{V2000}
    and the principal symbol can be easily calculated using quantized symbols.
\end{proof}

\begin{remark}\label{rem:rho fib vs explicit}
  Finally, we remark that in the above definitions of principal
  symbols, it is possible to replace $\ang{\tau, \zeta}^{-1}$ by any
  smooth boundary defining function $\rho_{\fib}$ of the fiber
  boundary.  We used $\ang{\tau, \zeta}^{-1}$ to be explicit but all
  factors of using a $\rho_{\fib}$ which is equal to $1/\tau$ near the
  characteristic set is also useful.
\end{remark}

\subsection{Commutators}\label{sec:commutators}

For two $\tsc$-operators $A \in \Psitsc^{m_1, r_1}, B \in \Psitsc^{m_2, r_2}$ the commutator in general does not decrease in order, i.e.\ 
in general $[A, B] \not \in \Psitsc^{m_1 + m_2 - 1, r_1 + r_2 - 1}$.
For example, if $A = D_z$ and $B = f(z)$ for some function $f \in S^0(\RR^n_z)$, then $[A, B]$ is the multiplication operator by $D_z f$, which is a
$\tsc$-operator of order $(0,0)$, since there is no decay in $t$.

The commutator drops in order exactly if the indicial operators
commute, 
\[
[\ffsymbz(A), \ffsymbz(B)] = 0,
\]
which is not the case in
the previous example.

Let $A \in \Psitsc^{m,r}$. We say that $A$ is in the $\tsc$-centralizer, \[A \in Z\Psitsc^{m,r} \text{ if and only if } [A, B] \in \Psitsc^{m + m' - 1, r + r' - 1}\] for all $B \in \Psitsc^{m',r'}$.
This is equivalent to the condition that
\begin{align}\label{eq:centrality}
    \ffsymb{r}(A)(\tau) = f(\tau) \id\,,
\end{align}
for some $f \in \CI(\Wperpo)$ (cf. Vasy~\cite[Lemma 6.5]{V2000}).

Now, we will calculate the indicial operator of the commutator $[A, B]$ in the case that $[\ffsymbz(A), \ffsymbz(B)] = 0$. For this we calculate the Taylor expansion of $[A, B] u$ at $x = 0$.

\begin{lemma}\label{lem:expansion_commutator}
    Let $A \in \Psitsc^{m_1,r_1}, B \in \Psitsc^{m_2, r_2}$ and $u \in \CI_{\ff}$, then
    \begin{align*}
        x^{r_1 + r_2} [A, B] u(x, z) &= [A_{\ff}, B_{\ff}] + x \big( [A_{\ff}' - D_\tau \hat{A}_{\ff}(0), B_{\ff}] + [A_{\ff}, B_{\ff}' - D_\tau \hat{A}_{\ff}(0)] \big) u_{\ff} \\
        &\phantom{=} + x [A_{\ff}, B_{\ff}] (\pa_x u)_{\ff} + O(x^2)\,.
    \end{align*}
\end{lemma}
\begin{proof}
    We use Lemma~\ref{lem:expansion_symbol} to see that
    \begin{align*}
        x^{r_1 + r_2} A B u(x, z)
        &= A_{\ff} B_{\ff} u_{\ff} + x \left( A_{\ff}' B_{\ff} + (r_2 - 1) D_\tau \hat{A}_{\ff}(0) B_{\ff} + A_{\ff} B_{\ff}' - A_{\ff} D_\tau  \hat{B}_{\ff}(0)\right) u_{\ff} \\
        &\phantom{=} + x A_{\ff} B_{\ff} u'_{\ff} + O(x^2)\,.
    \end{align*}
    Therefore,
    \begin{align*}
        [A, B] u(x, z) &= [A_{\ff}, B_{\ff}] + x \big( [A_{\ff}' - D_\tau \hat{A}_{\ff}(0), B_{\ff}] + [A_{\ff}, B_{\ff}' - D_\tau \hat{A}_{\ff}(0)] \big) u_{\ff} \\
        &\phantom{=} + x \left( r_2 D_{\tau} \hat{A}_{\ff}(0) B_{\ff} - r_1 D_{\tau} \hat{B}_{\ff}(0) A_{\ff} \right) u_{\ff} \\
        &\phantom{=} + x [A_{\ff}, B_{\ff}] u'_{\ff} + O(x^2)\,.
    \end{align*}
\end{proof}
\begin{proposition}\label{prop:commutator_indicial}
    If $A \in \Psitsc^{m_1,r_1}, B \in \Psitsc^{m_2, r_2}$ with $[\ffsymb{r_1}(A), \ffsymb{r_2}(B)] \equiv 0$, then
    \begin{align*}
        \ffsymb{r_1 + r_2 - 1}([A,B])(\tau) &= [\hat{A}'_{\ff}(\tau) - D_\tau\hat{A}_{\ff}(\tau), \hat{B}_{\ff}(\tau)] + [\hat{A}_{\ff}(\tau), \hat{B}'_{\ff}(\tau) - D_\tau\hat{B}_{\ff}(\tau)] \\
        &\phantom{=} + r_2 D_{\tau} \hat{A}_{\ff}(\tau) \hat{B}_{\ff}(\tau) - r_1 D_{\tau} \hat{B}_{\ff}(\tau) \hat{A}_{\ff}(\tau)\,.
    \end{align*}
\end{proposition}

If $Q = \Op_L(x^{-r} q)$ for $q \in \Ssc^{0,0}$ supported in a neighborhood of $\NP$ and $P_0 = D_t^2 - (\Delta + m^2)$, then
\begin{align*}
    \scnormsymbz(i[P_0, Q]) = x^{r-1} \Hamf(x^{-r} q)|_{x = 0} &= -2 x^{r} \left( \tau x \pa_x + (\zeta + \tau y) \pa_y \right) x^{-r} q|_{x = 0} \\
    &= 2 r \tau q|_{x = 0} - 2 (\zeta + \tau y) \pa_y q|_{x = 0}\,.
\end{align*}
At the north pole we have that
\begin{align*}
    \scnormsymbz(i[P_0, Q])|_{\NP} = 2 r \tau q(0,0,\tau,\zeta) - 2 \zeta \pa_y q(0,0,\tau,\zeta)\,.
\end{align*}
We can recover this from Lemma~\ref{lem:expansion_commutator} as follows:
we have that
\begin{align*}
    \hat{Q}_{\ff}(\tau) &= \Op_L(q(0,0,\tau,\zeta))\,,\\
    \hat{Q}_{\ff}'(\tau) &= \Op_L(\pa_x q(0,0,\tau,\zeta) + z \pa_y q(0,0,\tau, \zeta))\,,\\
    (\widehat{P_0})_{\ff}(\tau) &= \tau^2 - (\Delta + m^2)\,,\\
    (\widehat{P_0})_{\ff}'(\tau) &= 0\,.
\end{align*}
Since the indicial operator of a $\sc$-operator is a Fourier multiplier, we have that \[[(\widehat{P_0})_{\ff}(\tau_0), \hat{Q}_{\ff}(\tau_0)] = 0.\]
We have that
\begin{align*}
    \ffsymb{r - 1}(i[P_0,Q])(\tau) &= [\tau^2 - (\Delta + m^2), \hat{Q}'_{\ff}(\tau)] + 2 r \tau \hat{Q}_{\ff}(\tau) \\
    &= - 2 \Op_L( \zeta\pa_y q(0,0,\tau, \zeta)) + 2 r \tau \Op_L(q(0,0,\tau,\zeta))\,.
\end{align*}
This is a special case of the general identity
\begin{align*}
    \ffsymb{r_1 + r_2 - 1}(i[A, B])(\tau) = \Op_L( \scnormsymbz(x^{r_1 + r_2} H_a(b))|_{\NP})
\end{align*}
for $A = \Op_L(a) \in \Psisc^{m_1, r_1}$ and $B = \Op_L(b) \in \Psisc^{m_2, r_2}$.

\subsection{The indicial operator as a semiclassical scattering
  operator}\label{sec:norm fam}

Two features of semiclassical scattering operators will be crucial in
our work below, namely, we will need, for $A \in \Psitsc^{m,r}(\RR^{n + 1})$.  
\begin{itemize}
\item  to understand how the semiclassical scattering principal
  symbol of $\ffsymb{r}(A)$ can be seen as a function on certain parts
  of $\Ttsco^*[X;C]$, and
\item to recall mapping properties of $\ffsymb{r}(A)$ on semiclassical
  Sobolev spaces.
\end{itemize}
The former is used in formulating ellipticity of $A$ at $\pm \infty
\in \Wperpo$, while the latter is used crucially in that part of the
propagation estimates in Sections~\ref{sec:3sc_propagation} and
Section~\ref{sec:3sc_radial} below in which the indicial operator of the relevant
commutator for the free Klein-Gordon operator $P_0$ is compared to that of $P_V$.

For background on semiclassical analysis we refer to
\cite{DS1999,Zworski_semiclassical}.  Our work below follows more
closely the discussion of smooth semiclassical pseudodifferential
operators found in \cite{Vasy_asymp_hyper_high-energy}.

We work with semiclassical-scattering PsiDO's of order
$m,l,r$, specifically with (classical, smooth) elements in
$\Psisclsc^{m,r,k}$, which are
(by definition) semiclassical quantizations of symbols
$\tilde a \in h^{-k} \CI([0, 1)_h \times S^{m,r}_{\sc}(\RR^n))$, i.e. operators
\[
B(h) = \Op_{\scl}(\tilde a) = \frac{1}{h^{n/2}} \int e^{i \frac{z - z'}{h} \cdot
  \mu} \tilde a (z, \mu ;
h) d\mu
\]
where $h^{k} a (z, \mu ; h)$ is a family of scattering symbols of
order $m,r$ that is smooth in $h \in [0,1)$.

Given $A \in \Psitsc^{m,r}$, recalling $\aff$ from \eqref{eq:ff
  restrict}, the fact that, for
\[
h = \pm \tau,\  \mu = \zeta / \tau, \ \mbox{ as } \tau \to \pm \infty.
\]
the symbol
\begin{equation}
\tilde a (z, \mu ; h) = \aff(z, 1/h, \mu / h)\label{eq:scl rescaling
  norm op}
\end{equation}
is semiclassical can be seen by relating
the radial compactification $\overline{\RR_{\tau,
    \zeta}^{n+1}}$ to the semiclassical symbol space by blowing up the
$\tau = 0$ equator of $\partial \overline{\RR_{\tau,
    \zeta}^{n+1}}$.

Here, given $A \in \Psitsc^{m,r}$, the value of the principal symbol $\tau = + \infty$
will be the semiclassical principal symbol of $\Op_L(\aff) = \hat{A}_{\ff, r}(\tau)$,
i.e.\ the function of $z,\mu = \zeta / \tau$ given by the limit
$\lim_{\tau \to + \infty, \mu = \zeta / \tau}  \ang{\tau, \zeta}^{-m} \aff(z, \tau, \zeta)$.  Provided the restriction
makes sense (which it does for classical symbols) this
is exactly
\begin{align*}
    \ang{\tau, \zeta}^{-m} x^r a \rvert_{h = 1/\tau = 0 = x, \tau > 0}(z, \mu)
\end{align*}
This is just the restriction to the interior of the upper half-sphere in fiber infinity
over $\ff$.  Thus, since $\ang{\tau, \zeta}^{-m} x^r a$ is a smooth function on the whole of
$\Ttsco^*[X;C]$, if we define
\begin{equation}\label{eq:upper half sphere}
    UH_+ = \Stsc_{\ff}^*[X;\poles] \cap \cl(\{ \tau \ge 0 \}),
\end{equation}
we obtain the following:

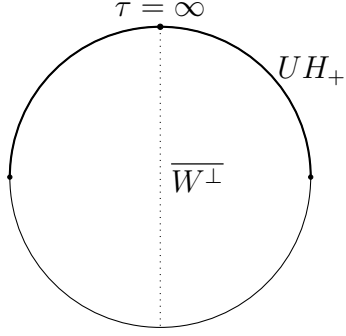
\begin{figure}
    \centering
    \begin{tikzpicture}[scale=2]
        \draw[thick] (1,0) arc [start angle=0, end angle=180, x radius=1cm, y radius=1cm];
        
        \draw (0,0) circle (1cm);
        \filldraw (0,1) circle (0.5pt) node [above] {$\tau = \infty$};
        \filldraw (-1cm,0) circle (0.4pt);
        \filldraw (1cm,0) circle (0.4pt);
        \draw[dotted] (0,1) -- (0,-1);
        \draw (0.707, 0.707) node [right] {$UH_+$};
        \draw (0,0) node [right] {$\Wperpo$};
    \end{tikzpicture}
    \caption{
    The upper half sphere $UH_+$.
    }
    \label{fig:UHplus}
  \end{figure}

\begin{lemma}\label{thm:normal is sclsc}
  Let $\Op_L(a) = A \in \Psitsc^{m,r}$ and $\aff = (x^{r} a) \rvert_{\ff}$.  Then
  $\tilde a$ in \eqref{eq:scl rescaling
  norm op} lies in the semiclassical symbol space $h^{-m} \CI([0, 1)_h
\times S^{m,0})$.  Thus, $\ffsymb{r}(A)(\tau) \in \Norop^{m,0,m}$, the space defined in
\eqref{eq:norop}.

For fixed $h = 1/\tau_0 = 0$, the scattering symbol $\scprinsymb{m,0}
(\ffsymb{r}(A)(\tau_0))$ satisfies
\begin{align*}
    \scfibsymb{m,0}(\ffsymb{r}(A)(\tau_0)),&= \ang{\tau_0, \zeta}^{-m} \aff \rvert_{\ff \times \fibeq} \\
    \scnormsymb{m, 0} (\ffsymb{r}(A)(\tau_0)) & = \ang{\tau_0, \zeta}^{-m} \aff \rvert_{\pa \ff \times \{ \tau = \tau_0 \}}.
\end{align*}
The semiclassical symbol at $h = 1/\tau = 0$ satisfies
\begin{equation}
    \sclsymb{h = 1/\tau}(\ffsymb{r}(A)(1/h)) =
    \ang{\tau, \zeta}^{-m} \aff  \rvert_{UH_+}.  \label{eq:scl symb of indicial}
\end{equation}
and similarly for $h = - 1/\tau$ and $UH_-$.

\end{lemma}
\begin{proof}
    The function $\aff$ satisfies $\ang{\tau, \zeta}^{-m} \aff \in
C^\infty(\ff \times \overline{\RR^{n + 1}_{\tau, \zeta}})$.  In
regions of the form $- C < \tau < C$, the lemma asserts only that
$\ffsymb{r}(A)(\tau) = \Op_{L,z}(\aff(z,\tau,\zeta))$ is a scattering operator of
order $m$, which follows since on each slice $\tau = 0$ the
restriction of $a_0$ is scattering of order $m$.

For regions of unbounded $\tau$, we consider the blown up space
\begin{equation}
\beta_{\scl} \colon [\overline{\RR_{\tau, \zeta}^{n+1}}; \fibeq] \lra
\overline{\RR_{\tau, \zeta}^{n+1}}.\label{eq:fibeq blow up}
\end{equation}
In this manifold with corners, the set $\mathrm{cl} (\beta_{\scl}^{-1} \{ \tau > 1 \})$ is
diffeomorphic to $[0, 1)_{1/\tau} \times \overline{\RR^{n}_{\mu = \zeta / \tau}}$
and thus smooth functions on the blown-up space define
classical semiclassical scattering operators in $h = 1/\tau$ as
$\tau \to \infty$.
Now, $\beta_{\sc}^{-1}(\ang{\tau, \zeta}^{-m} a_0)$ is smooth on the whole of $[\overline{\RR_{\tau, \zeta}^{n+1}};
\fibeq] $ (because pullbacks of
smooth functions via a blow down map are smooth) and thus the
restriction to $\tau > 1$ is smooth.  But in that region $\ang{\tau, \zeta}^{-1} \sim \ang{\mu}^{-1}(1/\tau)$,
so in face $\tau^{-m} \ang{\mu}^{-m} a_0$ is smooth, which is
what we wanted.

  The same goes for $h = - 1/\tau$ as
  $\tau \to - \infty$, i.e.\
  $\mathrm{cl} (\beta_{\scl}^{-1} \{ \tau < -1 \}) = [0, 1)_{- 1/\tau}
  \times \overline{\RR^{n}_{\mu = \zeta / \tau}}$.  
\end{proof}

Given the lemma, it makes sense to extend $\ffsymb{r}(A)(\tau)$ to
$h = 1 /\tau = 0$ as the semiclassical symbol of $\ffsymb{r}(A)$ at
$h = 0$.  Indeed, we define $\Wperpo$ to be the radial
compactification of $W^\perp$.  Thus, over $\NP$, it is
\begin{equation}
  \label{eq:Wperpo def}
\overline{\mathbb{R}_\tau} = \mathbb{R}_\tau \cup \{ \pm \infty \}.
\end{equation}
We then extend $\ffsymb{r}(A)$ to $\Wperpo$ be defining
\begin{equation}
  \label{eq:ff symb infty symbol}
  \ffsymb{r}(A)(\pm \infty) \coloneqq \ang{\tau,
  \zeta}^{-m} \aff  \rvert_{UH_\pm}
\end{equation}
This leads to a natural notion of ellipticity for $A$ at $+ \infty
\in \Wperpo$, namely that $\ffsymb{r}(A)$ be semiclassically
elliptic at $h = 1/\tau = 0$, i.e.\ its semiclassical symbol is
nowhere zero.  See Definition \ref{def:elliptic set} and after.

We recall that the semiclassical Sobolev space is defined as (cf. Zworski~\cite[Section 8.3]{Zworski_semiclassical})
\begin{align}\label{eq:scl sob spaces}
    \Sobscl^{s,\ell}(\mathbb{R}^n) &\coloneqq \curl{ u \in \schwartz'(\RR^n) :
                       \ang{\zeta}^s \cF_h (\ang{z}^\ell  u) \in L^2(\RR^n) }\,,\\
    \norm{u}_{\Sobscl^{s,\ell}}^2 &\coloneqq (2\pi h)^{-n} \int \ang{\zeta}^{2s} \abs{\cF_h (\ang{z}^{\ell} u)(\zeta)}^2 \,d\zeta\,.
\end{align}
Here,
\begin{align*}
    (\cF_h u)(\zeta) \coloneqq \int e^{-\frac{i}{h} \ang{z, \zeta}} u(z) \, dz
\end{align*}
denotes the semiclassical Fourier transform.

Semiclassical scattering operators are bounded on semiclassical Sobolev spaces:
\begin{proposition}\label{prop:semiclassical_bounded}
    If $A \in \Psisclsc^{m,r,k}$, then extends as a bounded operator
    \begin{align*}
        A : h^k \Sobscl^{m+s,r+\ell} \lra \Sobscl^{s,\ell}\,.
    \end{align*}
\end{proposition}

We also recall that for semiclassical operators ellipticity implies invertibility since we can bound the error term in the
parametrix construction by the semiclassical parameter and then use a Neumann series to show invertibility (see Zworski~\cite[Theorem 4.29]{Zworski_semiclassical}).
\begin{proposition}\label{prop:elliptic-invertible}
    Let $A = A(h) \in \Psisclsc^{m,r,k}$ be elliptic, then there exists $h_0 > 0$ such that for all $h \in (0,h)$, $A(h)^{-1}$ exists as a bounded operator
    \begin{align*}
        A(h)^{-1} : \Sobscl^{s,\ell} \lra h^k \Sobscl^{m+s,r+\ell}\,.
    \end{align*}
\end{proposition}
\subsection{Wavefront sets and elliptic sets}\label{sec:wavefront and elliptic}

We now develop an appropriate notion of operator wavefront set in this
context.  We note that there are distinct notions of operator
wavefront set even in the original works on the $\tsc$-calculus.
Given $A \in \Psitsc^{m,r}(X)$, there is $\WFtsc(A)$ and $\WFsc(A)$,
the former from Definition 9.1 of \cite{V2000} and the latter from
Definition 5.1 of \cite{V2001}.  Both of these do not treat fiber
infinity over $\partial X$, so our notion of wavefront set must
generalize these.  Rather than recalling both in detail, we give one
definition of operator wavefront set, which we will denote by
$\WFtsc(A)$.

First we point out the source of subtlety in the definition of
operator wavefront set in this context.  We wish in particular to have
a notion of wavefront set $\WFtsc'$ such that for all $A, B \in \Psitsc$ we have that
\begin{equation} \label{eq:wavefront product type}
    \WFtsc'(AB) \subset   \WFtsc'(A) \cap   \WFtsc'(B).
\end{equation}
The issue arises at points
$\alpha \in \Ttsco [X;C]$ lying over the
intersection of $\mf$ and $\ff$, that is,
$\alpha \in \Ttsc^*_{\mf \cap \ff}[X; C]$.  A natural notion of
wavefront may assert that $\alpha$ is \emph{not} in the wavefront set
of $\Op_L(a) = A$ if $a$ vanishes to infinite order in a neighborhood
of $\alpha$.  But any neighborhood $U$ of $\alpha$ in $\Ttsco$ has
that $U \cap \Ttsco^*_{\ff} [X;C]$ is non-empty and open.  But the
global nature of $\ffsymb{r}(A)$ means that the symbol of $A^2$ is not
necessarily trivial in $U \cap \Ttsco^*_{\ff} [X;C]$ if $a$ is.

To avoid this issue we essentially do not include  points in $\mf \cap
\ff$ in the operator wavefront set.  Rather, as we describe in detail now, away from $\poles$ we use
the scattering definition of $\WFtsc'$ and over $\ff$ only look at $\tau$
levels.  We define
\begin{equation} \label{eq:compressed cotan}
    \Tdoto^* X = \left( \Tsco^* X \setminus \Tsco^*_{\poles} X \right) \cup \Wperpo,
\end{equation}
the compactification of the compressed scattering cotangent bundle
from Section 5 of \cite{V2001}.  In that paper, interpreted in our
context, defines the compressed scattering cotangent bundle to be
\begin{align*}
    \Tdot^* X = \left( \Tsc^* X \setminus \Tsc^*_{\poles} X \right)
\cup W^\perp,
\end{align*}
i.e.\ not compactified.  This space has a mapping
\begin{equation}
\pi^\perp \colon \Tsc^* X \longrightarrow \Tdot^* X\label{eq:pi perp}
\end{equation}
which is
identity everywhere except on $T_\poles^* X$ where it is projection
onto $W^\perp$ (which is just the mapping $(\tau, \zeta) \mapsto \tau$
on the fiber.)

We now define the operator wavefront set $\WFtsc'(A)$ of a
$\tsc$-PsiDO $A$.  We follow \cite{V2001} here as opposed to
\cite{V2000} here as there are two different definitions, and we
indicate the difference.  Given $A \in \Psitsc^{m,r}$, 
$\prinsymb{m,r}(A)$ is defined as a function on the set
from \cite[Eq.\ 9.1]{V2000}
\begin{align}\label{eq:square analogue}
    \Ctsc \coloneqq \Stsc^*[X;C] \cup \Ttsco^*_{\mf}[X;C] \cup \Wperpo\,.
\end{align}
Its value on these three components are exactly the three components of
$\prinsymb{m,r}(A)$.  However, as described above, points at the
boundary of $\Ttsco^{*}_{\ff} X$ are problematic when considered as
part of the operator wavefront set.  Thus we instead define what is
effectively the locus in $\Tdoto^{*} X$ which avoids those points in
$\Ctsc$, namely
\begin{align}\label{eq:square analogue dot}
    \Ctscd \coloneqq \Ssc^{*}_{X \setminus \poles} X \cup
  \Tsco^{*}_{\pa X \setminus \poles} X \cup \Wperpo
\end{align}
and use this to define both $\WFtsc'$ and $\Elltsc$ below.

Since $\WFtsc'(A)$ will specify points in $\Ttsco^*[X; C]$ near which
the left symbol $a$ is rapidly decreasing, we need a mechanism for
relating subsets of $\Ctscd$ to subsets of $\Ttsco^*[X; C]$.  This
will work by associating $\Tsco^*X \setminus \Tsco_C^*X$ naturally to
its image via the blow-down $\beta_C$, while points $\Wperpo$ will
correspond to the natural $\tau$ slices over $\ff$.

Thus, we define
\begin{align}\label{eq:gamma tsc easy}
\gamma_{\tsc} : \Ctscd \lra
\mathcal{P}(\pa \Ttsco^* [X;C])
\end{align}
as 
\begin{align*}
    \gamma_{\tsc}(p) &= \{p\}&{} &\mbox{ for } p \in  \Ssc^{*}_{X \setminus \poles} X \cup
  \Tsco^{*}_{\pa X \setminus \poles} X\,, \\
    \gamma_{\tsc}(\tau) &= \beta_{C}^{-1} (\pi^\perp)^{-1}\{\tau\}&{} &\mbox{ for } \tau \in W^\perp\,,\\
    \gamma_{\tsc}(\pm \infty) &= UH_{\pm}&{} &\mbox{ for } \pm \infty \in \partial \Wperpo\,.
\end{align*}
By abuse of notation, for a set $S \subset \Ctscd$ we write
\begin{align}\label{eq:gamma_tsc}
    \gamma_{\tsc}(S) \coloneqq \bigcup_{p \in S} \gamma_{\tsc}(p)\,.
\end{align}
We used above that $\Ssc^{*}_{X \setminus \poles} X \cup
  \Tsco^{*}_{\pa X \setminus \poles} X$ is naturally identified with $\Stsc^{*}_{[X;C] \setminus \ff} [X;C] \cup
  \Ttsco^{*}_{\pa [X;C] \setminus \ff} [X;C]$ via the blow down map.

The topology on $\Tdot X$ is that in which a neighborhood basis near
$\alpha \in W^{\perp}$ is induced by the open neighborhoods of
$\gamma_{\tsc}(\alpha)$, while in $\Tdoto X$, open
sets around boundary points are defined as usual for radial compactifications.

For a symbol $a \in \Stsc^{m,r}$, we define the essential support
$\esssupptsc(a) \subset \pa \Ttsco^* [X;C]$ by declaring
$p \in \esssupptsc(a)^c$ if and only if there exists $U \subset
\Ttsco^* [X;C]$ open and $\chi \in \CcI(\Ttsco^* [X;C])$ such that $p
\in U$, $\chi|_U \equiv 1$ and $\chi a \in \Stsc^{-\infty, -\infty}$.

\begin{definition}\label{def:WFtsc}
    Let $A = \Op_L(a) \in \Psitsc^{m,r}(X)$.
    The operator wavefront set
    \begin{align*}
        \WFtsc'(A) \subset \Ctscd
    \end{align*}
is defined as
    follows: a point $p \in \Ctscd$ is \textit{not }in the wavefront set,
    \begin{align*}
        p \in \WFtsc'(A)^c \text{ if and only if }
        \gamma_{\tsc}(p) \cap \esssupptsc(a) = \varnothing\,.
    \end{align*}

    Moreover, we define
    \begin{align*}
        \WF_{\fib}'(A) &\coloneqq \WFtsc'(A) \cap \Ssc^{*}_{X \setminus \poles} X \,,\\
        \WF_{\mf}'(A) &\coloneqq \WFtsc'(A) \cap \Tsco^*_{\pa X \setminus \poles}X \,, \\
        \WF_{\ff}'(A) &\coloneqq \WFtsc'(A) \cap \Wperpo\,.
    \end{align*}
\end{definition}
We can write the complements of each of the components as
\begin{align*}
    \WF_{\fib}'(A)^c &= \{ \alpha \in \Ssc^{*}_{X \setminus \poles} X\colon \exists U \subset \Ssc^{*}_{X \setminus \poles} X \text{ open such that $\alpha \in U$} \\
    &\phantom{= \{} \text{and $a(A)$ vanishes to infinite order on $\overline{U}$}\}\,,\\
    \WF_{\mf}'(A)^c &= \{ \alpha \in \Tsco^*_{\pa X \setminus \poles}X \colon \exists U \subset \Tsco^*_{\pa X \setminus \poles}X \text{ open such that $\alpha \in U$} \\
    &\phantom{= \{} \text{and $a(A)$ vanishes to infinite order on $\overline{U}$}\}\,,\\
        \WF_{\ff}'(A)^c &=
        \{ \tau \in W^\perp \colon \exists\ \epsilon > 0 \text{ such that $a(A)$
                             vanishes to } \\
  &\phantom{= \{} \text{ infinite order on } 
                  \beta^{-1}(\pi^\perp)^{-1}[\tau - \epsilon, \tau + \epsilon]\}\,  \\
  &\phantom{= \{} \cup \{ \pm \infty :  \exists  \text{ open $U
    \subset \partial \Ttsco^*[X;C]$ such
    that } UH_\pm  \subset U \\
  & \phantom{= \{} \text{ and $a(A)$ vanishes to infinite order on }
    \overline{U} \}.
\end{align*}

Now we define the elliptic sets.  Over $\mf$ and fiber infinity, the
definition of ellipticity is exactly as in the standard scattering
case, i.e.\ non-vanishing (or, for operators acting on sections of
vector bundles, invertibility) of the principal symbol.  Over $\ff$ in
$W^\perp$, the correct notation of ellipticity is invertibility
between appropriate scattering Sobolev spaces.

To define the elliptic set, we note that the two components of the
symbol $\fibsymb{m, r}(A)$ and $\mfsymb{m, r}(A)$ define, by
restriction, functions on $\Ssc_{X \setminus \poles}^* X$ and
$\Tsco_{\pa X \setminus \poles}^* X$, respectively.  
\begin{definition}\label{def:elliptic set}
    Let $A \in \Psitsc^{m, r}$.  The $\tsc$-elliptic set
    $\Elltsc(A)$ (whose $m, r$ dependence is suppressed from the
    notation) is
    \begin{align*}
        \Elltsc(A) = \Ell_{\fib}(A) \cup \Ell_{\mf}(A) \cup \Ell_{\ff}(A) \subset \Ctscd\,,
    \end{align*}
    with
    \begin{align*}
        \Ell_{\fib}(A) &= \{\alpha \in \Ssc_{X \setminus \poles}^* X \colon
                         \fibsymb{m,r}(A)(\alpha) \neq 0  \}\,,\\
        \Ell_{\mf}(A) &= \{\alpha \in \Tsco_{\pa X \setminus \poles}^* X \colon
                        \mfsymb{m,r}(A)(\alpha) \neq 0  \}\,,\\
    \end{align*}
    while
  \begin{align*}
        \Ell_{\ff}(A) &= \{\tau \in W^\perp \colon
                        \ffsymb{r}(A)(\tau) \text{ is scattering
                        elliptic and invertible}\}\, \\
    &\phantom{= \{} \cup \{ \pm \infty \in \partial \Wperpo :
      \fibsymb{m,r}(A) \text{ is nowhere vanishing on } UH_\pm \}.
    \end{align*}

    Moreover, we set
    \begin{align}\label{eq:def_Chartsc}
        \Chartsc(A) \coloneqq \Ctscd \setminus \Elltsc(A)\,,
    \end{align}
    the $\tsc$-characteristic set of $A$.
\end{definition}

Note that it makes sense to speak of invertibility of scattering
elliptic operators.  Indeed, if $B \in \Psisc^{m,r}$ is (scattering) elliptic
then $B \colon \Sobsc^{m + s, r + \ell} \lra \Sobsc^{s , \ell}$ is Fredholm.
By scattering ellipticity, its kernel and cokernel consist of Schwartz
functions, and therefore the invertibility of this Fredholm mapping is
independent of $s, \ell$, i.e.\ if it is invertible for any $s,\ell$ then it
is invertible for all $s,\ell$.  Thus, we say a (globally) elliptic,
scattering operator $B \in \Psisc^{m,r}$ is ``invertible'' if any (and
thus all) of these Fredholm operators is invertible.

We now describe in more detail the significance of
$\tau_0 \in \Ell_{\ff}(A)$ and $\pm \infty \in \Ell_{\ff}(A)$.  For
the former, we first consider the assumption that
$\ffsymb{r}(A)(\tau_0)$ is scattering elliptic.  Recall that if $\Op_L(a) = A$, then
$\ffsymb{r}(A)(\tau_0) = \Op_{L,z} (\aff)$ \eqref{eq:norm fam
  explicit}.  Our discussion of symbols in Section \ref{sec:principal
  symb sec} allows us to identify the scattering principal symbol of
$\hat A_{\ff, r}(\tau)$ with the appropriate boundary restrictions of
$a$; following Remark \ref{rem:scat symb of indicial}, choosing
$\tau-$dependent defining function we have 
\begin{equation*}
  \label{eq:scattering symbol normal symbol}
  \begin{split}
      \scprinsymb{m,0}(\hat A_{\ff, l}(\tau_0)) &=
      (\ang{\tau,\zeta}^{-m}  \aff \rvert_{\ff \times \fibeq},
      \ang{\tau, \zeta}^{-m} \aff \rvert_{\Ttsco^*_{\ff \cap \mf} \cap \{ \tau = \tau_0 \}}).
  \end{split}
\end{equation*}
Thus if $\tau_{0} \in \Ell_{\ff}(A)$ then these two components are nowhere
vanishing, and the scattering operator $\ffsymb{r}(A)(\tau_{0})$ is invertible.

As for the condition $+ \infty \in \Ell_{\ff}(A)$ (the $-$ case is
similar) this is equivalent to $\ffsymb{r}(A)$ being
semiclassically elliptic as $\tau = 1/h \to + \infty$.  Indeed, we
have $\ang{\tau, \zeta}^{-m} x^{r} a$ is invertible on the
whole of $UH_+$.  For $A \in \Psitsc^{m,r}$, the indicial operator
$\ffsymb{r}(A)(\tau)$ is a semiclassical scattering PsiDO with
semiclassical principal symbols at $\tau = \pm \infty$, so ellipticity
at $+\infty \in \Wperpo$ (over $\ff_{+}$) is
exactly the condition that
\[
\sclsymb{m}
(\ffsymb{r}(A)) = x^r \ang{\tau, \zeta}^{-m} a \rvert_{UH_+}
\]
is nowhere vanishing.

\subsection{Elliptic
  regularity}\label{sec:reduced elliptic}

We have the following microlocal parametrix theorem, which extends
Lemma 9.3 and Remark 9.4 from Vasy~\cite{V2000} to include fiber infinity.
\begin{proposition}\label{prop:micro ell parametrix}
   Let $A \in \Psitsc^{m,r}(X)$ and $K \subset \Elltsc(A)$ be compact.
   Then there exists $G \in \Psitsc^{-m,-r}(X)$ such that
    \begin{align*}
        K \cap \WFtsc'(AG - \id) = K \cap \WFtsc'(GA - \id) = \varnothing\,.
    \end{align*}
  \end{proposition}
  
\begin{proof}
We include only the construction for $+ \infty \in \Wperpo$.
  
  Let $+ \infty \in \Ell_{\ff}(A)$.  Let $\aff = x^r a
  \rvert_{\ff}$, so $\aff \in \ang{\tau, \zeta}^m
  C^\infty(\ff \times \overline{\RR^{n + 1}_{\tau, \zeta}})$.
  The ellipticity condition at $+ \infty$ is exactly that $\ang{\tau, \zeta}^m  \aff \rvert_{UH_+}$ is invertible.
  This implies
that $\hat A_{\ff, r}(\tau)$ is invertible for sufficiently large
$\tau > \tau_0$ for $\tau_0 \gg 0$ fixed by Proposition~\ref{prop:elliptic-invertible}.
We will show that
one can quantize $\hat A_{\ff}(\tau)^{-1}$ to a $\tsc$ operator, and
use this to perform a parametrix construction for $A$ near $+ \infty
\in \Wperpo$. 

Let $UH_+
\subset U \subset U'$ where $U, U'$ are open sets such that: (1) $\ang{\tau, \zeta}^m  \aff$ is invertible on $U'$ and
(2) there exists
$\chi \in \CI(\ff \times \overline{\RR^{n + 1}_{\tau, \zeta}})$
with $\chi \rvert_U \equiv 1$ and $\supp \chi \subset U'$.  (The
second condition is easily achievable since $UH_+$ is an embedded
p-submanifold of $\ff \times \partial \overline{\RR^{n +
  1}_{\tau, \zeta}})$.)  One can, in particular, arrange that $\chi =
\chi(\tau, \zeta)$ (since $UH_+ = \ff \times \{ \tau \ge  0 \} \cap \partial \overline{\RR^{n +
  1}_{\tau, \zeta}})$) and that $\chi(\tau) \equiv 1$ for $\tau >
\tau_0$ for some fixed $\tau_0$.  Then
\[
\chi \aff^{-1} \in  \ang{\tau, \zeta}^{-m}
  C^\infty(\ff \times \overline{\RR^{n + 1}_{\tau, \zeta}}),
\]
and thus $\Op_{L, z}(\chi \aff^{-1}) \in \Norop^{-m, 0, -m}$.  Moreover, if $q
\in C^\infty(\ff \times \overline{\RR^{n + 1}_{\tau, \zeta}})$
has $\supp q \subset U$, then $q(\chi \aff^{-1} \aff - 1) \equiv
0$ and thus
\[
\Op_{L,z}(q) (\Op_{L, z}(\chi \aff^{-1}) \hat A_{\ff, r} - 1) \in \Norop^{-1,0,-1}
\]
 
The preceding is the first step in a semiclassical-scattering
parametrix construction which yields $\tilde g \in  \langle \tau, \zeta \rangle^{-m}
  C^\infty(\ff \times \overline{\RR^{n + 1}_{\tau, \zeta}}),$
  such that
\[
\Op_{L,z}(q) (\Op_{L, z}(\tilde g) \hat A_{\ff, r} - 1) \in \Norop^{-\infty,0,-\infty}.
\]
Indeed, following the standard elliptic parametrix construction, if
$g_0$ has
\[
E = \Op_{L,z}(q) (\Op_{L, z}(g_N) \hat A_{\ff, r} - 1) \in
\Norop^{-N,0,-N},
\]
then one can solve $-b_N \aff = \sigma_{\scl, \sc, -N,
  0, -N}(E)$ on the support of $\chi$ and then take a Borel sum to
obtain $\tilde g$.

Letting $\chi_{\tau_0} \in \CI(\RR)$  have $\chi_{\tau_0}(\tau)
= 1$ for $\tau > \tau_0 + 1$ and $\supp \chi_{\tau_0} \subset [\tau_0,
\infty)$, then
\[
\hat K(\tau) \coloneqq \chi_{\tau_0}(\tau) \left( \hat A_{\ff,r}^{-1} - \Op_{L, z}(\tilde g) \right) \in
\Norop^{-\infty, 0, - \infty},
\]
(This follows from the argument \cite{Mazzeo:Edge} used in the proof of
Corollary \ref{thm:inverse of elliptic tsc}, where one compares the
parametrix and actual operator and shows they differ by a residual operator.) But this
implies that
\[
(1 - \chi_{\tau_0}) \tilde g + \chi_{\tau_0}(\tau) \aff^{-1}  \in
\ang{\tau, \zeta}^{-m} \CI(\ff \times \overline{\RR^{n + 1}_{\tau, \zeta}}),
\]
Indeed, this follows since $\tilde g \in \ang{\tau, \zeta}^{-m}
\CI(\ff \times \overline{\RR^{n + 1}_{\tau, \zeta}})$ and
$\chi_{\tau_0}(\tau) \aff^{-1} - \tilde g$ is order $-\infty, 0,
-\infty$. 

Thus 
\[
\hat G \coloneqq (1 - \chi_{\tau_0}) \Op_{L,z}(\tilde g) + \chi_{\tau_0}(\tau) A_{\ff,r}^{-1}
\in \Norop^{-m, 0, -m},
\]
and
\[
\hat G - \Op_{L,z}(\tilde g)  \in \Norop^{-\infty, 0, \infty}.
\]

To quantize this $\hat G$ to a full $\tsc$-PsiDO, we need to prescribe
matching values on the other faces $\Ttsco^*_{\ff}[X;C]$ and $\mf$ of
$\Ttsco^*_{\ff}[X;C]$, but this is easily done.
Indeed, the symbols of $G$, $\Op_{L,z}(g)$, and $\aff^{-1}$ agree at fiber
$UH_+$ (meaning concretely that when multiplied by $\ang{\tau, \zeta}^{m}$
they restrict to the same function there) and thus
in a sufficiently small neighborhoods $V \subset V'$ of $UH_+$ in
$\Ttsco^*_{\ff}[X;C] \cup \mf$
one can choose another cutoff $\chi'$ which is $1$ on $V$ and
supported in $V'$, and
let $g = \chi \aff^{-1}$ on $V$.  Then this $g$ matches the
principal symbol of $\hat G$ and thus there is a smooth function which
we also call $g$ on the whole of $\Ttsco^*[X;C]$ such that $\Op_L(x^{-r}g) =
G \in \Psitsc^{-m, r, m}$ and: (1) $g = \aff^{-1}$ on $V$, $\hat
G_{\ff}(\tau) = \hat G(\tau) = \hat A_{\ff}^{-1}(\tau)$ for $\tau >
\tau_0$.

Thus for $Q \in \Psitsc^{0,0,0}$ close to $UH_+$,
\[
E = Q(G A - I) \in \Psitsc^{-1, -1, -1}.
\]
This is the first step in a parametrix construction which produces a
residual error.  Indeed, to find $G_1 \in \Psitsc^{-m-1,-r-1}$ such that $G_1 A + E \in
\Psitsc^{-2,-2, -2}$ near $UH_+$, one solves $(\hat G_1)_{\ff, -r-1}
\hat A_{\ff,r} = \hat E_{\ff, -1}$, i.e.\ $(\hat G_1)_{\ff, -r-1}
= \hat E_{\ff, -1}\hat A_{\ff,l}^{-1}$, for $\tau > \tau_0$.  We have
already seen that $\hat A_{\ff,r}^{-1}$ can be approximated in
$\Norop^{-\infty, 0, -\infty}$ by an operator in $\Norop^{-m, 0, -m}$,
so again using the ellipticity on near $UH_+$ on the other bhs's of
$\Ttsco^*[X;C]$, we have $G_1$.  The inductive construction of the
higher order approximations is standard.
\end{proof}

From the elliptic parametrix construction we obtain elliptic estimates.
We now record these estimates and use them to discuss an
important application to a specific globally elliptic $\tsc$-operator.

\begin{proposition}\label{thm:elliptic regularity theorem}
    Let $u \in \mathcal{S}'$ and $A \in \Psitsc^{m,r}$.  Let $B, Q' \in \Psitsc^{0,0}$.
    If $\WFtsc'(B) \subset \Elltsc(A) \cap \Elltsc(Q')$, then $Q' A u \in \Sobsc^{s-m,\ell-r}$ implies $Bu \in \Sobsc^{s, \ell}$,
    and for any $M, N \in \RR$ there is $C > 0$ such that
    \begin{align*}
        \norm{B u}_{s, \ell} \le C \left( \norm{Q' A u}_{s-m, \ell-r} + \normres{u}\right) .
    \end{align*}
\end{proposition}
\begin{proof}
This follows from the standard argument using the mapping property in
Proposition \ref{prop:bounded_Sobolev}.  Namely, Let $G$ be as in
Proposition \ref{thm:elliptic regularity theorem} with respect to $K =
\WFtsc'(B)$.  Then
\begin{align*}
    \norm{B u}_{s, \ell} &\le \norm{ (GA - I)B u}_{s, \ell} + \norm{G A B u}_{s, \ell} \\
    &\le \norm{ (GA - I)B u}_{s, \ell} + \norm{G (\Id - Q') A B u}_{s, \ell} + \norm{G Q' A B u}_{s, \ell} \\
    &\le C \left( \norm{Q' A u}_{s-m, \ell - r} + \normres{u} \right) .
\end{align*}
where we used Proposition \ref{prop:bounded_Sobolev} and~\eqref{eq:wavefront product type}.
\end{proof}

Thus an operator $A \in \Psitsc^{m,r}$ that is globally elliptic satisfies the (global) elliptic estimate,
namely that for any $M, N \in \RR$, there is $C > 0$ such that,
\begin{equation}
  \label{eq:global elliptic estimate}
  \norm{u}_{s,\ell} \le C(\norm{Au}_{s - m, \ell - r} + \normres{u}).
\end{equation}
In particular, as have now established,
there are global parametrices $G, G' \in \Psitsc^{-m,-r}$ such that
\begin{equation}
  GA - I , AG' - I\in \Psitsc^{-\infty, - \infty}.
  \label{eq:global parametrix}
\end{equation}
We use these parametrices in the corollary below.

In order to define microlocalizers over $\ff$, we will need an
operator which is globally elliptic and commutes with $P_{V_0}$.
To accomplish this, we switch the signs in $P_{V_0}$ to make it elliptic;
that is, we consider, for $E \ge 0$,
\begin{equation}
D_t^2 + H_{V_0} + E\label{eq:elliptified operator}
\end{equation}
We will use the inverse of this
operator below to define such microlocalizers in Section \ref{sec:actual func calc sect}.  Specifically, we will need:
\begin{corollary}\label{thm:inverse of elliptic tsc}
The operator $D_t^2 + H_{V_0} + E$ is invertible for some $E \ge 0$, and
\[
(D_t^2 + H_{V_0} + E)^{-1} \in \Psitsc^{-2,0}(\RR^{n + 1}).
\]
\end{corollary}
\begin{proof}
The operator $D_t^2 + H_{V_0} + E$ is elliptic on $\mf$ and fiber infinity
for any $E > 0$.  The indicial operator is
\[
\tau^2 + H_{V_0} + E
\]
and this operator is invertible for all $\tau$ for $E$ sufficiently large.
Indeed, all negative elements in the spectrum of $\Delta_z + {V_0}(z)$ are eigenvalues.
This means that for $-E < 0$ in the spectrum of $\Delta_z + m^2 + V_0$, there exists $w \in L^2(\RR^n_z)$ such that
\begin{align*}
    (H_{V_0} + E) w(z) = 0\,.
\end{align*}
But then by scattering ellipticity of $H_{V_0} + E$, $w \in \schwartz$.
On the other hand, for all Schwartz $w$,we can integrate by parts to obtain
\begin{align*}
    \ang{(\Delta_z + m^2 + V_0)w, w} = \int \abs{\nabla_z w}^2 +
  \ang{{V_0} w, w} + (m^2 + E) \abs{w(z)}^2 dz.
\end{align*}
Since $\ang{{V_0}w, w} \ge - \epsilon/2 |{V_0} w|^2 - (1/2 \epsilon) |w|^2$
and we can bound $|{V_0} w| \le C(|\nabla_z w|^2 + |w|^2)$, we
obtain a lower bound, for $\epsilon > 0$ sufficiently small and $E$ sufficiently large,
\[
\ang{(\Delta_z + m^2 + V_0 + E)w, w} \ge \frac 12 \int |\nabla_z w|^2
+ (m^2 + E) |w(z)|^2 dz > 0
\]
for all Schwartz $w$.
Thus for $E$ sufficiently large, $\ffsymbz(D_t^2 + H_{V_0} + E)(\tau)$
is invertible for all $\tau$, hence $D_t^2 + H_{V_0}  + E$ is globally
$\tsc$-elliptic.

On the other hand, by the same integration by parts argument, $D_t^2 +
H_{V_0}+
E$ is positive for $E$ sufficiently large, hence is
invertible.  

To show that $(D_t^2 + H_{V_0} +
E)^{-1}  \in \Psitsc^{-2,0}$, we employ an argument from
\cite{Mazzeo:Edge}.  Namely, by global ellipticity, we have
parametrices $G, G'$ for $(D_t^2 + H_{V_0} + 
E)^{-1}$ as in
\eqref{eq:global parametrix}, and writing $G (D_t^2 + H_{V_0} + 
E) - I= K$ and $(D_t^2 + H_{V_0} + 
E) G' - I = K'$ we obtain 
\begin{align*}
(D_t^2 + H_{V_0} + 
E)^{-1} &= (G (D_t^2 + H_{V_0} + 
E) + K) (D_t^2 + H_{V_0} + 
E)^{-1} ((D_t^2 + H_{V_0} + 
E) G' + K')
  \\
&= G ((D_t^2 + H_{V_0} + 
E) G' + K') + KG' K' + K (D_t^2 + H_{V_0} + 
E)^{-1} K'.
\end{align*}
Since for any operator $L \colon \mathcal{S} \longrightarrow
\mathcal{S}'$ we have $\Psitsc^{-\infty, - \infty} \circ L \circ \Psitsc^{-\infty,
  - \infty} \subset \Psitsc^{-\infty, - \infty}$, this and the
composition properties for $\tsc$-PsiDOs show that the right hand side
above lies in $\Psitsc^{-2,0}$, which is what we wanted. \end{proof}



\section{Functional calculus, commutators, and special symbol
  classes}\label{sec:full func calc sec}

As discussed in the introduction, our localizers near $\ff$ 
include functions of the limiting static operator $P_{V_{+}}$ which
act as microlocalizers to the characteristic set.  Indeed, from
Section \ref{sec:commutators}, we recall that $\zeta$-dependence in
the indicial operator can lead to failure of the basic commutator
order relations.  Thus, again following Vasy, our commutant will be
locally a composition of an operator whose symbol is purely
$\tau$-dependent over $\poles$ composed with a function of the
operator $P_{V_\pm}$.  To prepare for this, in this section, for
static $V_0$, we wish to consider functions of $P_{V_0}$.

However, we must take special care in defining these functions, as it
is not automatic that a function of a differential operator of
positive order is pseudodifferential.  As a cautionary example,
consider the free Hamiltonian function
$p = \tau^{2} - |\zeta|^{2} - m^{2}$.  This is a classical symbol,
but, for $\psi \in \CcI(\RR)$ with $\psi(0) = 1$, the composition
$\psi(p)$ is not a symbol, in fact $D_{\tau}^{k} (\psi(p))$ grows to
order $k$ as $\tau$ goes to infinity along the characteristic set.
Correspondingly, $\psi(P_{0})$ is not a pseudodifferential operator of
any order.

To handle this issue we consider instead the normalization of $p$ in
which we divide essentially by $\ang{\tau, \zeta}^{2}$ to
make an order $0$ classical symbol.  Consider, for $E \ge 0$,
the operator
\begin{align}\label{eq:def_Gpsi0}
    G_{\psi,0} \coloneqq \psi\left( (D_t^2 + D_{z} \cdot D_{z} + m^2 +
  E)^{-1} P_0 \right),
\end{align}
which is the Fourier multiplier for the function
\begin{equation}
  \label{eq:Gpsi0 Fourier Mult}
\psi \left( \frac{\tau^2 - |\zeta|^2 - m^2}{\tau^2 +
    |\zeta|^2 + m^2 + E} \right) =  \psi \left( \frac{p}{\tau^2 +
    |\zeta|^2 + m^2 + E} \right).
\end{equation}
This $G_{\psi,0}$ depends on $E$; the value of which is fixed such
that the operator in \eqref{eq:elliptified operator} is invertible.
Since the ratio $p / (\tau^{2} + |\zeta|^{2} + m^2 + E)$ is a smooth
function on the whole of $\Tsco^* X$, its composition with $\psi$ is
as well.  Thus we have
\begin{align*}
  G_{\psi,0} \in \Psisc^{0,0}(X)\,.
\end{align*}

\subsection{Functional calculus}\label{sec:actual func calc sect}

In this section, we construct the operator $G_{\psi}$ as an element in the 3-scattering calculus.
We will construct the functional calculus for general self-adjoint $\Psitsc^{0,0}$-operators.

Our main result is a generalization of Proposition 10.2 \cite{V2000},
which relates to functions of a many-body Hamiltonian $H$.  For such
$H$, one obtains that for $\psi \in \CcI(\RR)$,
\[
\psi(H) \in \Psitsc^{- \infty, 0},
\]
while in contrast, we see from the example of $G_{\psi,0}$ that we do
not expect smoothing operator when taking functions of the
Klein-Gordon operator.  We will apply the following proposition to 
functions of $(D_t^2 + H_{V_0} + E)^{-1} P_{V_0}$ in Definition
\ref{def:Gpsi} below.
\begin{proposition}\label{prop:func_calc}
    Let $\psi \in \CcI(\RR)$ and $A \in \Psitsc^{0,0}$ be self-adjoint.
    Then
    \[
      \psi(A) \in \Psitsc^{0,0}
    \]
    and its principal symbol satisfies
    $\prinsymbz(\psi(A)) = \psi(\prinsymbz(A))$, meaning
    \begin{equation}
        \label{eq:fib symb and mf symb}
        \fibsymbz(\psi(A)) = \psi(\fibsymbz(A)), \quad \mfsymbz(\psi(A)) =
        \psi(\mfsymbz(A)), \quad \ffsymbz(\psi(A)) = \psi(\ffsymbz(A))
    \end{equation}
\end{proposition}

\begin{remark}
    Since operators in $\Psitsc^{0,0}$ are bounded on $L^2$, the
    assumption that $A$ is self-adjoint is equivalent to $A$ being
    symmetric.  
\end{remark}

We start by recalling the results for the scattering calculus.

\begin{lemma}
    Let $A \in \Psisc^{0,0}(\RR^n)$ be self-adjoint.
    For each $k \in \NN$ there exists a family of order $k$ parametrices
    $B_k(z) \in \Psisc^{0,0}(\RR^n)$ for $z \in \CC \setminus \RR$ such that
    \begin{align*}
        (A - z) B_k(z) - \id = R_{1,z} &\in \Psisc^{-k,-k}(\RR^n)\,,\\
        B_k(z) (A - z) - \id = R_{2,z} &\in \Psisc^{-k,-k}(\RR^n)
    \end{align*}
    and the seminorms of order $k$ of $B_z, R_{1,z}, R_{2,z}$ are bounded by $C_k \abs{\Im z}^{-c(k)}$.
\end{lemma}
\begin{proof}
    The proof is the same as \cite[Lemma 10.1]{V2000}.
    The crucial point is that the principal symbol of $A - z$ is
    $\scprinsymbz(A) - z$ and since $A$ is self-adjoint, $\scprinsymbz(A)$ is real and therefore $\scprinsymbz(A) - z$ is invertible for $z \in \CC \setminus \RR$.
\end{proof}

\begin{proposition}\label{prop:func_calc_sc}
    Let $A \in \Psisc^{0,0}(\RR^n)$ be self-adjoint and $\psi \in \CcI(\RR)$. Then
    \begin{align*}
        \psi(A) \in \Psisc^{0,0}(\RR^n)
    \end{align*}
    with principal symbol
    \begin{align*}
        \scprinsymbz(\psi(A)) = \psi(\scprinsymbz(A))\,.
    \end{align*}
\end{proposition}
\begin{proof}
    The proof is the same as Proposition~10.2 in Vasy~\cite{V2000}, but using the above lemma for the parametrix.
    Namely, we use the Helffer--Sjöstrand formula to define
    \begin{align}\label{eq:Helffer-Sjostrand}
        \psi(A) \coloneqq -\frac{1}{2\pi i} \int \bar{\pa}_z \tilde{\psi}(z) (A - z)^{-1} \, dz \wedge d\bar{z}\,,
    \end{align}
    where $\tilde{\psi}$ is a compactly supported almost analytic extension of $\psi$.
    Then we define
    \begin{align*}
        A_{\psi, k} \coloneqq -\frac{1}{2\pi i} \int \bar{\pa}_z \tilde{\psi}(z) B_k(z) \, dz \wedge d\bar{z}\,,
    \end{align*}
    where $B_k(z)$ is the $k$-th order parametrix of $A - z$ as in the previous lemma.
    We can use asymptotic summation to obtain a limit $\tilde{A}_{\psi}$ of $A_{\psi,k}$ and we conclude that
    \begin{align*}
        \psi(A) - \tilde{A}_{\psi} : \CmI \lra \CdI
    \end{align*}
    and the formula for the principal symbol follows from the explicit form of $\psi(A)$.
\end{proof}

Now, we turn to the case of $A \in \Psitsc^{0,0}$. We will an analogous argument and therefore we start by constructing a parametrix for $A - z$.

\begin{lemma}
    Let $A \in \Psitsc^{0,0}(X)$ be self-adjoint.
    For each $k \in \NN$ there exists a family of order $k$ parametrices
    $B_z \in \Psitsc^{0,0}(X)$ for $z \in \CC \setminus \RR$ such that
    \begin{align*}
        (A - z) B_k(z) - \id = R_{1,k}(z) &\in \Psitsc^{-k,-k}(X)\,,\\
        B_k(z) (A - z) - \id = R_{2,k}(z) &\in \Psitsc^{-k,-k}(X)
    \end{align*}
    and the seminorms of order $k$ of $B_k(z), R_{1,k}(z), R_{2,k}(z)$ are bounded by $C_k \abs{\Im z}^{-c(k)}$.
\end{lemma}
\begin{proof}
To begin with, we claim that, for $z \in \CC \setminus \RR$,
that $A - z$ is globally $\tsc$-elliptic.  Indeed, by
self-adjointness, $\fibsymbz(A)$ and $\mfsymbz(A)$ are real valued,
so $\fibsymbz(A) - z$ and $\mfsymbz(A) - z$ are non-zero.
Moreover, $\ffsymbz(A)(\tau) \in \Psisc^{0,0}(\ff)$ is self-adjoint by Lemma~\ref{lem:adjoint_3sc},
and thus $\ffsymbz(A - z)(\tau) = \ffsymbz(A)(\tau) - z$ is invertible
for $z \in \CC \setminus \RR$.

The triple
\begin{align*}
    \left( (\fibsymbz(A) - z)^{-1}, (\mfsymbz(A) - z)^{-1}, (\ffsymbz(A)(\tau) - z)^{-1} \right)
\end{align*}
satisfies the conditions of Proposition~\ref{prop:3sc_quantization} and we find a symbol $b_1(z) \in \CI(\Ttsco^* X)$ such that
\begin{align*}
    \prinsymbz(b_1(z)) = \left( (\fibsymbz(A) - z)^{-1}, (\mfsymbz(A) - z)^{-1}, (\ffsymbz(A)(\tau) - z)^{-1} \right)\,.
\end{align*}
Therefore, we can quantize $b_1(z)$ to an operator $\Op_L(b_1(z)) = B_1(z) \in \Psitsc^{0,0}$ satisfying
\begin{align*}
    (A - z) B_1(z) - \id &= R_{1,1}(z)\,,\\
    B_1(z) (A - z) - \id &= R_{2,1}(z)
\end{align*}
with $R_{1,1}(z), R_{2,1}(z) \in \Psitsc^{-1,-1}(X)$
and by the chain rule $B_1(z), R_{1,1}(z), R_{2,1}(z)$ satisfy the seminorm bounds.

Define
\begin{align*}
    B_k(z) &= \left( \id + \sum_{j=1}^{k-1} R_{1,1}^j \right) B_1(z)\,,\\
    B_k'(z) &= B_1(z) \left( \id + \sum_{j=1}^{k-1} R_{2,1}^j \right) \,.
\end{align*}
\end{proof}

\begin{proof}[Proof of Proposition \ref{prop:func_calc}]
    Let $\tilde\psi \in \CcI(\CC)$ be an almost analytic extension of $\psi$. By the Helffer--Sjöstrand formula, we have that
    \begin{align*}
        \psi(A) &= -\frac{1}{2\pi i} \int_{\CC} \bar{\pa}_z \tilde\psi(z) (A - z)^{-1} \, dz \wedge d\bar{z}\,.
    \end{align*}
    Denote by $B_k(z)$ a family of order $k$ parametrices as in the previous lemma and define the operator $A_{\psi,k}$ by
    \begin{align*}
        A_{\psi, k} &= -\frac{1}{2\pi i} \int_{\CC} \bar{\pa}_z \tilde\psi(z) B_k(z) \, dz \wedge d\bar{z}\,.
    \end{align*}
    Since $\tilde\psi$ is compactly supported, we have that $A_{\psi, k} \in \Psitsc^{0,0}(X)$ for all $k$.
    Denote the error term by $F_k(z) = (A - z)^{-1} - B_k(z)$.
    We have that $\abs{\Im z}^{c'(k)} F_k(z)$ is uniformly bounded on
    $\LinOp(\Sobsc^{s,\ell}(X), \Sobsc^{s+k, \ell+k}(X))$ for some $c'(k)$ by
    the previous lemma and Proposition~\ref{prop:bounded_Sobolev}.
    Hence,
    \begin{align}\label{eq:psiA_mapping}
        \psi(A) - A_{\psi, k} \in \LinOp(\Sobsc^{s,\ell}(X), \Sobsc^{s+k, \ell+k}(X))\,.
    \end{align}
    Also we have that
    \begin{align*}
        A_{\psi, k+1} - A_{\psi, k} \in \Psitsc^{-(k+1), -(k+1)}(X)\,.
    \end{align*}
    By a standard asymptotic summation argument, we obtain $\tilde{A}_{\psi} \in \Psitsc^{0,0}$ such that
    \begin{align*}
        \tilde{A}_\psi \sim A_{\psi, 1} + \sum_{k=1}^\infty A_{\psi,k+1} - A_{\psi, k}\,.
    \end{align*}
    By \eqref{eq:psiA_mapping}, we have that
    \begin{align*}
        \psi(A) - \tilde{A}_\psi : \CmI(X) \lra \CdI(X)
    \end{align*}
    is continuous, hence an element in $\Psisc^{-\infty,-\infty}(X) = \Psitsc^{-\infty,-\infty}(X)$ and therefore
    \begin{align*}
        \psi(A) \in \Psitsc^{0,0}(X)\,.
    \end{align*}
    We calculate the principal symbol as
    \begin{align*}
        \prinsymbz(\psi(A))
        &= - \frac{1}{2\pi i} \int_{\CC} \bar{\pa}_z \tilde\psi(z) \prinsymbz( (A - z)^{-1}) \, dz \wedge d\bar{z}\\
        &= - \frac{1}{2\pi i} \int_{\CC} \bar{\pa}_z \tilde\psi(z) (\prinsymbz(A) - z)^{-1} \, dz \wedge d\bar{z}\\
        &= \psi(\prinsymbz(A))\,.
    \end{align*}
    Note that for the $\ffsymbz$-component, we have used Proposition~\ref{prop:func_calc_sc}.
\end{proof}

Now, we are able to define the operator $G_\psi$.
We fix $E \geq 0$ such that $D_t^2 + H_{V_0} + E$ is invertible
(see Corollary \ref{thm:inverse of elliptic tsc}).
\begin{definition}\label{def:Gpsi}
    For $\psi \in \CcI(\RR)$ and $V_{0} \in S^{-1}(\RR^n_z)$, we set
    \begin{align}\label{eq:def_Gpsi}
        G_{\psi} \coloneqq \psi\left( (D_t^2 + H_{V_0} + E)^{-1} P_{V_0} \right)\,.
    \end{align}
\end{definition}

\begin{remark}\label{rem:Gpsi pole remark}
  The operator $G_{\psi}$ depends on $V_{0}$ explicitly, but we drop this from the
  notation as it causes no confusion; when we work near $\NP$,
  $G_{\psi}$ is defined with $V_{0} = V_{+}$ and when we work near $\SP$,
  $G_{\psi}$ is defined with $V_{0} = V_{-}$.  We do not use
  $G_{\psi}$ away from $\poles$.
\end{remark}

From Proposition~\ref{prop:func_calc} it follows that $G_\psi \in
\Psitsc^{0,0}(X)$ with indicial operator
\begin{align}\label{eq:Gpsi ff}
    \ffsymbz(G_\psi)(\tau) = \psi\left( (\tau^2 + H_{V_0} + E)^{-1} (\tau^2 - H_{V_0}) \right)\,.
\end{align}
Moreover, both the fiber and $\mf$ components of the symbol of
$G_{\psi}$ are independent of $V$.  Indeed, by the assumptions in
Section \ref{sec:V assump}, we have $V \rvert_{\mf} = 0$, so since $V
\in \Psitsc^{0,0}$, we have
\begin{align}\label{eq:Gpsi fib}
      \fibsymbz(G_\psi)&= \left. \psi \left( \frac{p}{\tau^{2} + |\zeta|^{2} +
                         m^2 + E } \right)  \right|_{\Stsc [X;C]} = 
                         \psi \left(
                         \frac{\tau^{2} - |\zeta|^{2}}{\tau^{2} +
                         |\zeta|^{2} } \right)  \\ \label{eq:Gpsi mf}
      \mfsymbz(G_\psi)  &=  \left. \psi \left( \frac{p}{\tau^{2} + |\zeta|^{2} +
                         m^2 + E } \right)  \right|_{\Tsco^{*}_{\mf} [X;C]}
\end{align}

Moreover, we have an elliptic estimate for $B G_\varphi u$ by $Q G_\psi u$
given that $Q$ is elliptic on the wavefront set of $B G_\varphi$ and the support of $\varphi$ is contained in the set $\set{\psi = 1}$.
\begin{proposition}\label{prop:elliptic_regularity_Gpsi}
    Let $B, Q \in \Psitsc^{m,r}$ and $\varphi, \psi \in \CcI(\RR)$ with $\varphi \psi = \varphi$ and $\varphi(0) = 1$.
    If $\WFtsc'(B G_{\varphi}) \subset \Elltsc(Q)$, then for any $N, M \in \NN$ there exists $C > 0$ such that for $u \in \Sobsc^{-N, -M}$,
    \begin{align*}
        \norm{B G_{\varphi} u} \leq C \left( \norm{Q G_\psi u} + \normres{u}\right)\,.
    \end{align*}
\end{proposition}
\begin{proof}
    Using $\tsc$-elliptic regularity for $B G_{\varphi}$, then using
      $G_{\varphi} G_{\psi} = G_{\varphi}$,
      \begin{align*}
        \norm{B G_{\varphi} u} = \norm{(B G_{\varphi}) G_{\psi} u} \le
        C \left( \norm{Q G_\psi u} + \normres{u} \right),
      \end{align*}
      as claimed.
\end{proof}

It will be crucial below that, although $G_\psi$ is not smoothing, its
indicial operator is smoothing for each $\tau$.  This makes it
possible to compose $G_\psi$ with quantizations of functions of $\tau$
alone; an arbitrary composition of even a compactly supported function
of $\tau$ and a symbol in $\tau, \zeta$ is not necessarily a PsiDO.
On the other hand, we see from $G_{\psi,0}$ above that we expect
$\ffsymbz(G_\psi)$ to have symbol rapidly decaying as $|\zeta| \to
\infty$, and indeed it does.
\begin{proposition}\label{prop:Gpsi_indicial}
    We have that
    \begin{align*}
        \ffsymbz(G_\psi) \in \Psisclsc^{-\infty,0,0}\,.
    \end{align*}
\end{proposition}
\begin{proof}
    We can use the functional calculus for positive self-adjoint
    scattering operators of order $(m,0)$, \cite[Proposition
    10.2]{V2000} applied to the function, for $E \ge 0$,
    $\tilde{\psi}(t) \coloneqq \psi( (\tau^2 + t + E)^{-1}(\tau^2 - t) )$.\footnote{Strictly speaking this is not a smooth function on $\RR$, but since the spectrum of $H_{V_0}$ is bounded from below by $E$
    we can change $\tilde{\psi}$ to a smooth function without changing $\tilde{\psi}(H_{V_0})$}
    Then, for every fixed $\tau$, $\ffsymbz(G_\psi)(\tau) \in \Psisc^{-\infty, 0}(\RR^n)$. Since $\tilde{\psi}$ is smooth in $h = 1 / \tau$ up to $h = 0$, the assertion follows.
\end{proof}
Since $V_0$ is static, $D_t^2 + H_{V_0}$ and $D_t^2 - H_{V_0}$ commute and therefore
\begin{align}\label{eq:Gpsi_PV0_commute}
    [G_\psi, P_{V_0}] = 0\,.
\end{align}
Consequently,
\begin{align*}
    [G_\psi, P_V] = -[G_\psi, V'] \in \Psitsc^{1,-1}\,,
\end{align*}
which means that $G_\psi$ and $P_V$ commute to leading order.

The indicial operators of $G_\psi$ and $G_{\psi,0}$, which the Fourier multiplier defined by \eqref{eq:def_Gpsi0}, differ by a lower order operator:
\begin{lemma}\label{lem:diff_Gpsi_Gpsi0}
    We have that
    \begin{align*}
        \ffsymbz(G_\psi) - \ffsymbz(G_{\psi,0}) \in \Psisclsc^{-\infty, -1, -1}\,.
    \end{align*}
\end{lemma}
\begin{proof}
    We have that
    \begin{align*}
        \ffsymbz(G_\psi)(\tau) = \psi\left( (\tau^2 + H_{V_0} + E)^{-1} (\tau^2 - H_{V_0}) \right)\,.
    \end{align*}
    By \cite[Proposition 10.3]{V2000} and the fact that $\scprinsymb{2,0}(H_{V_0}) = \scprinsymb{2,0}(H_0)$, we have that
    \begin{align*}
        \scprinsymbz\left( \ffsymbz(G_\psi)(\tau) \right) = \scprinsymbz\left( \ffsymbz(G_{\psi,0})(\tau) \right)
    \end{align*}
    The semiclassical principal symbol is given by
    \begin{align*}
        \sclsymb{h=1/\tau}\left( \ffsymbz(G_\psi) \right) = \psi\left( (1 + h^2 \Delta)^{-1} (1 - h^2 \Delta) \right)\,,
    \end{align*}
    which is independent of $V$ and therefore agrees with the semiclassical principal symbol of $G_{\psi,0}$.
\end{proof}

Next, we prove the ``shrinking window'' lemma that states that if we choose $\psi \in \CcI$ supported in a sufficiently small neighborhood
of $0$, then the operator norm of $\ffsymbz(G_\psi)(\tau)$ is arbitrary small considered when considered as a map $\Sobscl^{1,1} \to L^2$.
This lemma is crucial in proving the positive commutator estimate, because we can control the error terms coming from commutators with the potential.

\begin{lemma}\label{lem:shrinking_window}
   Assuming the spectrum of $H_{V_{0}}$ is purely absolutely
   continuous near $[m^2, \infty)$, then there is
   $\delta > 0$, such that, 
    for every $\eps > 0$ there exists $\sigma > 0$ such that if $\psi \in \CcI(\RR)$ has
   $\supp \psi \subset (-\sigma, \sigma)$ and $0 \le
    \psi \le 1$, then
    \begin{align*}
        \norm{ \mathbbm{1}_{[m^{2} - \delta, +\infty)}(\tau) \ffsymbz(G_\psi)(\tau) \circ \iota}_{\LinOp(\Sobscl^{1,1}, L^2)} < \eps
    \end{align*}
    where $\iota \colon \Sobscl^{1,1} \hookrightarrow L^2$
    is the natural inclusion map.
\end{lemma}
\begin{remark}
    The same statement holds true for $(-\infty, -m^2 + \delta]$.
\end{remark}
\begin{proof}
  First, let $\kappa > m^{2}$ be fixed; we claim that there $\psi$ as
  in the lemma such that for $\tau \in [m^{2} - \delta, \kappa]$, that
  \begin{align*}
\ffsymbz(G_\psi)(\tau) \circ \iota \colon \Sobscl^{1,1} \lra L^2
    \end{align*}
    has mapping norm less than $\epsilon$.  To see this, recall that
    from Proposition~\ref{prop:semiclassical_bounded} we have
    $\ffsymbz(G_{\psi})$ bounded on $L^{2}$ uniformly in $\tau$.
    Since $\iota \colon \Sobscl^{1,1}(\ff) \hookrightarrow L^2(\ff)$
    is compact, it suffices to show that for any sequence of
    $\psi_n \in \CcI(\RR)$ with $\psi_n(s) = 1$ for $|s| \le 1/n$ and
    $0 \le \psi_n \le 1$, that $\ffsymbz(G_{\psi_n}) \to 0$ in the
    strong operator topology on $L^2$, uniformly for $\tau \in [m^{2}
    - \delta, \kappa]$.  This in turn follows from continuous
    functional calculus, specifically from \eqref{eq:Gpsi ff}, which
    identifies $\ffsymbz(G_{\psi_n})$ with $\psi_n$ of
    \begin{align*}
    (\tau^2 + H_{V_0} + E)^{-1} (\tau^2 - H_{V_0}) = (1 + \rho^2
    H_{V_0} + \rho^2 E)^{-1} (1 - \rho^2 H_{V_0}),
    \end{align*}
    which has purely absolutely continuous spectrum in a neighborhood
    of $0$ by assumption.

    On the other hand, the difference with the free operator satisfies
    \begin{equation}
      \label{eq:9}
      (\tau^2 + H_{V_0} + E)^{-1} (\tau^2 - H_{V_0}) - (\tau^2 + H_{0}
      + E)^{-1} (\tau^2 - H_{0}) \in \Psisclsc^{-1,-1,-1},
    \end{equation}
    and for any $\epsilon' > 0$ there is $\kappa > 0$ such that this difference has norm
    $< \epsilon'$ for $\tau \ge \kappa$.  Again functional calculus
    gives that $G_{\psi}(\tau) - G_{\psi_{0}}(\tau)$ has norm less
    then $\epsilon'$.  But $G_{\psi_{0}}(\tau)$ can be seen to satisfy
    the conclusion of the lemma explicitly, so the lemma follows for $G_{\psi}$.

    If $\tilde \psi$ satisfies $\tilde \psi \psi = \tilde \psi$, then
    also
    \begin{equation}
    G_{\tilde \psi} G_\psi = G_{\tilde \psi}, \label{eq:obvious}
    \end{equation}
    and
    thus the bound holds for all such
    $\ffsymbz(G_{\tilde \psi})$ and the lemma is proven.
\end{proof}

\begin{definition}\label{def:asymptotic_static}
    Let $V \in \rho_{\mf} \Psitsc^{1,0}$ and $r \in \ioo{0, \infty}$.
    We say that $V$ is asymptotically static of order $r$ at $\poles$, if there exist
    $V_{\pm} \in S_{\cl}^{-1}(\RR^n_z)$ and $\chi_{\pm} \in \CcI(X)$ such that
    \begin{enumerate}
        \item $\chi_+(\NP) = 1$ and $\chi_-(\SP) = 1$,
        \item $\chi_{\pm} \cdot (V - V_{\pm}) \in \Psitsc^{1, -r}$.
    \end{enumerate}
\end{definition}
We note that if $V$ is asymptotically static of order $r > 0$, then the static parts $V_{\pm}$ are uniquely determined.

Because the operators $G_{\varphi}$ act as localizers to the
characteristic set of $P_{V}$ over the north pole, $Bu$ should morally
be controlled by $BG_{\varphi}u$ and $P_{V}u$.  We now prove an
estimate in this direction when $B \in \Psitsc^{s, \ell}$ is localized
near $\NP$.  (An equivalent estimate holds near $\SP$.)  We bound $Bu$
in terms of $BG_\psi u$ and $G P_V u$ if $G$ is elliptic on the
wavefront set of $B$.  This estimate is used extensively below to
convert control of $BG_{\varphi}u$ (which is convenient to obtain via
our propagation estimates) into control of $Bu$ when $B$ is elliptic
(the latter control is harder to obtain because any elliptic set over
the front face necessitates that the entire fiber over the front face
is in the operatorial wavefront set):
\begin{proposition}\label{prop:elliptic_Gpsi_estimate}
    Let $V \in \rho_{\mf} \Psitsc^{1,0}$ be asymptotically static of order $r$.
    Let $\varphi \in \CcI(\RR)$ with $\varphi|_{(-\eps, \eps)} \equiv 1$ for some $\eps > 0$,
    $B, G', B' \in \Psitsc^{0, 0}$
    such that
    $B$ is microsupported near $\NP$ and
    \begin{align*}
        \WFtsc'(B) \subset \Elltsc(G') \cap \Elltsc(B')\,.
    \end{align*}
    For all $M, N \in \NN$ and $s, \ell \in \RR$ there exists $C > 0$ such that for all $u \in \Sobres$,
    \begin{align*}
        \norm{B u}_{s,\ell} \leq C \left( \norm{B G_\varphi u}_{s,\ell} + \norm{G' P_V u}_{s-2, \ell} + \norm{B' u}_{s-1,\ell-r} + \normres{u}\right)\,,
    \end{align*}
    provided that the right hand side is finite.
\end{proposition}

The main idea behind this proposition, which is proved below, is
that $I - G_\varphi$ should act as a localizer to the elliptic set
of $P_V$.  This requires explanation, as it is in fact not such a
localizer, i.e. $\WFtsc'(I - G_\varphi) \not \subset \Elltsc(P_V)$.
However, near $\NP$, the leading order part of the operator,
$P_{V_+}$, admits an explicit parametrix with error $G_{\varphi}$.
That is, there exists $E_{\varphi} \in \Psitsc^{-2, 0}$ solving
\begin{equation}
    E_\varphi P_{V_+} =  \id - G_{\varphi}.\label{eq:Evarphi}
\end{equation}
Indeed, this is satisfied by the operator
\begin{equation*}
    E_\varphi \coloneqq F_\varphi \circ (D_t^2 + H_{V_+} + E)^{-1}
\end{equation*}
where
\[
F_\varphi = f_\varphi((D_t^2 + H_{V_+} +
E)^{-1}P_{V_+}) \quad \mbox{ with } \quad f_\varphi(s) = \frac{1 - \varphi(s)}{ s}.
\]
This definition of $F_\varphi$ makes sense as a bounded operator on
$L^2$ by continuous functional calculus since
$(D_t^2 + H_{V_+} + E)^{-1}P_{V_+} \in \Psitsc^{0,0}$ is a bounded
symmetric operator.  Note that $f_\varphi$ is not compactly supported,
so Proposition \ref{prop:func_calc} does not apply directly to $f(A)$
with $A \in \Psitsc^{0,0}$.  However, given an almost analytic
extension $\tilde \varphi$ of $\varphi$, the function
$\tilde f_{\varphi}(z) = (1 - \tilde \varphi(z)) / z$ is an almost analytic
extension of $f_{\varphi}$.  Since $|\tilde{f}_{\varphi}(z)| \lesssim
1/\ang{z}$, the Helffer-Sjöstrand formula \eqref{eq:Helffer-Sjostrand}
and thus Proposition \ref{prop:func_calc} hold with $\varphi$ replaced
by $f_{\varphi}$.

Proposition \ref{prop:elliptic_Gpsi_estimate} is now a consequence of
the mapping properties of $E_{\varphi}$, applied to $P_V u$.  The
``one order better'' estimate from the operator $G'$ is obtained by
writing $P_V = P_{V_+} + V - V_+$ and using the assumptions on
$V - V_+$.  If $P_V$ itself admitted a parametrix with error
$G_\varphi$, we would obtain an estimate with purely
residual terms, i.e.\ no $G'$ present, but this is not to be expected
for $V$ which are not purely static for large $t$.

\begin{proof}[Proof of Proposition \ref{prop:elliptic_Gpsi_estimate}]
   Let $E_{\varphi} \in \Psitsc^{-2, 0}$ be as in \eqref{eq:Evarphi}.
    We write
    \begin{align*}
        B = B G_{\varphi} + B ( \id - G_{\varphi}) = B G_{\varphi} + \tilde{B} P_V + R\,,
    \end{align*}
    where $\tilde{B} = B E_\varphi \in \Psitsc^{-2,0}$ and
    $R = B E_\varphi (P_V - P_{V_+}) \in \Psitsc^{-1,-r}$.
    Then using Theorem~\ref{thm:elliptic regularity theorem}, we obtain
    \begin{align*}
        \norm{B u}_{s,\ell} \lesssim \norm{B G_\varphi u}_{s,\ell} + \norm{G' P_V u}_{s-2, \ell} + \norm{B' u}_{s-1, \ell-r} + \normres{u}\,,
    \end{align*}
    which is what we wanted.
\end{proof}

\subsection{Localization near the characteristic set}\label{sec:tsc
  localizers}

We now construct the operators which we use as microlocalizers and
commutators near $C$.  There will be of the form
\begin{align}\label{eq:gen comm form}
\Op_L(q) G_{\psi}\,,
\end{align}
where $q \in \CI(\Tsc^* X)$. We are mainly interested in the case that
\begin{align}\label{eq:q_central}
 q \rvert_{\Tsc_{\NP}^* X}(\tau, \zeta) = q(0,0,\tau, \zeta) = f (\tau)
\end{align}
for some $f \in \CI(\Wperpo)$, for this mimics the centrality
condition in \eqref{eq:centrality}, with the crucial difference that
in general these $q$ do not lie in $\Stsc^{m,l}(X; \poles)$.
Indeed, even for $f(\tau) \in \CcI(\mathbb{R}_{\tau})$, any function
$q$ satisfying \eqref{eq:q_central} is not a symbol since it does not exhibit
additional vanishing in $\zeta$ under application of $\pa_{\tau}$.

We will see that, despite the fact
that such $q$ are not $\tsc$-symbols, that the operator in
\eqref{eq:gen comm form} is a $\tsc$-operator provided 
\begin{equation}
  \label{eq:q in blow up}
  q \in \CI([\Tsco^* X; \fibeq])
\end{equation}
This is the same blow up that appears in the proof of Lemma
\ref{thm:normal is sclsc}.  Rather than describe this blow up in detail, we simply say that $q$
satisfies \eqref{eq:q in blow up} if:  (1) it is a classical symbol in
$\zeta$ of order zero and smooth in $\tau$ in regions $|\tau| < C$,
(2) in regions with $|\zeta| < C$, $\pm \tau > C > 0$ it is smooth in
$\rho = \pm 1/\tau$ down to $\rho = 0$, and (3) in regions where
$\mu = \zeta / \rho$ is bounded, it is a classical symbol in $\mu$,
smooth in $\rho$ down to $\rho = 0$.

While, \eqref{eq:q in blow up} will ensure that \eqref{eq:gen comm
  form} lies in $\Psitsc^{0,0}$, the condition \eqref{eq:q_central}
will ensure that commutators with $P_V$ have the correct order (i.e.\
lose one order compared with composition.)  This is established in the
following proposition, in which we use $\rho = 1/\tau$ to modify the
differential order so that the centrality condition is preserved; this
leads to the appearance of factors of $\ang{\tau, \zeta}/ \tau$ in the
local components of the symbol.  Also (see Remark \ref{rem:positive
  prefactor}) we use powers of $x$ to rescale the symbol.  Here we
also define all symbols using a fiber defining function $\rho_{\fib}$
which is equal to $\rho$ near the characteristic set.

    This proof follows from the same arguments as in Vasy~\cite[Proposition 13.1]{V2000}. The main difference is that since $G_\psi$ is of differential order $0$ and vanishes to infinite order at $\fibeq$
    (by Proposition~\ref{prop:Gpsi_indicial}),
    the product is a $\tsc$-operator if $q$ is smooth on $[\Tsco^* X; \fibeq]$.

\begin{proposition}\label{prop:Q_Gpsi}
    Let $q \in \CI([\Tsco^* X; \fibeq])$ satisfying \eqref{eq:q_central}.
    Then for $s, \ell \in \RR$,
    \begin{align*}
        \Op_L( x^{-\ell} \rho^{-s} q) G_\psi \in \Psitsc^{s,\ell}(X)\,.
    \end{align*}
    The components of the principal symbol are
    \begin{align*}
        \ffsymbz(\Op_L(x^{-\ell} \rho^{-s} q) G_\psi)(\tau) &=
        \rho^{-s} f(\tau) \cdot \ffsymbz(G_{\psi})(\tau)\,,
    \end{align*}
    and 
    \begin{align*}
        \fibsymbz(\Op_L(x^{-\ell} \rho^{-s} q) G_\psi) &=
        q|_{\Ssc^* X} \cdot \fibsymbz(G_\psi)\,,\\
        \mfsymbz(\Op_L(x^{-\ell} \rho^{-s} q) G_\psi) &=
        q|_{\Tsc^*_{\pa X} X} \cdot \mfsymbz(G_{\psi})\,,
    \end{align*}
    where we have implicitly used that $\rho_{\infty} = x, \rho_{\fib}
    = \rho$ near $\NP$ on $\supp (\fibsymbz(G_\psi))$ and on $\supp
    (\mfsymbz(G_{\psi}))$.  
\end{proposition}

\begin{remark}
    Note that we implicitly assume that
    \begin{align*}
        x^{-\ell} \rho^{-s} q \in \ang{t, z}^{-\ell} \ang{\tau, \zeta}^{-s} \CI([\Tsco^* X; \fibeq])\,.
    \end{align*}
    This condition is for instance satisfied if $q$ vanishes for $\abs{\tau}$ small and supported near $\NP$.
\end{remark}

While $\Op_L(q)$ is not in the $\tsc$-calculus in general, we can find $Q \in \Psitsc^{0,0}$ such that $Q G_\psi = \Op_L(q) G_\psi$.
\begin{lemma}\label{lem:QGpsi_eq_OpqGpsi}
    Let $q \in \CI([\Tsco^* X; \fibeq])$ satisfying \eqref{eq:q_central} and $s, \ell \in \RR$.
    Assume that there is $\tau_0 \in \Wperpo$ such that $f(\tau_0) \not = 0$.
    Then there exists $Q \in \Psitsc^{s,\ell}$ such that
    \begin{align*}
        Q G_\psi = \Op_L( x^{-\ell} \rho^{-s} q) G_\psi
    \end{align*}
    and $\tau_0 \in \Elltsc(Q)$.
\end{lemma}
\begin{proof}
    Choose $\varphi \in \CcI(\RR)$ with $\varphi|_{\supp \psi} \equiv 0$ so that $\varphi \psi = \psi$.
    We set
    \begin{align*}
        Q = \Op_L( x^{-\ell} \rho^{-s} q) G_\varphi + \tilde{Q} (\id - G_\varphi)\,,
    \end{align*}
    where $\tilde{Q} \in \Psisc^{s,\ell}$ with $\ffsymbz(\tilde{Q}) = f(\tau_0) \id$.
    This $Q$ satisfies the claimed equality, since $G_\varphi G_\psi = G_\psi$ and
    $\ffsymb{\ell}(Q)(\tau_0) = \tau_0^{s} f(\tau_0) \id$, and hence $\tau_0 \in \Elltsc(Q)$.
\end{proof}

We note that the commutator with $P_V$ decreases the order
as expected:
\begin{proposition}
    The commutator of $P_V$ and $Q G_\psi$ satisfies $[P_V, Q G_\psi] \in \Psitsc^{s+1,\ell-1}(X)$ and
    \begin{align*}
        \tau^{-s} \ffsymb{\ell-1}([P_V, Q G_\psi]) &= - f(\tau) [\pa_x V'|_{x = 0}, (\widehat{G_\psi})_{\ff}] \\
        &\phantom{=} - [H_{V_0}, \Op_L(\pa_x q(0,0,\tau,\zeta) + z \pa_y q(0,0,\tau,\zeta) )] (\widehat{G_{\psi}})_{\ff} \\
        &\phantom{=} - 2 i \ell \tau f(\tau) (\widehat{G_\psi})_{\ff}\,.
    \end{align*}
\end{proposition}
\begin{proof}
    Without loss of generality, we may assume that $s = 0$.
    By Proposition~\ref{prop:Q_Gpsi}, we have that the front face symbols of $Q G_\psi$ and $P_{V}$ commute.
    To calculate the indicial operator of the commutator, we use Proposition~\ref{prop:commutator_indicial} with $A = P_V$ and $B = Q G_\psi$.
    We have that
    \begin{align*}
        \hat{A}_{\ff} &= \tau^2 - H_{V_0}\,,\\
        D_\tau \hat{A}_{\ff} &= -2 i \tau\,,\\
        \hat{A}_{\ff}' &= - \pa_x V'|_{x = 0}\,,\\
        \hat{B}_{\ff} &= f(\tau) (\widehat{G_{\psi}})_{\ff}\,,\\
        D_\tau \hat{B}_{\ff} &= -if'(\tau) (\widehat{G_{\psi}})_{\ff} + f(\tau) D_{\tau} (\widehat{G_\psi})_{\ff}\,,\\
        \hat{B}_{\ff}' &= \Op_L(\pa_x q(0,0,\tau,\zeta) + z \pa_y q(0,0,\tau, \zeta)) (\widehat{G_{\psi}})_{\ff} + f(\tau) (\widehat{G_{\psi}})_{\ff}'\,.
    \end{align*}
    Using that $(\widehat{G_\psi})_{\ff}$ and $H_{V_0}$ commutate we calculate
    \begin{align*}
        [\hat{A}_{\ff}' - D_\tau \hat{A}_{\ff}, \hat{B}_{\ff}] &= - f(\tau) [\pa_x V'|_{x = 0}, (\widehat{G_\psi})_{\ff}]\,,\\ 
        [\hat{A}_{\ff}, \hat{B}_{\ff}' - D_\tau \hat{B}_{\ff}] &= - [H_{V_0}, \hat{B}_{\ff}']\,,\\ 
        &= -[H_{V_0}, \Op_L(\pa_x q(0,0,\tau,\zeta) + z \pa_y q(0,0,\tau, \zeta))] (\widehat{G_{\psi}})_{\ff} - f(\tau) [H_{V_0}, (\widehat{G_{\psi}})_{\ff}'] \,.
    \end{align*}
    The second summand vanishes because $(H_{V_0})' = 0$ implies that $(\widehat{G_\psi})'_{\ff} = 0$.
    This proves the claim.

\end{proof}
The positive commutator argument will evaluate the commutator of the unperturbed operator in the scattering calculus and therefore we need to
compare the commutator in the perturbed and unperturbed setting (cf. \cite[Corollary 13.4]{V2000}):
\begin{corollary}\label{cor:commutator_symbol}
    Let
    \begin{align*}
        R(\tau) &\coloneqq \tau^{-s} \left(\ffsymbz([P_V, Q G_\psi])(\tau) - \ffsymbz([P_0, Q G_{\psi,0}])(\tau)\right)\,.
    \end{align*}
    Then,
    \begin{align*}
        R \in \Psisclsc^{-\infty,-1, 0}
    \end{align*}
    and
    \begin{align*}
        \norm{R(\tau)}_{\LinOp(L^2, \Sobscl^{1,1})} \lesssim \sup\{ \abs{ D_{x,y,\tau}^\alpha D^\beta_\zeta q(0,0,\tau,\zeta)} \colon \abs{\alpha} \leq 1, \abs{\beta} \leq c n \}\,,
    \end{align*}
    where $c > 0$ is a universal constant and the implied constant is independent of $\tau$ and $q$.
\end{corollary}
\begin{proof}
    For brevity, we set $q_x(\tau, \zeta) \coloneqq \pa_x q(0,0,\tau,\zeta)$ and $q_y(\tau, \zeta) = \pa_y q(0,0,\tau,\zeta)$.
    From the previous proposition, we calculate that
    \begin{align*}
        R(\tau) &= - f(\tau) [\pa_x V'|_{x=0}, (\widehat{G_\psi})_{\ff}] + [V_0, \Op_L( q_x + z q_y)] (\widehat{G_\psi})_{\ff} \\
        &\phantom{=} -i \left( -2 \Op_L(\zeta q_y) + 2 \ell \tau f(\tau) \right) \left( (\widehat{G_\psi})_{\ff} - (\widehat{G_{\psi,0}})_{\ff} \right)\,.
    \end{align*}
    We have that $(\widehat{G_\psi})_{\ff} \in \Psisclsc^{-\infty,0,0}$ and therefore the first term is in $\Psisclsc^{-\infty,-1,0}$ and bounded independently of $q$ and $\tau$.
    The commutator $[V_0, \Op_L(q_x + z q_y)]$ is in $\Psisclsc^{0, -1, 0}$ and we can estimate its operator norm as a map $\Sobsc^{1,0} \to \Sobsc^{1,1}$ by
    \begin{align*}
        C \sup\{ \abs{ D_{x,y}^\alpha D^\beta_\zeta q(0,0,\tau,\zeta)} \colon \abs{\alpha} = 1, \abs{\beta} \leq c n \}\,,
    \end{align*}
    and composition with $(\widehat{G_\psi})_{\ff}$ is a bounded map $L^2 \to \Sobsc^{1,1}$ with the same norm.

    From Lemma~\ref{lem:diff_Gpsi_Gpsi0}, we have that $(\widehat{G_\psi})_{\ff} - (\widehat{G_{\psi,0}})_{\ff} \in \Psisclsc^{-\infty,-1,-1}$ and
    \begin{align*}
        \Op_L(\zeta q_y) - \ell \tau f(\tau) \in \Psisclsc^{0,0,1}
    \end{align*}
    and the operator norm is bounded by
    \begin{align*}
        C \sup\{ \abs{ D_y^\alpha D^\beta_\zeta q(0,0,\tau,\zeta)} \colon \abs{\alpha} \leq 1, \abs{\beta} \leq c n \}\,,
    \end{align*}
    which completes the proof.
\end{proof}

\subsection{G\aa{}rding type theorems}

In this section, we state and prove a sharp G\aa{}rding type theorem for $\tsc$-operators.
In contrast to Vasy~\cite{V2000}, we use a localization operator $G_\psi$ that is in $\Psitsc^{0,0}$ and therefore we have to
take the fiber principal symbol in account.
We will consider the general situation of a localizer $\psi(A)$, where $A \in \Psitsc^{0,0}$ is self-adjoint.
In the case of the Klein-Gordon equation, we take $A = (D_t^2 +
H_{V_0} + E^2)^{-1} P_{V_0}$ in which case $\psi(A) = G_\psi$.  

Our proof again follows \cite{V2000}, in particular using a method for
construction of square roots of operators, which we recall now.  In
\cite{V2000} this appears as Lemma C.1, but as we do not use it
directly here, we merely state the result and recall the method of
proof.  

Assume that we are given self-adjoint operators $A, Q \in
\Psitsc^{0,0}(X)$, $c > 0$ and $\psi \in \CcI(\RR)$ real-valued with
$\psi(x) = 1$ for $\abs{x} < \delta$ for some $\delta > 0$.  Then,
if we have a bound from below of the form
    If
    \begin{align*}
        \psi(A) Q \psi(A) \geq c \psi(A)^2\,,
    \end{align*}
    then for any $c' \in (0,c)$ and $\phi \in \CcI(\RR)$ with $\phi
    \psi = \phi$,  we can find a square root $B \in \Psitsc^{0,0}(X)$
    in the sense that
    \begin{align}\label{eq:square root bulk}
        \phi(A) (Q - c') \phi(A) = \phi(A) B^* B \phi(A)\,.
    \end{align}

    We recall the proof almost verbatim from \cite{V2000}, the main
    difference being that our $\psi(A)$ is in $\Psitsc^{0,0}$ as
    opposed to $\Psitsc^{- \infty, 0}$.  Define
    \begin{align*}
        P = \psi(A) Q \psi(A) + c (\id - \psi(A)^2) \in \Psitsc^{0,0}(X)\,.
    \end{align*}
    Since $P \geq c$ we have that $P - c' \geq c - c' > 0$.
    Then we can apply Proposition~\ref{prop:func_calc} to take the
    square root of $P - c'$, i.e.\ we take $f(P - c')$ with a function $f \in
    \CcI(\RR)$ such that $f(t) = \sqrt{t}$ on the spectrum of $P - c'$. 
    The function exists because $\sigma(P - c') \subset [c-c', C]$ for
    some $C > 0$. We then have that
    \begin{align*}
        \tilde{P} \coloneqq (P - c')^{1/2} \in \Psitsc^{0,0}(X)\,.
    \end{align*}
    We choose a $\psi_1 \in \CcI(\RR)$ with $\psi_1 \equiv 1$ on $\supp \phi$ and $\psi_1 \equiv 0$ on $\supp (1 - \psi)$. We calculate that
    \begin{align*}
        \psi_1(A) \tilde{P}^2 \psi_1(A) &= \psi_1(A) (P - c') \psi_1(A) \\
        &= \psi_1(A) (A - c') \psi_1(A)\,.
    \end{align*}
    Let $B = \tilde{P} \psi_1(A)$, then multiplying the previous
    equation yields the equation \eqref{eq:square root bulk}.

We now have the non-sharp G\aa{}rding inequality which operates under the assumption that the principal symbol is strictly positive.
\begin{proposition}\label{thm:garding}
    Let $A, Q, C \in \Psitsc^{0,0}(X)$ be self-adjoint and assume
    $\ffsymbz(A) \in \Norop^{-\infty, 0,0}$.

    Suppose that the principal symbol of $C$ satisfies
    \begin{align*}
        \fibsymbz(C) &= c_{\fib} \cdot \psi_0(\fibsymbz(A))^2\,,\\
        \mfsymbz(C) &= c_{\mf} \cdot \psi_0(\mfsymbz(A))^2\,,\\
        \ffsymbz(C) &= c_{\ff} \cdot \psi_0(\ffsymbz(A))^2\,,
    \end{align*}
    where $\psi_0 \in \CcI(\RR)$, $c_{\fib} \in \CI(\Ssc^*_{X \setminus C} X)$, $c_{\mf} \in \CI(\Tsc^*_{\pa X \setminus C} X)$ and $c_{\ff} \in \CI(\Wperpo)$.
    We assume that
    \begin{enumerate}
        \item $c_{0}\le c_{\bullet} \le c_{0}'$ for $\bullet \in \{\fib, \mf,
          \ff\}$ and some $c_{0}', c_{0} > 0$,
        \item $\psi_0(x) = 1$ for $\abs{x} < \delta_0$,
        \item there exists $\psi \in \CcI(\RR)$ with $\psi(x) = 1$ for $\abs{x} \leq \delta_1$ and $\supp \psi \cap \supp (1 - \psi_0) = \varnothing$ such that
            \begin{equation} \label{eq:Garding_prinsymb_ineq}
            \begin{aligned}
                \psi(\fibsymbz(A)) \fibsymbz(Q) \psi(\fibsymbz(A)) &\geq c_{\fib} \psi(\fibsymbz(A))^2\,,\\
                \psi(\mfsymbz(A)) \mfsymbz(Q) \psi(\mfsymbz(A)) &\geq c_{\mf} \psi(\mfsymbz(A))^2\,,\\
                \psi(\ffsymbz(A)) \ffsymbz(Q) \psi(\ffsymbz(A)) &\geq c_{\ff} \psi(\ffsymbz(A))^2\,.
            \end{aligned}
            \end{equation}
    \end{enumerate}

    Then for any $\eps \in (0,1)$ and $\phi \in \CcI(\RR)$ with $\supp \phi \cap \supp (1 - \psi) = \varnothing$,
    there exists $R \in \Psitsc^{-1,-1}(X)$ such that
    \begin{align*}
        \phi(A) Q \phi(A) \geq (1 - \eps) \phi(A) C \phi(A) + R\,.
    \end{align*}
\end{proposition}
\begin{proof}
    The idea is to construct a square-root of $Q - (1 - \eps) C$
    modulo lower order terms.  We follow the methodology described
    before the proof to take square roots first of each of the symbols
    individually.  That is, we write, for $\bullet \in \{ \mf, \ff\}$,
    \begin{align*}
        P_{\bullet}(\tau) = \hat N_{\bullet}(\psi(A) Q \psi(A)) +
      c_{\bullet} (\id - \hat N_{\bullet}(\psi(A))^2) 
    \end{align*}
    and
        \begin{align*}
        P_{\fib}(\tau) = \fibsymbz(\psi(A) Q \psi(A)) +
      c_{\fib} (\id - \fibsymbz(\psi(A))^2).
        \end{align*}
        Then, with $f \in \CcI(\RR)$ with $f(x) = \sqrt{x}$ for
        $\epsilon \le x \le c$ for $c$ sufficiently large, we let
        \begin{align*}
          B_{\bullet} &= f(P_{\bullet} - (1 - \epsilon) c_{\bullet})
                        \psi_{1}(A),\quad \bullet \in \{ \fib, \mf, \ff\}\,.
        \end{align*}

    The symbols $B_{\fib}, B_{\ff}, B_{\mf}$ satisfy the conditions of
    Proposition~\ref{prop:3sc_quantization}; the matching conditions
    follow easily, and the fact that $B_{\ff}$ lies $\Norop^{0,0,0}$ and
    satisfies the appropriate smoothness condition follows from the
    fact that $P_{\ff}$ is semiclassically scattering elliptic and
    $\ffsymbz(\psi_{1}(A)) \in \Norop^{-\infty, 0 , 0}$.

    Therefore we find a $B \in \Psitsc^{0,0}(X)$
    with $\fibsymbz(B) = B_{\fib}, \ffsymbz(B) = B_{\ff}$ and
    $\mfsymbz(B) = B_{\mf}$.
    Hence there is a $R \in \Psitsc^{-1,-1}(X)$ such that
    \begin{align*}
        \phi(A) \left( Q - (1 - \eps) C \right) \phi(A) &= \phi(A) B^* B \phi(A) + R\,.
    \end{align*}
    Since $\phi(A) B^* B \phi(A) \geq 0$ this proves the proposition.
\end{proof}

Lastly, we have the sharp G\aa{}rding inequality that only assumes that the principal symbol is non-negative.
\begin{proposition}\label{prop:sharp_Garding}
    Let $A, Q, C \in \Psitsc^{0,0}(X)$ be self-adjoint and assume
    $\ffsymbz(A) \in \Norop^{-\infty, 0,0}$.

    Suppose that the principal symbol of $C$ satisfies
    \begin{align*}
        \fibsymbz(C) &= c_{\fib} \cdot \psi_0(\fibsymbz(A))^2\,,\\
        \mfsymbz(C) &= c_{\mf} \cdot \psi_0(\mfsymbz(A))^2\,,\\
        \ffsymbz(C) &= c_{\ff} \cdot \psi_0(\ffsymbz(A))^2\,,
    \end{align*}
    where $\psi_0 \in \CcI(\RR)$, $c_{\fib} \in \CI(\Ssc^*_{X \setminus C} X)$, $c_{\mf} \in \CI(\Tsc^*_{\pa X \setminus C} X)$ and $c_{\ff} \in \CI(\Wperpo)$.

    Assume that
    \begin{enumerate}
        \item $c_{\bullet} \geq 0$ for $\bullet \in \{\fib, \mf,
          \ff\}$,
          \item either $c_{\ff}$ vanishes in a neighborhood of $\pm
            \infty$ or $c_{\ff}(\pm \infty) > 0$,
        \item if $c_{\fib}(\xi) = 0$ for $\xi \in \Ssc^*_{X \setminus C} X$, then $\fibsymbz(Q)(\xi) = 0$ and the analogous condition for $c_{\mf}$ and $c_{\ff}$,
        \item $\sqrt{c_{\bullet}} \in \CI$ and vanishes to infinite order at points $\xi$ where $c_{\bullet}(\xi) = 0$,
        \item \label{it:assump_bounded}  the symbols
          \[
            c_{\ff}^{-1} \ffsymbz(Q),\quad c_{\mf}^{-1} \mfsymbz(Q), \quad
            c_{\fib}^{-1} \fibsymbz(Q)
          \]
          are bounded together with all their
derivatives on $\Wperpo, \Ttsco^{*}_{\mf}[X;C], \Stsc^{*}[X;C]$, respectively,

        \item $\psi_0(x) = 1$ for $\abs{x} < \delta_0$,
        \item there exists $\psi \in \CcI(\RR)$ with $\psi(x) = 1$ for $\abs{x} \leq \delta_1$ and $\supp \psi \cap \supp (1 - \psi_0) = \varnothing$ such that
            \begin{equation} \label{eq:sharp_Garding_prinsymb_ineq}
            \begin{aligned}
                \psi(\fibsymbz(A)) \fibsymbz(Q) \psi(\fibsymbz(A)) &\geq c_{\fib} \psi(\fibsymbz(A))^2\,,\\
                \psi(\mfsymbz(A)) \mfsymbz(Q) \psi(\mfsymbz(A)) &\geq c_{\mf} \psi(\mfsymbz(A))^2\,,\\
                \psi(\ffsymbz(A)) \ffsymbz(Q) \psi(\ffsymbz(A)) &\geq c_{\ff} \psi(\ffsymbz(A))^2\,.
            \end{aligned}
            \end{equation}
    \end{enumerate}

    Then for any $\eps \in (0,1)$ and $\phi \in \CcI(\RR)$ with $\supp \phi \cap \supp (1 - \psi) = \varnothing$,
    there exists $R \in \Psitsc^{-1,-1}(X)$ such that
    \begin{align*}
        \phi(A) Q \phi(A) \geq (1 - \eps) \phi(A) C \phi(A) + R\,.
    \end{align*}
\end{proposition}
\begin{proof}
In the case that $c_{\ff}(\tau) = 0$ for $\tau \gg 0$, the argument is
identical to that in \cite[Proposition C.3]{V2000}.  In the case
$c_{\ff}(+ \infty) \neq 0$, we put the previous lemma microlocally
near $+ \infty \in \Wperpo$ and then the same argument again from \cite{V2000}.
\end{proof}

\section{Propagation estimates over \texorpdfstring{$\poles$}{C}}\label{sec:3sc_propagation}

We now prove propagation of singularities estimates over $\poles$.  As
elsewhere, since the arguments are identical at the two points in $\poles$, we focus on $\NP$.
Our commutant construction follows \cite[Chap.\ 14]{V2000} closely,
and as such we attempt to be faithful to the notation there for ease
of comparison, although we make some changes in order to decrease the
overall amount of notation, which is substantial.  

Note that $\tau$ is preserved along the flow at $\NP$; as described above, the global nature of the operator above $\ff$
leads to propagation phenomenon analogous to diffraction, namely that
singularities entering at a given $\tau$ level in the characteristic
set at $\NP$ may emerge, still at level $\tau$, in any direction.
Thus, if we wish to control a distribution $u$ at a specific $\tau_0$
in $\Wperpo$ over $\NP$, we must assume a priori control of $u$
along all bicharacteristics which enter $\NP$ at that $\tau_0$ level.

To formulate this rigorously, recall $\gamma_{\tsc} : \Ctscd \lra
\mathcal{P}( \pa \Ttsco^*[X;C])$ defined in \eqref{eq:gamma tsc easy},
which in particular associates to each $\tau_0 \in \Wperpo$ the full
$\{ \tau = \tau_0 \}$ slice above $\ff$.
We also define the projection onto $\tau$ levels as follows.
Recalling that $\Char(P_0) \subset \Tsco^* X$, i.e.\ that we include fiber
infinity in the characteristic set of $P_0$, we define the map that
records both the spacetime location and the $\tau$ level of a point in
the characteristic set of $P_0$, (possibly
$\pm \infty$),
\begin{equation} \label{eq:pi X tau}
\begin{aligned}
    \pi_{X, \tau} : \Char(P_0) &\lra X \times \overline{\RR}\,,\\
    (x,y,\tau, \zeta) &\mapsto (x, y, \tau)\,.
\end{aligned}
\end{equation}
This map is well-defined up to the fiber boundary since $\Char(P_0)$
has empty intersection with the fiber equator.

We introduce the notion of control under the Hamilton flow that arises
in the propagation estimates.
\begin{definition}
    Let $U, V, W \subset \Ctscd$. We say that
    $U$ is (backward) controlled by $V$ through $W$ if
    for all $\alpha \in \Char(P_0)$ that are incoming to $U$, in sense that
    \begin{align*}
        \pi_{X, \tau}(\alpha) \in \pi_{X,\tau}( \gamma_{\tsc}(U) \cap \Char(P_0))\,,
    \end{align*}
    there exists $s_{\alpha} <0$ such that
    \begin{align*}
        \exp(s_{\alpha} \Hamsc_p)(\alpha) \in V
    \end{align*}
    and for all $s \in [s_{\alpha}, 0]$,
    \begin{align*}
        \exp(s \Hamsc_p)(\alpha) \in W\,.
    \end{align*}
\end{definition}

The following is the main propagation estimate we use near $\poles$ and
away from the radial sets.
\begin{proposition}\label{prop:propagation_localized}
    Let $V \in \rho_{\mf} \Psitsc^{1,0}$ be asymptotically static of order $r \geq 1$ and
    \begin{align*}
        V - V^* \in \Psitsc^{0,-2}\,.
    \end{align*}
    Moreover, we assume that $H_{V_{\pm}}$ have purely absolutely continuous spectrum in $[m^2, \infty)$.
    
    Let $\delta > 0$ sufficiently small, $\varphi, \psi_1, \psi_2 \in \CcI(\RR)$ with $\supp \varphi \subset \ioo{-\delta, \delta}$ and $\psi_j|_{\ioo{-\delta, \delta}} \equiv 1$,
    and $B, E, G, G', B' \in \Psitsc^{0,0}$ such that
    \begin{enumerate}
        \item $\WF_{\ff}'(E) = \varnothing$,
        \item $\Hamsc_p(\alpha) \not = 0$ for all $\alpha \in \gamma_{\tsc}(\Elltsc(G))$,
        \item $\WFtsc'(B G_\varphi)$ is controlled by $\Elltsc(E)$ through $\Elltsc(G)$,
        \item $\WFtsc'(B) \subset \Elltsc(G') \cap \Elltsc(B')$.
    \end{enumerate}

    For all $M, N, s, \ell \in \RR$ and $u \in \Sobres$ with
    $E G_{\psi_1} u \in \Sobsc^{s,\ell}$,
    $G G_{\psi_2} P_V u \in \Sobsc^{s-1, \ell+1}$,
    $G' P_V u \in \Sobsc^{s-2, \ell}$,
    and $B' u \in \Sobsc^{s-1, \ell-r}$,
    it follows that $Bu \in \Sobsc^{s,\ell}$ and
    \begin{align*}
        \norm{B u}_{s,\ell} &\leq C \big( \norm{E G_{\psi_1} u}_{s,\ell} + \norm{G G_{\psi_2} P_V u}_{s-1, \ell+1} + \norm{G' P_V u}_{s-2, \ell} + \norm{B' u}_{s-1, \ell-r} \\
        &\phantom{\leq C \big(} + \normres{u} \big)\,.
    \end{align*}
\end{proposition}

The proof of this proposition comes at the end of this section.

\begin{remark}
In words, the proposition states that, to obtain estimates at a given
$\tau$ level over $\NP$, we must control the backward flow out of
the entire sphere $|\zeta|^2 = \tau^2 + m^2$ in a neighborhood of
$\NP$.  The elliptic set of $E_{0}$ must control this set in the sense
that it must contain a transversal of the sphere earlier along the
flow.  In particular, the elliptic set of $G$ over $\ff$ must contain the whole of
$\WF'_{\ff}(B G_{\varphi})$.
\end{remark}

We have two types of estimates in this section, one
microlocalized near $\tau \in W^\perp$ with $m < |\tau|
< \infty$, for which the arguments follow most closely those in
\cite{V2000}, and the other microlocalized at $\tau = \pm \infty
\in \Wperpo$, which requires more substantial modifications.

The proofs in these two settings are similar, not only in their
overall structure, but in the specific functions which define the
commutators and the proofs of the various properties of the attendant
operators.  We thus focus on the case $\pm \infty \in \Wperpo$, and in
fact to $\tau = + \infty$, and the argument for finite $\tau$ (as for
$\tau = - \infty)$ is a
straightforward adaptation.

Both cases involve the consideration of the set $\Sigma$ (not the
characteristic set!) which is the closure in $\Tsco^*X$ of the set
$\zeta \cdot \ y = 0$ in a region $|y| < c$ of $\NP$.  In the
coordinates $(x, y, \rho, \mu)$ above in equation \eqref{eq:great
  coords}, $\Sigma$ is
\begin{equation}
  \label{eq:Sigma}
 \Sigma = \{ (x, y, \rho, \mu) : \mu \cdot y = 0, \mu
  \neq 0 , |y| < c\},
\end{equation}
and the value of $c$ is irrelevant below as we will localize our
estimates in small neighborhoods of $\NP$. In these coordinates it is
clear that $\Sigma$ is smooth up to fiber infinity.  (This same
$\Sigma$ is used in Vasy, but there only its finite $\xi, \zeta$
points are relevant; we use it out to infinity.)

\begin{remark}
    Since the estimates are microlocal, it suffices to prove the case $V_{\pm} = V_0$ for some $V_0 \in \Psitsc^{1,0}$ and
    \begin{align*}
        V - V_0 \in \Psitsc^{1, -r}\,.
    \end{align*}

    We also assume for simplicity for most of this section that
    \begin{align*}
    V - V^{*}= 0.
    \end{align*}
    Indeed, without this assumption the commutators which arise below
    involve $P_{(V + V^{*})/2}$ and the $V - V^{*}$ appears as an error, but it is
    clearer to make this realness assumption and use commutators with
    $P_V$, and then discuss the generalization of $V$ later under the
    assumption in \eqref{eq:sub prin assump}.
\end{remark}

We first work with the free Klein-Gordon operator $P_0$ and then
relate its commutators to those of $P_V$. The commutators of $P_0$
can be analyzed directly using its Hamilton vector field, and we
need in particular to analyze this vector field's behavior at $\Sigma$, both
near and away from fiber infinity over $\NP$.  
\begin{lemma}
    The set $\Sigma$ is a smooth submanifold of
    $\Tsco^*X$.  Moreover,
    the Hamilton vector field $\Hamscp$
    is transversal to $\Sigma \cap \Tsco^*_{\NP} X$ in a neighborhood
    of $\Char(P_0)$.  
\end{lemma}
\begin{proof}
  At $y = 0$, away from fiber infinity, $\Hamf  = - 2x (\zeta \cdot
  \partial_y)$, so the rescaled Hamilton vector field $\Hamscp$ is
  $\Hamscp = -  \mu \cdot \partial_y$ at $\NP$   \eqref{eq:hammy scat}.  Since
  the condition defining $\Sigma$ is $\mu \cdot y = 0$, thus
  \begin{equation}
      \Hamscp (\mu \cdot \eta) = -  |\mu|^2\label{eq:transversal to Sigma}
  \end{equation}
on $\NP$ and is thus non-zero near the
  characteristic set near $\NP$, which is what we wanted.  
\end{proof}

\begin{remark}\label{rem:solvability_hamf}
    Consequently, there exists a neighborhood $U' \subset \Tsco^* X$ of $\Char(P_0) \cap \Tsco^*_{\NP} X$ on which we can solve the Cauchy problem
    \begin{align*}
        \Hamf f = 0\,,\quad f|_{\Sigma} &= f_0\,.
    \end{align*}
\end{remark}

The proposition will follow from estimates localized near points in
$\Wperpo$, and the positive commutator argument we use to establish
such estimates is accomplished using commutators localized near $\NP$
and near $\tau$.
Formally speaking, we will use a positive commutator argument
analogous to that used in the scattering setting discussed in Section
\ref{sec:model-case}.  We use the commutators constructed for this
purpose in Section \ref{sec:tsc localizers} of the form
\begin{equation}\label{eq:meet the commutator}
i [P_V, G_\psi Q^* Q G_\psi]
\end{equation}
where $\psi \in \CcI$ and $Q$ is constructed  analogously to the
corresponding commutant in \cite{V2000}.  In particular, we will take $q \in \CI([\Tsco^* X;
\fibeq])$ satisfying the centrality condition $q \rvert_{\NP} = f(\rho)$ where $f
\in \CI(\Wperpo)$.  Concretely,
\begin{equation}
    Q = \Op_L(x^{-(\ell+1/2)} \rho^{-(s-1/2)} q) \,, \quad q = \chi_\partial(x) \tilde q\,,\label{eq:q notation def}
\end{equation}
for a cutoff function $\chi_\partial$ supported near $0$ and $\tilde q$ a
function on $\Tsc^*_{\partial X} X$, and $\tilde q$ essentially given
by \cite[Eq.\ 14.20]{V2000}, with modification that we clarify below.
Furthermore, $\tilde q$ itself is defined first near the
characteristic set (in a neighborhood of the sort depicted by Figure~\ref{fig:localization-near-charset}), this is the function $\tilde q_0$ in \cite[Eq.\
14.11]{V2000}, and then on a complement of a neighborhood of the
characteristic set using a partition of unity.  We describe this in
detail below.

First, we make the following remarks about this commutator and its
important features, both of which appear in the proof in Vasy \cite{V2000}.
\begin{itemize}
\item As discussed in Section \ref{sec:commutators}, $Q G_\psi \in
  \Psitsc^{s - 1/2, \ell + 1/2}$ satisfies
  \begin{equation*}
  [P_V, G_\psi Q^* Q G_\psi] \in \Psitsc^{2s,
    2 \ell},
\end{equation*}
and, by choosing the support of $f = q \rvert_{\NP}$
  localized around a given $\tau_0 \in \Wperpo$, we will obtain
  estimates localized near that $\tau_0$.  (Note we want more than
  simply localization as we need a positive commutator.)
\item We do not use (or more accurately we do not attempt to define)
  the Hamilton vector field of $P_V$ directly, and thus we do not
  directly compute the principal symbol of $i[P_V, G_\psi Q^* Q
  G_\psi]$ in terms of some action on $q$.  Instead, we \emph{compare} $i[P_V, G_\psi Q^* Q
  G_\psi]$ to an operator whose principal symbol we know explicitly.
  (See just below these remarks for an elaboration on this comparison.)
\end{itemize}

It is instructive to consider first the case $V = 0$ and the commutator
$[P_0, G_{\psi,0} Q^* Q G_{\psi,0}]$, where $G_{\psi,0}$ is the
corresponding function of the free Klein-Gordon operator $P_0$ in
equation \eqref{eq:def_Gpsi0}.  In this case compute $\ffsymbz(i[P_0,
G_{\psi,0} Q^* Q G_{\psi,0}]) $ in terms of
\begin{align*}
f(\rho) \coloneqq \tilde q \rvert_{\NP},
\end{align*}
and $\Hamscp \tilde q$.  Recalling, from \eqref{eq:hammy
  3}-\eqref{eq:hammy scat}, that away from the characteristic set we
use the rescaling $\Hamscp = (\rho / x) H_p$, we have 
\begin{equation}\label{eq:free commutator}
  \ffsymb{2 \ell}(i[P_0, G_{\psi,0} Q^* Q G_{\psi,0}])(\tau) = \rho^{-2s}(\widehat{G_{\psi,0}})_{\ff}\left( 2 \left( \Hamscp
  \tilde q   \right) f(\tau) + 2 (2\ell + 1)
f(\tau)^2 \right) (\widehat{G_{\psi,0}})_{\ff}.
\end{equation}
(This is an easy consequence of $\ffsymb{\ell}(A)(\tau) =
\scnormsymbz(A)(0,0,\tau,D_z)$ for $A \in \Psisc^{m,\ell}$.)
Thus, we seek a $\tilde q$ which gives positivity when differentiated
by the Hamilton vector field, and we proceed to the construction of
$\tilde{q}$ now.

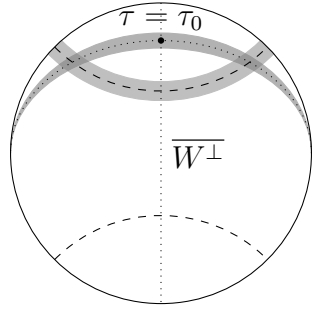
\begin{figure}
    \centering
    \begin{tikzpicture}[scale=2]
        \filldraw[gray, opacity=0.6] (1,0) arc [start angle=0, end angle=180, x radius=1cm, y radius=0.8cm] arc [start angle=180, end angle=0, x radius=1cm, y radius=0.7cm];
        \filldraw[gray,opacity=0.5] (-0.75,0.661) arc (-135:-45:1.06cm) --
        (0.661,0.75) arc (-45:-135:0.935cm) -- (-0.75,0.661);

        \draw[dotted] (1,0) arc [start angle=0, end angle=180, x radius=1cm, y radius=0.75cm];
        
        \draw (0,0) circle (1cm);
        \draw[dashed] (-0.707, 0.707) arc (-135:-45:1cm);
        \draw[dashed] (-0.707, -0.707) arc (135:45:1cm);
        
        \filldraw (0,0.75) circle (0.5pt) node [above] {$\tau = \tau_0$};
        \draw[dotted] (0,1) -- (0,-1);
        \draw (0,0) node [right] {$\Wperpo$};
    \end{tikzpicture}
    \caption{Localization near the characteristic set and at $\tau_0 \in \Wperpo$.}
    \label{fig:localization-near-charset}
\end{figure}

Let $\chi_0 \in \CI(\RR)$ given by
\begin{align}\label{eq:chi 0}
    \chi_0(t) = \begin{cases} e^{-1/t} & t \geq 0\,,\\
        0 & t \leq 0\,.
    \end{cases}
\end{align}
A key feature of $\chi_0$ is that
\begin{align*}
  \chi_0  (t) &= t^2 \chi_0' (t) ,
\end{align*}
Choose $\chi_1 \in \CI(\RR, [0,1])$ such that
\begin{align*}
    \chi_1(t) &= 0\quad \text{ for } t \leq 0\,,\\ 
    \chi_1(t) &= 1\quad \text{ for } t \geq 1\,,\\ 
    \chi_1'(t) &\geq 0\,. 
\end{align*}

Vasy's construction of $\tilde q_0$ uses functions $N$ and $\omega$, and we retain
this notation with appropriate modifications of their definitions.
We choose a neighborhood
\[
  U' \subset \Char(P_0) \cap \Tsco^*_{\NP} X
\]
as in Remark~\ref{rem:solvability_hamf}
and we define a function $N \in \CI(U')$, which will act as our flow parameter
from $\Sigma$, by
\begin{align*}
    \Hamscp N = 1\,,\quad N|_{\Sigma} = 0\,,
\end{align*}
and by the transversality of $\Hamscp$ to $\Sigma$ in
\eqref{eq:transversal to Sigma}, we see that away from $\mu = 0$ and
near $\Sigma$, i.e.\ on sets of the form $|\mu| \ge c > 0$, $| \mu
\cdot y | \le c$, $|y| < c$, we have
\begin{align*}
    c_1 (\mu \cdot y) \leq N \leq c_2 (\mu \cdot y)\,.
\end{align*}
for some $c_1 < c_2$.  Thus $N$ has the dual features that it is
commensurable with $\mu \cdot y$ (near $\Sigma$ and away from $\mu =
0$) and parallel along the flow.  This $N$ is \emph{identical} to that
in \cite{V2000}, we merely use that it is smooth up to the fiber
boundary $\rho = 0$ near the characteristic set.

In proving estimates at $+ \infty \in \Wperpo$, we define an $\omega$
which differs from the one in Vasy, namely we let $\omega \in \CI(U')$ be given by
the solution of 
\begin{equation}
  \label{eq:omega infinity}
  \Hamscp \omega = 0\,,\quad  \omega \rvert_{\Sigma} = \abs{y}^2 + \rho^2.
\end{equation}
This $\omega$ is used to localize near $\omega = 0$, which here is the
set $\rho = 0, y = 0$, i.e.\ fiber infinity over $\NP$.
Note that at finite $\tau_0$ levels one can use, exactly as in Vasy,
$\omega \rvert_{\Sigma} = \abs{y}^2 +  \abs{\tau - \tau_0}^2$.

Thus exactly as in \cite[Eq.\ 14.10-14.11]{V2000}, for $\epsilon, \delta, \beta > 0$, we set
\begin{align*}
    \phi = N+ \omega / \epsilon,
\end{align*}
and define the function
\begin{align*}
    \tilde{q}_0(x,y,\rho,\theta) \coloneqq \chi_0(\beta^{-1}(2 -
  \phi/\delta)) \chi_1(N/\delta + 2)\, ,
\end{align*}
depicted in Figure~\ref{fig:pic-of-q0}.
In particular the support of $\tilde{q}_0$ is contained in $\phi \leq
2 \delta$ and $N \geq - 2\delta$.
Thus, we have the bounds
\begin{align}\label{eq:N omega bound}
    \abs{N} \leq 2\delta\,,\quad \abs{\omega} \leq 4 \eps \delta\,,\quad \abs{\phi} \leq 6 \delta
\end{align}
on the support of $\tilde{q}_0$.
In particular, the choice of $\delta$ determines how far from $\NP$ we have to control $u$ with $E$ and $\epsilon$ determines the interval around $+ \infty \in \Wperpo$ we want to control.

\begin{figure}
    \centering
    \includegraphics{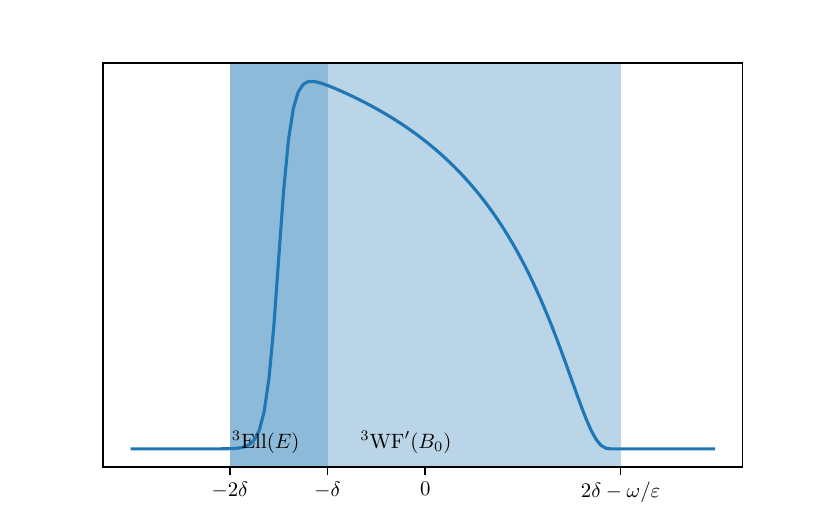}
    \caption{The function $\tilde{q}_0$ in the direction of the flow
      measured by $N$. The north pole is at $N = 0$.}
    \label{fig:pic-of-q0}
\end{figure}

We now wish to extend $\tilde{q}_{0}$ to a function $\tilde{q}$
defined on the whole of $\Tsco^{*}X$, i.e.\ also away from the
characteristic set; since $\tilde{q}_{0}$ is defined on the
characteristic set we can do this easily enough with a bump function
that is $1$ near $\Char(P_{0})$ and in fact we will choose such a
function that is smooth on $\Tsco^{*}X$.

Such cutoffs to the characteristic set were discussed in Section \ref{sec:full
  func calc sec}, and we recall that smooth functions of $p /
(\tau^{2} + |\zeta|^{2} + 1)$ are smooth on the whole of $\Tsco^{*} X$.
Thus we define, for $\chi \in \CcI(\RR)$ with $\chi$
identically $1$ near zero,
\begin{equation}
\chi_{\Char} = \chi \left( \frac{\tau^2 - |\zeta|^2 - m^2}{\tau^2 +
    |\zeta|^2 + 1} \right) = \chi \left( \frac{p}{\tau^2 +
    |\zeta|^2 + 1} \right).\label{eq:def of chichar}
\end{equation}
We choose $\chi$ such that $\supp \chi_{\Char} \subset U'$ in order that $\chi_{\Char} \tilde{q}_0$ is well-defined.
The choice of a function $\chi$ as opposed to the $\psi$ used in
Section~\ref{sec:full func calc sec} is deliberate,
as we hope to avoid confusion in the
use of $\chi_{\Char}$ as a cutoff to the characteristic set and $\psi$
which is used to take functions of an operator. For the same reason
the denominator has $1$ instead of $m^2 + E$, as the symbol of
$\chi_{\Char}$ is irrelevant; it is just a cutoff to the
characteristic set.  We note that $\chi_{\Char}$ differs from Vasy's cutoff
$\rho(g)$ \cite[Eq.\ 14.20]{V2000}; in the setting of that paper, the
characteristic set is compactly contained in the interior of
$\Tsc^{*}X$ so the behavior near fiber infinity is irrelevant.

Define $\tilde q$ by
\begin{equation}
  \label{eq:def of qtilde}
  \tilde q = \chi_{\Char} \tilde q_0 + (1 - \chi_{\Char}) \chi_0(\beta^{-1}(2 - \omega_0/(\tilde \epsilon \delta))),
\end{equation}
Note that $\tilde{q}$ is constant in $\zeta$ on $\NP$.
We have that
\begin{align*}
    N|_{\NP} = 0\,,\quad \omega|_{\NP} = \rho^2\,,
\end{align*}
therefore we may define $f \in \CI(\Wperpo)$ by
\begin{align}\label{eq:def f}
f(\rho) = \tilde q \rvert_{\NP} = \chi_0(\varrho(\rho) )\,,
\end{align}
where
\begin{align*}
    \varrho(\rho) \coloneqq \beta^{-1} \left(2 - \frac{\rho^2}{\epsilon \delta} \right)
\end{align*}
and thus $\tilde q$ satisfies the centrality condition
\eqref{eq:centrality}.

We now compute the Hamilton vector field applied to $\tilde q_0$:
\begin{align*}
    \Hamscp(\tilde{q}_0) &= - \beta^{-1} \delta^{-1} \chi_0'(\beta^{-1} (2 - \phi/\delta)) \chi_1(N/\delta + 2) + \delta^{-1} \chi_0(\beta^{-1} (2 - \phi/\delta)) \chi_1'(N/\delta + 2)\,.
\end{align*}
We have used here that $\Hamscp N = 1$ and $\Hamscp \omega =
0$.  Note that $\chi_1(N/\delta + 2)$ is constantly $1$ in a
neighborhood of $\ff$.  Thus we now draw the important conclusion that over $\NP$ this
expression simplifies to
\begin{equation}
  \label{eq:def of f flat}
  f^\flat(\rho) \coloneqq - \Hamscp(\tilde{q}_0)|_{\NP} = \beta^{-1}
  \delta^{-1} \chi_0'(\varrho) \in
  C^\infty(\Wperpo),
\end{equation}
the minus sign being included as $f^{\flat} \ge 0$ will be used as
an upper bound below.    We will use below that we can bound $f$ in
terms of $f^{\flat}$,
\begin{equation}
  \label{eq:fflat to f}
  f(\rho) \le \frac{4\delta}{\beta} f^{\flat}(\rho).
\end{equation}

We may define $g_0$ and $\tilde e_{0}$ by
\begin{equation*}
  g_0^2 \coloneqq  2 \beta^{-1}\delta^{-1}\left( \chi_0(\beta^{-1}(2 - \phi/\delta)) \chi_0 '(\beta^{-1}(2 - \phi/\delta)) \right)
  \chi_1(N/\delta + 2)^2, 
\end{equation*}
and
\begin{equation}\label{eq:tilde_e_0}
  \tilde e_{0}^2 \coloneqq 2 \delta^{-1} \chi_0^2 \chi_1\chi_1'.
\end{equation}
Setting
\begin{align*}
    r &\coloneqq \frac{M\delta}{2 \beta} \left(2 - \phi/\delta\right)^2\,,
\end{align*}
we have the following.
\begin{lemma}\label{lem:hamf_q_leq}
    For any $\eps, \delta, M > 0$ there exist $\beta > 0$ large such that $r \in [0,1)$ and
    \begin{align}\label{eq:q0 M decomp}
        \Hamscp(\tilde{q}_0^2) + M \tilde{q}_0^2 = - (1 - r)g_0^2 + \tilde{e}_{0}^2
    \end{align}
    on $U'$.
\end{lemma}
\begin{proof}
    We start by observing that from the bound of $N$ and $\omega$
    in \eqref{eq:N omega bound}, we have that \[\abs{2 - \phi/\delta} \leq 8.\]
    Therefore, if we choose $\beta > 0$ such that
    \begin{align*}
        \beta > 8^2 \delta M\,,
    \end{align*}
    then $r < 1$.

    To establish \eqref{eq:q0 M decomp}, we first observe that
    \begin{align*}
        \Hamscp(\tilde{q}_0^2) = - g_0^2 + \tilde e_{0}^2\,.
    \end{align*}
    Moreover, we calculate 
    \begin{align*}
      M\tilde{q}_0^2 &= M \chi_0^2 \left( \beta^{-1}(2-\phi /
                      \delta)\right) \chi_1^2  \\
      &= M\frac{(2-\phi / \delta)^2}{\beta^2}\left( \chi_0
        \chi_0 ' \right) \left(\beta^{-1}(2-\phi / \delta)\right)
        \chi_1^2 \\
      &= r g_0^2\,,
    \end{align*}
    where we used the explicit relationship between $\chi_{0}(t) =
    t^{2}\chi_{0}'(t)$.
  \end{proof}

Taking $\beta > 0$ as in the previous lemma, we define
\begin{align*}
    \tilde{b}_0 &\coloneqq (1 - r)^{1/2} g_0\,.
\end{align*}
Noting that $\tilde b_0$ is defined only on $U'$, we can extend to
a function $b$ as with $\tilde q$ and $\tilde q_0$ above, namely by
writing
\begin{align*}
b \coloneqq \chi_\pa(x) \tilde b
\end{align*}
with
\begin{align*}
    \tilde b &\coloneqq \chi_{\Char} \tilde b_0 + (1 - \chi_{\Char}) \left( (1 - M \delta \varrho ) \frac{2}{\beta \delta} \chi_0(\varrho) \chi_0'(\varrho) \right)^{1/2}\,.
\end{align*}
Here the parenthetical term on the right is equal to $\tilde b_{0}$
over $\NP$, and thus $\tilde b$ gives a globally defined function
which restricts to
\begin{equation}
  \label{eq:tilde b front fact symbol}
  \tilde b \rvert_{\NP}(\rho) = \left( 2 (1 - M\delta \varrho) f(\rho) f^{\flat}(\rho) \right)^{1/2}\,.
\end{equation}

Similarly, we set
\begin{align*}
  \tilde e = \chi_{\Char} \tilde e_0, \quad e = \chi_{\pa}(x) \tilde{e}_{0}, 
\end{align*}
We note that
\begin{align*}
  \left( \supp e \right) \cap \Tsco^{*}_{\NP} X = \varnothing,
\end{align*}
and thus $e$ is in fact a standard scattering symbol.
From \eqref{eq:tilde_e_0}, we see that
 \begin{equation}
   \label{eq:support e}
   \supp e \subset \{ N + \omega / \epsilon \le 2 \delta \} \cap \{ -
   2 \delta \le N \le - \delta \},
 \end{equation}
and in addition is supported near the characteristic set thanks to the
$\chi_{\Char}$.  Thus $e$ \emph{is supported near the bicharacteristic
rays flowing into $\rho = 0$ over $\NP$ as $N$ increases, i.e.\ along the
direction of the Hamiltonian flow.}

To utilize $q, b$ and $e$ in an expression related to commutator
$i [P_V, G_\psi Q^* Q G_\psi]$, we will choose a $\psi \in \CcI(\RR)$
with sufficiently small support that, away from the front face, only
the function $\tilde q_{0}$ appears in the principal symbol of $Q G_{\psi}$.
Indeed, recall that for arbitrary $\psi \in \CcI(\RR)$, by Proposition
\ref{prop:Q_Gpsi} we have
\begin{align*}
\fibsymb{s - 1,\ell+1}(Q G_\psi)  &= (q |_{\Ssc^{*} X}) \cdot
  \fibsymbz(G_{\psi})\\
 \mfsymb{s - 1,\ell+1}(Q G_\psi)  &= (q |_{\Tsco^{*}_{\pa X} X}) \cdot
   \mfsymbz(G_{\psi})\\
   \ffsymb{\ell + 1}(Q G_{\psi})&=  \rho^{-(s-1/2)} f \cdot \ffsymbz(G_{\psi}).
\end{align*}
Thus, if $\psi$ is chosen with sufficiently small support such that
\begin{align*}
  \chi_{\Char} \cdot \psi \left( (\tau^{2} + |\zeta|^{2} +
  m^2 + E )^{-1}p \right) = \psi\left( (\tau^{2} + |\zeta|^{2} + 
  m^2 + E )^{-1}p \right),
\end{align*}
We have that
\begin{align*}
\left( q \psi \left( (\tau^{2} + |\zeta|^{2} +
  m^2 + E )^{-1}p \right) \right|_{\pa X}= \tilde q_{0} \left. \psi
  \left( (\tau^{2} + |\zeta|^{2} +
  m^2 + E )^{-1}p \right) \right|_{\pa X},
\end{align*}
so over the boundary away from $\ff$ only $\tilde q_{0}$ appears in
the symbol of $Q G_\psi$.

We thus obtain our desired positivity at the level of the principal
symbol in the first two components, i.e.\ at fiber infinity and over
$\mf$.  Indeed, directly from Lemma~\ref{lem:hamf_q_leq}, we have
\begin{align}
    \label{eq:commutator_fibsymb}
    \fibsymb{2s, 2\ell}(i[P_V, G_\psi Q^* Q G_\psi]) &=
        \left( - \tilde b_{0}^{2} + \tilde e_{0}^{2} + ( 2(2\ell + 1) - M) \tilde q_{0} ^2 \right)  \fibsymbz(G_{\psi})^{2}\\
    \label{eq:commutator_mfsymb}
    \mfsymb{2s, 2\ell}(i[P_V, G_\psi Q^* Q G_\psi]) &=
        \left( - \tilde b_{0}^{2} + \tilde e_{0}^{2} + ( 2(2\ell + 1) - M) \tilde q_{0} ^2 \right)  \mfsymbz(G_{\psi})^{2}.
\end{align}
In order to use the sharp G\aa{}rding type theorem in
Proposition \ref{prop:sharp_Garding}, we must have a similar
inequality over $\ff$.  This can be done by again possibly reducing
the size of the support of $\psi$, and using operator norm bounds to
compare the indicial operator of $[P_V, G_\psi Q^* Q G_\psi]$ to that
of $[P_0, G_{\psi,0} Q^* Q G_{\psi,0}]$.  We note that we have a formula
for the free indicial operator in terms of the functions defined
above, namely, from \eqref{eq:free commutator},
\begin{equation}
  \label{eq:free indicial commutator}
  \ffsymb{2\ell}\left( i [P_0, G_{\psi,0} Q^* Q G_{\psi,0}] \right) = \left( - 2 (1 - M\delta \varrho) f f^\flat + (2 (2\ell+1) - M)f^2 \right) \rho^{-2s} \ffsymbz(G_{\psi,0})^2\,,
\end{equation}

Thus, for $P_0$, the indicial operator has similar structure to the
other symbol components, and $f$ can be bounded in terms of $f^\flat$.
However, we cannot easily calculate the indicial operator of the perturbed commutator.
Following Vasy, we will use a ``window shrinking'' argument to prove that the
difference of the free and perturbed commutators are small provided that the cutoff function $\psi$ in the
localizer $G_\psi$ is supported sufficiently close to $0$.
\begin{lemma}\label{lem:commutator_ffsymb}
    For every $\eps' > 0$ there exists $\psi \in \CcI$ with $\psi = 1$ on $(-\delta, \delta)$ where $\delta = \delta(\eps')$ such that
    \begin{gather*}
        \rho^{2s} \ffsymb{2\ell}(i [P_V, G_\psi Q^* QG_\psi]) - \left(
      2(2\ell + 1) - M \right) f^2 \cdot \ffsymbz(G_\psi)^2 \\ \leq  - (2 -
      \eps') (1 - M \delta \varrho) f^\flat f \cdot  \ffsymbz(G_\psi)^2\,.
    \end{gather*}
\end{lemma}
\begin{proof}
    Let $\varphi \in \CcI(\RR)$ with $\varphi(s) = 1$ for $\abs{s} \leq 1$
    to be chosen later.
    We consider the operator defined on $\ff$ which is the difference
    of the commutator and the local part of the free commutator in
    \eqref{eq:free indicial commutator} but with the $G_{\varphi}$ localizer:
    \begin{align*}
        R_{\varphi} \coloneqq{}& \rho^{2s} \ffsymb{2\ell}(i [P_V, G_\varphi
                         Q^* Q G_\varphi]) - \left(-2(1 - M \delta
                         \varrho) f f^\flat + \left( 2(2\ell + 1) -
                         M\right) f^2\right) \ffsymbz(G_\varphi)^2.
    \end{align*}
    Note that, if $\psi \in \CcI(\mathbb{R})$ has $\psi \phi = \psi$ then
    \begin{align*}
      R_{\psi} = G_{\psi} R_{\varphi} G_{\psi}.
    \end{align*}
    We have that
    \begin{align*}
   R_{\varphi}     ={}& \rho^{2s} \left( \ffsymb{2\ell}\{ i [P_V, G_\varphi Q^* Q G_\varphi] - i [P_0, G_{\varphi,0} Q^* Q G_{\varphi,0}] \}\right) \\
        &+ \rho^{2s} \ffsymb{2\ell}\left(i [P_0, G_{\varphi,0} Q^* Q G_{\varphi,0}] \right) - \left(-2(1 - M \delta \varrho) f f^\flat + \left( 2(2\ell + 1) - M\right) f^2\right) \ffsymbz(G_{\varphi,0})^2 \\
        &+ \left(-2(1 - M \delta \varrho) f f^\flat + \left( 2(2\ell + 1) - M\right) f^2 \right) \left( \ffsymbz(G_{\varphi,0})^2 - \ffsymbz(G_\varphi)^2 \right)\,.
    \end{align*}
    We claim that
        \begin{align*}
        R_{\varphi} \in \Psisclsc^{-\infty, -1, 0}
        \end{align*}
        and that there is a $C > 0$ independent of $\rho = 1/\tau$ such
that
    \begin{align}\label{eq:fantastic estimate}
        \norm{R_{\varphi}(\rho)}_{\LinOp(L^2, \Sobscl^{1,1})} \leq C f^\flat(\rho) f(\rho)\,.
    \end{align}
To see this, note that in the expression for $R_{\varphi}(\rho)$, line two
vanishes by \eqref{eq:free indicial commutator}.  For the third line
in that expression, one obtains \eqref{eq:fantastic estimate} from
Lemma~\ref{lem:diff_Gpsi_Gpsi0} (and Proposition
\ref{prop:Gpsi_indicial}) and estimating $f$ by $f^{\flat}$ as in
\eqref{eq:fflat to f}.  Finally, for the first line, we use
Corollary~\ref{cor:commutator_symbol} (and Lemma
\ref{lem:diff_Gpsi_Gpsi0} and Proposition
\ref{prop:Gpsi_indicial} again).  
    
Taking $\eps = \eps' / C$ and $\psi$ as in
Lemma~\ref{lem:shrinking_window}, we obtain 
    \begin{equation}\label{eq:ffsymb_commutator_aux}
        \begin{aligned}
        \rho^{2s}\ffsymb{2s,2\ell }(&i [P_V, G_{\psi} Q^*
        QG_{\psi}]) - (2(2\ell + 1) - M) f^2 \ffsymbz(G_{\psi})^2 \\
        & = 2(1 - M \delta \varrho) f^\flat f \cdot \ffsymbz(G_{\psi})^2
        + \ffsymbz(G_{\psi}) R_{\varphi} \ffsymbz(G_{\psi})   \\
        &\leq 2 (1 - M \delta \varrho) f^\flat f \cdot
        \ffsymbz(G_{\psi})^2 - \eps'\cdot f^{\flat} f\,.
        \end{aligned}
    \end{equation}
This is nearly what the lemma claims, only missing the
$\ffsymbz(G_{\psi})^2$ on the $\epsilon'$ term, but is obtained by
multiplying both sides by yet another $\ffsymbz(G_{\tilde\psi})$ for
yet another $\tilde \psi$ with $\tilde \psi \psi = \tilde \psi.$
\end{proof}

Now we can fix $\eps' \in (0, 1/4)$ and $\psi \in \CcI$ as in the lemma above. Choose $\phi \in \CcI$ with $\phi(s) = 1$ for $\abs{s} \leq \delta(\eps')/2$ and $\supp \phi \subset (-\delta(\eps'), \delta(\eps'))$.
We define
\begin{align} \label{eq:B final on Gpsi}
    B_0 &= \Op_L(x^{-\ell} \rho^{-s} b) G_\phi\,,\\ \label{eq:E final on Gpsi}
    E_0 &= \Op_L(x^{-\ell} \rho^{-s} e) G_\phi\,,\\
    Q_0 &= \Op_L(x^{-(\ell + 1/2)} \rho^{-(s-1/2)} q) G_\phi\,,\\
    Q_0' &= \Op_L(x^{-\ell} \rho^{-s} q) G_\phi\,.
\end{align}

\begin{lemma}\label{lem:commutator_ineq}
    Let $O \subset \Ctscd$ with $\WFtsc'(B) \subset O$.
    There exists $F \in \Psitsc^{2s-1, 2\ell-1}$ and $\eps \in (0,1)$ such that $\WFtsc'(F) \subset O$ and
    \begin{align*}
        i[P_V, Q_0^* Q_0] + (M - 2 (2\ell + 1)) (Q_0')^* Q_0' \leq -(1 - \eps) B^* B + E^* E + F\,.
    \end{align*}
\end{lemma}
\begin{proof}
    We will prove the lemma by using Proposition~\ref{prop:sharp_Garding}.
    By the definition of $\phi$, there exists $\psi_1 \in \CcI$ such that
    \begin{align*}
        \supp \phi \cap \supp (1 - \psi_1) = \varnothing\,,\\
        \supp \psi_1 \cap \supp (1 - \psi) = \varnothing\,.
    \end{align*}
    We set
    \begin{align*}
        A &\coloneqq i[P_V, G_\psi Q^* Q G_\psi] + (M - 2(2\ell+1)) G_\psi \Op_L( x^{-\ell} \rho^{-s} q)^* \Op_L(x^{-\ell} \rho^{-s} q) G_\psi \\
        &\phantom{=} - G_\psi \Op_L(x^{-\ell} \rho^{-s} e)^* \Op_L(x^{-\ell} \rho^{-s} e) G_\psi\,,\\
        C &\coloneqq - G_\psi \Op_L( x^{-\ell} \rho^{-s} b)^* \Op_L(x^{-\ell} \rho^{-s} b) G_\psi\,.
    \end{align*}
    We have that
    \begin{align*}
        G_\phi A G_\phi &= i[P_V, Q_0^* Q_0] + (M - 2(2\ell + 1)) (Q_0')^* Q_0' - E^* E\,,\\
        G_\phi C G_\phi &= -B^* B\,.
    \end{align*}
    Hence, to prove the claim, we can apply Proposition~\ref{prop:sharp_Garding} with $\psi_1$.
    It remains to verify the inequality of the principal symbols.
    For the main face and fiber symbol, this is \eqref{eq:commutator_fibsymb} and \eqref{eq:commutator_mfsymb}.
    For $\ffsymbz$, we observe that
    \begin{align*}
        \rho^{2s} \ffsymb{2\ell}(A) &= \rho^{2s} \ffsymb{2\ell}(i[P_V, G_\psi Q^* Q G_\psi]) + (M - 2(2\ell + 1)) f^2 \ffsymbz(G_\psi)^2\,,\\
        \rho^{2s} \ffsymb{2\ell}(C) &= -2 (1 - M \delta \varrho) f(\rho) f^\flat(\rho) \ffsymbz(G_\psi)^2
    \end{align*}
    and by
    Lemma~\ref{lem:commutator_ffsymb}, we have that
    \begin{align*}
        \ffsymb{2\ell}(A) \leq (1 - \eps'/2) \ffsymb{2\ell}(C)\,.
    \end{align*}
\end{proof}
\begin{remark}
    As opposed to the case of scattering operators in Section~\ref{sec:model-case}, we cannot use a simplified version of the G\aa{}rding inequality, because we truly only have an
    inequality of principal symbols.
\end{remark}

For $r \in (0,1)$ and $\delta_1, \delta_2 \in (0,\infty)$, we define
\begin{align*}
    J_{r,\delta_1, \delta_2} \coloneqq (1 + r/x)^{-\delta_2} (1 + r/\rho)^{-\delta_1}\,.
\end{align*}

We have that
\begin{align*}
    \Op_L(J_{r,\delta_1, \delta_2}) \to \Id \text{ as } r \to 0
\end{align*}
strongly in $\LinOp(\Sobsc^{s', \ell'})$ for every $s' > 0, \ell' > 0$ (cf. Vasy~\cite[p. 408]{Vasy13}).

Using \eqref{eq:hammy 3}, we calculate
\begin{align*}
    \Hamf J_{r,\delta_1, \delta_2} = - 2 \tau x \delta_2 \frac{r}{x + r} J_{r,\delta_1, \delta_2}\,.
\end{align*}
Hence, we can choose $\tilde{M} > 0$ independent of $r$, such that for all $\tau, x$, we have
\begin{align}\label{eq:bound_regularizer}
    \tilde{M} J_{r,\delta_1, \delta_2} \geq \frac{1}{\tau x} \Hamf J_{r,\delta_1,\delta_2}\,.
\end{align}
Taking $M > \tilde{M} + 2(2\ell + 1)$ and using Lemma~\ref{lem:hamf_q_leq}, we have that
\begin{align}\label{eq:hamf_Jq_leq}
    \Hamf( x^{-(2\ell+1)} \rho^{-(2s - 1)} J_{r,\delta_1, \delta_2}^2 q^2) \leq x^{-2\ell} \rho^{-2s} J_{r,\delta_1, \delta_2}^2 (- b^2 + e^2 + O(\rho \cdot x))
\end{align}
on $U'$.

\begin{align*}
    Q_r &\coloneqq \Op_L(x^{-(\ell + 1/2)} \rho^{-(s - 1/2)} J_{r,\delta_1, \delta_2} q) G_\phi\,,\\
    B_r &\coloneqq \Op_L(x^{-\ell} \rho^{-s} J_{r,\delta_1, \delta_2} b) G_\phi\,,\\
    E_r &\coloneqq \Op_L(x^{-\ell} \rho^{-s} J_{r,\delta_1, \delta_2} e) G_\phi\,.
\end{align*}
We have that
\begin{align*}
    Q_r \in \Psitsc^{s-1/2-\delta_1, \ell+1/2 - \delta_2}\text{ and }
    B_r, E_r \in \Psitsc^{s- \delta_1, \ell- \delta_2}\,.
\end{align*}

Since
\begin{align*}
    J_{r,\delta_1, \delta_2} = \rho^{\delta_1} x^{\delta_2} (\rho + r)^{-\delta_1} (x + r)^{-\delta_2}\,,
\end{align*}
we have that
\begin{align*}
    \ffsymb{\ell+1/2 - \delta_2}(Q_r) &= \rho^{\delta_1} \ffsymb{\ell+1/2}(Q_0)\,,\\
    \mfsymbz(Q_r) &= \rho^{\delta_1} x^{\delta_2} \mfsymbz(Q_0)\,,\\
    \fibsymbz(Q_r) &= \rho^{\delta_1} x^{\delta_2} \fibsymbz(Q_0)
\end{align*}
and similarly for $B_r$ and $E_r$.

\begin{lemma}\label{lem:commutator_ineq_reg}
    Let $O \subset \Ctscd$ with $\WFtsc'(B) \subset O$.
    There exists $F_r' \in \Psitsc^{2s-1-2\delta_1, 2\ell-1-2\delta_2}$ such that $\WFtsc'(F_r') \subset O$ and
    \begin{align*}
        i[P_V, Q_r^* Q_r] \leq -(1 - \eps) B_r^* B_r + E_r^* E_r + F_r'\,,
    \end{align*}
    and $F_r' \in \LinOp(\Sobsc^{2s-1, 2\ell-1})$ is uniformly bounded as $r \to 0$.
\end{lemma}
\begin{proof}
    The argument is essentially the same as the proof of Lemma~\ref{lem:commutator_ineq}, but using \eqref{eq:hamf_Jq_leq} to obtain the
    symbol inequalities.
\end{proof}

We now relax the condition on $V - V^*$. For
this we write
\begin{align*}
    \tilde{V} \coloneqq (V + V^*)/2\,,\quad V'' \coloneqq (V - V^*)/(2i)\,,
\end{align*}
and we have that
\begin{align*}
    P_V = P_{\tilde{V}} - i V''\,.
\end{align*}
We have that both $P_{\tilde{V}}$ and $V''$ are formally self-adjoint.
The assumption of Proposition~\ref{prop:propagation_localized} implies that
\begin{align}\label{eq:sub prin assump}
    V'' \in \Psitsc^{0,-2}\,.
\end{align}

We first prove a variant of the propagation estimate with the specific
$B_0$, $Q_0$, and $E_0$.
\begin{lemma}\label{lem:propagation_small}
    With $B_0$, $Q_0$ and $E_0$ defined as above,
    and $B_1 \in \Psitsc^{0,0}$ such that $\WFtsc'(Q_0) \subset \Elltsc(B_1)$.
    If $u \in \Sobres$ with $B_1 u \in \Sobsc^{s-1/2, \ell-1/2}$, $Q_0 P_V u \in \Sobsc^{-1/2, 1/2}$, and $E_0 u \in L^2$, then
    $B_0 u \in L^2$ and
    \begin{align*}
        \norm{B_0 u} \leq C \left( \norm{Q_0 P_V u}_{-1/2, 1/2} + \norm{E_0 u} + \norm{B_1 u}_{s-1/2, \ell-1/2} + \normres{u}\right)\,.
    \end{align*}
\end{lemma}
\begin{proof}
    Using the operators and parameters defined above, we take $\delta_1 = s + N$ and $\delta_2 = \ell + M$. Then for $r \in (0,1)$, we have that
    \begin{align*}
        Q_r \in \Psitsc^{-N-1/2, -M+1/2}\,.
    \end{align*}
    Therefore, $Q_r u \in \Sobsc^{-1/2,1/2}$ and $Q_r P_V u \in \Sobsc^{1/2, -1/2}$ and the pairing of $Q_r u$ and $Q_r P_V u$ is well-defined.
    We have that
    \begin{align}\label{eq:Im_PV}
        2 \Im \ang{Q_r P_V u, Q_r u} &= \ang{i [P_{\tilde{V}}, Q_r^* Q_r] u, u} - \ang{(Q_r^* Q_r V'' + V'' Q_r^* Q_r) u, u}\,.
    \end{align}
    Since $V'' \in \Psitsc^{0,-2}$, we have that
    \begin{align*}
        F_r'' \coloneqq F_r' -(Q_r^* Q_r V'' + V'' Q_r^* Q_r) \in \Psitsc^{-2N - 1, -2M - 1}
    \end{align*}
    and using Lemma~\ref{lem:commutator_ineq_reg} with $P_{\tilde{V}}$, we obtain that
    \begin{align*}
        \norm{B_r u}^2 \lesssim \norm{E_r u}^2 - 2 \Im \ang{Q_r P_V u, Q_r u} + \ang{ F_r'' u, u}\,.
    \end{align*}
    We have
    \begin{align*}
        \abs{\Im \ang{Q_r P_V u, Q_r u}} \geq -\frac{1}{4\mu} \norm{Q_r P_V u}_{-1/2, 1/2}^2 - \mu \norm{Q_r u}_{1/2, -1/2}^2\,.
    \end{align*}
    Since $Q_r$ and $B_r$ have the same principal symbol up to a constant factor, by Lemma~\ref{lem:QGpsi_eq_OpqGpsi} and Proposition~\ref{prop:elliptic_regularity_Gpsi}, we have that
    \begin{align*}
        \norm{Q_r u}_{1/2,-1/2}^2 \lesssim \norm{B_r u}^2 + \norm{\tilde{F}_r G_\psi u}^2 + \normres{u}
    \end{align*}
    for some $\tilde{F}_r \in \Psitsc^{-N-1,-M-1}$ with $\WFtsc'(Q_r) \subset \Elltsc(\tilde{F}_r)$ and $\WFtsc'(\tilde{F}_r G_\psi) \subset \Elltsc(B_1)$ and $\tilde{F}_r$ being uniformly bounded in $\Psitsc^{s-1, \ell-1}$.
    Hence, choosing $\mu > 0$ sufficiently small, we can absorb $Q_r u$ into $B_r u$, and arrive at the estimate
    \begin{align*}
        \norm{B_r u}^2 \lesssim \norm{Q_r P_V u}_{-1/2, 1/2}^2 + \norm{E_r u}^2 + \norm{\tilde{F}_r G_\psi u}^2 + \abs{\ang{F''_r u, u}}\,.
    \end{align*}
    The right-hand side is bounded as $r \to 0$ and since $B_r \to B_0$ as $r \to 0$, we conclude that $B_0 u \in L^2$ and
    \begin{align*}
        \norm{B_0 u}^2 \lesssim \norm{Q_0 P_V u}_{-1/2, 1/2}^2 + \norm{E_0 u}^2 + \norm{B_1 u}_{s-1/2, \ell-1/2} + \normres{u}
    \end{align*}
    using that by elliptic regularity, Theorem~\ref{thm:elliptic regularity theorem}, we have that
    \begin{align*}
        \norm{\tilde{F}_r G_\psi u}^2 + \abs{\ang{F''_r u, u}} &\lesssim \norm{B_1 u}_{s-1/2, \ell-1/2}^2 + \normres{u}^2\,.
    \end{align*}
\end{proof}

Now we can prove Proposition \ref{prop:propagation_localized}.
\begin{proof}[Proof of Proposition \ref{prop:propagation_localized}]
    By using a partition of unity and the propagation estimates in the scattering calculus, it suffices to prove the
    estimate for $B, E \in \Psitsc^{0,0}$ such that there exist $\tau_0 \in \Wperpo$ with $\abs{\tau_0} > m$ and $U \subset \Elltsc(G)$ an arbitrary small neighborhood of $\tau_0$
    with $\WFtsc'(BG_{\varphi})$ being controlled by $\Elltsc(E)$ through $U$ and $\tau_0 \in \Elltsc(B)$.

    We take the $B_0, E_0$, and $Q_0$ as defined above with respect to the particular $\tau_0$. Since we can take $\WFtsc'(B G_{\varphi})$ to be arbitrary small provided that $\tau_0 \in \Elltsc(B)$, we have that
    \begin{align*}
        \norm{B G_{\varphi} u}_{s,\ell} \lesssim \norm{B_0 u} + \normres{u}
    \end{align*}
    by Lemma~\ref{lem:QGpsi_eq_OpqGpsi} and Proposition~\ref{prop:elliptic_regularity_Gpsi}.

    In particular this means that $\WFtsc'(Q_0) \subset \Elltsc(G)$ and therefore, again by Lemma~\ref{lem:QGpsi_eq_OpqGpsi} and Proposition~\ref{prop:elliptic_regularity_Gpsi}, we obtain
    \begin{align*}
        \norm{\Lamsc_{-1/2, 1/2} Q_0 P_V u} &\lesssim \norm{G G_{\psi_{2}}P_V u}_{s-1, \ell + 1} + \normres{u}
    \end{align*}
    and similarly,
    \begin{align*}
        \norm{E_0 u} &\lesssim \norm{E G_{\psi_1} u}_{s,\ell} + \normres{u}\,.
    \end{align*}
    
    We choose $B_1 \in \Psitsc^{0,0}$ with
    \begin{enumerate}[(i)]
        \item $\WFtsc'(B G_{\varphi}) \cap \WFtsc'(Q_0) \subset \Elltsc(B_1) \subset U$,
        \item $\WFtsc'(B_1 G_\varphi) \subset U$,
        \item $\WFtsc'(B_1) \subset \Elltsc(G') \cap \Elltsc(B')$.
    \end{enumerate}
    Note that we can arrange the first two conditions with a construction as in Section~\ref{sec:tsc localizers} and the third condition is trivial near $\NP$.

    We obtain the estimate
    \begin{align*}
        \norm{B G_{\varphi} u}_{s,\ell} \lesssim \norm{E G_{\psi_1} u}_{s,\ell} + \norm{G G_{\psi_2} P_V u}_{s-1, \ell + 1} + \norm{B_1 u}_{s-1/2,\ell-1/2} + \normres{u}\,.
    \end{align*}
    The proof is finished by applying Proposition~\ref{prop:elliptic_Gpsi_estimate} to $B u$ and $B_1 u$,
    \begin{align*}
        \norm{B u}_{s,\ell} &\lesssim \norm{B G_{\varphi} u}_{s,\ell} + \norm{G' P_V u}_{s-2, \ell} + \norm{B' u}_{s-1, \ell-r} + \normres{u}\,,\\
        \norm{B_1 u}_{s-1/2,\ell-1/2} &\lesssim \norm{B_1 G_{\varphi} u}_{s-1/2,\ell-1/2} + \norm{G' P_V u}_{s-5/2, \ell} + \norm{B' u}_{s-3/2,\ell-r-1/2} + \normres{u}
    \end{align*}
    and using an standard inductive argument to absorb the term $\norm{B_1 G_{\varphi} u}_{s-1/2,\ell-1/2}$ into $\normres{u}$.
\end{proof}

\section{Radial point estimates over \texorpdfstring{$\poles$}{C}}\label{sec:3sc_radial}

In the previous section, we had to exclude the points
$\tau_0 \in \{\pm m\} \subset \Wperpo$ because the Hamilton vector
field vanishes there.  Indeed, by \eqref{eq:rad_future}, we have that
$\Tsco^*_{\NP} X \cap \Rad^f = \{x = 0, y = 0, \zeta = 0, \tau = \pm
m\}$.  In Section~\ref{sec:localized rad est}, we proved localized
radial point estimates away from the poles.  Hence it remains to prove
radial point estimates near $\poles$ and $\tau_0 \in \{\pm m\}$.  These
estimates take a similar form as the scattering radial point
estimates.
As usual, we work only near $\NP$.

The propositions in this section are conceptually a combination of the
radial points estimates reviewed in Section~\ref{sec:localized rad
  est} and the commutator construction with operators of the form
$Q G_\psi$ used in Section~\ref{sec:3sc_propagation}.  As in the
scattering case, the form of the symbol $q$ used in the radial points
commutator is simpler than that of the standard propagation, in that
it is essentially a bump function localizing to the radial set
multiplied by a
weight; for us it is a bump function localizing to
$\tau = \pm m \in W^\perp$, see \eqref{eq:radial com symb}.  Since the
commutators $Q G_\psi$ are localized in bounded $\tau$, they are
smoothing, and thus the only relevant weight is the spacetime weight.

As in Section \ref{sec:3sc_propagation}, the theorem is microlocal only in
the $\tsc$-sense near $\pm m \in W^\perp$, meaning only in the inverse
image of $\pm m$ under the $\pi^{\perp}$ map in \eqref{eq:pi
  perp}. For example, the above threshold
estimates imply that if a distribution $u$ lies in the
above threshold space $\Sobsc^{s, -1/2 + \epsilon}$ near
$(\pi^{\perp})^{-1}(m)$, and $P_V u$ lies in $\Sobsc^{s - 1, 1/2 +
  \epsilon'}$ there, then $u$ lies in $\Sobsc^{s, -1/2 + \epsilon'}$ in a
smaller neighborhood of $(\pi^{\perp})^{-1}(m)$.

As in the previous section we assume that $V \in \rho_{\mf} \Psitsc^{1, 0}$ is asymptotically static of order $r \geq 1$ and
the Hamiltonian of its static part $V_{\pm} \in S^{-1}(\RR^n_z)$, $H_{V_\pm} = \Delta + m^2 + V_\pm$ have purely absolutely continuous spectrum near $\ico{m^2, \infty}$.
Moreover, we assume that
\begin{align*}
    V - V^* \in \Psitsc^{0,-2}\,.
\end{align*}

\begin{proposition}[Above threshold radial point estimate]
    \label{prop:tsc_radial_above}
    Let $\delta > 0$ sufficiently small, $\varphi, \psi_1, \psi_2 \in \CcI(\RR)$ with $\supp \varphi \subset \ioo{-\delta, \delta}$ and $\psi_j|_{\ioo{-\delta, \delta}} \equiv 1$,
    $\tau_0 \in \set{\pm m} \subset W^\perp$,
    and $B, B_1, G, G', B' \in \Psitsc^{0,0}$ such that
    \begin{enumerate}
        \item $\WFtsc'(B G_{\varphi})$ is contained in a sufficiently small neighborhood of $\tau_0 \in \Ctscd$,
        \item $\tau_0 \in \Ell_{\ff}(B) \cap \Ell_{\ff}(B_1)$,
        \item $\WFtsc'(B G_{\varphi}) \subset \Elltsc(B_1) \cap \Elltsc(G)$,
        \item $\WFtsc'(B) \subset \Elltsc(B') \cap \Elltsc(G')$.
    \end{enumerate}

    For all $M, N, s, s', \ell, \ell' \in \RR$ with $\ell > \ell' > -1/2$, $s > s'$,
    and $u \in \Sobres$ with
    $B_1 G_{\phi} u \in \Sobsc^{s', \ell'}$,
    $G G_{\psi} P_V u \in \Sobsc^{s-1, \ell+1}$,
    $G' P_V u \in \Sobsc^{s-2, \ell}$,
    and $B' u \in \Sobsc^{s- 1, \ell - r}$,
    it follows that $Bu \in \Sobsc^{s, \ell}$ and
    \begin{align*}
        \norm{B u}_{s,\ell} &\leq C \big( \norm{B_1 G_{\psi_1} u}_{s',\ell'} + \norm{G G_{\psi_2} P_V u}_{s-1, \ell+1} + \norm{G' P_V u}_{s-2, \ell} + \norm{B' u}_{s-1, \ell-r} \\
        &\phantom{\leq C \big(} + \normres{u} \big)\,.
    \end{align*}
\end{proposition}

\begin{proposition}[Below threshold radial point estimate]
    \label{prop:tsc_radial_below}
    Let $\delta > 0$ sufficiently small, $\varphi, \psi_1, \psi_2 \in \CcI(\RR)$ with $\supp \varphi \subset \ioo{-\delta, \delta}$ and $\psi_j|_{\ioo{-\delta, \delta}} \equiv 1$,
    and $B, E, G, G', B' \in \Psitsc^{0,0}$ such that

    \begin{enumerate}
        \item $\WFtsc'_{\ff}(E) = \varnothing$,
        \item $\WFtsc'(B G_{\varphi}) \cup \Elltsc(E) \subset \Elltsc(G)$,
        \item $\Elltsc(G) \cap \Radtsc_{\sources} = \varnothing$,
        \item $\WFtsc'(B) \subset \Elltsc(B') \cap \Elltsc(G')$,
        \item $\WFtsc'(B G_{\varphi}) \setminus \Radtsc_{\sinks}$ is backward controlled by $\Elltsc(E)$ through $\Elltsc(G)$.
    \end{enumerate}

    For all $M, N, s, \ell \in \RR$ with $\ell < -1/2$ and $u \in \Sobres$ with
    $E G_{\psi_1} u \in \Sobsc^{s,\ell}$,
    $G G_{\psi_2} P_V u \in \Sobsc^{s-1, \ell+1}$,
    $G' P_V u \in \Sobsc^{s-2, \ell}$,
    and $B' u \in \Sobsc^{s-1, \ell-r}$,
    it follows that $B u \in \Sobsc^{s,\ell}$ and
    \begin{align*}
        \norm{B u}_{s,\ell} &\leq C \big( \norm{E G_{\psi_1} u}_{s,\ell} + \norm{G G_{\psi_2} P_V u}_{s-1, \ell+1} + \norm{G' P_V u}_{s-2, \ell} + \norm{B' u}_{s-1, \ell-r} \\
        &\phantom{\leq C \big(} + \normres{u} \big)\,.
    \end{align*}

    The same statement holds if the roles of $\Radtsc_{\sources}$ and $\Radtsc_{\sinks}$ are interchanged and forward control is used instead of backward control.
\end{proposition}

Again for the construction of the commutator, we start by assuming that $V - V^{*} = 0$
and then later we can run the argument under the assumption \eqref{eq:sub prin assump}.

We proceed to the commutator construction.
We will work near $\tau_0 = m$, as the result near $\tau_0 = -m$ follows
via identical arguments.  To avoid cumbersome distinction between the
above and below threshold cases, we introduce a constant $\kappa \in
\{\pm 1\}$ in the arguments below. The case that $\kappa = 1$ is used
in the above threshold estimate and the case that $\kappa = -1$ is
used in the below threshold estimate.

Let $\delta_1, \delta_2 > 0$ to be chosen later.
Let $\chi_1, \chi_2 \in \CI(\RR)$ be non-negative cutoff functions  with $\supp \chi_i \subset [-\delta_i, \delta_i]$
and $\chi_i(s) = 1$ for $s \in [-\delta_i / 2, \delta_i / 2]$ , $i = 1,2$.
Below, $\delta_1$ will be taken sufficiently small with respect to $\delta_2$ so that an error
term arising on $\mf$ has a definite sign.

Consider the function
\begin{equation} \label{eq:radial com symb}
    q(x, y, \tau) = \chi_2(\abs{y}^2) \chi_1(x) \chi_1(\tau - m),
\end{equation}
which is identically $1$ on a small neighborhood of the
set $\tau = m$ over $\poles$.

As opposed to the previous section, we directly calculate the commutator with the regularized symbols.
Since $q$ is compactly supported in $\tau$, we only need a regularization
for the decay order. Hence, for $r \in (0,1)$, we set
\begin{align}
    J_{r,\delta} \coloneqq (1 + r/x)^{-\delta}\,.
\end{align}
We have that $J_{r, \delta} = x^\delta (x + r)^{-\delta}$
and therefore
\begin{align*}
    \ffsymb{-\delta}(J_{r,\delta}) = r^{-\delta}\,.
\end{align*}

As before, by the form of the Hamilton vector field~\eqref{eq:hammy 3}, we have
\begin{align*}
    \Hamf J_{r,\delta} = -2 \tau x \delta \frac{r}{x + r} J_{r,\delta}
\end{align*}
Setting
\begin{align*}
    q_r^\flat(x, y, \tau, \zeta) = \chi_1(x) \chi_1(\tau - m) \left( \left(2\ell + 1 - 2 \delta \frac{r}{x + r}\right) \tau \chi_2(\abs{y}^2) - 4 (\zeta + \tau y) \cdot y \, \chi_2'(\abs{y}^2)\right)
\end{align*}
we have that
\begin{align*}
    \Hamf (x^{-(\ell +1/2)} J_{r,\delta} q) &= x^{-(\ell - 1/2)} J_{r,\delta} \cdot \left( q_r^\flat + O(x) \right)\,.
\end{align*}
Since $x \geq 0$, we have that
\begin{align*}
    2\ell + 1 - 2 \delta \frac{r}{x + r} \geq 2 (\ell + 1/2 - \delta)
\end{align*}
and thus for $\delta > \ell - 1/2$, we have $2\ell + 1 - 2r/(x + r) > 0$.

Let $\ell \in \RR$, $\kappa (\ell + 1/2) > 0$, and define
\begin{align}
    \label{eq:b above thresh 1}
    b_r(x, y, \tau) &= \sqrt{2 \tau \kappa \left(2\ell + 1 - 2\delta\frac{r}{x + r}\right)} \chi_2(\abs{y}^2) \chi_1(x) \chi_1(\tau - m) \\
    &= \sqrt{2 \tau \kappa \left(2\ell + 1 - 2 \delta \frac{r}{x + r}\right)} \cdot q(x, y, \tau) \,, \nonumber\\
    \nonumber
    e(x, y, \tau, \zeta) &=  \sqrt{ 8 (\zeta + \tau y) \cdot y \chi_2(\abs{y}^2) (-\chi_2'(\abs{y}^2))} \chi_1(x) \chi_1(\tau - m)
\end{align}
and we can write
\begin{align*}
    2 q_{r}^\flat q &= \kappa b_r^2 + e^2
\end{align*}
and consequently,
\begin{align}\label{eq:qhamf}
    \Hamf(x^{-(2\ell+1)} J_{r,\delta}^2 q^2) &= x^{-2\ell} J_{r,\delta}^2 \left( \kappa b_r^2 + e^2 + O(x) \right)\,.
\end{align}

We then define
\begin{align*}
    f(\tau) &\coloneqq q \rvert_{C} = \chi_1(\tau - m)\,, \\
    f^\flat(\tau) &\coloneqq q_r^\flat \rvert_{C} = (2\ell+1) \tau \chi_1(\tau - m)
\end{align*}
and observe that
\begin{align*}
    2 f f^\flat = \kappa \, b_r\rvert_{C}^2\,.
\end{align*}

We define the microlocal commutant as
\begin{align*}
    \tilde{Q}_r \coloneqq \Op_L(x^{-(\ell + 1/2)} J_{r,\delta} q)\,.
\end{align*}

\begin{lemma}\label{lem:ffsymb_ineq_radial}
    For every $\eps' > 0$ there exist $\delta = \delta(\eps')$ and $\psi \in \CcI$ with $\psi = 1$ on $(-\delta, \delta)$ such that
    \begin{align*}
        \kappa \ffsymb{2\ell}(i [P_V, G_\psi \tilde{Q}_r^* \tilde{Q}_r G_\psi]) &\geq (2 - \eps') r^{-2\delta} f f^\flat \ffsymbz(G_\psi)^2\,.
    \end{align*}
\end{lemma}
\begin{proof}
    The proof is almost identical to the proof of Lemma~\ref{lem:commutator_ffsymb}.
    We write
    \begin{align*}
        R_1 \coloneqq{}& \kappa \ffsymb{2\ell}(i [P_V, G_\varphi \tilde{Q}_r^* \tilde{Q}_r G_\varphi]) - 2 r^{-2 \delta} f f^\flat \ffsymbz(G_\varphi)^2 \\
        ={}& \ffsymb{2\ell}( i [P_V, G_\varphi \tilde{Q}_r^* \tilde{Q}_r G_\varphi] - i [P_0, G_{\varphi,0} \tilde{Q}_r^* \tilde{Q}_r G_{\varphi,0}] ) \\
        &+ \ffsymb{2\ell}\left(i [P_0, G_{\varphi,0} \tilde{Q}_r^* \tilde{Q}_r G_{\varphi,0}] \right) - 2 r^{-2\delta} f f^\flat \ffsymbz(G_{\varphi,0})^2 \\
        &+ 2 r^{-2\delta} f f^\flat \left( \ffsymbz(G_{\varphi,0})^2 - \ffsymbz(G_\varphi)^2 \right)\,.
    \end{align*}
    We use Corollary~\ref{cor:commutator_symbol}, equation~\eqref{eq:qhamf}, and Lemma~\ref{lem:diff_Gpsi_Gpsi0} to obtain that
    $R_1 \in \Psisclsc^{-\infty,-1,0}$ and
    \begin{align*}
        \norm{R_1(\tau)}_{\LinOp(L^2, \Sobscl^{1,1})} \leq C f(\tau) f^\flat(\tau)\,.
    \end{align*}
    Since we assumed that $H_{V_0}$ has purely absolutely continuous spectrum near $[m^2, \infty)$, we can apply Lemma~\ref{lem:shrinking_window} to obtain the claimed inequality.
\end{proof}

Again we choose $\eps' \in (0, 1/4)$ and $\psi$ in the lemma above and we choose $\phi \in \CcI$ with $\phi(s) = 1$ for $\abs{s} \leq \delta(\eps')/2$ and $\supp \phi \subset (-\delta(\eps'), \delta(\eps'))$.
Define
\begin{align*}
    Q_r &\coloneqq \Op_{L}(x^{-(\ell + 1/2)} J_{r,\delta} q) G_{\phi}\,,\\
    B_r &\coloneqq \Op_{L}(x^{-\ell} J_{r,\delta} b_r) G_{\phi}\,,\\
    E_r &\coloneqq \Op_{L}(x^{-\ell} J_{r,\delta} e) G_{\phi}\,.
\end{align*}

We have that $Q_r \in \Psitsc^{-\infty, \ell+1/2 - \delta}(X)$ and $B_r, E_r \in \Psitsc^{-\infty, \ell - \delta}(X)$.
The indicial operators are given by
\begin{equation}
\begin{split}
    \ffsymb{\ell+1/2 - \delta}(Q_r)(\tau) &= r^{-\delta} \chi_1(\tau - m) \ffsymbz(G_{\phi})(\tau)\,,\\
    \ffsymb{\ell - \delta}(B_r)(\tau) &= 2 r^{-\delta} \sqrt{\tau \kappa (\ell + 1/2)} \chi_1(\tau - m) \ffsymbz(G_{\phi})(\tau)\,,\\
    \ffsymb{\ell - \delta}(E_r)(\tau) &\equiv 0\,.
\end{split}\label{eq:principal_symbols_radial_pt}
\end{equation}

\begin{lemma}
    There exists $F_r \in \Psitsc^{-\infty, 2(\ell-\delta)-1}$ such that $\WFtsc'(F_r) \subset \WFtsc'(B_r)$ and
    \begin{align*}
        (1 - \eps) B_r^* B_r &\leq \kappa (i[P_V, Q_r^* Q_r] - E_r^* E_r) + F_r
    \end{align*}
    and
    \begin{align*}
        F_r \in \LinOp(\Sobsc^{s,2\ell-1}, L^2)
    \end{align*}
    is uniformly bounded as $r \to 0$.
\end{lemma}
\begin{proof}
    Using the sharp G\aa{}rding type theorem, Proposition~\ref{prop:sharp_Garding}, we only have to prove that
    \begin{align*}
        \mfsymb{2\ell}(B_r^* B_r) &= \kappa \mfsymb{2\ell}(i[P_V, Q_r^* Q_r] - E_r^* E_r)\,,\\
        (1 - \eps') \ffsymb{2\ell}(B_r^* B_r) &\leq \kappa \ffsymb{2\ell}(i[P_V, Q_r^* Q_r] - E_r^* E_r) \,.
    \end{align*}
    The first claim follows from \eqref{eq:qhamf} and the fact that
    \begin{align*}
        \mfsymb{2\ell}(i [P_V, Q_r^* Q_r]) &= H_p( x^{-(2\ell+1)} q^2) \mfsymbz(G_\phi^2)\,.
    \end{align*}

    For the second equality, we observe that
    \begin{align*}
        \ffsymb{2\ell}(E_r^* E_r) = 0
    \end{align*}
    together with Lemma~\ref{lem:ffsymb_ineq_radial} implies the claimed inequality.
\end{proof}

We define
\begin{align*}
    Q_0 &\coloneqq \Op_L(x^{-(\ell + 1/2)} q) G_{\phi}\,,\\
    B_0 &\coloneqq \Op_L(x^{-\ell} b_0) G_{\phi}\,,\\
    E_0 &\coloneqq \Op_L(x^{-\ell} e) G_{\phi}\,.
\end{align*}

We now only assume that
\begin{align*}
    (V - V^*)/(2i) \in \Psitsc^{0,-2}\,.
\end{align*}

\begin{lemma}[Above threshold estimate]
    \label{lem:above_threshold_v1}
    Let $\ell > -1/2$.
    Let $B_0$ and $Q_0$ defined be as above
    and $B_1 \in \Psitsc^{s-1/2, \ell-1/2}$ such that
    \begin{align*}
        \WFtsc'(Q_0) \subset \Elltsc(B_1)\,.
    \end{align*}
    If $u \in \Sobres$ with $Q_0 P_V \in \Sobsc^{-1/2, 1/2}$, and $B_1 u \in \Sobsc^{s-1/2, \ell-1/2}$, then $B_0 u \in L^2$ and
    \begin{align*}
        \norm{B_0 u} \leq C \left( \norm{Q_0 P_V u}_{-1/2, 1/2} + \norm{B_1 u}_{s-1/2, \ell-1/2} + \normres{u} \right)
    \end{align*}
\end{lemma}

\begin{lemma}[Below threshold estimate]
    \label{lem:below_threshold_v1}
    Let $\ell < -1/2$.
    Let $B_0, E_0$ and $Q_0$ defined be as above
    and $B_1 \in \Psitsc^{s-1/2, \ell-1/2}$ such that
    \begin{align*}
        \WFtsc'(Q_0) \subset \Elltsc(B_1)\,.
    \end{align*}
    If $u \in \Sobres$ with $E_0 u \in L^2$, $Q_0 P_V \in \Sobsc^{-1/2, 1/2}$, and $B_1 u \in \Sobsc^{s-1/2, \ell-1/2}$, then $B_0 u \in L^2$ and
    \begin{align*}
        \norm{B_0 u} \leq C \left( \norm{E_0 u} + \norm{Q_0 P_V u}_{-1/2, 1/2} + \norm{B_1 u}_{s-1/2, \ell-1/2} + \normres{u} \right)
    \end{align*}
\end{lemma}

\begin{proof}[Proof of Lemma~\ref{lem:above_threshold_v1}]
    We take $\delta = \ell - \ell'$, so that $Q_r \in \Psitsc^{-\infty, \ell' + 1/2}$ and since we are above threshold, we have that $\kappa = 1$.
    As in \eqref{eq:Im_PV}, we have that
    \begin{align*}
        2 \Im \ang{Q_r P_V u, Q_r u} &= \ang{i [P_{\tilde{V}}, Q_r^* Q_r] u, u} - \ang{(Q_r^* Q_r V'' + V'' Q_r^* Q_r) u, u}\,,
    \end{align*}
    where $\tilde{V} \coloneqq (V + V^*)/2$ and $V'' \coloneqq (V - V^*)/(2i)$ and $P_V = P_{\tilde{V}} - i V''$.
    Since $V'' \in \Psitsc^{0,-2}$, we have that
    \begin{align*}
        F_r' \coloneqq F_r  + Q_r^* Q_r V'' + V'' Q_r^* Q_r \in \Psitsc^{- \infty, 2\ell' - 1}
    \end{align*}
    and
    \begin{align*}
        \norm{B_r u}^2 &\lesssim \ang{ i [\PtV, Q_r^* Q_r] u, u} - \norm{E_r u}^2 + \ang{ F_r u, u} \\
        &= 2 \Im \ang{Q_r P_V u, Q_r u} - \norm{E_r u}^2 + \ang{F_r' u, u} \\
        &\leq \mu \norm{x^{1/2} Q_r u}^2 + \frac{1}{4\mu} \norm{ x^{-1/2} Q_r P_V u}^2 + \abs{\ang{F_r' u, u}} \\
    \end{align*}
    Taking $\mu \in (0,1)$ small enough we can absorb $x^{1/2} Q_r u$ into $B_r u$ since $x^{1/2} Q_r$ is a 0th order multiple of $B_r$.
    Hence, we have that
    \begin{align*}
        \norm{B_r u}^2 \lesssim \norm{ x^{-1/2} Q_r P_V u}^2 + \abs{\ang{F_r' u,u}}\,.
    \end{align*}
    By assumption the right-hand side is bounded uniformly as $r \to 0$, and since $B_r \to B_0$, we have that $B_0 u \in L^2$.
    The claimed estimate follows from a standard elliptic estimate for $F_r'$:
    \begin{align*}
        \abs{ \ang{ F_r' u, u}} \lesssim \norm{B_1 u}_{s-1/2, \ell-1/2}^2 + \normres{u}^2\,.
    \end{align*}
\end{proof}

\begin{proof}[Proof of Lemma~\ref{lem:below_threshold_v1}]
    We use the same argument, but start with $u \in \Sobres$ and use $\delta = \ell - M$. Moreover, since $\kappa = -1$ we have to keep the term
    $\norm{E_r u}^2$.
\end{proof}

\begin{proof}[Proof of Proposition~\ref{prop:tsc_radial_below}]
    The proof follows from exactly the same arguments as the propagation of singularities result, Proposition~\ref{prop:propagation_localized} follows from the
    explicit construction in Lemma~\ref{lem:propagation_small}.
\end{proof}

\begin{proof}[Proof of Proposition~\ref{prop:tsc_radial_above}]
    Again the argument is very similar, but we can decrease the decay order $\ell - k/2$ in $\norm{B_1 u}_{s-k/2,\ell-k/2}$ only to $\ell'$ for $\ell' > -1/2$ and then we use
    Proposition~\ref{prop:elliptic_Gpsi_estimate} to estimate
    \begin{align*}
        \norm{B_1 u}_{s',\ell'} \lesssim \norm{B_1 G_{\phi_1} u}_{s',\ell'} + \norm{G' P_V u}_{s-2, \ell} + \norm{B' u}_{s-r_1-1, \ell-r_2-1} + \normres{u}\,.
    \end{align*}
\end{proof}

\section{Construction of the causal propagators}\label{sec:3sc_propagators}

Using the estimates above, we can now complete the construction of the
causal propagators.  The argument is similar in structure to the
construction done for scattering perturbations $V$ in
Section~\ref{sec:fred for scat} and Section~\ref{sec:scat perturb}.
We will now state the precise assumptions for a $\tsc$ perturbation $V$.

Let $V \in \rho_{\mf} \Psitsc^{1,0}$ such that
\begin{enumerate}
    \item \label{it:lead}
        $V$ is asymptotically static of order $1$ in the sense of Definition~\ref{def:asymptotic_static}.
    \item \label{it:imag}
        The imaginary part $(V - V^*)/(2i)$ satisfies
        \begin{align*}
            (V - V^*)/(2i) \in \Psitsc^{0,-2}\,.
        \end{align*}
    \item \label{it:spec}
        The Hamiltonians $H_{V_\pm} = \Delta + m^2 + V_{\pm}$ have purely absolutely continuous spectrum near $[m^2, \infty)$
        and finitely many eigenvalues in $(-\infty, m^2)$.
\end{enumerate}
If the Hamiltonians $H_{V_{\pm}}$ have bound states we need more decay for $V - V_{\pm}$.
In that case we additionally have to assume positivity of the Hamiltonians to conclude invertibility of $P_V$.

We construct the causal propagators, working in $\cX$ spaces based on
exactly the same scattering Sobolev spaces used in the scattering
setting in Section \ref{sec:fred for scat}.  Recall the
spacetime-dependent forward and backward weight functions
$\ellvar_\pm$ from Definition \ref{def:forward backward weights}.  In
addition to being monotone along the flow, we must have that the
weight functions $\ellvar_\pm$ are constant in an open neighborhood of
$\Tsco^{*}_{\poles} X$.  Indeed, all of our estimates over the poles
are proven with constant weights near $\poles$, and in fact we do not
even define the indicial operator in the presence of variable
weights.  Note that we can take
$\ellvar_{\pm} \in \CI(X)$ i.e.\ a function on spacetime, which is
constant in neighborhoods of both past and future causal infinity, as
discussed in the proof of Proposition \ref{thm:global char set scattering}.
Note that this cannot be achieved for the Feynman weights, since they have to satisfy $\ellvar > -1/2$ at $\NP \cap \{\tau = -m\}$ and $\ellvar < - 1/2$ on $\NP \cap \{\tau = m\}$.

We fix $s \in \RR$ and admissible forward and backward weights $\ellvar_+$, $\ellvar_-$.
Define
\begin{align*}
    \cX^{s,\ellvar_{\pm}} &\coloneqq \{u \in \Sobsc^{s,\ellvar_{\pm}} \colon P_V u \in \Sobsc^{s-1, \ellvar_{\pm}+1}\}\,,\\
    \cY^{s,\ellvar_{\pm}} &\coloneqq \Sobsc^{s,\ellvar_{\pm}} \,.
\end{align*}
\begin{remark}
    Note that these $\cX^{s,\ellvar_{\pm}}$ spaces depend on $V$, but we do
    not include this in the notation. However, if $W \in \Psitsc^{1, -1}$, then
    \begin{align*}
        W : \Sobsc^{s, \ellvar_{\pm}} \lra \Sobsc^{s-1, \ellvar_{\pm} + 1}\,,
    \end{align*}
    hence $\cX^{s,\ellvar_{\pm}}$ only depends on $V_{\pm}$.
\end{remark}

The first theorem is for the case that $H_{V_\pm}$ have no eigenvalues.
\begin{theorem}\label{thm:fredholm causal 3sc}
    Let $V \in \rho_{\mf} \Psitsc^{1,0}$ satisfying \eqref{it:lead}, \eqref{it:imag}, and \eqref{it:spec} above.
    If the Hamiltonians $H_{V_{\pm}}$ have no discrete spectrum, then
    the mapping
    \begin{equation*}
        P_V \colon \cX^{s, \ellvar_+} \lra \cY^{s - 1, \ellvar_+ + 1}
    \end{equation*}
    is invertible
    and its inverse is the forward propagator in the sense of \eqref{eq:forward solution}.
    The same is true for the backward propagator with $\ellvar_+$ replaced by $\ellvar_-$.
\end{theorem}

In the case that the Hamiltonians $H_{V_\pm}$ have discrete spectrum we have to strengthen the assumptions on the decay of $V - V_{\pm}$,
namely we assume to have an additional order of spacetime decay.
\begin{theorem}\label{thm:result with bound states}
    Let $V \in \rho_{\mf} \Psitsc^{1,0}$ be asymptotically static of order $2$ and satisfies \eqref{it:imag} and \eqref{it:spec}.
    Then the mapping
    \begin{equation} \label{eq:fredholm 3scat causal true}
        P_V \colon \cX^{s, \ellvar_+} \lra \cY^{s - 1, \ellvar_+ + 1}
    \end{equation}
    is Fredholm.

    Moreover,
    \begin{enumerate}
    \item if $H_{V_-}$ is positive, then \eqref{eq:fredholm 3scat causal true} is injective and
    \item if $H_{V_+}$ is positive, then \eqref{eq:fredholm 3scat causal true} is surjective.
    \end{enumerate}
    If $H_{V_{\pm}}$ are both positive, then $P_V$ is invertible and its inverse is the forward propagator in the sense of \eqref{eq:forward solution}.
    The same is true for the backward propagator with $\ellvar_+$ replaced by $\ellvar_-$.
\end{theorem}

Finally, if $V$ is static we can drop the assumption of no decaying modes.
\begin{theorem}\label{thm:result for static}
    Let $V = V(z) \in S^{-2}(\RR^n_z)$ with \eqref{it:imag} and \eqref{it:spec}, then
    \begin{align*}
        P_V \colon \cX^{s, \ellvar_+} \lra \cY^{s - 1, \ellvar_+ + 1}
    \end{align*}
    is invertible.
\end{theorem}

\begin{remark}
    The conclusions of the three theorems remain true if $D_t^2 -
    \Delta_z$ is replaced by $\square_g$ for a non-trapping, asymptotically
    Minkowski perturbation $g$.
\end{remark}

\subsection{Assuming no bound states}
\label{sec:no-bound-states}

In this section, we prove Theorem~\ref{thm:fredholm causal 3sc}.
The assumption that there are no eigenvalues in $(-\infty, m^2)$ implies that
\begin{align}
  \label{eq:no bound assump}
  (-m , m ) \subset \Ell_{\ff}(P_V).
\end{align}
Indeed, the assumption implies directly that
$\hat{N}_{\ff_{\pm}}(P_V)(\tau) = \tau^2 - (\Delta_z + m^2 + V_{\pm}(z))$
is invertible for $\abs{\tau} < m$, which implies the
ellipticity statement.  (See Section \ref{sec:reduced
  elliptic}.)  Thus in this section we will have
\eqref{eq:no bound assump} at both $\NP$ and $\SP$.

We first prove that $P_V : \cX^{s, \ellvar_+} \to \cY^{s-1, \ellvar_+ + 1}$ is a
Fredholm mapping.  This reduces to showing the analogue of
\eqref{eq:global scat fredholm}, namely that for any $N, M, s \in\RR$, there is $C > 0$ such that, provided all quantities
are finite,
\begin{align}
  \label{eq:global scat fredholm perturb}
  \norm{u}_{s, \ellvar_{+}} &\leq C \left( \norm{P_V u}_{s-1,\ellvar_{+}+1} + \norm{u}_{-N, -M} \right), \\
  \norm{u}_{1-s, -1 - \ellvar_{+}} &\leq C\left( \norm{P_V u}_{-s, -\ellvar_{+}} + \norm{u}_{-N,-M} \right). \notag
\end{align}

To obtain these estimates we again argue as in Section
  \ref{sec:fred for scat}.  We choose
an open cover $O_{1}, O_2, O_3, O_4$ of the compressed cotangent
bundle  $\Tdoto ^{*} X$.
Note here that, due to our assumption of no bound states, $\Chartsc
P_V = \Chartsc P_0$.  With $c_1, c_2 > 0$, 
\begin{enumerate}
\item $\Rad_{\pm}^{p} \subset O_{1} \subset \{ \ellvar_{+} = -1/2 +
 c_1\}$, in particular $\pm m \in O_{1} \cap W^{\perp}$ over $\SP$.
\item $\Rad_{\pm}^{f}\subset O_{2}\subset \{ \ellvar_{+}=
  -1/2-c_2\}$, in particular $\pm m \in O_{2} \cap W^{\perp}$ over $\NP$.
\item $\Chartsc P_V \subset O_{1}\cup O_{2}\cup O_{3}$.  
\item $O_{3} \cap \Chartsc P_V$ is controlled along $\Hamscp$ by $O_{1}$, 
\item $O_{2}\setminus \Rad^{f}$ is controlled along $\Hamscp$ by $O_{3}$,
  and 
\item $O_{4} \subset \Elltsc(P_{V})$.
 \end{enumerate}

In this context, the meaning of item (3) is that, away from $\poles$,
the standard characteristic set $\tau^{2} - |\zeta|^{2} - m^{2}$ lies
in $O_{1} \cup O_{2} \cup O_{3}$, while over $\poles$, 
\begin{equation*}
  [- \infty, - m] \cup [m, + \infty] \subset
  O_{1} \cup O_{2} \cup O_{3} \cap \Wperpo
\end{equation*}
In fact, we will choose $O_{1}, O_{2}, O_{3}$ such that, for some
$\epsilon > 0$,  
\begin{align}
  (-m - \epsilon, -m + \epsilon) \cup   (m - \epsilon, m +
  \epsilon) &= O_{1} \cap W^{\perp} \mbox{ over } \SP \\
  (-m - \epsilon, -m + \epsilon) \cup   (m - \epsilon, m +
  \epsilon) &= O_{2} \cap W^{\perp} \mbox{ over } \NP \\
  [- \infty, -m - \epsilon/2) \cup   (m + \epsilon/2, + \infty] &=
  O_{3} \cap \Wperpo \mbox{ over both } \SP \mbox{ and } \NP.  \label{eq:O3 over Wperpo}
\end{align}
In particular, $O_1$ and $O_2$ are neighborhoods of $\Rad^p$ and
$\Rad^f$ respecitively, but over $\NP$ and $\SP$ we restrict their
size for convenience.

The meaning of (4) must now be understood in $\Tdoto X$, in the sense
that any point $\tau_0 \in \Wperpo \cap O_3$ over, say, $\SP$,
corresponds to all points $\tau_0^2 - |\zeta|^2 - m^2 = 0$ over $\SP$,
and the control assumption implies that for any such $\zeta$, there
exists some point $q \in \Tsco^* X$ such that for some $s \in
\RR$: (1) $\exp_{\Hamscp}(sq) = (\SP, \tau_0, \zeta)$ and (2)
either $q \in O_1 \cap (\Tsco^* X \setminus \Tsco_{\SP}^* X)$ or $q
\in \Tsco_{\SP}^* X$ and $\tau(q) \in O_1$.  Similarly for $O_2$ being
controlled by $O_3$.

Such a collection $O_{1}, O_2, O_3, O_4$ of open sets can be
constructed as follows.  Our $O_1$ must be a set which contains the
$\SP$ radial sets and $\pm m \in W^\perp$ over $\SP$.  For this we can
take, near $\SP$, a set of the form $(-m - \epsilon, -m + \epsilon)
\cup (m - \epsilon, m + \epsilon) \subset W^\perp$ over $\SP$,
union with the set in $\Tsco^* X \setminus \Tsco^*_{\SP} X$:
\begin{equation}
\begin{gathered}
  \bigcup_{\pm} \left( \left\{ |\tau \pm m| < \epsilon \right\}\right) \bigcup \left( \left\{ \frac{|\tau \pm m|}{\langle \tau,
      \zeta \rangle} < \epsilon \right\} \bigcap \{ \langle \tau,
  \zeta \rangle   > 1/\epsilon \} \right) \\
  \bigcap \{ t < - 
  1/\epsilon, 0 < |y| < \epsilon  \}.
\end{gathered}\label{eq:explicit set}
\end{equation}
This is a union of a basic neighborhood around $\tau = \pm m$ in the
uncompactified $\dot T^*X$ union with an open set around the limit
points of $\tau \pm m$ on the fiber boundary, all localized near $\SP$
by the intersection in the second line.  We take the union of this
with an open set containing $\Rad^p$ away from $\SP$, for example with
the coordinate function $w$ which defines the radial set (see
\eqref{eq:great coords}), simply taking
\begin{equation}
\left\{ |w| < \epsilon  \right\}  \bigcap \{ t < -
1/\epsilon,  |y| > \epsilon / 2 \}\label{eq:near rad away poles}
\end{equation}
works.   For
$O_2$ we define the set in the exact same way but near $\NP$.

For $O_3$ we will take an open neighborhood of $\Chartsc(P_V)$ in the
complement of $O_1 \cup O_2$.  Again we can be explicit.  We take
$O_3$ to be as in \eqref{eq:O3 over Wperpo} on $\Wperpo$ over both
$\poles$, thus nearby $O_3$ we take as in \eqref{eq:explicit set},
namely with identical to \eqref{eq:explicit set} except with $\left\{
  |\tau \pm m| >  \epsilon/2 \right\}$ in the first term.  Doing the
exact same over $\NP$, we take the union of these with a small
neighborhood of $\Chartsc(P_V)$ from the radial sets and the poles,
for example, with 
\[
\left\{ |w| > \epsilon/2  \right\}  \bigcap \{  |y| > \epsilon / 2 \}
\bigcap \{  |\sigma_{\sc, 2, 0}(P_V)| < \epsilon \}.
\]
Then $O_4$ we take any open neighborhood of the closure of the
complement $O_1 \cap O_2 \cap O_3$.  Necessarily $O_4 \subset
\Elltsc(P_V)$.

Note that, due to the nature of $\Tdoto^*X$ and, correspondingly,
$\tsc$-ellipticity, the sets $O_1$ and $O_3$, for example, necessarily
overlap at the intersection of the closures of $\{ \tau = \pm m \}$
with the fiber boundary. This is simply because all closures of sets
of the form $\{ \tau = c \}$ intersect over the fiber boundary at the
``fiber equator''.

Keeping this in mind, we now choose a collection $B_{1}, B_{2},
B_{3}, B_{4} \in \Psitsc^{0,0}$ (in fact in $\Psisc^{0,0}$) analogous
to the $B_{i}$ in the scattering construction in Section \ref{sec:fred
  for scat}.  However, recalling that
\begin{equation*}
\Ell_{\ff}(B) \neq \varnothing \implies \Wperpo = \WF'_{\ff}(B),
\end{equation*}
we cannot have the
$\WFtsc'$ subordination condition \eqref{eq:same old WF condition} on
the nose.  In fact, this condition is only a convenient means of
relating the elliptics sets of the $B_i$ to the assumed Hamiltonian
dynamics.  Here we instead take the approach that the $B_i$ will have ellip

Due to the overlap just
described, the conditions $\WFtsc' (B_{i})\subset
O_{i}$ are ambiguous, in that they do not indicate the behavior of the
$B_i$ at the fiber equator above $\poles$.  Here we
instead require that
\begin{equation}\label{eq:elliptic B assump tsc}
  \Ctscd \subset \bigcup_{i = 1}^4 \Elltsc(B_i),
\end{equation}
meaning each point in $\Ctscd$ is in the elliptic set of some
$B_i$ \eqref{eq:square analogue dot}.
Any collection $B_i$ satisfying
\eqref{eq:elliptic B assump tsc} gives the ``same control'' as a
partition of unity, namely, for any $N, M \in \RR$, and any $s, \ell$,
\begin{equation} \label{eq:Bi tsc bound}
    \norm{u}_{s, \ell} \lesssim
    \norm{B_1 u}_{s, \ell} + \norm{B_2 u}_{s, \ell} + \norm{B_3 u}_{s, \ell} + \norm{B_4 u}_{s, \ell} + \normres{u}
\end{equation}
(Indeed, in that case $\sum_{i = 1}^4 B_i^* B_i$ is globally $\tsc$-elliptic and
\eqref{eq:Bi tsc bound} follows from the Fredholm property for
globally $\tsc$-elliptic operators and boundedness of $B_i^*$.)

As a
replacement for the condition $\WFtsc' (B_{i})\subset
O_{i}$, we assume that, for some $\delta > 0$
\begin{align*}
  \Elltsc(B_i) &\subset O_i,  \\
  \WF_{\ff}'(B_i G_\psi) &\subset O_i \mbox{ for } \psi \in
                           \CcI(\mathbb{R}) \mbox{ with } \supp \psi
                           \subset (-\delta,\delta), \mbox { and }\\
  \WF_{\bullet}'(B_i) \cap \Chartsc(P_V) &\subset O_i \mbox{ for } \bullet
                                       \in \{ \fib, \mf \}.
\end{align*}
To construct such $B_i$ which satisfy both these conditions and
\eqref{eq:elliptic B assump tsc} we can use cutoff
functions and the expressions for the $O_i$ above.  For example, we
can take $B_1 = \Op_L(\chi_1)$ with $\chi_1 \colon \Tsco^* X \lra
\RR$ identically $1$ on sets of the forms \eqref{eq:explicit
  set} and \eqref{eq:near rad away poles} with the $\epsilon$ replace
by a smaller $\epsilon' > 0$, and with support in the union of
\eqref{eq:same old WF condition}.  Such $B_1$ has $\ffsymbz(B_1)(\tau)
= \id$ for $\tau \in [- m - \epsilon', -m + \epsilon'] \cup [m -
\epsilon', m + \epsilon']$ and indeed $\Rad^p \subset \Elltsc(B_1)$.
We define $B_2$ and $B_3$ similarly.

We now follow the process in Section \ref{sec:fred for scat}, meaning
we begin with above threshold estimates to bound $B_{1} u$,
propagation estimates to bound $B_{3} u$ (in terms of $B_{1} u$), below
threshold estimates to bound $B_{2} u$ (in terms of $B_{3} u$) and
then elliptic estimates for $B_{4} u$.
For example, for $B_{1} u$, Proposition \ref{prop:tsc_radial_above}
implies that, for any $\ell'$ with $-1/2 < \ell' < \ell$,
\begin{align}\label{eq:above thresh in Fredholm}
        \norm{B_{1} u}_{s,\ell} \leq C \left( \norm{\tilde B  u}_{s-1,
  \ell'} + \norm{P_V u}_{s-1,\ell+1} + \normres{u} \right)\,.
\end{align}
where $B\in \Psitsc^{0,0}$, $\WFtsc'(B_{1} G_{\psi}) \subset
\Elltsc(\tilde B)$.  To see that this follows from the
proposition, note first that all the terms in the above threshold
estimate involving $P_{V} u$ can be bounded by the global norm of $P_{V}
u$.  Similarly, the terms on the right hand side involving
$G_{\phi} u $ and $B' u$ are both bounded by the $\tilde B u$ term in
\eqref{eq:above thresh in Fredholm}.  Similarly,
\begin{align}\label{eq:propagation in Fredholm}
        \norm{B_{3} u}_{s,\ell} \leq C \left( \norm{B_{1} u}_{s,\ell}
   + \norm{P_V u}_{s-1,\ell+1} + \norm{\tilde B'  u}_{s-1,
  \ell - 1}+ \normres{u} \right)\,,
\end{align}
where $\tilde B'$ comes from the $B'$ term in Proposition \ref{prop:propagation_localized}.
For the below threshold estimate we have
    \begin{align}
        \norm{B_{2} u}_{s,\ell} &\leq C \big( \norm{B_{3}
                                  u}_{s,\ell} + \norm{P_V
                                  u}_{s-1, \ell+1} + \normres{u} \big)\,.
    \end{align}
Note here that the term from Proposition \ref{prop:tsc_radial_below}
with $B' u \in \Sobsc^{s-r_1-1, \ell-r_2 - 1}$ does not appear as it
can be iterated and absorbed into the residual term.

Putting these estimates together with the elliptic estimates
(Proposition \ref{thm:elliptic regularity theorem}) 
and using \eqref{eq:youngs inequality} to remove the $s-1, \ell-1$
term yields the global estimate which in turn gives the Fredholm
result.

To see that $P_V$ is invertible under the given assumptions,
we need only check that its kernel and cokernel are zero.
Distributions
\begin{align*}
u \in  \ker \left(P_V \colon \cX^{s, \ellvar_+} \lra \cY^{s - 1, \ellvar_+ + 1} \right)
\end{align*}
must, thanks to the above threshold radial points estimates, 
be rapidly decaying along with all their derivatives to the past.
This allows us to use a standard argument to show that the energy must
be zero.

Indeed, we use the energy functionsl $E_{u}(t)$ with
\begin{equation}
  \label{eq:ham fun}
  E_{u}(t) \coloneqq \frac 12 \int_{\RR^n} | \pa_t u|^2 + |
  \nabla_z u |^2 + (Vu) u + m^2 u^2 dz \ge c(t) \int_{\RR^{n}} |u(t, z)|^{2} dz,
\end{equation}
where $c(t)$ is the minimum of
$\spec(\Delta + m^2 + V(t,z))$.  If there are no bound
states of $V_{-}$, then
$\spec(\Delta + m^2 + V_{-}) \ge c_{0} > 0$, so by the
Kato--Rellich theorem, there is $t_{0}$ such that for any $t < t_{0}$,
$E_{u}(t) \ge (c - \delta) \| u \|^{2}$.
By a standard   Gr\"{o}nwall argument, this
gives that $u(t, z)
\equiv 0$ for $t \gg 0$.
Indeed, $d E_{u}(t) / dt \le k(t) E_{u}(t)$ where $k(t) = O(1/|t|)^{2}$ as $t \to - \infty$. 
Since $E_u(t) \to 0$ as $t \to -\infty$ the differential Gr\"{o}nwall inequality shows
that $E_{u}(t) \equiv 0$ for $t \ll 0$.
By the positivity of the energy, we obtain that $u(t) = 0$ for $t \ll 0$.
Then uniqueness of solutions to the Cauchy
problem shows $u \equiv 0$ globally.

The cokernel of \eqref{eq:fredholm 3scat causal true} can be
identified with
\begin{align*}
    \ker \left(P_{V^*} \colon \cX^{1-s, -1 - \ellvar_+} \lra \cY^{-s, - \ellvar_+ } \right)
\end{align*}
and the same argument but using the energy estimates at $t \to +
\infty$ shows the cokernel is zero.

We have shown that $P_V : \cX^{s, \ellvar_+} \to \cY^{s-1, \ellvar_+ + 1}$ for any $s \in \RR$ and admissible forward weight $\ellvar_+$
is invertible. Indeed the value of the inverse mapping is independent of the specific choice of $s$ and $\ellvar_+$.
To see this, let $s' \in \RR$ and $\ellvar_+'$ an admissible weight and given
$f \in \Sobsc^{s-1, \ellvar_+ + 1} \cap \Sobsc^{s'-1, \ellvar_+' + 1}$, let $u \in \cX^{s,\ellvar_+}$ and $u' \in \cX^{s', \ellvar_+'}$ with
$P_V u = f = P_V u'$. Then $u - u' \in \cX^{\tilde{s}, \tilde{\ellvar_+}}$ for some $\tilde{s} \in \RR$ and $\tilde{\ellvar_+}$ admissible forward weight.
Then $P_V (u - u') = 0$ and therefore $u = u'$.

Thus, we can unambiguously speak of the inverse of $P_V$, which we denote by $(P_V)^{-1}_{\forw}$. The fact that $(P_V)^{-1}_{\forw}$ satisfies the forward condition,
\eqref{eq:forward solution}, follows again from energy arguments. Indeed if $f \in \Sobsc^{s-1,\ellvar_+ + 1}$ and $\supp f \subset \{t \geq T\}$, then $u = (P_V)^{-1}_{\forw} f$
satisfies the above threshold estimates at $\Rad^p$ and thus is a Schwartz function as $t \to -\infty$ and the same energy argument shows that $\supp u \subset \{t \geq T\}$.

\subsection{With bound states}\label{sec:3sc_bound}
 
In this section, we prove Theorem~\ref{thm:result with bound states} and Theorem~\ref{thm:result for static}.
We make appropriate adjustments to the above
propagator construction to include the possibility that there are
bound states of the Hamiltonian $\Delta + m^2 + V_{\pm}$ with (negative) energy
bigger than $0$.  Such states appear as elements in the kernel of
$\ffsymbz(P_V)(\tau)$ for $\tau \in (-m,m)$.

In discussing bound states, it is useful to distinguish the behavior
of $P_V$ at $\NP$ and $\SP$, so for the remainder of this section we
include the pole in the notation for the indicial operator:
\begin{equation*}
  \ffsymbpz(P_V)(\tau) = \mbox{ indicial operator of $P_V$
    at $\NP$},
\end{equation*}
while $\ffsymbmz(P_V)(\tau)$ is the indicial operator at $\SP$.
Similarly we write
\begin{equation*}
  W^\perp_+ = W^\perp \mbox{ over } \NP
\end{equation*}
and $W^\perp_- = W^\perp \mbox{ over } \SP$.
Recall that, by scattering ellipticity, for $|\tau| < m$,
\[
(\ffsymbpmz(P_V)(\tau) w = 0 \implies w \in \schwartz(\RR^n),
\]
and thus by self-adjointness of $\ffsymbpmz(P_V)(\tau)$,
\begin{equation}\label{eq:elliptic condition PV}
    \tau \in \Ell_{\ff, \pm}(P_V) \iff
    \ker( \ffsymbpmz(P_V)(\tau)) = \set{0}.
\end{equation}

We therefore have the elliptic estimate for $P_V$ over $\ff$.
\begin{lemma}\label{thm:elliptic PV ff}
    Let $K \subset (-m, m) \subset W^\perp_+$ be a compact set such that
    \begin{equation}
        \tau\in K \implies \ker( \ffsymbpz(P_V)(\tau)) = \set{0}.\label{eq:6}
    \end{equation}
    Then there is
    $Q \in \Psitsc^{0,0}(\RR^{n + 1})$ with $K \in \Ell_{\ff}(Q)$ such that for any
    $M, N \in \RR$, $s, \ell \in \RR$, and any $Q' \in
    \Psitsc^{0,0}(\RR^{n + 1})$ with $\WFtsc'(Q) \subset \Elltsc(Q') \cap \Elltsc(P_V)$, there is $C > 0$
    such that,
    \begin{align*}
        \norm{Q u}_{s, \ell} \le C \left( \norm{Q' P_V u}_{s-2, \ell} + \normres{u}\right) .
    \end{align*}
    The same goes near $\SP$ with the relevant $+$'s replaced by $-$'s.
\end{lemma}
\begin{proof}
    This follows immediately from Proposition
    \ref{thm:elliptic regularity theorem} and the fact that
    $K \subset \Ell_{\ff}(P_V)$.
\end{proof}

It is possible that
there are finitely many points $\tau_0 \in (-m, m)$ such that
$\ffsymbpmz(P_V)(\tau_0)$ is \textit{not} invertible, and the
remainder of this section proves estimates near such $\tau_0 \in
W^\perp$.

First we establish some notation.  Let $\lambda(\tau) \coloneqq \sqrt{m^2 - \tau^2}$, define the eigenspace
(set of bound states) of $\Delta_z + V_\pm(z)$ at frequency $\lambda$:
\begin{align*}
    E_\pm(\lambda(\tau_0)) \coloneqq \ker\left( \ffsymbpmz(P_V)(\tau_0) \colon \Sobsc^{s,\ell}(\ff) \longrightarrow
    \Sobsc^{s-2, \ell}(\ff)\right).
\end{align*}
Thus $E_+(\lambda(\tau_0))$ is the collection of bound
states of $\Delta_z + V_+$ with frequency $\lambda(\tau)$.
Again, $E_\pm(\lambda(\tau_0)) \subset \mathcal{S}(\RR^n)$, so the
eigenspace is independent of $s, \ell$.  We define the set of $\tau$
values of bound states:
\begin{equation}
  \label{eq:bound states}
  \BSspm \subset \set{ \tau \in (-m, m) \colon E_\pm(\lambda(\tau)) \neq \set{ 0 } }
\end{equation}
As discussed above, we assume that $\BSspm$ is \emph{finite}, and
\begin{equation} \label{eq:no zero tau state}
    0 \not \in \BSspm,
\end{equation}
i.e.\ $\Delta_z + m^2 + V_+(z)$ has no eigenvalue at zero.

Our theorem for the causal propagators in the presence of bound states
will follow the statement and proof of Theorem  \ref{thm:fredholm
  causal 3sc} closely; we show that \emph{the same mapping}
\eqref{eq:fredholm 3scat causal true} is Fredholm, and then prove
under additional assumptions on decaying modes that it is invertible.

The theorem is proven at the end of this section using the additional
estimates proven near $\BSspm$ in Proposition \ref{thm:elliptic region
  TRUE}.

We now fix $\tau_0 \in \BSsp$. There is therefore $w \in
E(\lambda(\tau_0))$ with $w \not \equiv 0$ and we therefore have the
oscillatory solution
\begin{equation*}
P_{V_+} (e^{\pm i \tau_0 t} w) \equiv 0.
\end{equation*}
Since $w \subset \mathcal{S}(\RR^n_z)$, this implies
\[
e^{\pm i \tau_0 t} w(z) \in \Sobsc^{\infty, -1/2 - \epsilon}(\RR^{n + 1}).
\]
for any $\epsilon > 0$ but not for $\epsilon = 0$. We therefore expect a spacetime weight threshold of $-1/2$ also for
estimates localized near $\BSsp$.

Near $\tau_0 \in W^\perp$, we will prove estimates with a
spacetime weight loss identical to those in the radial points
estimates above.  We will see that this is a consequence of the fact
that our approximate projection onto the solutions $e^{i t \tau_0}
w(z)$ intertwine $P_V$ with $D_t^2 - \tau_0^2$ to leading order.  Our
main result will be the following.

\begin{proposition}\label{thm:elliptic region TRUE}
    Assume that $V$ is asymptotically static of order $0,1$
    \eqref{eq:no zero tau state} holds.
    Then for any $\epsilon >
    0$, there exists $Q \in \Psitsc^{0,0}(\RR^{n + 1})$ such that
    \begin{equation}
    \label{eq:elliptic region TRUE}
        [- m + \epsilon, m - \epsilon] \subset \Ell_{\ff}(Q),
    \end{equation}
    and for any $G \in \Psitsc^{0,0}$ with $\WFtsc'(Q) \subset \Elltsc(G)$,
    then for $s, \ell, N, M \in \RR$, we have the following estimates.

    If $\ell < -1/2$, there is $C > 0$ such that, if
    $Q u \in \Sobsc^{s,\ell}(\RR^{n + 1})$ and
    $G P_V u \in \Sobsc^{s - 2,\ell + 1}(\RR^{n + 1})$, then
    \begin{equation} \label{eq:elliptic region below}
        \norm{Q u}_{s, \ell} \le C(\norm{G P_V u}_{s - 2,\ell + 1} +  \normres{u}).
    \end{equation}

    If $\ell > -1/2$, and $\ell' \in \RR$ has $\ell > \ell' > -1/2$, then
    there is $C > 0$ such that, if $G u \in \Sobsc^{-N,\ell'}(\RR^{n + 1})$
    and $G P_V u \in \Sobsc^{s - 2,\ell + 1}(\RR^{n + 1})$, then $Qu \in
    \Sobsc^{s, \ell}$ and
    \begin{equation}
        \norm{Q u}_{s, \ell} \le C( \norm{G P_V u}_{s - 2,\ell + 1} + \norm{G u}_{s - 2,\ell'} + \normres{u}).\label{eq:elliptic region above}
    \end{equation}
\end{proposition}

The proposition follows from the elliptic regularity estimates in Lemma
\ref{thm:elliptic PV ff}, and the following lemma, which is simply
Proposition \ref{thm:elliptic region TRUE} microlocalized near a fixed $\tau_0
\in \BSsp$.  
\begin{lemma}\label{thm:elliptic non elliptic}
    Let $\tau_0 \in \BSsp$, i.e.\ let $\tau_0 \in (-m,m)$, $\tau_0 \neq 0$
    and $E_+(\lambda(\tau_0)) \neq \{ 0 \} $.  For any $Q \in
    \Psitsc^{0,0}(\RR^{n + 1})$ with $\tau_0 \subset \Ell_{\ff}(Q)$ and
    $\WFtsc'(Q)$ sufficiently small, for any, $G \in \Psitsc^{0,0}(\RR^{n+1})$ with $\WFtsc'(Q) \subset \Elltsc(G)$,
    then for $s, \ell, N, M \in \RR$, the estimates in
    Proposition \ref{thm:elliptic region TRUE} hold.
\end{lemma}

The lemma is proven at the end of this section.

The proof proceeds by approximating projection onto the solutions $e^{
  i \tau_0 t} E_+(\lambda(\tau_0))$.  If
$\Pi_{\tau_0} = \Pi_{\tau_0}(z, z')$ is the integral kernel of
orthogonal projection in $L^2(\RR^n_z)$ onto
$E_+(\lambda(\tau_0))$, i.e. for some orthonormal bases $\{ w_j \}_{j
  = 1}^N$ of $E_+(\lambda(\tau_0))$,
\begin{align*}
  \Pi_{\tau_0}(z, z') \coloneqq \sum_{j = 1}^N w_j(z) \cdot w_j(z') \in
  \mathcal{S}(\RR^n \times \RR^n),
\end{align*}
then $e^{i (t - t') \tau_0} \Pi_{\tau_0}(z, z')$ is the integral kernel
of this projection.  However,  as we will see below, $e^{i (t - t') \tau_0}\Pi_{\tau_0}(z, z')$ is not the integral kernel of a
$\tsc$-PsiDO.

We proceed by smoothing in $t$ in addition to projecting onto
$E_+(\lambda(\tau_0))$; that is, we project onto a small $\tau$-window around
$\tau_0$ on the $t$-Fourier transform side.  Let $\chi_{\ge t_0} \in
C^\infty(\RR)$ be a bump function supported near $+ \infty$
with $\chi_{\ge t_0}(t) = 1$ for $t \ge t_0 - 1$ and
$\chi_{\ge t_0}(t) = 0$ for $t \le t_0 - 2$, and let
$\chi_{\tau_0} \in C^\infty(\RR_{\tau})$ be a bump function supported
near $\tau_0$, so $\chi_{\tau_0}(\tau) = 1$ for $|\tau - \tau_0| <
\delta$ and $\chi_{\tau_0}(\tau) = 0$ for $|\tau - \tau_0| \ge 2 \delta$.  Consider the operator defined by the integral kernel which:
(1) cuts off to large time, (2) localizes in the $t$-momentum variable
around $\tau_0$, (3) projects onto $E_+(\lambda(\tau_0))$:
\begin{align}
  \label{eq:normal kernel cutoffs}
  K_{\tau_0}(t,z, t', z') &\coloneqq  \chi_{\ge t_0}(t) \cdot \Pi_{\tau_0} \circ
                            \mathcal{F}^{-1}_{\tau \to t} \circ 
                            \chi_{\tau_0}(\tau) \cdot  \mathcal{F}_{t' \to
                            \tau} \circ \chi_{\ge t_0}(t')  \\
  &=  \frac{1}{2 \pi} \int_{- \infty}^\infty e^{i(t - t') \tau}\chi_{\ge t_0}(t)
    \cdot \chi_{\tau_0}(\tau) \cdot \Pi_{\tau_0}(z, z') \cdot
    \chi_{\ge t_0}(t')  \  d\tau
\end{align}
This $K_{\tau_0}$ will be used as a stand-in for projection onto $e^{i \tau_0 t} E_+(\lambda(\tau_0))$ near $t = + \infty$.
\begin{lemma}
  For $\delta > 0$ sufficiently small in the definition of
  $\chi_{\tau_0}$, 
  \begin{align}
    K_{\tau_0}  \in \Psitsc^{-\infty,0}, \text{ and }
    \ffsymbz(K_{\tau_0})(\tau) = \chi_{\tau_0}(\tau)
    \Pi_{\tau_0},\label{eq:projected guy}
  \end{align}
  and
\begin{equation}
\WFtsc'[P_{V_+}, K_{\tau_0}] = \varnothing \label{eq:commutator P K}.
\end{equation}
\end{lemma}
\begin{proof}
Proving that an operator lies in $\Psitsc$ can be done using the
double space characterization in Section 3 of \cite{V2000}, but we
argue directly using our work above.
First we note that, for any $\chi \in \CcI(\RR)$, the operator without the time cutoffs:
\begin{align}\label{eq:no time cutoff}
\tilde K_{\chi} \coloneqq \int_{- \infty}^\infty e^{i(t - t') \tau}
    \cdot \chi(\tau) \cdot \Pi_{\tau_0}(z, z') \, d\tau
\end{align}
is a $\tsc$-operator since $\Pi_{\tau_0}(z, z') \in \Psisc^{-\infty, - \infty}(\RR^n)$,
in fact $\tilde K_{\chi} = \Op_L(k)$ for $k = \chi(\tau) a(z, \zeta)$ with $a \in
\schwartz(\RR^n_z \times \RR^n_{\zeta})$. This shows directly that $\tilde K_{\chi}$ is in $\Psitsc^{0,0}$ and has
indicial operator as in \eqref{eq:projected guy}.  Then multiplying on
the left by $\chi_{\ge t_0}(t)$ gives an operator $\Op_L(\chi_{\ge
  t_0}(t) k)$ which remains in $\Psitsc^{0,0}$ since $k$ is rapidly
decaying in $z$.  The adjoint of that operator $\tilde K_{\chi}
\chi_{\ge t_0}$ is thus also in $\Psitsc^{0,0}$.  Then $K_{\tau_0}$
itself can be expressed as a composition $\chi_{\ge t_0} \tilde
K_{\chi_1} \circ \tilde
K_{\chi_2} \chi_{\ge t_0}$ where $\chi_1 = \chi_{\tau_0}$ and $\chi_2
\chi_{\tau_0} = \chi_{\tau_0}$, and is thus in $\Psitsc^{0,0}$.  The indicial operator statement
follows from the composition and adjunction properties of indicial
operators.

Finally, $[P_{V^+}, \tilde K_{\chi} ] = 0$ for any $\chi$.  Since
$[P_{V_+} , K_{\tau_0}]$ differs from $[P_{V^+}, \tilde
K_{\chi_{\tau_0}} ]$ only by terms with derivatives falling on
$\chi_{\ge t_0}(t)$, the commutator $[P_{V_+} , K_{\tau_0}]$ is
microsupported on
\[\chi_{\ge t_0}'(t) \subset \{ t \in [t_0 - 2, t_0] \},\]
but then $\chi_{\ge t_0}'(t)\Pi_{\tau_0}$ has symbol which is rapidly
decaying to the spacetime boundary, so \eqref{eq:commutator P K} holds.
\end{proof}

\bigskip

We will use the approximate projection $K_{\tau_0}$ mainly by
exploiting the following feature.
\begin{corollary}
  $\tau_0 \in \Ell_{\ff}(P_V + K_{\tau_0})$.
\end{corollary}
\begin{proof}
This follows directly from
$\ffsymbz(P_V + K_{\tau_0})(\tau_0) = \ffsymbz(P_V)(\tau_0) + \Pi_{\tau_0}$,
since $\Pi_{\tau_0}$ is
exactly projection onto the kernel of the self-adjoint operator $\ffsymbz(P_V)(\tau_0)$.
\end{proof}

To obtain estimates for $P_V u$ near $\tau_0 \in W^\perp$, we first
use the elliptic estimates for $P_V + K_{\tau_0}$ near $\tau_0 \in
W^\perp$, namely, that for any $Q, G \in \Psitsc^{0,0}$ with $\tau_0 \in
\Ell_{\ff}(Q)$ and $\WFtsc'(Q) \subset \Elltsc(G) \cap \Elltsc(P_V +
K_{\tau_0})$, for any $M, N \in \RR$ there is $C > 0$ such that
\begin{equation}
\begin{split}
    \norm{Q u}_{s + 2, \ell} &\le C \left( \norm{G (P_V + K_{\tau_0}) u}_{s, \ell} + \normres{u} \right) \\
    &\le C (\norm{GP_V u}_{s, \ell} + \norm{G K_{\tau_0} u}_{s, \ell} + \normres{u}).
\end{split}\label{eq:elliptic with K}
\end{equation}
What we will show below is that, off the spacetime weight $\ell = -1/2$,
for $Q$ with sufficiently small support around $\tau_0 \in W^\perp$,
that the $K_0 u$ term on the RHS can be bounded by $P_V u$.  This will
follow by applying ODE techniques to the easily verified formula
\begin{equation}
  \label{eq:ODE reduction}
P_{V_+} K_{\tau_0} u = (D_t^2 - \tau_0^2) K_{\tau_0} u  + R u 
\end{equation}
where $R \in \Psitsc^{-\infty, - \infty}$, to obtain estimates for
$K_{\tau_0} u$ in terms of $P_{V_+} K_{\tau_0} u$.  Then using
\eqref{eq:commutator P K} we will remove the $K_{\tau_0} u$ from
\eqref{eq:elliptic with K} entirely.  Then applying ODE methods to the
first term on the RHS, or using a positive commutator argument akin to
that of Section \ref{sec:3sc_radial}, we obtain the following lemma.
\begin{lemma}\label{thm:elliptic on kernel}
  Let $\ell \in \RR$.  Provided $\ell < -1/2$, for any $s_0, M, N
  \in \RR$ there is $C$ such that, if $K_{\tau_0}u \in \Sobsc^{s,
    \ell}$ and $ P_{V_+}K_{\tau_0}u \in \Sobsc^{s, \ell + 1}$, then for $Q, Q' \in
  \Psitsc^{0,0}$ with $\tau_0 \in \Ell_{\ff}(Q)$ and $\WFtsc'(Q) \subset
  \Elltsc(Q')$, we have $K_{\tau_0}u \in \Sobsc^{s_0,
    r}$ and $ P_{V_+}K_{\tau_0}u \in \Sobsc^{s_0, \ell + 1}$, and
  \begin{equation}
      \norm{Q  K_{\tau_0}u}_{s_0, \ell} \le
      C( \norm{Q' P_{V_+}K_{\tau_0}u}_{s,\ell + 1} +  \normres{u}).\label{eq:ODE estimate}
  \end{equation}
  The same is true if, in the definition of $K_{\tau_0}$, the
  projection $\Pi_{\tau_0}$ is replaced by orthogonal projection onto
  any subspace of $E_+(\lambda(\tau_0))$.

  If instead $\ell > -1/2$, for any $s_0, M, N \in \RR$ there is $C$ such that, if $K_{\tau_0}u \in \Sobsc^{s,
    \ell}$ and $ P_{V_+}K_{\tau_0}u \in \Sobsc^{s, \ell + 1}$, possibly after
  taking $\chi_{\tau_0}$ in $K_{\tau_0}$ with smaller support, we have
  $K_{\tau_0}u \in \Sobsc^{s_0, \ell}$ and $ P_{V_+}K_{\tau_0}u \in \Sobsc^{s_0, \ell + 1}$,
  and $\ell > \ell' > -1/2$,
  \begin{equation}
      \norm{Q  K_{\tau_0}u}_{s_0, \ell} \le C(
      \norm{Q' P_{V_+}K_{\tau_0}u}_{s,\ell + 1} +
      \norm{u}_{-N,\ell'}).\label{eq:ODE estimate2}
  \end{equation}
  The same is true if, in the definition of $K_{\tau_0}$, the
  projection $\Pi_{\tau_0}$ is replaced by orthogonal projection onto
  any subspace of $E_+(\lambda(\tau_0))$.
\end{lemma}

With the lemma, we can now conclude the proof.
\begin{proof}[Proof of Lemma \ref{thm:elliptic non elliptic}]
    From \eqref{eq:elliptic with K} and Lemma \ref{thm:elliptic on
    kernel}, treating $\ell <  -1/2$ first, we take $Q, G$ and $G'$ with $\WFtsc'(G)
    \subset \Elltsc(G')$ to obtain
    \begin{align*}
        \norm{Q u}_{s + 2, \ell}   &\lesssim  \norm{GP_V u}_{s, \ell} + \norm{G K_{\tau_0} u}_{s, \ell} + \normres{u} \\
        &\lesssim \norm{GP_V u}_{s, \ell} + \norm{G' P_{V_+} K_{\tau_0} u}_{s, \ell + 1} + \normres{u} \\
        &\lesssim \norm{GP_V u}_{s, \ell} + \norm{G' K_{\tau_0} P_{V_+} u}_{s, \ell + 1} + \normres{u} \\
        &\lesssim \norm{GP_V u}_{s, \ell} + \norm{G' K_{\tau_0} P_{V} u}_{s, \ell + 1} + \norm{G K_{\tau_0} (P_{V} - P_{V_+}) u}_{s, \ell + 1} + \normres{u}\,,
    \end{align*}
    where we used $[P_{V_+}, K_{\tau_0}] \in \Psitsc^{-\infty, - \infty}$ in the fourth line.

    The first two terms may both be bounded by $G'' P_V u$ for $G''$
    with $\WFtsc'(G) \cup \WFtsc'(G') \subset \Elltsc(G'')$ by choosing
    $\WFtsc'(K_{\tau_0})$ sufficiently small.
    Moreover, $P_V - P_{V_+} = - (V - V_+)$ and since $V$ is asymptotically static of order $0,1$,
    the third term is controlled by $\norm{u}_{-N, \ell - 1}$ for any $N$,
    and we obtain overall that there is
    $C > 0$ such that, if all the terms below are finite, then we have
    an estimate
    \begin{align*}
        \norm{Q u}_{s + 2, \ell} &\le  \norm{G'' P_V u}_{s, \ell} + \norm{G'' u}_{-N,\ell - 1} + \normres{u}\,.
    \end{align*}
    Iterating this estimate allows us to drop the $\ell -1$ term on the
    right.

    The proof for $\ell > -1/2$ is similar.
\end{proof}

We can now use Proposition \ref{thm:elliptic region TRUE} and the
arguments in Section \ref{thm:fredholm causal 3sc} to
prove Theorem \ref{thm:result with bound states}.
\begin{proof}[Proof of Theorem \ref{thm:result with bound states}]
  That $P_V$ acting in \eqref{eq:fredholm 3scat causal true} is
  Fredholm follows exactly as in the proof of Theorem
  \ref{thm:fredholm causal 3sc} using exactly the same methodology,
  adding in $B_{5, \pm}$ elliptic on
  $(-m + \epsilon', m - \epsilon') \subset W^\perp_\pm$ and the
  estimates in Proposition \ref{thm:elliptic region TRUE}.

  The injectivity follows from exactly the same energy estimate
  argument, and the surjectivity is that same energy estimate argument
  applied to the adjoint.

  The property \eqref{eq:forward solution} follows from the same argument as
  in the proof of Theorem~\ref{thm:fredholm causal 3sc}.
\end{proof}

\begin{proof}[Proof of Theorem \ref{thm:result for static}]
  For static $V = V(z)$, the Fredholm statement holds even in the
  presence bound states with energy less than $-m^2$.  The fact that
  there are no elements in the kernel can be concluded directly from
  seperation of variables since on the finite family of
  eigenfunctions $H_V$ the solutions are explicit and orthogonal to
  this family the energy argument holds.
\end{proof}

\section*{Index of Notation}

\begin{itemize}
\item $\schwartz$ is Schwartz functions, $\mathcal{S}'$ tempered distributions
\item $\CcI$ is smooth and compactly supported
\item $\CdI(M)$ for a manifold with corners $M$ is the space of
smooth functions which vanish to infinite order together with
all their derivatives at the boundary
\item $\spec(A)$ the spectrum of an operator $A$
\item $\lesssim$, used in an inequality when an unspecified
positive constant is needed on the right hand side

\item $H_V \coloneqq \Delta + m^2 + V$, the Hamiltonian, for $V$ possibly
time-dependent, page \pageref{eq:def of PV}
\item $P_V \coloneqq D_t^2 - H_V$, the
Klein-Gordon operator, page \pageref{eq:def of PV}
\item $\Sobsc^{s,\ellvar}(\RR^{n + 1}), \cY^{s,\ellvar}, \cX^{s,\ellvar}$ the
weighted $L^2$-based Sobolev spaces and the a priori spaces,
page \pageref{eq:first X spaces}

\item $P_{0}$, the free Klein-Gordon operator, page \pageref{def:P0}

\item $\Diffsc^m$ and $\Diffsc^{m,r}$ the scattering differential
operators, page \pageref{eq:sample-scattering-diff-op}
\item $X$, the radial compactification of $\RR^{n + 1}_{t, z}$, page \pageref{eq:def of X}

\item $\Tsc^*X, \Tsco^*X$ the scattering cotangent and its fiber
compactification, page \pageref{eq:scattering phase space}
\item  $  \rho_{\mathrm{base}}$, $\rho_{\mathrm{fib}}$, boundary
defining functions for $\Tsco^* X$, page \pageref{eq:bhss}
\item $\scprinsymb{m,r}(A)$ the scattering principal symbol, page
\pageref{eq:principal}
\item $\scfibsymb{m,r}(A), \scnormsymb{m,l}(A)$ the fiber and
normal components of the scattering principal symbol,
respectively, page \pageref{eq:scfibsymb}

\item $\Ssc^{m,r}(X)$, scattering symbols, page \pageref{eq:scattering symbols},
\item $\Psisc^{m,r}$ of scattering operators, page \pageref{eq:scattering symbols},
\item $C_{\sc}(X)$ the boundary of $\Tsco^* X$, page \pageref{eq:square}

\item $\WF'(A)$ the scattering operator wavefront set, Section  \ref{sec:elliptic-estimates-scattering}
\item $\Ell(A)$ the scattering elliptic set, Section \ref{sec:elliptic-estimates-scattering}
\item $\Char(A) = C_{\sc}(X) \setminus \Ell(A)$, the characteristic
set of $A$, Section~\ref{sec:elliptic-estimates-scattering}
\item $H_p$ and $\Hamscp$ the Hamilton vector field and its rescaling,
equations~\eqref{eq:hammy 1} and~\eqref{eq:scattering hammy def},
page~\pageref{eq:hammy 1}

\item $\Rad$ the radial set, and $\Rad^f_\pm, \Rad^p_\pm$
its four components, page \pageref{eq:radial set first
def} and below
\item $\ellvar$ variable order spacetime weight, Section \ref{sec:vari-weight-spac}
\item $\ellvar_\pm$ forward and backward weights, Section \ref{sec:fred for scat}

\item $V_+(z) = \lim_{t \to \infty} V(t)$, the limiting potential,
Section \ref{sec:V assump}, similarly for $V_{-}$

\item $\poles \subset X$, the ``poles'', $\poles = \NP \cup \SP$,
page \pageref{eq:def of C}
\item $[X;C]$ the blow up  of the poles in $X$, page \pageref{eq:X blow C def}
\item $\beta_C$ the blow down map, page \pageref{eq:X blow C def}
\item $\Difftsc^{m,r}$ the $\tsc$-differential operators, page \pageref{eq:1}

\item $\Ttsc^* [X;C], \Ttsco^*[X;C]$ the $\tsc$-cotangent
bundle and compactification, page \pageref{eq:three scat cotangent},
\item $\rho_{\ff},\rho_{\mf},\rho_{\fib}\,$, boundary
defining functions for $\Ttsco^*X$, page \pageref{eq:bdfs}
\item  $ W^{\perp}$, the lines of $\tau$ over the poles, page
\pageref{eq:piWperp}

\item $\pi \colon \Ttsc^*_C X \longrightarrow W^\perp$ the projection
  on $W^\perp$, page \pageref{eq:piWperp}

 \item $\Radtsc$ the $\tsc$-radial set, page \pageref{eq:tsc radial set}
\item $\fibeq$ the fiber equator, page \pageref{eq:fibeq def}

\item $\Stsc^{m,r}(X)$,
$\tsc$-symbols, page \pageref{eq:symbol space}
\item $\Psitsc^{m,r}$ the
$\tsc$-operators page \pageref{eq:3}
\item $\ffsymbz(A) = \hat{A}_{\ff}$ and $\ffsymb{\ell}(A)$ the indicial operators, page \pageref{eq:general ff boundary symbol}
\item $\aff$ the weighted front face restriction,
page \pageref{eq:ff restrict}

\item $\prinsymb{m,r}(A)$ the principal symbol,
equation~\eqref{eq:actual symbol},
page~\pageref{eq:actual symbol}
\item $\fibsymb{m,r}(A)$ the fiber symbol,
$\mfsymb{m,r}(A)$ the ``main face'' symbol
(restriction to the spacetime boundary), page
\pageref{eq:actual symbol}

\item $\Norop^{m,r,k}$ the two-sided semiclassical scattering
operators, page \pageref{eq:norop}

\item $UH_+$ the upper half of fiber infinity over $\ff$,
page \pageref{eq:upper half sphere}
\item $\Sobscl^{s,\ell}$ the semiclassical Sobolev spaces
of order $s, \ell$, page \pageref{eq:scl sob spaces}
\item $\Wperpo$ the compactification of $W^\perp$, page \pageref{eq:Wperpo def}

\item  $\Tdot^*X, \Tdoto^*
X$ is the compressed cotangent bundle and its
compactification, page \pageref{eq:compressed cotan}
\item $\pi^\perp$ the projection to $\Tdot^* X$, page \pageref{eq:pi perp}
\item $\Ctsc$, $\Ctscd$, equations~\eqref{eq:square
analogue} and~\eqref{eq:square analogue dot},
page \pageref{eq:square analogue}
\item $\gamma_{\tsc}$, page \pageref{eq:gamma tsc easy}

\item $\WFtsc'(A)$ the $\tsc$-operator wavefront set,
    Definition \ref{def:WFtsc}, page~\pageref{def:WFtsc}

\item $\Elltsc(A)$ the $\tsc$-elliptic set, Definition
    \ref{def:elliptic set}, page~\pageref{def:elliptic set}
\item $\Chartsc(A)$ the $\tsc$-characteristic set, page~\pageref{eq:def_Chartsc}

\item $G_{\psi,0}$, functional localizer, equation~\eqref{eq:def_Gpsi0},
    page~\pageref{eq:def_Gpsi0}
\item $G_\psi$, functional localizer, Definition~\eqref{def:Gpsi},
    page~\pageref{def:Gpsi}

\item $\pi_{X, \tau}$,
    page \pageref{eq:pi X tau}
\item $\Sigma$, equation~\eqref{eq:Sigma},
    page~\pageref{eq:Sigma}
\item $\chi_0, \chi_1$ special cutoff functions,
    page \pageref{eq:chi 0}

\end{itemize}

\newpage

\bibliographystyle{plain}
\bibliography{KGPCbib}

\end{document}